\documentclass[english,10pt]{smfbook}

\usepackage[english]{babel}
\usepackage[utf8]{inputenc}
\usepackage{amsmath,amssymb,amsthm,stmaryrd,appendix,imakeidx}
\usepackage{eqparbox}
\usepackage{graphicx,xcolor}
\usepackage{etoolbox}

\usepackage{enumitem}
\setlist[enumerate]{label=\mbox{(\textit{\roman*}\hspace{.08em})},font=\rm,itemsep=.25em}
\setlist[itemize]{itemsep=.25em}

\newcommand{\itemi}{\mbox{(\textit{i}\hspace{.08em})}}
\newcommand{\itemii}{\mbox{(\textit{ii}\hspace{.08em})}}
\newcommand{\itemiii}{\mbox{(\textit{iii}\hspace{.08em})}}

\newcommand{\hair}{\ifmmode\mskip1.5mu\else\kern0.08em\fi}

\usepackage{tikz}
\usepackage{mathtools}
\usetikzlibrary{backgrounds,matrix,arrows,decorations.markings,decorations.pathmorphing,shapes}
\tikzset{>=stealth}
\renewcommand{\to}{\mathrel{\tikz[baseline]\draw[ ->,line width=.4pt] (0ex,0.65ex) -- (3ex,0.65ex);}}
\renewcommand{\mapsto}{\mathrel{\tikz[baseline]\draw[|->,line width=.4pt] (0pt,0.65ex) -- (3ex,0.65ex);}}
\renewcommand{\rightarrow}{\mathrel{\tikz[baseline]\draw[->,line width=.4pt] (0ex,0.65ex) -- (3ex,0.65ex);}}

\renewcommand{\leftarrow}{\mathrel{\tikz[baseline]\draw[<-,line width=.4pt] (0ex,0.65ex) -- (3ex,0.65ex);}}
\renewcommand{\rightrightarrows}{\mathrel{\tikz[baseline]{\draw[->,line width=.4pt] (0ex,0.25ex) -- (3ex,0.25ex); \draw[->] (0ex,1.05ex) -- (3ex,1.05ex);}}}

\renewcommand{\hookrightarrow}{\mathrel{\tikz[baseline]\draw[right hook->,line width=.4pt] (0ex,0.65ex) -- (3ex,0.65ex);}}
\renewcommand{\leftrightarrow}{\mathrel{\tikz[baseline]\draw[<->,line width=.4pt] (0ex,0.65ex) -- (3.5ex,0.65ex);}}

\newcommand{\toarg}[1]{\mathrel{\tikz[baseline]\path[->,line width=.4pt] (0ex,0.65ex) edge node[above=-.4ex, overlay, font=\scriptsize] {$#1$} (3.5ex,.65ex);}}

\newcommand{\toarglong}[1]{\mathrel{\tikz[baseline]\path[->,line width=.4pt] (0ex,0.65ex) edge node[above=-.4ex, overlay, font=\scriptsize] {$#1$} (5ex,.65ex);}}

\newcommand{\tosim}{\mathrel{\tikz[baseline] \path[->,line width=.4pt] (0ex,0.65ex) edge node[above=-.9ex, overlay, font=\normalsize, pos=.45] {$\sim$} (3.5ex,.65ex);}}

\newcommand{\MyTo}{\mathbin{\tikz[baseline] \draw[<-,line width=.3pt] (0ex,0.4ex) -- (2ex,0.4ex);}}

\renewcommand{\varprojlim}{%
	\mathop{%
	\mathchoice
		{\mathrm{lim}\mkern-24mu\tikz[baseline]\draw[<-,line width=.4pt](-2.9ex,-.55ex)--(-.1ex,-.55ex);\mkern1mu}
		{\mathrm{lim}\mathllap{\rule[-1.2pt]{0pt}{1.2pt}\tikz[baseline,overlay] \draw[<-,line width=.4pt] (-2.9ex,-.55ex) -- (-.1ex,-.55ex);}\mkern1mu}
		{\mathrm{lim}\tikz[baseline=-.5ex] \draw[<-,line width=.3pt] (0ex,.3ex) -- (2ex,.3ex);}
		{\lim_{\MyTo{2.0ex}}
}}}

\usepackage{scalerel}

\newcommand{\hdot}{{\scaleobj{.95}{\bullet}}}
\newcommand{\ldot}{{\scaleobj{.95}{\bullet}}}

\newcommand{\llrr}[1]{%
\llbracket #1 \rrbracket}

\newcommand{\overbar}[1]{{\mkern1.5mu\overline{\mkern-1.5mu #1\mkern-1.5mu}\mkern1.5mu}}

\newcommand{\compl}{{\widehat{\;\;}}}

\newcommand{\hatotimes}{\mathbin{\widehat{\otimes}}}

\newcommand{\blank}{\mkern.5mu {-} \mkern.5mu}

\renewcommand{\theta}{\vartheta}
\renewcommand{\phi}{\varphi}
\renewcommand{\emptyset}{\varnothing}

\numberwithin{equation}{chapter}
\numberwithin{figure}{chapter}
\newtheorem{theorem}[equation]{Theorem}
\newtheorem*{theorem*}{Theorem}
\newtheorem{proposition}[equation]{Proposition}

\newtheorem{question}[equation]{Question}
\newtheorem{lemma}[equation]{Lemma}
\newtheorem*{lemma*}{Lemma}
\newtheorem{corollary}[equation]{Corollary}
\newtheorem*{corollary*}{Corollary}
\theoremstyle{definition}
\newtheorem{definition}[equation]{Definition}
\newtheorem{example}[equation]{Example}

\newtheorem{notation}[equation]{Notation}
\newtheorem{heuristic}[equation]{Heuristic}
\theoremstyle{remark}
\newtheorem{remark}[equation]{Remark}
\newtheorem{remarks}[equation]{Remarks}
\newtheoremstyle{none}
{.5\linespacing plus .5\linespacing}        
{.5\linespacing plus .5\linespacing}              
{}              
{}    
{}              
{}             
{0pt}         
{}  
\theoremstyle{none}
\newtheorem*{none*}{}
\patchcmd{\thmhead}{(#3)}{#3}{}{} 

\renewcommand{\Bar}{\operatorname{Bar}}
\DeclareMathOperator{\coh}{coh}

\DeclareMathOperator{\Def}{Def}

\newcommand{\e}{\mathrm{e}}

\DeclareMathOperator{\GL}{GL}
\DeclareMathOperator{\gldim}{gldim}
\DeclareMathOperator{\gr}{gr}

\DeclareMathOperator{\HH}{HH}
\DeclareMathOperator{\Hom}{Hom}
\newcommand{\Irr}{\mathrm{Irr}}
\DeclareMathOperator{\MC}{MC}

\DeclareMathOperator{\Mod}{Mod}
\newcommand{\id}{\mathrm{id}}
\DeclareMathOperator{\im}{im}

\newcommand{\Sym}{\mathrm{S}}
\newcommand{\T}{\mathrm{T}}
\DeclareMathOperator{\tip}{tip}

\newcommand{\U}{\mathrm{U}}

\makeindex

\begin{document}

\frontmatter


\title{Deformations of path algebras of quivers with relations} 

\begin{abstract}
Let $A = \Bbbk Q / I$ be the path algebra of any finite quiver $Q$ modulo any two-sided ideal $I$ of relations and let $R$ be any reduction system satisfying the diamond condition for $I$. We introduce an intrinsic notion of deformation of reduction systems and show that there is an equivalence of deformation problems between deformations of the associative algebra $A$ and deformations of the reduction system $R$, the latter being controlled by a natural, explicit L$_\infty$ algebra. It follows in particular that any formal deformation of the associative multiplication on $A$ can, up to gauge equivalence, be given by a combinatorially defined star product, and the approach via reduction systems can be used to give a concrete and complete description of the deformation theory of $A$. 

For the polynomial algebra in a finite number of variables, this combinatorial star product can be described via bidifferential operators associated to graphs, which we compare to the graphs appearing in Kontsevich's universal quantization formula.

Using the notion of admissible orders on the set of paths of the quiver $Q$, we give criteria for the existence of algebraizations of formal deformations, which we also interpret geometrically as algebraic varieties of reduction systems. In this context the Maurer--Cartan equation of the L$_\infty$ algebra can be viewed as a generalization of the Braverman--Gaitsgory criterion for Poincaré--Birkhoff--Witt deformations of Koszul algebras.
\end{abstract}




\author{Severin Barmeier}
\email{s.barmeier@gmail.com}
\address{Max-Planck-Institut für Mathematik, Vivatsgasse 7, 53111 Bonn, Germany}
\address{Albert-Ludwigs-Universität Freiburg, Mathematisches Institut, Ernst-Zermelo-Str.~1, 79104 Freiburg im Breisgau, Germany}
\address{Universität zu Köln, Weyertal 86-90, 50931 Köln, Germany}

\author{Zhengfang Wang}
\email{wangzg@mathematik.uni-stuttgart.de}
\address{Max-Planck-Institut für Mathematik, Vivatsgasse 7, 53111 Bonn, Germany}
\address{Universität Stuttgart, Institut für Algebra und Zahlentheorie, Pfaffenwaldring 57, 70569 Stuttgart, Germany}

\thanks{Both authors were supported by the Max Planck Institute for Mathematics Bonn and the Hausdorff Research Institute for Mathematics Bonn funded by the Deutsche Forschungsgemeinschaft (DFG, German Research Foundation) under Germany's Excellence Strategy -- EXC-2047/1 -- 390685813. The first named author was also supported by the DFG Research Training Group GK1821 ``Cohomological Methods in Geometry'' at the University of Freiburg. The second named author was also supported by a Humboldt Research Fellowship from the Alexander von Humboldt Foundation and by a National Natural Science Foundation of China (NSFC) grant with number 11871071.}

\maketitle

\tableofcontents

\listoffigures

\mainmatter

\chapter{Introduction}
\section{Background and motivation}

The deformation theory of associative algebras is the systematic study of variations of the associative multiplication along a single parameter or multiple parameters, first studied in a series of articles by Gerstenhaber \cite{gerstenhaber,gerstenhaber2,gerstenhaber3,gerstenhaber4}. 

A cornerstone of the theory is the Gerstenhaber bracket\index{Gerstenhaber bracket}, which endows the (shifted) Hochschild cochain complex\index{Hochschild!cochain complex} $\Hom_{A^\e} (\Bar_{\ldot+1}, A)$ with the structure of a DG Lie algebra\index{DG Lie algebra}\index{algebra!DG Lie} --- which is precisely the structure used in the Maurer--Cartan formalism of deformation theory via DG Lie and L$_\infty$ algebras. It is well known that first-order deformations correspond to Hochschild $2$-cocycles and two such first-order deformations are equivalent precisely when the $2$-cocycles differ by a $2$-coboundary so that first-order deformations up to equivalence are parametrized by the second Hochschild cohomology $\HH^2 (A, A)$ and obstructions to extending such deformations to higher orders lie in $\HH^3 (A, A)$.

This classical description is both powerful and elegant. However, the classical theory works with multilinear maps and with the ``higher structure'' on  the complex $\Hom_{A^\e} (\Bar_{\ldot + 1}, A)$, where $\Bar_\ldot$ is the bar resolution\index{resolution!bar} of $A$
\[
\dotsb \to A \otimes_{\Bbbk} A^{\otimes_{\Bbbk} n} \otimes_{\Bbbk} A \to \dotsb \to A \otimes_{\Bbbk} A \otimes_{\Bbbk} A \to A \otimes_{\Bbbk} A  \to 0.
\]
The size of the bar resolution\index{resolution!bar} is in fact often the main obstacle for obtaining a concrete description of the deformations of $A$. Indeed, such a concrete description is difficult to come by even in the cases of, say, finite-dimensional algebras \cite{gabriel} or commutative polynomial algebras \cite{kontsevich1}.

From a representation-theoretic point of view deformations of $A$ can be identified with deformations of the module category $\Mod (A)$ as Abelian category \cite{lowenvandenbergh1,lowenvandenbergh2} and the effect of deformations of $A$ on $\Mod (A)$ can be studied both in representation theory and (noncommutative) algebraic geometry, where $\Mod (A)$ can be viewed as the category of quasi-coherent sheaves on a noncommutative scheme. More generally, deformations of $A$ can be viewed as deformations of the derived category $\mathrm D (A)$ \cite{keller1} which can be naturally related to many different algebraic or geometric objects via derived equivalence, e.g.\ to algebraic varieties via the derived category of quasi-coherent sheaves, or to symplectic manifolds via the Fukaya category. It is a compelling and difficult problem to understand concretely how algebraic deformations of $A$ change the representation-theoretic or geometric properties of these associated objects.

In this book we develop a general combinatorial method to study the deformation theory of path algebras of quivers with relations. More precisely, when $A$ can be written as the path algebra $\Bbbk Q / I$ of a finite quiver $Q$ modulo a two-sided ideal $I$ of relations, we study the deformations of $A$ via the combinatorics of a so-called reduction system, a notion which was introduced by Bergman \cite{bergman} to state and prove his Diamond Lemma (see \S\ref{subsection:reductionsystems}). Our method relies on the choice of a reduction system, which always exists (Proposition \ref{proposition:existencereduction}) and which can be computed algorithmically --- often by hand --- or obtained from a noncommutative Gröbner basis, if available. This combinatorial approach to deformations of associative algebras turns out to be of great practical value, allowing one to easily compute particular classes of examples, as well as give criteria for the existence of algebraizations of formal deformations.

As every finitely generated associative algebra can be written as $\Bbbk Q / I$ where $Q$ has a single vertex with a finite number of loops at this vertex, this method can be used for instance to study deformations of
\begin{enumerate}
\item \label{list1} commutative algebras such as the polynomial algebra $\Bbbk [x_1, \dotsc, x_d]$ whose deformation theory is related to quantizations of Poisson structures
\item \label{list2} graded associative algebras such as $N$-Koszul algebras.
\end{enumerate}
Associative algebras of the form $\Bbbk Q/I$, where the quiver $Q$ has multiple vertices, also naturally appear in various guises for example in representation theory, commutative and noncommutative algebraic geometry or symplectic geometry, for example as
\begin{enumerate}[resume]
\item \label{list3} finite-dimensional associative algebras
\item \label{list4} endomorphism algebras of tilting bundles on algebraic varieties
\item \label{list5} diagram algebras associated to a presheaf of algebras 
\item \label{list6} noncommutative resolutions of singularities
\item \label{list7} derived endomorphism algebras of (formal) generators of Fukaya categories.
\end{enumerate}
We study examples of type \ref{list1}--\ref{list3} in Chapters \ref{section:pbw} and \ref{section:relationtoquantization} of the present book and \ref{list4}--\ref{list6}, which describe the deformation theory of the Abelian category of (quasi)coherent sheaves, in \cite{barmeierwang}. In \ref{list7} the derived endomorphism algebras of formal generators are naturally graded and generalizations of this method for A$_\infty$ deformations of graded algebras, with applications to deformations of Fukaya and Fukaya--Seidel categories, are studied in \cite{barmeierwang2,barmeierschrollwang}.

\section{Main results}

Let $\Bbbk$ be any field of characteristic $0$. Our main theoretical result is the following.

\begin{theorem}[(Theorem \ref{theorem:deformationreduction})]
\label{theorem:main}
Given any finite quiver $Q$ and any two-sided ideal of relations $I \subset \Bbbk Q$, let $A = \Bbbk Q / I$ be the quotient algebra and let $R$ be any reduction system satisfying \textup{($\diamond$)} for $I$.

There is an equivalence of formal deformation problems between
\begin{enumerate}
\item \label{main1} deformations of the associative algebra $A$
\item \label{main2} deformations of the reduction system $R$
\item \label{main3} deformations of the relations $I$.
\end{enumerate}
\end{theorem}

This equivalence can be understood as an isomorphism of the corresponding deformation functors. (See \S\ref{subsection:summary} for a description of these deformation problems.) For the second deformation problem, we introduce a new and intrinsic notion of equivalence of reduction systems (see Definitions \ref{definition:equivalencereduction} and \ref{definition:equivalenceformalreductionsystem}).

From the equivalence of \ref{main1} and \ref{main2} we obtain a bijection between
\begin{equation}
\begin{array}{c}
\text{formal deformations of $A$} \\ \text{up to equivalence}
\end{array}
\,\leftrightarrow\,
\begin{array}{c}
\text{formal deformations of $R$} \\ \text{up to equivalence.}
\end{array}
\end{equation}
The right-hand side is often easier to determine explicitly (see e.g. \S\ref{subsubsection:formalbrauertree}). The point of view of reduction systems can thus also be used to simplify the classification of deformations up to gauge equivalence.

Let us give a rough idea on the proof of  the equivalence between \ref{main1} and \ref{main2} in Theorem \ref{theorem:main}. For any choice of reduction system $R$ satisfying the diamond condition ($\diamond$) for the ideal $I$, Chouhy--Solotar \cite{chouhysolotar} proved the existence of an $A$-bimodule resolution $P_\ldot$ of $A$, generalizing Bardzell's resolution\index{resolution!Bardzell}\index{resolution!Chouhy--Solotar} for monomial algebras \cite{bardzell1,bardzell2}. We give a recursive formula for the differential of $P_\ldot$ (see Proposition \ref{proposition:resolution}) and construct an explicit homotopy deformation retract between the bar resolution\index{resolution!bar} $\Bar_\ldot$ and $P_\ldot$ (see Chapter \ref{section:homotopydeformationretract}), giving the following theorem.

\begin{theorem}[(Theorems \ref{theorem:linfinitytransfer} and \ref{theorem:equivalenceformal})]
\label{theorem:mainquasi}
\index{L$_\infty$!algebra}\index{algebra!L$_\infty$}
There exists an L$_\infty[1]$ algebra $\mathbf p (Q, R)$
with underlying cochain complex $P^{\hdot+2} = \Hom_{A^\e} (P_{\ldot+2}, A)$ and an L$_\infty[1]$ quasi-isomorphism\index{L$_\infty$!quasi-isomorphism} between $\mathbf p(Q, R)$ and $\Hom_{A^\e} (\Bar_{\ldot+2}, A)$ with its DG Lie$[1]$\index{DG Lie algebra}\index{algebra!DG Lie} algebra structure.

This L$_\infty[1]$ \index{L$_\infty$!algebra}\index{algebra!L$_\infty$} algebra naturally controls the deformation theory of $R$ and in light of the L$_\infty[1]$ quasi-isomorphism also the deformation theory of $A$.
\end{theorem}

\begin{remark}
Here, and indeed throughout most of the book, we work with the notions of L$_\infty[1]$\index{L$_\infty$!algebra} algebras and DG Lie$[1]$ algebras\index{DG Lie algebra}\index{algebra!DG Lie}\index{DG Lie algebra!DG Lie$[1]$|textbf}. We have been careful in working with a consistent choice of sign conventions. The notions of L$_\infty[1]$ and DG Lie$[1]$ algebras --- which are simply shifted versions of the more classical notions of L$_\infty$ algebras and DG Lie algebras (see Remark \ref{ordinarydgalgebra}) --- allow us to avoid otherwise rather complicated signs that would appear in the various formulas involving higher brackets. Except when taking a close look at signs, we invite the reader simply to ignore the $[1]$.
\end{remark}

In a general L$_\infty$ or L$_\infty [1]$ algebra\index{L$_\infty$!algebra}\index{algebra!L$_\infty$}, the notion of equivalence between Maurer--Cartan elements\index{Maurer--Cartan!elements} is that of homotopy equivalence, which can be rather cumbersome in practice. Theorem \ref{theorem:mainquasi} implies in particular that we may instead work with the intrinsic and practicable notion of equivalence between reduction systems.

To any complete local Noetherian $\Bbbk$-algebra $(B, \mathfrak m)$, e.g.\ $B = \Bbbk \llrr{t}$ and $\mathfrak m = (t)$, and any element $\widetilde \phi \in P^2 \hatotimes \mathfrak m$ we associate a certain combinatorially defined operation $\star \colon A \otimes A \to A \hatotimes B$, which we call {\it combinatorial star product}\index{star product!combinatorial}\index{combinatorial star product}. Roughly speaking, this operation can be described by performing rightmost reductions with respect to a new, formal reduction system associated to $\widetilde \phi$ (see Definition \ref{definition:star_2}). 

We show that $\star$ coincides with the bilinear operation obtained from transferring $\widetilde \phi$ via the L$_\infty[1]$ quasi-isomorphism of Theorem \ref{theorem:mainquasi} (see Theorem \ref{theorem:interpretation-star}), which has the following two important consequences.  On the one hand, $\star$ can be used to give a combinatorial criterion for the Maurer--Cartan equation\index{Maurer--Cartan!equation} of $\mathbf p (Q, R) \hatotimes \mathfrak m$ (Theorem \ref{theorem:higher-brackets}). On the other hand, when $\widetilde \phi$ is a Maurer--Cartan element of $\mathbf p (Q, R) \hatotimes \mathfrak m$, then $\star$ gives an explicit formula for the corresponding formal deformation of the associative multiplication on $A$. Indeed, up to gauge equivalence, any formal deformation of $A$ over $(B, \mathfrak m)$ is of the form $(A \hatotimes B, \star)$ for some Maurer--Cartan element\index{Maurer--Cartan!elements} $\widetilde \phi$ (Corollary \ref{theorem-summary}).

When restricting to deformations over $B = \Bbbk [\epsilon] / (\epsilon^2)$, the combinatorial star product provides a simple combinatorial method for computing $\HH^2 (A, A)$ --- as first-order deformations of the reduction system $R$ modulo equivalence (see \S\ref{subsection:firstorder}).

For proving the equivalence between \ref{main1} and \ref{main3} of Theorem \ref{theorem:main} we use the same idea underlying the construction of the Gutt star product\index{Gutt star product}\index{star product!Gutt} (see Definition \ref{definition:gutt}). For any Maurer--Cartan element $\widetilde \phi$ of $\mathbf p (Q, R) \hatotimes \mathfrak m$, the associated Gutt star product coincides with the combinatorial star product\index{star product!combinatorial}\index{combinatorial star product} (Theorem \ref{theorem:equivalenceformal}). In particular, any formal deformation of $A = \Bbbk Q / I$ can be obtained from a formal deformation of the ideal $I$ and the point of view of reduction systems provides a systematic method for constructing such a formal deformation. \\


Using the equivalences in Theorem \ref{theorem:main} and the explicit nature of the combinatorial star product\index{star product!combinatorial}\index{combinatorial star product} one obtains a surprisingly workable description of the deformation theory of $A$.

Let us also mention two further aspects of the deformation theory of reduction system we address in this book.

\subsection{Actual deformations}

Our main result stated in Theorem \ref{theorem:main} is about {\it formal} deformations. In Chapter \ref{section:nonformal} we consider certain finiteness conditions under which deformations of reduction systems can be understood in the non-formal setting. We also consider finiteness conditions which are weaker than reduction-uniqueness, which naturally give rise to a generalization of the Diamond Lemma (see Theorem \ref{theorem:nonformal}).

It is often desirable to obtain ``actual'' deformations of an algebra, say, by finding an algebraization of a formal deformation and evaluating the deformation parameters to a constant --- if possible. In this case, the original algebra and its ``actual'' deformation have the {\it same} basis, and may thus be compared on equal footing. In Chapter \ref{section:pbw} we study this problem of algebraization from a geometric point of view: by considering certain subspaces of $P^2$ containing elements satisfying certain degree conditions.

In particular, this allows us to define a variety of reduction systems, whose equations are given by the Maurer--Cartan equation\index{Maurer--Cartan!equation} of $\mathbf p (Q, R)$ (Theorem \ref{theorem:variety}). These varieties were studied by Green--Hille--Schroll \cite{greenhilleschroll} and we show that the natural notion of equivalence of reduction systems gives rise to a natural groupoid action on the variety, where two reduction systems lie in the same orbit if and only if their associated algebras are isomorphic. These varieties can provide algebraizations of formal deformations, even for possibly infinite-dimensional algebras (see \S\ref{subsubsection:geometricinterpretation}).

Moreover, we show how in this non-formal context the Maurer--Cartan equation\index{Maurer--Cartan!equation} of $\mathbf p (Q, R)$ can be viewed as a natural generalization of the Braverman--Gaitsgory criterion for PBW deformations\index{PBW deformation} of Koszul algebras, to the setting of non-Koszul algebras and non-PBW deformations (see Proposition \ref{proposition:pbw}). 

\subsection{A graphical description of the combinatorial star product}

In Chapter \ref{section:relationtoquantization} we show that for the polynomial algebra $A = \Bbbk [x_1, \dotsc, x_d]$ the combinatorial star product $\star$ admits the following graphical description.

\begin{theorem}[(Theorem \ref{theorem:loopless} and Proposition \ref{proposition:graphs})]
\label{theorem:main3}
Given any degree zero element $\widetilde \varphi \in \mathbf p (Q, R) \hatotimes \mathfrak m$, the associated combinatorial star product\index{star product!combinatorial!graphical calculus}\index{combinatorial star product!graphical calculus} $\star$ can be given as
\begin{flalign}
\label{combinatorialstarintro}
&& f \star g &= \sum_{k \geq 0} \sum_{\Pi \in \mathfrak G_{k,2}^{\mathrm C}} C_\Pi (f, g) && \mathllap{f, g \in A}
\end{flalign}
where $\mathfrak G_{k,2}^{\mathrm C}$ is a set of graphs in bijection to the set of Kontsevich graphs without cycles and with an ordering of the incoming edges at each vertex, and each $C_\Pi$ is a bidifferential operator associated to a graph $\Pi$.
\end{theorem}

Here $Q$ is the quiver with a single vertex and $d$ loops $x_1, \dotsc, x_d$ and $R = \{ (x_j x_i, x_i x_j) \}_{1 \leq i < j \leq d}$ is a reduction system for the ideal of commutativity relations (see Example \ref{example:polynomialreduction}).

Phrased in the context of deformation quantization, and working with deformations over $\Bbbk \llrr{\hbar}$, we have the following result.\index{deformation quantization}

\begin{theorem}[(Theorem \ref{theorem:theoretical} and Proposition \ref{proposition:quantizations})]
Any Poisson structure on $\mathbb A^d$ can be quantized using the combinatorial star product.\index{Poisson structure}\index{star product!combinatorial}\index{combinatorial star product}

Given any Maurer--Cartan element $\widetilde \phi = \widetilde \phi_1 \hbar + \widetilde \phi_2 \hbar^2 + \dotsb$ of $\mathbf p (Q, R) \hatotimes (\hbar)$, the combinatorial star product gives an explicit formula for the quantization of the Poisson structure corresponding to $\widetilde \phi_1$.\index{deformation quantization}
\end{theorem}

It is interesting that the purely combinatorial approach we take in this book recovers the graphs without cycles used in Kontsevich's universal quantization formula \eqref{kontsevichstar}, which can be motivated from Feynman diagrams \cite{cattaneofelder}. However, no extra weights enter the formula \eqref{combinatorialstarintro}, which may prove useful in obtaining a quantization formula defined over $\mathbb Q$ (for examples see \S\ref{subsection:combinatorialquantization}). It has also proved useful in studying convergence and continuity properties of star products needed to produce strict quantizations \cite{barmeierschmitt}.\index{deformation quantization!strict}

\section{Further applications}

The methods developed in this book can also be applied to study deformations arising in algebraic geometry, singularity theory, representation theory and symplectic geometry. In \cite{barmeierwang} we show how deformations of path algebras of quivers with relations can be used to study deformations of the Abelian category of coherent sheaves on any separated Noetherian scheme in the sense of Lowen--Van den Bergh \cite{lowenvandenbergh1,lowenvandenbergh2}. Deformations of $\coh (X)$ for some separated Noetherian scheme $X$ are equivalent to deformations of the diagram algebra $\mathcal O_X \vert_{\mathfrak U}!$ obtained from restricting the structure sheaf of $X$ to an affine open cover $\mathfrak U$ closed under intersections. This diagram algebra can naturally be written as the path algebra of a quiver with relations. We further study deformations of $\coh (X)$ for varieties with a tilting bundle, whose endomorphism algebra is also of the form $\Bbbk Q / I$ which also give rise to noncommutative deformations of singularities.

This method also allows us to study the effect of associative, or more generally, A$_\infty$ deformations of graded gentle algebras on their representation theory and on the symplectic geometry of the associated surface models \cite{barmeierschrollwang}. In \cite{barmeierwang2} we also apply our results in this book to settle Stroppel's Conjecture from her ICM 2010 address \cite{stroppel2} on the Hochschild cohomology and A$_\infty$ deformations of extended Khovanov arc algebras appearing in geometric representation theory, link homology and symplectic geometry \cite{stroppel}, which also describe algebraic deformations of Fukaya--Seidel categories associated to certain Hilbert schemes on surfaces.

\section{Structure of the book}

In Chapter \ref{section:quivers} we recall some basic notions on quivers and their path algebras. 
In Chapter \ref{section:preliminaries} we recall the notion of a reduction system and Bergman's Diamond Lemma. We introduce a new intrinsic notion of equivalence between reduction systems (Definition \ref{definition:equivalencereduction}). We also recall the notion of a noncommutative Gröbner basis which naturally gives a reduction system. We also study the combinatorics of reductions. In Chapter \ref{section:resolution} we recall the projective resolution $P_\ldot$ constructed by Chouhy--Solotar and give explicit recursive formulae for the differential and the homotopy in $P_\ldot$ (see Proposition \ref{proposition:resolution}). 
In Chapter \ref{section:homotopydeformationretract} we construct an explicit homotopy deformation retract between the projective resolution $P_\ldot$ and the bar resolution $\Bar_\ldot$. We describe explicitly the maps in low degrees, which will be used to describe the deformation theory of $A$. 
In Chapter \ref{section:deformationtheory} we give a brief review of deformation theory via DG Lie and L$_\infty [1]$ algebras and we recall the classical theory of ``actual'' and formal deformations of associative algebras.
In Chapter \ref{section:deformations-of-path-algebras} we prove our main result on the equivalences between the formal deformations of reduction systems, ideals and path algebras of quivers with relations. In \S\ref{subsection:firstorder} we also give a combinatorial description for computing $\HH^2 (A, A)$ by restricting to deformations over $\Bbbk [\epsilon] / (\epsilon^2)$. In Chapter \ref{section:nonformal} we study deformations in the non-formal setting and give a deformation-theoretic interpretation for Bergman's Diamond Lemma. In Chapter \ref{section:pbw} we study algebraic varieties of reduction systems, whose points may be viewed as ``actual'' deformations of $A$, by introducing certain degree conditions. We show that the Maurer--Cartan equation of $\mathbf p (Q, R)$ can be viewed as a generalization of the classical Braverman--Gaitsgory criterion for PBW deformations of $N$-Koszul algebras (see Proposition \ref{proposition:pbw}) and give several examples.
In Chapter \ref{section:relationtoquantization} we give a graphical description of the combinatorial star product for the polynomial algebra. The combinatorial star product can be used to give explicit formulae for deformation quantizations of algebraic Poisson structures\index{Poisson structure} on $\mathbb A^d$.\index{deformation quantization}

\chapter{Quivers, path algebras and relations}
\label{section:quivers}

A {\it quiver} $Q$ consists of a set $Q_0$ of vertices and a set $Q_1$ of arrows together with source and target maps $\mathrm s, \mathrm t \colon Q_1 \to Q_0$ assigning to each arrow $x \in Q_1$ its source and target vertices $\mathrm s (x)$ and $\mathrm t (x)$, respectively. A quiver $Q$ is called {\it finite} if $Q_0$ and $Q_1$ are finite sets. {\it All quivers in this book are finite quivers.}

To each vertex $i \in Q_0$ we associate {\it a path of length $0$} denoted by $e_i$ with $\mathrm s (e_i) = i = \mathrm t (e_i)$. For $n \geq 1$ a {\it path of length $n$} in $Q$ is a sequence $p = x_1 x_2 \cdots x_n$ of $n$ arrows with $\mathrm t (x_i) = \mathrm s (x_{i+1})$ for $1 \leq i < n$ and we sometimes denote the path length of an arbitrary path $p$ by $\lvert p \rvert$. We denote the set of paths of length $n$ by $Q_n$ and write $Q_{\geq N} = \bigcup_{n \geq N} Q_n$ for the set of paths of length $\geq N$ and also $Q_\ldot$ for the set $Q_{\geq 0}$ of all paths. For $p = x_1 x_2 \cdots x_n$, we call $\mathrm s (x_1)$ the source of $p$, denoted by $\mathrm s (p)$, and call $\mathrm t (x_n)$ the target of $p$, denoted by $\mathrm t (p)$. The paths $x_k \cdots x_l$ for $1 \leq k \leq l \leq n$ are called {\it subpaths} of $p$.

An arrow $x \in Q_1$ with $\mathrm s (x) = \mathrm t (x)$ is called a {\it loop} and a path $p$ of length $\geq 2$ with $\mathrm s (p) = \mathrm t (p)$ is called a {\it cycle}. A finite quiver $Q$ is called {\it acyclic} if it contains no loops or cycles in which case $Q_n = \emptyset$ for $n \gg 0$.

The {\it path algebra} $\Bbbk Q = \bigoplus_{n \geq 0} \Bbbk Q_n$ has the set of all paths as a $\Bbbk$-basis and the product $p q$ of two paths $p$ and $q$ is defined to be their concatenation if $\mathrm t (p) = \mathrm s (q)$ and zero otherwise. The paths of length $0$ are orthogonal idempotents (i.e.\ $e_i e_j = 0$ for $i \neq j$ and $e_i^2 = e_i$) satisfying $e_{\mathrm s (p)} p = p = p e_{\mathrm t (p)}$ for each path $p$ and their sum $\sum_{i \in Q_0} e_i = 1_{\Bbbk Q}$ is the identity element of the path algebra $\Bbbk Q$.

Two paths $p, q$ are said to be {\it parallel} if $\mathrm t (p) = \mathrm t (q)$ and $\mathrm s (p) = \mathrm s (q)$ and for a linear combination of paths $f = \sum_k \lambda_k q_k \in \Bbbk Q$, we say that $p$ is parallel to $f$ if $p$ is parallel to $q_k$ whenever $\lambda_k \neq 0$.

The subspace $\Bbbk Q_0 \simeq \Bbbk \times \dotsb \times \Bbbk$ is a subalgebra of $\Bbbk Q$ and $\Bbbk Q_1$ is a $\Bbbk Q_0$-bimodule. The path algebra $\Bbbk Q$ is isomorphic to the tensor algebra of $\Bbbk Q_1$ over $\Bbbk Q_0$
$$
\Bbbk Q \simeq \mathrm T_{\Bbbk Q_0} (\Bbbk Q_1) = \Bbbk Q_0 \oplus \bigoplus_{n\geq 1} (\Bbbk Q_1)^{\otimes_{\Bbbk Q_0} n}.
$$
The path algebra $\Bbbk Q$ can be written as a matrix algebra $\Bbbk Q = (\Bbbk \hair e_i Q e_j)_{ij}$ where $e_i Q e_j$ is the set of paths from $i$ to $j$ and the multiplication is given by matrix multiplication.

Given a two-sided ideal $I \subset \Bbbk Q$, the algebra $A = \Bbbk Q / I$ is called a {\it path algebra of a quiver with relations}.

\begin{figure}
\centering
\begin{tikzpicture}[baseline=-2.75pt,x=1em,y=1em]
\node at (0,0) {$Q$};
\node[anchor=east] at (-4.5,-3) {\itemi};
\node[anchor=east] at (-4.5,-7) {\itemii};
\node[anchor=east] at (-4.5,-10) {\itemiii};
\begin{scope}[shift={(-1.2,-3)}]
\draw[line width=1pt, fill=black] (0,0) circle(0.2ex);
\node[shape=circle, scale=0.7](L) at (0,0) {};
\node[shape=circle, scale=0.9](LL) at (0,0) {};
\path[->,line width=.4pt,font=\scriptsize, looseness=16, in=35, out=325,transform canvas={xshift=-.5pt,yshift=-.3pt}]
(L.340) edge (L.20)
;
\path[->,line width=.4pt,font=\scriptsize, looseness=18, in=40, out=320,transform canvas={xshift=-3pt},overlay]
(LL.320) edge (LL.40)
;
\node[font=\scriptsize] at (1.25,-1.3) {$x_1{,} ..., x_d$}
;
\node[font=\scriptsize] at (1.8,0) {$...$}
;
\end{scope}
\node at (13,0) {$I$};
\begin{scope}[shift={(13,-3)}]
\node at (0,0) {$\langle x_j x_i - x_i x_j\rangle_{1 \leq i < j \leq d}$};
\end{scope}
\begin{scope}[shift={(13,-7)}]
\node at (0,0) {$\langle x_1 x_2, y_2 y_1, x_2 y_2 - y_1 x_1\rangle $};
\end{scope}
\begin{scope}[shift={(13,-10)}]
\node at (0,0) {$\langle x_0 y_j - x_1 y_{j-1}, y_{j-1} x_1 - y_j x_0\rangle_{0 < j \leq k-1}$};
\end{scope}
\begin{scope}[shift={(-3.8,-7)},x=3.8em]
\draw[line width=1pt, fill=black] (0,0) circle(0.2ex);
\draw[line width=1pt, fill=black] (1,0) circle(0.2ex);
\draw[line width=1pt, fill=black] (2,0) circle(0.2ex);
\node[shape=circle, scale=0.7](0) at (0,0) {};
\node[shape=circle, scale=0.7](1) at (1,0) {};
\node[shape=circle, scale=0.7](2) at (2,0) {};
\path[->,line width=.4pt,font=\scriptsize, looseness=1, in=155, out=25]
(0) edge node[above=-.2ex] {$x_1$} (1)
(1) edge node[above=-.2ex] {$x_2$} (2)
;
\path[<-,line width=.4pt,font=\scriptsize, looseness=1, in=-155, out=-25]
(0) edge node[below=-.2ex] {$y_1$} (1)
(1) edge node[below=-.2ex] {$y_2$} (2)
;
\end{scope}
\begin{scope}[shift={(-2,-10)}]
\draw[line width=1pt, fill=black] (4,0) circle(0.2ex);
\draw[line width=1pt, fill=black] (0,0) circle(0.2ex);
\node[shape=circle, scale=0.7](L) at (0,0) {};
\node[shape=circle, scale=0.7](R) at (4,0) {};
\path[->, line width=.4pt]
(L) edge[transform canvas={yshift=.6ex}] node[above=-.3ex, font=\scriptsize] {$x_0, x_1$} (R)
(L) edge[transform canvas={yshift=-.2ex}] (R)
;
\path[->, line width=.4pt, line cap=round]
(R) edge[out=-160, in=-20, transform canvas={yshift=-.6ex}] (L)
(R) edge[out=-125, in=-55, transform canvas={yshift=-.6ex}, looseness=1.2] node[below=-.4ex, font=\scriptsize] {$y_0{,}..., y_{k-1}$} (L);
\draw (2,-.95) node[font=\scriptsize] {$.$};
\draw (2,-1.15) node[font=\scriptsize] {$.$};
\draw (2,-1.35) node[font=\scriptsize] {$.$};
\end{scope}
\end{tikzpicture}
\caption{Three examples of quivers with relations}
\label{figure:quivers}
\end{figure}
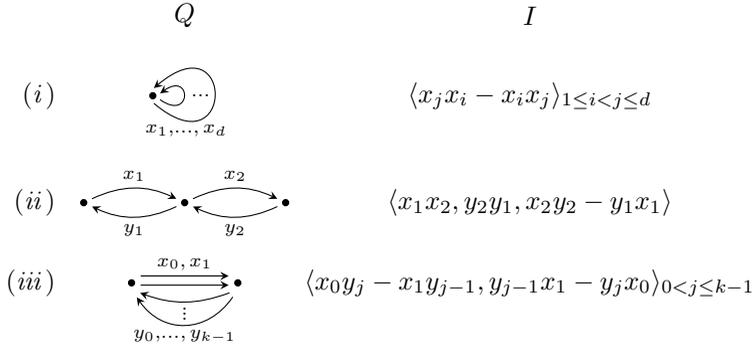

Three simple examples of quivers with relations are given in Fig.~\ref{figure:quivers}, where the quotient algebras $\Bbbk Q / I$ are the following
\begin{enumerate}
\item \label{alg1} the polynomial algebra on $d$ variables $\Bbbk [x_1, \dotsc, x_d]$
\item \label{alg2} the $10$-dimensional Brauer tree algebra
\item \label{alg3} a noncommutative resolution of the toric $\frac1k (1,1)$ singularity.
\end{enumerate}
Deformations of the algebras in \ref{alg1} and \ref{alg2} are studied in Chapter \ref{section:relationtoquantization} and \S\ref{subsection:brauertree}, respectively. (See \cite{barmeierwang} for deformations of the algebra in \ref{alg3} and for further applications in algebraic geometry.)

Note that if a finite quiver $Q$ is acyclic, then $\Bbbk Q$ is finite-dimensional, and so is $\Bbbk Q / I$ for any two-sided ideal $I$ (e.g.\ Fig.~\ref{figure:quivers} {\itemii}). We generally allow the quivers to have loops and cycles, in which case quotients of their path algebras are finitely generated, but may also be infinite-dimensional over $\Bbbk$ (e.g.\ Fig.~\ref{figure:quivers} {\itemi} and {\itemiii}).

\begin{notation}
\label{notation:tensorhom}
Unless indicated otherwise, we use the following shorthand notations
\begin{align*}
\otimes &= \otimes_{\Bbbk Q_0} \\
\Hom &= \Hom_{\Bbbk Q_0^\e}
\end{align*}
where $\Bbbk Q_0^\e = \Bbbk Q_0 \otimes_{\Bbbk} \Bbbk Q_0^{\mathrm{op}}$ is the enveloping algebra of $\Bbbk Q_0$.
\end{notation}

\chapter{Reduction systems and the Diamond Lemma}
\label{section:preliminaries}
\index{Diamond Lemma|textbf}
\index{reduction system|textbf}

The notion of a reduction system introduced by Bergman \cite{bergman} is central to our description of deformations. It formalizes the choice of a $\Bbbk$-basis for the path algebra $A = \Bbbk Q / I$ of a quiver $Q$ with ideal of relations $I$ and deformations of $A$ may then be described in this basis. For example, given a path $s$ in $Q$ which appears in a relation for $I$, we have that $s - \varphi_s \in I$ for some linear combination of paths $\varphi_s \in \Bbbk Q$ and $s = \varphi_s$ in $A$. A reduction system is a collection of pairs $(s, \varphi_s)$, where $s$ will be ``reduced'' to $\varphi_s$. (And deformations of $A$ may then be given by reducing $s$ to $\varphi_s + \widetilde \varphi_s$.)

The choice of a reduction system will give a $\Bbbk$-basis of $A$ consisting of ``irreducible'' paths (see the Diamond Lemma \ref{lemma:basis}). Following Chouhy--Solotar \cite{chouhysolotar} the combinatorics of the reduction system can be used to construct a projective $A$-bimodule resolution of $A$ (see \S\ref{subsection:projectiveresolutions}), which we later use to study the deformation theory of $A$ (see Chapter \ref{section:deformations-of-path-algebras}).

We also introduce a notion of equivalence between reduction systems in Definition \ref{definition:equivalencereduction} which in the context of formal deformations will be equivalent to the notion of (gauge) equivalence between formal deformations of the associative multiplication.

\section{Reduction systems}
\label{subsection:reductionsystems}

\begin{definition}[{(Bergman \cite[\S 1]{bergman})}]
\index{reduction system}
A {\it reduction system} $R$ for $\Bbbk Q$ is a set of pairs
\[
R = \{ (s, \varphi_s) \mid s \in S \text{ and } \varphi_s \in \Bbbk Q \}
\]
where
\begin{itemize}
\item $S$ is a subset of $Q_{\geq 2}$ such that $s$ is not a subpath of $s'$ when $s \neq s' \in S$
\item for all $s \in S$, $s$ and $\varphi_s$ are parallel
\item for each $(s, \varphi_s) \in R$, $\varphi_s$ is irreducible, i.e.\ it is a linear combination of irreducible paths.
\end{itemize}
Here a path is {\it irreducible} if it does not contain elements in $S$ as a subpath and we denote by $\Irr_S = \Irr_S (Q) = Q_\ldot \setminus Q_\ldot S \hair Q_\ldot$ the set of all irreducible paths.

Given a two-sided ideal $I$ of $\Bbbk Q$, we say that a reduction system $R$ {\it satisfies the condition \textup{($\diamond$)} for $I$} if
\begin{enumerate}
\item $I$ is equal to the two-sided ideal generated by the set $\{ s - \varphi_s \}_{(s, \varphi_s) \in R}$
\item every path is {\it reduction-unique}\index{reduction system!reduction-unique} (see Definition \ref{definitionreductionfiniteunique} below).
\end{enumerate}
We call a reduction system $R$ {\it finite} if $R$ is a finite set. 
\end{definition}

\begin{remark}
\label{remark:f}
It follows from the definition that a reduction system $R = \{ (s, \varphi_s) \}$ is uniquely determined by the set $S \subset Q_{\geq 2}$ together with a $\Bbbk Q_0$-bimodule homomorphism $\phi \in \Hom (\Bbbk S, \Bbbk \Irr_S)$ with $\phi (s) = \phi_s$ (cf.\ Notation \ref{notation:tensorhom}). We sometimes write $R = R_\phi$ to indicate the dependence on the map $\phi$.
\end{remark}

Let $(s, \varphi_s) \in R$ and let $q, r \in Q_\ldot = \bigcup_{n \geq 0} Q_n$ be two paths such that $qsr \neq 0$ in $\Bbbk Q$. Following \cite[\S 2]{chouhysolotar} a {\it basic reduction} $\mathfrak r_{q, s, r} \colon \Bbbk Q \to  \Bbbk Q$ is defined as the $\Bbbk$-linear map uniquely determined by the following: for any path $p \in Q_\ldot$
$$
\mathfrak r_{q, s, r} (p) =
\begin{cases}
q \hair \varphi_s r & \text{if $p = q s r$}\\
p & \text{if $p \neq qsr.$}
\end{cases}
$$
A {\it reduction}\index{reduction} $\mathfrak r$ is defined as a composition $\mathfrak r_{q_n, s_n, r_n} \circ \dotsb \circ \mathfrak r_{q_2, s_2, r_2} \circ \mathfrak r_{q_1, s_1, r_1}  $ of basic reductions for some $n \geq 1$. 

A path may contain many subpaths which lie in $S$ and so one may obtain different elements in $\Bbbk Q$ after performing different reductions.  
\begin{definition}
\label{definitionreductionfiniteunique}
We say that a path $p \in Q_\ldot$ is
\begin{itemize}
\item {\it reduction-finite}\index{reduction system!reduction-finite|textbf} if for any infinite sequence of reductions $(\mathfrak r_i)_{i \in \mathbb N}$ there exists $n_0 \in \mathbb N$ such that for all $n \geq n_0$, we have $\mathfrak r_n \circ \dotsb \circ \mathfrak r_1 (p) = \mathfrak r_{n_0} \circ \dotsb \circ \mathfrak r_1(p)$
\item {\it reduction-unique}\index{reduction system!reduction-unique|textbf} if $p$ is reduction-finite and, moreover, for any two reductions $\mathfrak r$ and $\mathfrak r'$ such that $\mathfrak r(p)$ and $\mathfrak r'(p)$ are both irreducible, we have $\mathfrak r (p) = \mathfrak r' (p)$.
\end{itemize}
\end{definition}

Note that an element $a \in \Bbbk Q$ is irreducible if and only if $\mathfrak r (a) = a$ for all reductions $\mathfrak r$. The combinatorics of reductions is described in \S\ref{subsubsection:reductions}.

We will see many examples of reduction systems\index{reduction system}, but two simple and important examples are the following.

\begin{example}
\label{example:monomialreduction}
{\it Monomial algebras.}\; Let $A = \Bbbk Q / \langle S\rangle$ be a monomial algebra where $\langle S \rangle$ is the two-sided ideal generated by the set $S \subset Q_{\geq 2}$ of minimal relations, i.e.\ $s$ is not a subpath of $s'$ when $s \neq s' \in S$. Any path in the ideal $\langle S \rangle$, i.e.\ any path containing an element of $S$ as a subpath, is equal to $0$ in $A$ and $R = \{ (s, 0) \mid s \in S \}$ is a reduction system\index{reduction system} satisfying the condition ($\diamond$) for the ideal $\langle S \rangle$. In this case, it is clear that the set of irreducible paths forms a $\Bbbk$-basis of $A$ (even without invoking the Diamond Lemma \ref{lemma:basis} below). 
\end{example}

\begin{figure}
\centering
\begin{tikzpicture}[baseline=-2.6pt,description/.style={fill=white,inner sep=1.75pt}]
\matrix (m) [matrix of math nodes, outer sep=1pt, row sep={3.4em,between origins}, text height=1.2ex, column sep={3.4em,between origins}, text depth=0.25ex, ampersand replacement=\&]
{
\&\& x_3 x_2^2 x_1 \&\& \\
\& x_2 x_3 x_2 x_1 \&\& x_3 x_2 x_1 x_2 \\
x_2^2 x_3 x_1 \&\& x_2 x_3 x_1 x_2 \&\& x_3 x_1 x_2^2 \\
x_2^2 x_1 x_3 \&\& x_2 x_1 x_3 x_2 \&\& x_1 x_3 x_2^2 \\
\& x_2 x_1 x_2 x_3 \&\& x_1 x_2 x_3 x_2 \\
\&\& x_1 x_2^2 x_3 \\
};
\path[|->,line width=.4pt] 
(m-1-3) edge (m-2-2)
(m-1-3) edge (m-2-4)
(m-2-2) edge (m-3-1)
(m-2-2) edge (m-3-3)
(m-2-4) edge (m-3-3)
(m-2-4) edge (m-3-5)
(m-3-1) edge (m-4-1)
(m-3-3) edge (m-4-3)
(m-3-5) edge (m-4-5)
(m-4-1) edge (m-5-2)
(m-4-3) edge (m-5-2)
(m-4-3) edge (m-5-4)
(m-4-5) edge (m-5-4)
(m-5-2) edge (m-6-3)
(m-5-4) edge (m-6-3)
;
\end{tikzpicture}
\caption{An example of reductions in a reduction system for the polynomial algebra}
\label{figure:diamond}
\end{figure}
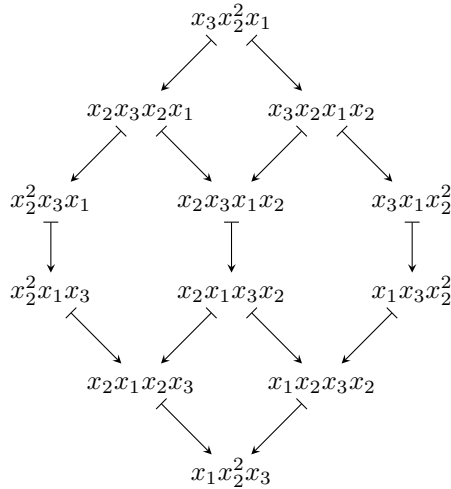

\begin{example}
\label{example:polynomialreduction}
{\it Polynomial algebras.}\; Let $Q$ be the quiver with one vertex and $d$ loops $x_1, \dotsc, x_d$ and let $I = \langle x_j x_i - x_i x_j \rangle_{1 \leq i < j \leq d}$ as in Fig.~\ref{figure:quivers} {\itemi}. Then $A = \Bbbk Q / I \simeq \Bbbk [x_1, \dotsc, x_d]$ and
\[
R = \bigl\{ (x_j x_i, x_i x_j) \bigr\}_{1 \leq i < j \leq d}
\]
is a reduction system\index{reduction system} satisfying the condition ($\diamond$) for $I$. See Fig.~\ref{figure:diamond} for an example of a ``diamond'' of reductions with respect to $R$, where every arrow corresponds to a basic reduction.
\end{example}

Reductions can be described systematically by defining maps
\begin{flalign*}
&& \widetilde{\mathrm{split}}{}_2^* &\colon \Bbbk Q \to \Bbbk Q \otimes \Bbbk S \otimes \Bbbk Q && \mathllap{* \in \{ \mathrm R, \mathrm L, \emptyset \}}
\end{flalign*}
by
\begin{flalign}
\label{split_2}
&& \widetilde{\mathrm{split}}{}_2 (p) &= \sum_{\substack{s \in S \\ q s r = p}} q \otimes s \otimes r &
\begin{aligned}
\widetilde{\mathrm{split}}{}^{\mathrm R}_2 (p) &= q \otimes s_{\mathrm R} \otimes r \\
\widetilde{\mathrm{split}}{}^{\mathrm L}_2 (p) &= q \otimes s_{\mathrm L} \otimes r
\end{aligned} &&
\end{flalign}
where $s_{\mathrm R}$ (resp.\ $s_{\mathrm L}$) is the rightmost (resp.\ leftmost) subpath of $p$ which lies in $S$ as illustrated in Fig.~\ref{figure:split} and $\widetilde{\mathrm{split}}{}^{*}_2(p) = 0$ if $p$ is irreducible. (In \S\ref{subsubsection:reductions} we will define $\mathrm{split}_2^*$ by passing to the quotient $A = \Bbbk Q / I$ on the first and last factor of the tensor product, whence the notation $\widetilde{\mkern20mu}$ here, as well as ``higher analogues'' of this map denoted by $\mathrm{split}_n^*$, whence the index $_2$ here.)

\begin{figure}
\centering
\begin{tikzpicture}
[x=2em,y=1em]
\node[right] at (10.7,0) {$\in \Bbbk Q$};
\node (p) at (-1.9,-.1) {$p$};
\node at (-1,-.1) {$=$};
\path[|->,line width=.4pt] (p) edge ++(0,2.2);
\path[|->,line width=.4pt] (p) edge ++(0,-2.2);
\foreach \n in {0,1,2,3,4,5,6,7,8,9,10} {
\node[inner sep=0] (\n) at (\n,0) {};
}
\path[->,line width=.4pt] (0) edge (1);
\path[->,line width=.4pt] (1) edge (2);
\path[->,line width=.4pt] (2) edge (3);
\path[->,line width=.4pt] (3) edge (4);
\path[->,line width=.4pt] (4) edge (5);
\path[->,line width=.4pt] (5) edge (6);
\path[->,line width=.4pt] (6) edge (7);
\path[->,line width=.4pt] (7) edge (8);
\path[->,line width=.4pt] (8) edge (9);
\path[->,line width=.4pt] (9) edge (10);
\node at (3,.9) {$\overbracket[.3pt]{\hspace{3.85em}}^{\in S}$};
\node at (6,.9) {$\overbracket[.3pt]{\hspace{3.85em}}^{\in S}$};
\node at (7.5,-.9) {$\underbracket[.3pt]{\hspace{5.85em}}_{\in S}$};
\begin{scope}[shift={(0,-3.3)}]
\node at (-2.3,.1) {$\widetilde{\mathrm{split}}{}_2^{\mathrm R} (p)$};
\node at (-1,-.1) {$=$};
\node[right] at (10.7,0) {$\in \Bbbk Q \otimes S \otimes \Bbbk Q$};
\foreach \n in {0,1,2,3,4,5,6,7,8,9,10} {
\node[inner sep=0] (\n) at (\n,0) {};
}
\path[->,line width=.4pt,transform canvas={xshift=-1.2em}] (0) edge (1);
\path[->,line width=.4pt,transform canvas={xshift=-1.2em}] (1) edge (2);
\path[->,line width=.4pt,transform canvas={xshift=-1.2em}] (2) edge (3);
\path[->,line width=.4pt,transform canvas={xshift=-1.2em}] (3) edge (4);
\path[->,line width=.4pt,transform canvas={xshift=-1.2em}] (4) edge (5);
\path[->,line width=.4pt,transform canvas={xshift=-1.2em}] (5) edge (6);
\path[->,line width=.4pt] (6) edge (7);
\path[->,line width=.4pt] (7) edge (8);
\path[->,line width=.4pt] (8) edge (9);
\path[->,line width=.4pt,transform canvas={xshift=1.2em}] (9) edge (10);
\node at (5.7,0) {$\otimes$};
\node at (9.3,0) {$\otimes$};
\end{scope}
\begin{scope}[shift={(0,3.3)}]
\node[right] at (10.7,0) {$\in \Bbbk Q \otimes S \otimes \Bbbk Q$};
\node at (-2.3,.1) {$\widetilde{\mathrm{split}}{}_2^{\mathrm L} (p)$};
\node at (-1,-.1) {$=$};
\foreach \n in {0,1,2,3,4,5,6,7,8,9,10} {
\node[inner sep=0] (\n) at (\n,0) {};
}
\path[->,line width=.4pt,transform canvas={xshift=-1.2em}] (0) edge (1);
\path[->,line width=.4pt,transform canvas={xshift=-1.2em}] (1) edge (2);
\path[->,line width=.4pt] (2) edge (3);
\path[->,line width=.4pt] (3) edge (4);
\path[->,line width=.4pt,transform canvas={xshift=1.2em}] (4) edge (5);
\path[->,line width=.4pt,transform canvas={xshift=1.2em}] (5) edge (6);
\path[->,line width=.4pt,transform canvas={xshift=1.2em}] (6) edge (7);
\path[->,line width=.4pt,transform canvas={xshift=1.2em}] (7) edge (8);
\path[->,line width=.4pt,transform canvas={xshift=1.2em}] (8) edge (9);
\path[->,line width=.4pt,transform canvas={xshift=1.2em}] (9) edge (10);
\node at (1.7,0) {$\otimes$};
\node at (4.3,0) {$\otimes$};
\end{scope}
\end{tikzpicture}
\caption{An illustration of the maps $\widetilde{\mathrm{split}}{}^{\mathrm L}_2$ and $\widetilde{\mathrm{split}}{}^{\mathrm R}_2$}
\label{figure:split}
\end{figure}

The map $\widetilde{\mathrm{split}}{}_2^{\mathrm R}$ records the rightmost subpath lying in $S$ and can be used to define the {\it rightmost reduction}, which is the $\Bbbk$-linear map \index{reduction}\index{reduction!rightmost}
\begin{equation}
\label{rightmost-reduction}
\mathrm{red}_\varphi \colon \Bbbk Q \to \Bbbk Q
\end{equation}
given by 
\[
\mathrm{red}_\varphi(p)=
\begin{cases}
p & \text{if $p$ is an irreducible path}\\
q \varphi_s r & \text{if $p$ is a path such that $\widetilde{\mathrm{split}}{}^{\mathrm R}_2(p)=q \otimes s \otimes r$}
\end{cases}
\]
and will be used to define a combinatorial star product in \S\ref{subsection:combinatorialstarproduct}.

If $R$ is reduction-finite, then for every path $p$ we have $\mathrm{red}_\varphi^N (p) = \mathrm{red}_\varphi^{N+1} (p)$ for $N \gg 0$ so that $\mathrm{red}_\varphi^N (p)$ is irreducible for $N \gg 0$ and we define $\mathrm{red}_\varphi^{(\infty)} \colon \Bbbk Q \to \Bbbk \Irr_S$ as 
\begin{align}\label{align:redinfty}
\mathrm{red}_\varphi^{(\infty)} (p) = \mathrm{red}_\varphi^N (p) \quad \text{for $N \gg 0$}.
\end{align}
(This definition can also be recovered from a more general construction by taking $z = 0$ in \eqref{red}.) Note that $\mathrm{red}_\phi^{(\infty)} \rvert_{\Bbbk \Irr_S} = \id_{\Bbbk \Irr_S}$.

We will now recall Bergman's {\it Diamond Lemma}\index{Diamond Lemma}, which shows that a reduction system\index{reduction system} $R$ satisfying ($\diamond$) for an ideal $I \subset \Bbbk Q$ will give rise to a $\Bbbk$-basis of $A = \Bbbk Q / I$.

\begin{definition}[{\cite[\S 1]{bergman}}]
\label{definition:overlap}
\index{overlap ambiguity|textbf}\index{ambiguity!overlap|textbf}
Let $R$ be a reduction system for $\Bbbk Q$. A path $pqr \in Q_{\geq 3}$ for $p,q,r \in Q_{\geq 1}$ is an {\it overlap ambiguity} (or simply {\it overlap}) of $R$ if $pq, qr \in S$.

We say that an overlap ambiguity\index{overlap ambiguity}\index{ambiguity!overlap|textbf} $pqr$ with $pq = s$ and $qr = s'$ is {\it resolvable} if $\varphi_s r$ and $p \varphi_{s'}$ are reduction-finite and $\mathfrak r (\varphi_s r) = \mathfrak r' (p \varphi_{s'})$ for some reductions $\mathfrak r, \mathfrak r'$.
\end{definition}

\begin{theorem}[(Diamond Lemma) {\cite[Thm.~1.2]{bergman}}]
\label{lemma:basis}
\index{Diamond Lemma|textbf}\index{overlap ambiguity}\index{ambiguity!overlap}
Let $R = \{ (s, \phi_s) \}$ be a reduction system for $\Bbbk Q$ and let $S = \{ s \mid (s, \phi_s) \in R \}$. Denote by $I = \langle s - \varphi_s \rangle_{s \in S} \subset \Bbbk Q$ be the corresponding two-sided ideal and by $A = \Bbbk Q / I$ the quotient algebra. If $R$ is reduction-finite, then the following are equivalent:
\begin{enumerate}
\item all overlap ambiguities of $R$ are resolvable \label{diamond1}
\item $R$ is reduction-unique, i.e.\ $R$ satisfies \textup{($\diamond$)} for $I$ \label{diamond2}
\item the image of the set $\Irr_S$ of irreducible paths under the projection $\pi \colon \Bbbk Q \to A$ forms a $\Bbbk$-basis of $A$. \label{diamond3}
\end{enumerate}
\end{theorem}

\noindent (We will give a deformation-theoretic interpretation of the Diamond Lemma in \S\ref{subsection:proofdiamond}.)

The Diamond Lemma \ref{lemma:basis} implies that we have a $\Bbbk$-linear isomorphism $\sigma \colon A \to \Bbbk \Irr_S$
\begin{equation}
\label{sigma}
\begin{tikzpicture}[baseline=-2.6pt,description/.style={fill=white,inner sep=1pt,outer sep=0}]
\matrix (m) [matrix of math nodes, row sep=0em, text height=1.5ex, column sep=3em, text depth=0.25ex, ampersand replacement=\&, inner sep=3pt]
{
\Bbbk Q \\[.3em]
\rotatebox{90}{$\subset$} \\
\Bbbk \Irr_S \& A \\
};
\path[->,line width=.4pt,font=\scriptsize]
(m-1-1) edge node[above=-.2ex] {$\pi$} (m-3-2)
;
\path[->,line width=.4pt,font=\scriptsize]
(m-3-2) edge node[below=-.3ex,pos=.45] {$\sigma$} (m-3-1)
;
\end{tikzpicture}
\end{equation}
such that  $\pi \sigma = \id_A$ and $\sigma \pi = \mathrm{red}_\phi^{(\infty)}$. In particular
\begin{itemize}
\item $\sigma \pi (u) = u$ for all irreducible paths $u \in \Irr_S$ and
\item $\sigma \pi (s) = \phi_s$ for any $s \in S$.
\end{itemize}

\begin{remark}
\label{remark:reduction}
Note that $\pi \rvert_{\Bbbk \Irr_S} \colon \Bbbk \Irr_S \to A$ is a $\Bbbk Q_0$-bimodule isomorphism with inverse $\sigma$. In particular we may naturally identify $\Hom (\Bbbk S, \Bbbk \Irr_S) \simeq \Hom (\Bbbk S, A)$ (cf.\ Remark \ref{remark:f}).
\end{remark}

The following result shows that the deformation theory developed in later sections via reduction systems can be applied to any algebra $A = \Bbbk Q / I$ and the choice of a reduction system can be thought of the choice of a basis in which the deformations can be described explicitly.

\begin{proposition}[(Chouhy--Solotar {\cite[Prop.~2.7]{chouhysolotar}})]
\label{proposition:existencereduction}
If $I \subset \Bbbk Q$ is any two-sided ideal, then there exists a reduction system\index{reduction system!existence} $R$ satisfying the condition \textup{($\diamond$)} for $I$.
\end{proposition}


\begin{remark}\label{remark:algorithm}
The proof of this proposition consists essentially of an algorithm similar to the Buchberger algorithm used to compute commutative or noncommutative Gröbner bases. This algorithm starts with choosing a total order on $Q_{\leq 1}$. There may be other natural choices of reduction systems which cannot be obtained from this algorithm (see \cite[Ex.~2.10.1]{chouhysolotar} for an example).
\end{remark}

In practice it is not always necessary to choose a total order on $Q_{\leq 1}$. Instead one may apply the following heuristic, which is less systematic, but can often be used to find a natural reduction system, at least when the ideal $I$ is finitely generated.

\begin{heuristic}
\label{heuristic}
Let $Q$ be a finite quiver and $I = \langle f_1, \dotsc, f_m \rangle \subset \Bbbk Q$ be the two-sided ideal generated by a finite set $X = \{ f_1, \dotsc, f_m \} \subset \Bbbk Q$, each $f_i$ being a (finite) linear combination of paths, say $f_i = \lambda_{i,1} p_{i,1} + \dotsb + \lambda_{i,n_i} p_{i,n_i}$ for some $n_i \geq 1$, paths $p_{i,j} \in Q_\ldot$ and coefficients $\lambda_{i,j} \in \Bbbk$. Construct a reduction system $R$ (starting from the empty set) as follows.

\begin{description}[labelindent=0em,itemindent=0em,labelwidth=3.5em,leftmargin=4.5em,before={\renewcommand\makelabel[1]{\bfseries ##1}},labelsep=!]
\item[\bf Step 1.] For each $f_i = \lambda_{i,1} p_{i,1} + \dotsb + \lambda_{i,n_i} p_{i,n_i} \in X$, choose some leading path $p_{i,k_i}$ with $1 \leq k_i \leq n_i$, divide through by its coefficient and add the pair
\[
\Bigl( p_{i,k_i}, - \textstyle\sum\limits_{j \neq k_i} \tfrac{\lambda_{i,j}}{\lambda_{i,k_i}} p_{i,j} \Bigr)
\]
to the set $R$. At this step, try to avoid overlaps of the paths $p_{i,k_i}$ with themselves or with other paths $p_{j,k_j}$ already chosen.
\item[\bf Step 2.] Compute the overlaps of the paths $p_{i,k_i}$ for all $1 \leq i \leq m$.
\item[\bf Step 3.] When the overlaps do not resolve, add the difference as a new element to $X$ and proceed as in Steps 1 and 2.\index{overlap ambiguity}\index{ambiguity!overlap}
\item[\bf Step 4.] Prove that the final reduction system is reduction-finite. (If so, it will by construction be reduction-unique.)
\end{description}
\end{heuristic}

\noindent (See Remark \ref{remark:strong}, \S\ref{subsection:brauertree} and \cite{barmeierwang} for applications of this heuristic.)

\subsection{Reduction systems from noncommutative Gröbner bases}
\label{subsubsection:reduction}
\index{reduction system!from noncommutative Gröbner basis}

In the theory of noncommutative Gröbner bases (also called Gröbner--Shirshov bases), an {\it admissible order} $\prec$ \index{admissible order|textbf} is a total order on $Q_\ldot = \bigcup_{n\geq 0} Q_n$ such that
\begin{enumerate}
\item every non-empty subset of $Q_\ldot$ has a minimal element, and
\item for any elements $p, q, r \in Q_\ldot$ the following conditions hold
\begin{itemize}
\item if $p \prec q$ then $pr \prec qr$ whenever $pr$ and $qr$ are non-zero
\item if $p \prec q$ then $rp \prec rq$ whenever $rp$ and $rq$ are non-zero
\item and $q \preceq pqr$ whenever $pqr$ is non-zero.
\end{itemize}
\end{enumerate} 

Fix an admissible order $\prec$ on $Q_\ldot$. Let $F = \sum_{p\in Q_\ldot} \lambda_p p$ be a non-zero element of  $\Bbbk Q$, where $\lambda_p \in \Bbbk$ and almost all $\lambda_p$ are zero. Then we define 
$$
\tip_\prec (F) = \tip (F) := p \quad \text{if $\lambda_p \neq 0$ and $p \succ q$ for all $q$ with $\lambda_q \neq 0$}.
$$
If $X \subset \Bbbk Q$ then let $\tip (X) = \bigl\{ \tip (F) \mid F \in X \backslash \{ 0 \} \bigr\}$ and $\langle \tip (X) \rangle$ be the ideal of $\Bbbk Q$ generated by $\tip (X)$.

Furthermore, an element $F \in \Bbbk Q$ is called {\it uniform} if it is a linear combination of parallel paths.

\begin{definition}
\index{Gröbner basis!noncommutative|textbf}
Let $I$ be an ideal of $\Bbbk Q$ and let $\prec$ be an admissible order on $Q_\ldot$. The {\it (reduced) noncommutative Gröbner basis} for $I$ with respect to $\prec$ is the set $\mathcal G$ of uniform elements in $I$ such that
$$
\langle \tip (I) \rangle = \langle \tip (\mathcal G) \rangle
$$
and such that the coefficient $\lambda_{\tip (F)} = 1$ for any $F \in \mathcal G$ and $\tip (F_i)$ is not a subpath of $\tip (F_j)$ when $F_i \neq F_j \in \mathcal G$.
\end{definition}

For any fixed admissible order $\prec$ the (reduced) noncommutative Gröbner basis is unique (see e.g.\ \cite[Def.~3.2]{greenhilleschroll}).

Let $\{ F_j \}_{j \in J}$ be a noncommutative Gröbner basis for an ideal $I \subset \Bbbk Q$ (with respect to an order $\prec$). Then one immediately obtains a reduction system\index{reduction system!from noncommutative Gröbner basis}
\begin{equation}
\label{reductiongroebner}
R = \bigl\{ \bigl( \tip (F_j), -F_j + \tip (F_j) \bigr) \bigr\}_{j \in J}
\end{equation}
satisfying ($\diamond$) for $I$.

So every noncommutative Gröbner basis gives a reduction system, but conversely not every reduction system is obtained from a Gröbner basis (cf.\ \cite[Ex.~2.10.1]{chouhysolotar}). In the commutative case, any ideal admits a finite commutative Gröbner basis which may be computed by Buchberger's algorithm. For ideals in noncommutative algebras a similar algorithm computes a noncommutative Gröbner basis and this algorithm terminates if and only if the ideal admits a finite Gröbner basis\index{Gröbner basis!noncommutative!finite}. If $\Bbbk Q / I$ is commutative or finite-dimensional, then $I$ admits a finite noncommutative Gröbner basis \cite{eisenbudpeevasturmfels,green1} and this property also extends to certain diagrams of commutative algebras \cite{barmeierwang}. However, in general the property of admitting a finite noncommutative Gröbner basis is undecidable (see e.g.\ \cite{bokutkolesnikov,fmora,tmora}). That said, noncommutative Gröbner bases have been computed for a growing number of noncommutative algebras --- see e.g.\ \cite{bokutchen,bueso} for examples including Weyl algebras, Iwahori--Hecke algebras, quantum groups, Ore extensions, universal enveloping algebras of finite-dimensional Lie algebras, \cite{schedler} for preprojective algebras, and also \cite{bremnerdotsenko} for noncommutative Gröbner bases in the more general setting of algebras over an operad. All of these can be used to describe the deformations of these algebras using the deformation theory via reduction systems as developed in the present book.

A basic example is the Gröbner basis for the polynomial algebra.

\begin{example}
\label{example:reductiongroebnersymmetric}
Let $R = \{ (x_j x_i, x_i x_j) \}_{1 \leq i < j \leq d}$ be the reduction system for the algebra $$\Bbbk \langle x_1, \dotsc, x_d \rangle / \langle x_j x_i - x_i x_j \rangle \simeq \Bbbk [x_1, \dotsc, x_d]$$ as in Example \ref{example:polynomialreduction}. (The corresponding quiver is drawn in Fig.~\ref{figure:quivers} {\itemi}.)

Then $R$ can be given via a Gröbner basis as in (\ref{reductiongroebner}) where the Gröbner basis\index{Gröbner basis!noncommutative} is obtained from the ordering $x_1 \prec \dotsb \prec x_d$ on the linear monomials, extended to all monomials using the so-called \index{degree--lexicographic order|textbf}\index{admissible order!degree--lexicographic|textbf}{\it degree--lexicographic ordering}, i.e.\ monomials are ordered first by degree (i.e.\ path length) and monomials of the same degree are ordered by the lexicographic order on linear monomials. In particular, $x_j x_i \succ x_i x_j$ whenever $j > i$, so that $\tip (x_j x_i - x_i x_j) = x_j x_i$.
\end{example}

\section{Equivalence of reduction systems}
\label{subsection:equivalencereduction}

We introduce the following new notion of equivalence of reduction systems. Roughly speaking, an equivalence of reduction systems is a $\Bbbk$-linear automorphism of the irreducible paths which ``commutes with'' performing reductions.

\begin{definition}
\label{definition:equivalencereduction}
\index{reduction system!equivalence|textbf}
Let $\mathrm{MC} \subset \Hom (\Bbbk S, \Bbbk \Irr_S)$ denote the subset of elements whose associated reduction systems (cf.\ Remark \ref{remark:f}) are reduction-unique.\footnote{Later $\mathrm{MC}$ will be seen to be the set of Maurer--Cartan elements for the L$_\infty[1]$ algebra $\mathbf p (Q, R)$ controlling the deformation theory of the reduction system $R$.} Given any $\varphi, \varphi' \in \mathrm{MC}$, let $R$ and $R'$ denote the corresponding reduction systems. We say that $R$ and $R'$ (or $\varphi$ and $\varphi'$) are {\it equivalent} if there exists a $\Bbbk$-linear automorphism $T \in \GL (\Bbbk \Irr_S)$ satisfying $T (e_i) = e_i$ for each $i \in Q_0$ and 
\begin{equation}
\label{equivred}
T (\mathrm{red}_{\phi'}^{(\infty)} (p)) = \mathrm{red}_{\phi}^{(\infty)} (T (p_1) \dotsb T (p_m))
\end{equation}
for any path $p = p_1 \dotsb p_m$ with $p_i \in Q_1$, where $\mathrm{red}_{\phi}^{(\infty)}$ was defined in \eqref{align:redinfty}.

Let $G (\varphi, \varphi')$ denote the set of all such $\Bbbk$-linear automorphisms of $\Bbbk \Irr_S$. Then we have a groupoid $G \rightrightarrows \mathrm{MC}$ (see e.g.\ \cite[Part VI]{cannasdasilvaweinstein}) and $R$ and $R'$ are equivalent precisely when $G (\varphi, \varphi') \neq \emptyset$.
\end{definition}

\begin{lemma}
\label{lemma:equivalencereduction}
Let $R, R'$ be two reduction-unique reduction systems determined by $\phi, \phi' \in \Hom (\Bbbk S, \Bbbk \Irr_S)$ respectively. Then the following are equivalent:
\begin{enumerate}
\item \label{equivred1} $R$ and $R'$ are equivalent.
\item \label{equivred2} There exists $T \in \GL (\Bbbk \Irr_S)$ satisfying $T (e_i) = e_i$ and
\begin{flalign*}
&& T (u) &= \mathrm{red}_{\varphi}^{(\infty)} (T (u_1) \dotsb T (u_m)) && \llap{$u = u_1 \dotsb u_m$} \\
&& T (\phi' (s)) &= \mathrm{red}_{\varphi}^{(\infty)} (T (s_1) \dotsb T (s_n)) && \mathllap{s = s_1 \dotsb s_n}
\end{flalign*}
for any $u \in \Irr_S$ and any $s \in S$.
\item \label{equivred3} The algebras $A = \Bbbk Q / \langle s - \phi_s \rangle_{s \in S}$ and $A' = \Bbbk Q / \langle s - \phi'_s \rangle_{s \in S}$ are isomorphic.
\end{enumerate}
\end{lemma}

\begin{proof}
That \ref{equivred1} implies \ref{equivred2} is clear since $\mathrm{red}_{\phi'}^{(\infty)} (u) = u$ and $\mathrm{red}_{\phi'}^{(\infty)} (s) = \phi' (s)$.

Since $R$ and $R'$ are reduction-unique, the Diamond Lemma implies that $A \simeq \Bbbk \Irr_S \simeq A'$ as $\Bbbk Q_0$-bimodules. Let $\pi \colon \Bbbk Q \to A$ and $\pi' \colon \Bbbk Q \to A'$ be the natural projections. Let $\Psi \colon \Bbbk Q \to A$ be the algebra homomorphism uniquely determined by $\Psi (x) = \pi \, T (x)$ for $x \in Q_1$, since $\Bbbk Q$ is a free tensor algebra over $\Bbbk Q_0$.

Now let $p = p_1 \dotsb p_m$ be any path. Then
\begin{equation}
\label{psired}
\begin{split}
\Psi (p) &= \Psi (p_1) \dotsb \Psi (p_m) \\
&= \pi \, T (p_1) \dotsb \pi \, T (p_m) \\
&= \pi (T (p_1) \dotsb T (p_m)) \\
&= \pi \, \mathrm{red}_{\phi}^{(\infty)} (T (p_1) \dotsb T (p_m))
\end{split}
\end{equation}
where the last equality follows from the observation that $\sigma \pi = \mathrm{red}_{\phi}^{(\infty)}$, whence $\pi = \pi \sigma \pi = \pi \, \mathrm{red}_{\phi}^{(\infty)}$ (cf.\ \eqref{sigma}). Now \ref{equivred2} implies $\Psi (u) = \pi \, T (u)$ for any $u \in \Irr_S$ and $\Psi (s) = \pi \, T (\phi_s')$ for any $s \in S$. Since $\phi_s'$ is irreducible, this implies $\Psi (s - \phi_s') = 0$ and thus $\Psi$ induces an algebra homomorphism $\bar \Psi \colon A' \to A$ so that the following diagram commutes.
\begin{equation*}
\begin{tikzpicture}[baseline=-2.6pt,description/.style={fill=white,inner sep=1pt,outer sep=0}]
\matrix (m) [matrix of math nodes, row sep=2em, text height=1.5ex, column sep=3em, text depth=0.25ex, ampersand replacement=\&, inner sep=3pt]
{
A' \& A \\
\Bbbk \Irr_S \& \Bbbk \Irr_S \\
};
\path[->,line width=.4pt,font=\scriptsize]
(m-1-1) edge node[above=-.3ex] {$\bar \Psi$} (m-1-2)
;
\path[->,line width=.4pt,font=\scriptsize]
(m-2-1) edge node[above=-.3ex] {$T$} (m-2-2)
;
\path[->,line width=.4pt,font=\scriptsize]
(m-2-1) edge node[left=-.3ex] {$\pi'$} (m-1-1)
;
\path[->,line width=.4pt,font=\scriptsize]
(m-2-2) edge node[right=-.3ex] {$\pi$} (m-1-2)
;
\end{tikzpicture}
\end{equation*}
Since $\pi, \pi', T$ are isomorphisms of $\Bbbk$-vector spaces, $\bar \Psi$ is an algebra isomorphism. Thus \ref{equivred2} implies \ref{equivred3}.

Finally, to show that \ref{equivred3} implies \ref{equivred1}, let $\bar \Psi \colon A' \to A$ be an algebra isomorphism and define $T = \pi^{-1} \bar \Psi \pi'$. Similar to \eqref{psired}, using $\pi = \pi \, \mathrm{red}_{\phi}^{(\infty)}$ and $\pi' = \pi' \hair \mathrm{red}_{\phi'}^{(\infty)}$, we may prove that for any path $p = p_1 \dotsb p_m$ we have
\begin{align*}
\pi \, T (\mathrm{red}_{\phi'}^{(\infty)} (p)) 
&= \pi \, \mathrm{red}_\phi^{(\infty)} (T (p_1) \dotsb T (p_m)).
\end{align*}
Since $\pi \rvert_{\Bbbk \Irr_S}$ is an isomorphism, \eqref{equivred} holds, i.e.\ $R$ and $R'$ are equivalent.
\end{proof}

\begin{remark}
Lemma \ref{lemma:equivalencereduction} \ref{equivred2} implies that $T \in G (\varphi, \varphi')$ is determined by its value on arrows and thus $T$ is determined by a map in $\Hom (\Bbbk Q_1, \Bbbk \Irr_S)$. In \S\ref{subsection:projectiveresolutions}, $\Hom (\Bbbk Q_1, \Bbbk \Irr_S)$ will appear as the space of $1$-cochains in a cochain complex computing the Hochschild cohomology of $A = \Bbbk Q / I$ with $\Bbbk$-basis $\Irr_S$.

We will see a formal version of the equivalence between \ref{equivred1} and \ref{equivred3} of Lemma \ref{lemma:equivalencereduction} in \S\ref{subsubsection:notation}, where up to equivalence formal deformations of the reduction system are in one-to-one correspondence with formal deformations of the associative algebra structure.
\end{remark}

\section{Higher ambiguities}
\label{subsection:chains}
\index{ambiguity!higher, $n$-|textbf}

Let $Q$ be a finite quiver and let $S$ be any subset of $Q_{\geq 2}$ such that $s$ is not a subpath of $s'$ when $s \neq s' \in S$. We now recall the definition of $n$-ambiguities for $n \geq 0$, which can be viewed as a generalization of overlap ambiguities.\index{overlap ambiguity}\index{ambiguity!overlap}\index{ambiguity!higher, $n$-|textbf}

\begin{definition}
Let $p \in Q_{\geq 0}$ be a path. If $p = qr$ for some paths $q, r$ we call $q$ a {\it proper left subpath} of $p$ if $p \neq q$. 

Now let $n \geq 0$. A path $p \in Q_\ldot$ is a {\it (left) $n$-ambiguity}\index{ambiguity!higher, $n$-|textbf} if there exist $u_0 \in Q_1$ and irreducible paths $u_1, \dotsc, u_{n+1}$ such that
\begin{enumerate}
\item $p = u_0 \dotsb u_{n+1}$
\item for all $i$, $u_i u_{i+1}$ is reducible, and $u_i d$ is irreducible for any proper left subpath $d$ of $u_{i+1}$.
\end{enumerate}
The notion of right $n$-ambiguity\index{ambiguity!higher, $n$-|textbf} may be defined analogously (as $p = v_0 \dotsb v_{n+1}$ with $v_{n+1} \in Q_1$ and $v_0, \dotsc, v_n$ irreducible paths satisfying analogous conditions), but these notions turn out to be equivalent \cite[Lem.~3.1]{bardzell2}.

The set of $n$-ambiguities\index{ambiguity!higher, $n$-|textbf} can be visualized as ``overlaps'' of at least $n+1$ elements in $S$, as illustrated in Fig.~\ref{figure:overlap} where we have also illustrated the paths $u_i, v_i$ appearing in the definition of left or right $n$-ambiguity\index{ambiguity!higher, $n$-|textbf}. (Note that for each $1 \leq i \leq n$ the overlap of $u_i$ and $v_i$ is an overlap of subpaths lying in $S$.)
\end{definition}

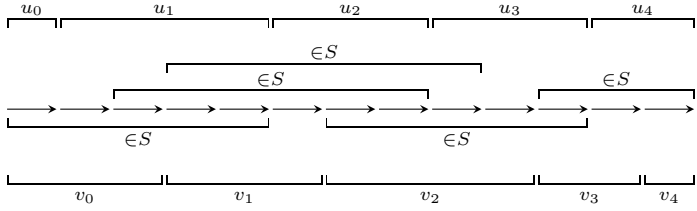
\begin{figure}
\centering
\begin{tikzpicture}[x=2em,y=1em]
\foreach \n in {0,1,2,3,4,5,6,7,8,9,10,11,12,13} {
\node[inner sep=0] (\n) at (\n,0) {};
}
\path[->,line width=.4pt] (0) edge (1);
\path[->,line width=.4pt] (1) edge (2);
\path[->,line width=.4pt] (2) edge (3);
\path[->,line width=.4pt] (3) edge (4);
\path[->,line width=.4pt] (4) edge (5);
\path[->,line width=.4pt] (5) edge (6);
\path[->,line width=.4pt] (6) edge (7);
\path[->,line width=.4pt] (7) edge (8);
\path[->,line width=.4pt] (8) edge (9);
\path[->,line width=.4pt] (9) edge (10);
\path[->,line width=.4pt] (10) edge (11);
\path[->,line width=.4pt] (11) edge (12);
\path[->,line width=.4pt] (12) edge (13);
\node at (2.5,-.9) {$\underbracket[.3pt]{\hspace{9.85em}}_{\in S}$};
\node at (5,.9)    {$\overbracket[.3pt]{\hspace{11.85em}}^{\in S}$};
\node at (6,1.9)   {$\overbracket[.3pt]{\hspace{11.85em}}^{\in S}$};
\node at (8.5,-.9) {$\underbracket[.3pt]{\hspace{9.85em}}_{\in S}$};
\node at (11.5,.9) {$\overbracket[.3pt]{\hspace{5.85em}}^{\in S}$};
\node at (.5,3.5)  {$\overbracket[.3pt]{\hspace{1.85em}}^{u_0}$};
\node at (3,3.5)   {$\overbracket[.3pt]{\hspace{7.85em}}^{u_1}$};
\node at (6.5,3.5) {$\overbracket[.3pt]{\hspace{5.85em}}^{u_2}$};
\node at (9.5,3.5) {$\overbracket[.3pt]{\hspace{5.85em}}^{u_3}$};
\node at (12,3.5)  {$\overbracket[.3pt]{\hspace{3.85em}}^{u_4}$};
\node at (1.5,-3)  {$\underbracket[.3pt]{\hspace{5.85em}}_{v_0}$};
\node at (4.5,-3)  {$\underbracket[.3pt]{\hspace{5.85em}}_{v_1}$};
\node at (8,-3)    {$\underbracket[.3pt]{\hspace{7.85em}}_{v_2}$};
\node at (11,-3)   {$\underbracket[.3pt]{\hspace{3.85em}}_{v_3}$};
\node at (12.5,-3) {$\underbracket[.3pt]{\hspace{1.85em}}_{v_4}$};
\end{tikzpicture}
\caption{A $3$-ambiguity in $S_5$}
\label{figure:overlap}
\end{figure}

Now set
\[
\begin{aligned}
S_0 &= Q_0 \\
S_1 &= Q_1 \\
S_2 &= S \subset Q_{\geq 2}
\end{aligned}
\]
and let $S_{n+2} \subset Q_{\geq n+2}$ for $n \geq 0$ denote the set of $n$-ambiguities\index{ambiguity!higher, $n$-|textbf}.

When $S \subset Q_2$, then overlaps can have no ``gaps'' and can thus be visualized as follows
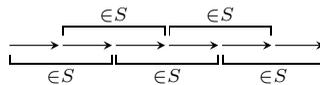
\[
\begin{tikzpicture}[x=2em,y=1em]
\foreach \n in {0,1,2,3,4,5,6} {
\node[inner sep=0] (\n) at (\n,0) {};
}
\path[->,line width=.4pt] (0) edge (1);
\path[->,line width=.4pt] (1) edge (2);
\path[->,line width=.4pt] (2) edge (3);
\path[->,line width=.4pt] (3) edge (4);
\path[->,line width=.4pt] (4) edge (5);
\path[->,line width=.4pt] (5) edge (6);
\node at (1,-.9) {$\underbracket[.3pt]{\hspace{3.85em}}_{\in S}$};
\node at (2,.9) {$\overbracket[.3pt]{\hspace{3.85em}}^{\in S}$};
\node at (3,-.9) {$\underbracket[.3pt]{\hspace{3.85em}}_{\in S}$};
\node at (4,.9) {$\overbracket[.3pt]{\hspace{3.85em}}^{\in S}$};
\node at (5,-.9) {$\underbracket[.3pt]{\hspace{3.85em}}_{\in S}$};
\end{tikzpicture}
\]
so that in this case
\[
S_n = \bigl\{ x_1 \dotsb x_n \in Q_n \bigm| x_i \in Q_1 \text{ \rm for all } i \text{ \rm and } x_{i-1} x_i \in S \bigr\} \subset Q_n
\]
(cf.\ \cite[Prop.~3.4]{chouhysolotar}).

\begin{remark}
The definition of $n$-ambiguities\index{ambiguity!higher, $n$-} was first given in Anick \cite{anick} under the name of ``$(n+1)$-chain''. This naming would be confusing in the context of this book since in \S\ref{subsection:projectiveresolutions} elements of the form $a \otimes w \otimes b$ for $w \in S_{n+2}$ will be $(n+2)$-chains in a projective $A$-bimodule resolution of $A$ (see Theorem \ref{theorem:projective} and Proposition \ref{proposition:resolution}).

Although our numbering also does not agree with the degree of the resulting element in the chain complex, we have that a $0$-ambiguity is just an element in $S$ (performing reductions is unambiguous) and a $1$-ambiguity is an overlap ambiguity (see Definition \ref{definition:overlap}).\index{overlap ambiguity}\index{ambiguity!overlap}\index{ambiguity!higher, $n$-}
\end{remark}

\subsection{The combinatorics of reductions}
\label{subsubsection:reductions}

We now study the higher combinatorics of reductions. The maps defined here will be used to give a recursive formula for the differential in the projective resolution of Theorem \ref{theorem:projective} and later to give an explicit description of the deformations in Chapter \ref{section:deformations-of-path-algebras}.

Generalizing (\ref{split_2}) we define maps $\mathrm{split}_n^*$, $\delta_n$ and $\gamma_{n-1}$ for any $n \geq 1$
\begin{flalign*}
&&&
\begin{tikzpicture}[baseline=-2.6pt,description/.style={fill=white,inner sep=1pt,outer sep=0}]
\matrix (m) [matrix of math nodes, row sep=0em, text height=1.5ex, column sep=3em, text depth=0.25ex, ampersand replacement=\&, inner sep=3pt]
{
\Bbbk Q \& A \otimes \Bbbk S_n \otimes A \&[-.5em] A \otimes \Bbbk S_{n-1} \otimes A \\
};
\path[->,line width=.4pt,font=\scriptsize]
(m-1-1) edge node[above=-.4ex] {$\mathrm{split}^*_n$} (m-1-2)
;
\path[->,line width=.4pt,font=\scriptsize]
(m-1-2) edge node[above=-.4ex] {$\delta_n$} (m-1-3)
;
\path[->,line width=.4pt,font=\scriptsize,out=200,in=340,looseness=1]
(m-1-3.185) edge node[below=-.4ex] {$\gamma_{n-1}$} (m-1-2.355)
;
\end{tikzpicture}
&& \mathllap{* \in \{ \varnothing, {\mathrm R}, {\mathrm L} \}}
\end{flalign*}
by
\begin{flalign*}
\label{split}
&& \mathrm{split}_n (p) &= \sum_{\substack{w \in S_n \\ q w r = p}} \pi (q) \otimes w \otimes \pi (r) &
\begin{aligned}
\mathrm{split}^{\mathrm R}_n (p) &= \pi (q) \otimes w_{\mathrm R} \otimes \pi (r) \\
\mathrm{split}^{\mathrm L}_n (p) &= \pi (q) \otimes w_{\mathrm L} \otimes \pi (r) \\
\end{aligned}
&&
\end{flalign*}
where $w_{\mathrm R}$ (resp.\ $w_{\mathrm L}$) is the rightmost (resp.\ leftmost) subpath of $p$ which is an element of $S_n$. We set $\mathrm{split}^*_n(p) = 0$ if there are no subpaths of $p$ contained in $S_n$.

Define $\delta_n$ as the $A$-bimodule morphism determined by
\[
\delta_n(1 \otimes w \otimes 1) = \begin{cases}
\pi (w) \otimes 1 - 1 \otimes \pi (w) & \text{if $n = 1$} \\
\mathrm{split}{}^{\mathrm R}_{n-1} (w) - \mathrm{split}{}^{\mathrm L}_{n-1} (w) & \text{if $n > 2$ is odd} \\
\mathrm{split}{}_{n-1} (w) & \text{if $n \geq 2$ is even}
\end{cases}
\]
and $\gamma_{n-1}$ as the morphism of {\it left} $A$-mod\-ules determined by 
\begin{flalign*}
&& \gamma_{n-1} (1 \otimes w \otimes \pi(u)) &= (-1)^n \mathrm{split}_n (wu) && \llap{for any $u \in \Irr_S$.}
\end{flalign*}

We will use the following identities for $n \geq 1$ in the proof of Proposition \ref{proposition:resolution}:
\begin{equation}
\label{gammadelta}
\begin{aligned}
\gamma_{n-1} \gamma_{n-2} = 0 \quad \text{and} \quad \gamma_{n-1} \delta_n \gamma_{n-1} = \gamma_{n-1}
\end{aligned}
\end{equation}
where we set $\gamma_{-1} \colon A \to A \otimes A$ defined by $\gamma_{-1} (a) = a \otimes 1$ for any $a \in A$.

The first identity in (\ref{gammadelta}) can be seen to hold since for any $w \in S_{n-2}$ and any irreducible path $u$ the path $wu$ never contains elements in $S_n$ as a subpath. From Sköldberg \cite[Thm.~1]{skoldberg} it is not difficult to show $\gamma_{n-1} = \gamma_{n-1} \gamma_{n-2} \delta_{n-1} + \gamma_{n-1} \delta_n \gamma_{n-1}$ so that the second identity in (\ref{gammadelta}) follows then directly from the first.

We make repeated use of the map $\mathrm{split}_1 \colon \Bbbk Q \to A \otimes \Bbbk Q_1 \otimes A$ which for $p = x_1 \cdots x_m$ with $x_i \in Q_1$ is given by
\begin{equation}
\label{split1}
\begin{split}
\mathrm{split}_1 (p) &= 1 \otimes x_1 \otimes \pi (x_2 \cdots x_m) \\
&\quad{} + \sum_{1 < i < m} \pi (x_1 \cdots x_{i-1}) \otimes x_i \otimes \pi (x_{i+1} \cdots x_m) \\
&\quad{}+ \pi (x_1 \cdots x_{m-1}) \otimes x_m \otimes 1.
\end{split}
\end{equation}
We define $\mathrm{split}_1 (e_i) = 0$ for each $i \in Q_0$.  Note that $\mathrm{split}_1$ satisfies a Leibniz rule
\begin{equation}
\label{split1pq}
\mathrm{split}_1 (p q) = \pi (p) \mathrm{split}_1 (q) + \mathrm{split}_1 (p) \pi (q).
\end{equation}

\chapter{Projective resolutions}
\label{section:resolution}

The Hochschild cohomology of an associative algebra $A$ can be calculated from any projective $A$-bimodule resolution of $A$. We recall the definition of the standard and the normalized bar resolutions (\S\ref{subsection:barresolutions}) and the construction of a resolution for path algebras of quivers with relations obtained from a reduction system\index{reduction system} (\S\ref{subsection:projectiveresolutions}).

\section{Bar resolutions}
\label{subsection:barresolutions}
\index{resolution!bar|textbf}

The {\it bar resolution} of an associative $\Bbbk$-algebra $A$ is given by
\[
\dotsb \to A \otimes_{\Bbbk} A^{\otimes_{\Bbbk} n} \otimes_{\Bbbk} A \to \dotsb \to A \otimes_{\Bbbk} A \otimes_{\Bbbk} A \to A \otimes_{\Bbbk} A  \to 0
\]
with differential $d_n \colon A \otimes_\Bbbk A^{\otimes_{\Bbbk} n} \otimes_\Bbbk A \to A \otimes_\Bbbk A^{\otimes_{\Bbbk} n-1} \otimes_\Bbbk A$ given by 
\begin{equation}
\label{differential}
\begin{aligned}
d_n (a_0 \otimes_{\Bbbk} a_{1 ... n} \otimes_{\Bbbk} a_{n+1}) &= a_0 a_1 \otimes_{\Bbbk} a_{2 ... n} \otimes_{\Bbbk} a_{n+1}\\
&\quad + \sum_{i=1}^{n-1} (-1)^i a_0 \otimes_{\Bbbk} a_{1 ... i-1} \otimes_{\Bbbk} a_i a_{i+1} \otimes_{\Bbbk} a_{i+2 ... n} \otimes_{\Bbbk} a_{n+1} \\
&\quad + (-1)^{n} a_0 \otimes_{\Bbbk} a_{1 ... n-1} \otimes_{\Bbbk} a_n a_{n+1}
\end{aligned}
\end{equation}
where for $i \leq j$ we have written $a_{i ... j}$ to denote $a_i \otimes_{\Bbbk} \dotsb \otimes_{\Bbbk} a_j \in A^{\otimes_{\Bbbk} (j-i+1)}$.

For $A = \Bbbk Q / I$ one can consider the ({\it $\Bbbk Q_0$-relative}) {\it normalized bar resolution}\index{resolution!bar!normalized}\index{resolution!bar} $\Bar_\ldot$
\[
\dotsb \to A \otimes \bar A^{\otimes n} \otimes A \to \dotsb \to A \otimes \bar A \otimes A \to A \otimes A  \to 0
\]
where here $\otimes = \otimes_{\Bbbk Q_0}$ (cf.\ Notation \ref{notation:tensorhom}) and $\bar A = A / (\Bbbk Q_0 \cdot 1_A)$ is the quotient $\Bbbk Q_0$-bimodule. The differential \eqref{differential} induces a differential on $\Bar_\ldot$, which we continue to denote by $d_n$, and we shall refer to this simply as ``the bar resolution''\index{resolution!bar} (cf.\ Remark \ref{relativestandard}). It is well known that $\Bar_\ldot$ is a projective $A$-bimodule resolution of $A$, where the augmentation map $d_0 \colon A \otimes A \to A$ is given by the multiplication of $A$ (see e.g.\ \cite[Ch.~1]{loday}).

\section{Projective resolutions from reduction systems}
\label{subsection:projectiveresolutions}
\index{resolution!Chouhy--Solotar|textbf}

For $A = \Bbbk Q / I$, a reduction system\index{reduction system} $R$ satisfying ($\diamond$) for $I$ gives rise to a much smaller resolution that uses the set of $n$-ambiguities\index{ambiguity!higher, $n$-} (see \S\ref{subsection:chains}). The resolution is based on the combinatorics of the reduction system $R = \{ (s, \varphi_s) \}$, replacing subpaths which lie in $S = \{ s \mid (s, \varphi_s) \in R \}$ by linear combinations of irreducible paths. (Under certain conditions this resolution can sometimes be shown to be minimal\index{resolution!minimal}, see \cite[Thm.~8.1]{chouhysolotar}.)

\begin{theorem}[(Chouhy--Solotar {\cite[\S 4]{chouhysolotar}})]
\label{theorem:projective}
\index{resolution!Chouhy--Solotar|textbf}
Let $A = \Bbbk Q / I$, let $R$ be a reduction system satisfying \textup{($\diamond$)} for $I$, let $S_m = Q_m$ for $m = 0, 1$ and let $S_{n+2}$ for $n \geq 0$ denote the set of $n$-ambiguities\index{ambiguity!higher, $n$-}.
\begin{enumerate}
\item There is a projective $A$-bimodule resolution $P_\ldot$ of $A$
\[
\dotsb \toarglong{\partial_{n+1}} P_n \toarglong{\partial_n} P_{n-1} \toarglong{\partial_{n-1}} \dotsb \toarglong{\partial_2} P_1 \toarglong{\partial_1} P_0 
\]
where $P_n = A \otimes \Bbbk S_n \otimes A$ and the augmentation map $\partial_0 \colon A \otimes A \to A$ is given by the multiplication of $A$. 
\item For each $n \geq 0$, there is a homomorphism of left $A$-modules $\rho_{n-1} \colon P_{n-1}\to P_n$, where we set $P_{-1} = A$, such that for any $n \geq 0$
\begin{align*}
\partial_{n} \rho_{n-1} + \rho_{n-2} \partial_{n-1} = \id_{P_{n-1}} \qquad\text{and}\qquad \rho_{n}(a \otimes w \otimes 1) = 0
\end{align*}
for any $w\in S_{n}$ and $a\in A$. Here we set $\partial_{-1} = 0$ and $\rho_{-2} = 0$.
\end{enumerate}
\end{theorem}

The actual maps in the resolution $P_\ldot$ of Theorem \ref{theorem:projective} shall be useful later, so we give explicit formulae for $\partial_n$ and $\rho_n$ using the maps $\gamma_n$ and $\delta_n$ defined in \S\ref{subsubsection:reductions}. 

\begin{proposition}
\label{proposition:resolution}
\index{resolution!Chouhy--Solotar|textbf}
Let $A = \Bbbk Q / I$ and let $R$ be a reduction system satisfying \textup{($\diamond$)} for $I$. Then the differential $\partial_\ldot$ and the homotopy $\rho_\ldot$ in Theorem \ref{theorem:projective} can be defined from   
\begin{align*}
\partial_0 (a \otimes b) = a b \qquad\text{and}\qquad \rho_{-1} (a) = a \otimes 1
\end{align*}
by the following recursive formulae for $n \geq 1$
\begin{flalign*}
&& \partial_n (a \otimes w \otimes b) &= a \big( (\id - \rho_{n-2} \partial_{n-1}) \delta_n \big) (1 \otimes w \otimes 1) b && \mathllap{a, b \in A, \, w \in S_n} \\
&& \rho_{n-1} &= \gamma_{n-1} + \textstyle\sum\limits_{i\geq 1} \gamma_{n-1} (\delta_n \gamma_{n-1} - \partial_n \gamma_{n-1})^i.
\end{flalign*}
\end{proposition}

\begin{proof}
Note that the formula for $\rho_{n-1}$ can be rewritten as 
$$
\rho_{n-1} = \gamma_{n-1} + \textstyle\sum\limits_{i \geq 1} \gamma_{n-1} (\id - \rho_{n-2} \partial_{n-1} - \partial_n \gamma_{n-1})^i
$$ 
since by $\gamma_{n-1} \gamma_{n-2} = 0$ in (\ref{gammadelta}) we get that $\gamma_{n-1} \rho_{n-2} = 0$ and thus
\begin{align*}
\gamma_{n-1} (\id - \rho_{n-2} \partial_{n-1} - \partial_n \gamma_{n-1})^i &= \gamma_{n-1} (\id - \partial_n \gamma_{n-1})^i = \gamma_{n-1} (\delta_n \gamma_{n-1} - \partial_n \gamma_{n-1})^i
\end{align*}
where the second identity uses $\gamma_{n-1} \delta_n \gamma_{n-1} = \gamma_{n-1}$ in (\ref{gammadelta}). 

We now prove that $\partial_n$ is a differential, i.e.\ that $\partial_n \partial_{n+1} = 0$, and $\rho_n$ is a homotopy, i.e.\ $\partial_n \rho_{n-1} + \rho_{n-2} \partial_{n-1} = \id$. (Note that this implies that $(P_\ldot, \partial_\ldot)$ is a projective resolution of $A$.) The proof proceeds by induction on $n$. For this, it is clear that $\partial_0\partial_{1} = 0$ and $\partial_0 \rho_{-1} + \rho_{-2} \partial_{-1} = \id$. Assume that this holds for all $i \leq n-1$. We need to prove it holds for $n$. For simplicity, we set
\begin{align}\label{align:Cn}
C_n := \id - \rho_{n-2} \partial_{n-1} - \partial_n \gamma_{n-1}
\end{align}
so that $\rho_{n-1} = \gamma_{n-1} +   \sum_{j\geq 1} \gamma_{n-1} C_n^j$. Here, $C_n^j$ denotes the $j$th iterated composition of $C_n$. We set $C_n^0 = \id$. Note that $\partial_{n-1} C_n = 0$ since 
\begin{align*}
\partial_{n-1} C_n  &= \partial_{n-1} - \partial_{n-1} \rho_{n-2} \partial_{n-1} -\partial_{n-1} \partial_n \gamma_{n-1} \\
&= \partial_{n-1} - (\id - \rho_{n-3} \partial_{n-2})\partial_{n-1} \\
&= 0
\end{align*}
where by the induction hypothesis we have $\partial_{n-1}\partial_n = 0 = \partial_{n-2} \partial_{n-1}$.  

Let us prove that $(\partial_n\rho_{n-1} + \rho_{n-2}\partial_{n-1})(y)=y$ for any nonzero $y \in P_{\ldot}$. By the proof of \cite[Prop.~5.8]{chouhysolotar} we have that $C_n^i (y)  = 0$ for $i \gg 0$. Then there exists $k \geq 1$ (depending on $y$) such that  $C_n^k(y)=0$ but $C_{n}^{k-1}(y)\neq 0$. 

We claim that for any $0\leq i\leq k$ we have
\begin{equation}
\label{partialrhocn}
(\partial_n \rho_{n-1} + \rho_{n-2} \partial_{n-1}) (C_n^{k-i}(y)) = C_n^{k-i}(y). 
\end{equation}
Note that \eqref{partialrhocn} yields the desired identity since $C_n^0(y)=y$.
Indeed, \eqref{partialrhocn} can be proved by induction on $i$ as follows. Clearly, it holds for $i = 0$ since $C_n^k(y) = 0$. Assume that (\ref{partialrhocn}) holds for $i-1 < k$.  Then it also holds for $i$ since
\begin{align*}
& (\partial_n \rho_{n-1} + \rho_{n-2} \partial_{n-1})(C_n^{k - i}(y))\\
={} & (\partial_n \gamma_{n-1} + \textstyle\sum\limits_{j\geq 1} \partial_n \gamma_{n-1} C_n^{j} + \rho_{n-2} \partial_{n-1})(C_n^{k - i}(y))\\
={} &\textstyle\sum\limits_{j\geq 1} \partial_n \gamma_{n-1} C_n^{k - i  + j}(y) +  C_n^{k-i}(y) - C_n^{k-(i-1)}(y) \\
={} & C_n^{k-i}(y) - \rho_{n-2} \partial_{n-1}(C_n^{k-(i-1)}(y))\\
={} & C_n^{k-i}(y)
\end{align*}
where the first identity follows from $\rho_{n-1} = \gamma_{n-1} + \sum_{j\geq 1} \gamma_{n-1} C_n^j$, the second identity from \eqref{align:Cn}, the third identity from the induction hypothesis $(\partial_n \rho_{n-1} + \rho_{n-2} \partial_{n-1}) (C_n^{k-(i-1)}(y)) = C_n^{k-(i-1)}(y)$, and the fourth identity from $\partial_{n-1} C_n = 0$. This proves the claim. 

Let us prove that $\partial_{n}\partial_{n+1}=0$. For this, we take $y: = \partial_n \delta_{n+1} (1 \otimes w \otimes 1)$ into $(\partial_n \rho_{n-1} + \rho_{n-2} \partial_{n-1}) (y) = y$ and get
\begin{flalign*}
&& \partial_n \rho_{n-1} \partial_n \delta_{n+1} (1 \otimes w \otimes 1) = \partial_n \delta_{n+1}(1 \otimes w \otimes 1) && \llap{for $w\in S_{n+1}$.}
\end{flalign*}
This yields
$\partial_n\partial_{n+1}(1\otimes w\otimes 1)=(\partial_n \delta_{n+1}-\partial_n\rho_{n-1}\partial_n \delta_{n+1})(1\otimes w\otimes 1)=0.$
\end{proof}

\begin{lemma}
\label{lemma:rhosplit}
\begin{enumerate} 
\item \label{rhosplit1} For $n \geq 0$ we have
$$
\rho_{n} = \gamma_{n} + \rho_{n} (\delta_{n+1} \gamma_{n} - \partial_{n+1} \gamma_{n}).
$$

\item \label{rhosplit2} For $u \in \Irr_S$ an irreducible path and for $x \in Q_1$, we have the following useful identities
\begin{align*}
\partial_1 (a\otimes x\otimes b) & = a\pi(x)\otimes b-a\otimes \pi(x)b\\
\partial_2 (a \otimes s \otimes b) & = a \, \mathrm{split_1} (s-\varphi_s) b \\
\rho_0 (a \otimes \pi(u)) &= -a \, \mathrm{split}_1 (u) \\
\rho_1 (a \otimes x \otimes \pi(u)) &= \begin{cases}
0 & \text{if $xu \in \Irr_S$}\\
a \otimes s \otimes 1 + a \rho_1(\mathrm{split}_1(\phi_s) \pi(v))
 & \text{if $xu = sv$ with $s \in S$.}
\end{cases}
\end{align*}
\end{enumerate} 
\end{lemma}

\begin{proof}
The first assertion follows since by definition  
\begin{align*}
\rho_{n} &= \gamma_{n} + \Big( \gamma_{n} + \textstyle\sum\limits_{i\geq 1} (\gamma_{n} \delta_{n+1} - \gamma_{n} \partial_{n+1})^i \gamma_{n} \Big) (\delta_{n+1} \gamma_{n} - \partial_{n+1} \gamma_{n}) \\
&= \gamma_{n} + \rho_{n} (\delta_{n+1} \gamma_{n} - \partial_{n+1} \gamma_{n}).
\end{align*}
Let us prove the second assertion. Since $\partial_0 \delta_1 =0$ it follows that $$\partial_1 = (\id - \rho_{-1} \partial_0)  \delta_1 = \delta_1$$ 
which yields the first identity. Thus by \ref{rhosplit1} for $n = 0$ we have $$\rho_0 (a \otimes \pi(u)) = \gamma_0 (a \otimes \pi (u)) = -a \, \mathrm{split}_1 (u).$$

Let us compute $\partial_2$: 
\begin{align*}
\partial_2(a \otimes s \otimes b) & = a (\id - \rho_0 \partial_1) \delta_2(1 \otimes s \otimes 1) b\\
& =a\, \mathrm{split}_1(s) b - a \rho_0 ( \phi_s \otimes 1 - 1 \otimes \phi_s)b \\
& = a\, \mathrm{split}_1(s - \phi_s)b 
\end{align*}
where the third identity follows from the fact that $\rho_0 ( \phi_s \otimes 1) = 0$. 

If $xu\in \Irr_S$ is irreducible, then $\gamma_1 (a \otimes x \otimes \pi(u)) = 0$. It follows from \ref{rhosplit1} that $\rho_1 (a \otimes x \otimes \pi(u)) = 0$.

If $xu = s v$ with $s \in S$, then $\gamma_1 (a \otimes x \otimes \pi(u)) = a \otimes s \otimes \pi(v)$. By \ref{rhosplit1} we obtain $
\rho_1 (a \otimes x \otimes \pi(u)) = a \otimes s \otimes \pi(v) + a \rho_1(\mathrm{split}_1(\phi_s) \pi(v)).$ 
\end{proof}

We give two simple and well-known examples of the resolution\index{resolution!Chouhy--Solotar} $P_\ldot$ in Theorem \ref{theorem:projective} which may be obtained from the reduction systems of Examples \ref{example:monomialreduction} and \ref{example:polynomialreduction}.

\begin{example}
\label{example:bardzell}
{\it Bardzell's resolution\index{resolution!Bardzell|textbf} of monomial algebras.}\; Let $A = \Bbbk Q /\langle S \rangle$ be a monomial algebra and let $R = \{(s, 0) \mid s\in S\}$ be the reduction system for $A$. Then the resolution\index{resolution!Chouhy--Solotar} $P_\ldot$ coincides with Bardzell's resolution \cite{bardzell1,bardzell2} and the homotopy $\rho_\ldot $ with the one given in Sköldberg \cite{skoldberg}. 
\end{example}

\begin{example}
\label{example:polynomialkoszul}
{\it Koszul resolution\index{resolution!Koszul}\index{Koszul!resolution} of the polynomial algebra.}\; Let $A = \Bbbk [x_1, \dotsc, x_d]$ and let $R = \bigl\{ (x_j x_i, x_i x_j) \bigr\}_{1 \leq i < j \leq d}$ be the reduction system in Example \ref{example:polynomialreduction}. Then we have that $S = \{ x_j x_i\}_{1\leq i < j \leq d}$ and 
\begin{align*}
S_k & =\begin{cases} \{ x_{i_k} x_{i_{k-1}} \dotsb x_{i_1} \}_{1 \leq i_1 < \dotsb < i_k \leq d} & \text{for $1\leq k\leq d$}\\
 \emptyset & \text{for $k > d$}.
 \end{cases}
\end{align*}
The resolution $P_\ldot$ is given by 
\[
0 \toarglong{} P_d \toarglong{\partial_d} P_{d-1} \toarglong{\partial_{d-1}} \dotsb \toarglong{\partial_2} P_1 \toarglong{\partial_1} P_0 \toarglong{\partial_0} A
\]
with differential
\begin{align*}
\partial_k(1 \otimes x_{i_k} \dotsb x_{i_1} \otimes 1) ={} \sum_{j=1}^k (-1)^{k-j} \biggl( \biggl\lceil
\begin{aligned}
x_{i_j} & \otimes x_{i_k} \dotsb \widehat x_{i_j} \dotsb x_{i_1} \otimes 1 \\ - 1 &\otimes x_{i_k} \dotsb \widehat x_{i_j} \dotsb x_{i_1} \otimes x_{i_j}
\end{aligned}
\biggr\rfloor \biggr)
\end{align*}
for any $x_{i_k} x_{i_{k-1}} \dotsb x_{i_1} \in S_k$. (Here we use $\lceil \blank \rfloor$ to indicate a line break within the parentheses.) Note that $P_\ldot$ coincides with the usual Koszul resolution\index{resolution!Koszul}\index{Koszul!resolution} of $A$ via the map $1 \otimes x_{i_k} \cdots x_{i_1} \otimes 1 \mapsto 1 \otimes x_{i_k} \wedge \dotsb \wedge x_{i_1} \otimes 1$.
\end{example}

\chapter{Homotopy deformation retract}
\label{section:homotopydeformationretract}

In this section we construct a homotopy deformation retract between the projective resolution\index{resolution!Chouhy--Solotar} $(P_\ldot, \partial_\ldot)$ given in Theorem \ref{theorem:projective} and the bar resolution\index{resolution!bar} $(\Bar_\ldot, d_\ldot)$ (see \S\ref{subsection:barresolutions}).

\begin{notation}
Let $\varpi \colon A \to \bar A = A / (\Bbbk Q_0 \cdot 1_A)$ denote the natural projection and for $n \geq 0$ let $\varpi_n = \id_{A \otimes \bar A^{\otimes n}} \otimes \varpi$ so that
\begin{align}
\label{varpi}
\varpi_n \colon A \otimes \bar A^{\otimes n} \otimes A \to A \otimes \bar A^{\otimes n} \otimes\bar A.
\end{align}
We also set $\varpi_{-1} = \id_A\colon A \to A$.

Given a path $p \in \Bbbk Q$ we write $\bar p = \varpi \pi (p)$, where $\pi \colon \Bbbk Q \to \Bbbk Q / I = A$ is the natural projection.
\end{notation}

Note that by the definition of $d_n$, the maps $\varpi_n$ satisfy the following identities
\begin{flalign}
\label{varpid}
&& d_{n+1} (\varpi_n (y) \otimes b) &= \varpi_{n-1} d_n (y) \otimes b + (-1)^{n+1} y b && \mathllap{\text{for } n \geq 0.}
\end{flalign}

\section{Comparison maps and homotopy}

We first define comparison maps
\begin{align*}
\begin{tikzpicture}[baseline=-2.6pt,description/.style={fill=white,inner sep=1pt,outer sep=0}]
\matrix (m) [matrix of math nodes, row sep=0em, text height=1.5ex, column sep=3em, text depth=0.25ex, ampersand replacement=\&, inner sep=3pt]
{
(P_\ldot, \partial_\ldot) \& (\Bar_\ldot, d_\ldot) \\
};
\path[->,line width=.4pt,font=\scriptsize, transform canvas={yshift=.4ex}]
(m-1-1) edge node[above=-.4ex] {$F_\ldot$} (m-1-2)
;
\path[->,line width=.4pt,font=\scriptsize, transform canvas={yshift=-.4ex}]
(m-1-2) edge node[below=-.4ex] {$G_\ldot$} (m-1-1)
;
\end{tikzpicture}
\end{align*}
where $F_\ldot$ and $G_\ldot$ are morphisms of complexes of $A$-bimodules satisfying $G_\ldot F_\ldot = \id_{P_\ldot}$. (In the monomial case similar comparison maps were given in Redondo--Román \cite{redondoroman}.)

The complexes $P_\ldot$ and $\Bar_\ldot$ are zero in negative degrees. However, it is convenient to set $P_{-1} = \Bar_{-1} = A$ and $F_{-1} = G_{-1} = \varpi_{-1} = \id_A$ and let $\partial_0 = d_0 \colon A \otimes A \to A$ denote the augmentations given by multiplication in $A$.

We can then define
\begin{flalign*}
&& \begin{tikzpicture}[baseline=-2.6pt,description/.style={fill=white,inner sep=1pt,outer sep=0}]
\matrix (m) [matrix of math nodes, row sep=0em, text height=1.5ex, column sep=3em, text depth=0.25ex, ampersand replacement=\&, inner sep=3pt]
{
A \otimes \Bbbk S_n \otimes A \& A \otimes \bar A^{\otimes n} \otimes A \\
};
\path[->,line width=.4pt,font=\scriptsize, transform canvas={yshift=.4ex}]
(m-1-1) edge node[above=-.4ex] {$F_n$} (m-1-2)
;
\path[->,line width=.4pt,font=\scriptsize, transform canvas={yshift=-.4ex}]
(m-1-2) edge node[below=-.4ex] {$G_n$} (m-1-1)
;
\end{tikzpicture}
&& \mathllap{n \geq 0}
\end{flalign*}
recursively by
\begin{align}
F_n (a \otimes \eqmakebox[usw]{$w$} \otimes b) &= (-1)^{n} \varpi_{n-1} F_{n-1} \partial_n (a \otimes w \otimes 1) \otimes b \label{Fn} \\
G_n (a \otimes \eqmakebox[usw]{$y$} \otimes b) &= (\rho_{n-1} G_{n-1} d_n (a \otimes y \otimes 1)) b \label{Gn}
\end{align}
where $\rho_{n-1} \colon A \otimes \Bbbk S_{n-1} \otimes A \to A \otimes \Bbbk S_n \otimes A$ are the morphisms of {\it left} $A$-mod\-ules in Proposition \ref{proposition:resolution}. (Note that $P_0 = \Bar_0 = A \otimes A$ and $F_0 = G_0 = \id_{A \otimes A}$.)

\begin{remark}
Since $\varpi (1) = 0$, the definitions of $F_n$ and $\varpi_n$ immediately give that
\begin{flalign}
\label{Fn1}
&& \varpi_n F_n (a \otimes w \otimes 1) &= 0 && \mathllap{\text{for } n \geq 0.}
\end{flalign}
\end{remark}

\begin{lemma}\label{lemma:FG}
$F_n$ and $G_n$ satisfy the following identities for $n \geq 0$
\begin{enumerate}
\item \eqmakebox[lemmaFG][r]{$d_n F_n = {}$}$F_{n-1} \partial_n$ \label{FG1}
\item \eqmakebox[lemmaFG][r]{$\partial_n G_n = {}$}$G_{n-1} d_n$ \label{FG2}
\item \eqmakebox[lemmaFG][r]{$G_{n-1} F_{n-1} = {}$}$\id_{P_{n-1}}$. \label{FG3}
\end{enumerate}
\end{lemma}
\begin{proof}
Each part can be proved by induction on $n$ using the identity (\ref{varpid}), the recursive definitions of $F_n$ and $G_n$ given in (\ref{Fn}) and (\ref{Gn}), the fact that $\rho_\ldot$ is a homotopy retract and the fact that $d_{n-1} d_n = 0$ and $\partial_{n-1} \partial_n = 0$. For each part the case $n = 0$ is immediate from the definitions. A proof can also be found in \cite[Lem.~2.4 \& Lem.~2.5]{ivanovivanovvolkovzhou}.

We illustrate the proof of \ref{FG3}, where indeed $G_0 F_0 = \id_{A \otimes A}$ by definition. For $n \geq 1$, we then have 
\begin{align*}
G_n F_n (a \otimes w \otimes b) &= (-1)^n (\rho_{n-1} G_{n-1} d_n (\varpi_{n-1} F_{n-1} \partial_n (a \otimes w \otimes 1) \otimes 1)) b \\
&= (-1)^{n} (\rho_{n-1} G_{n-1} (\varpi_{n-2} F_{n-2} \partial_{n-1} \partial_n (a \otimes w \otimes 1) \otimes 1)) b \\ &\quad + (\rho_{n-1} G_{n-1} F_{n-1} \partial_n (a \otimes w \otimes 1)) b \\
&= (\rho_{n-1} \partial_n (a \otimes w \otimes 1)) b \\
&= a \otimes w \otimes b
\end{align*}
where the second identity follows from (\ref{varpid}) and $d_{n-1}F_{n-1} = F_{n-2}\partial_{n-1}$, the third identity  from $\partial_{n-1}\partial_n = 0$ and the induction hypothesis $G_{n-1}F_{n-1}=\id$, and the last identity  from $\id_{P_n} = \rho_{n-1} \partial_n + \partial_{n+1} \rho_{n}$ and $\rho_{n} (a \otimes w \otimes 1) = 0$. 
\end{proof}

We can now give a homotopy $h_\ldot$ by $A$-bimodule morphisms $h_n \colon \Bar_n \rightarrow \Bar_{n+1}$ for $n \geq 0$, defining 
\begin{align}
\label{varphi1}
h_n (a_0 \otimes \bar a_{1 ... n} \otimes a_{n+1}) &= \sum_{i=1}^{n}(-1)^{i+1} \varpi_i F_i G_i (a_0 \otimes \bar a_{1 ... i} \otimes 1) \otimes \bar a_{i+1 ... n} \otimes a_{n+1}. 
\end{align}
Note that for $n = 0$ the sum is empty and $h_0 = 0$. Alternatively, $h_n$ can be defined recursively for $n \geq 1$ by
\begin{align}
\label{varpih}
\begin{split}
h_n (a_0 \otimes \bar a_{1 ... n} \otimes a_{n+1}) &= \varpi_n h_{n-1}(a_0 \otimes \bar a_{1 ... n-1}\otimes a_{n})\otimes a_{n+1} \\
&\quad {}+ (-1)^{n+1} \varpi_n F_n G_n (a_0 \otimes \bar a_{1 ... n} \otimes 1) \otimes a_{n+1}.
\end{split}
\end{align}

\begin{theorem}\label{theorem:hdr}
We have a homotopy deformation retract
\begin{align}
\begin{tikzpicture}[baseline=-2.6pt,description/.style={fill=white,inner sep=1pt,outer sep=0}]
\matrix (m) [matrix of math nodes, row sep=0em, text height=1.5ex, column sep=3em, text depth=0.25ex, ampersand replacement=\&, inner sep=1.5pt]
{
(P_\ldot, \partial_\ldot) \& (\Bar_\ldot, d_\ldot) \\
};
\path[->,line width=.4pt,font=\scriptsize, transform canvas={yshift=.4ex}]
(m-1-1) edge node[above=-.4ex] {$F_\ldot$} (m-1-2)
;
\path[->,line width=.4pt,font=\scriptsize, transform canvas={yshift=-.4ex}]
(m-1-2) edge node[below=-.4ex] {$G_\ldot$} (m-1-1)
;
\path[->,line width=.4pt,font=\scriptsize, looseness=8, in=30, out=330]
(m-1-2.355) edge node[right=-.4ex] {$h_\ldot$} (m-1-2.5)
;
\end{tikzpicture}
\end{align}
namely $G_n F_n = \id$ and $F_n G_n - \id = h_{n-1} d_n + d_{n+1} h_n$ for each $n \geq 0$ (setting $h_{-1} = 0$).
Moreover, the homotopy deformation retract is {\rm special}, i.e.\ it satisfies
\begin{flalign}
\label{special}
&& h_n F_n &= 0 && G_{n+1} h_n=0 && h_{n+1} h_n=0. &&
\end{flalign}
\end{theorem}

\begin{proof}
The first identity is proved in Lemma \ref{lemma:FG} \ref{FG3}. The other identities can be proved similarly by induction, using the identities of Lemma \ref{lemma:FG} as well as identities (\ref{varpid})--(\ref{Fn1}) and (\ref{varpih}) as well as $\rho_n (a \otimes w \otimes 1) = 0$ (cf.\ Theorem \ref{theorem:projective}).
\end{proof}

\begin{remark}
\label{relativestandard}
For an augmented algebra $A \toarg{\epsilon} \Bbbk Q_0$ (i.e.\ augmented over $\Bbbk Q_0 \simeq \Bbbk \times \dotsb \times \Bbbk$), the maps $\bar A \hookrightarrow A$ given by $\bar A \simeq \ker \epsilon \subset A$ and $\varpi \colon A \to \bar A = A / (\Bbbk Q_0 \cdot 1_A)$ induce comparison maps between the ($\Bbbk Q_0$-relative) normalized bar resolution\index{resolution!bar}\index{resolution!bar!normalized} and the standard bar resolution. Moreover, it is straightforward to construct a homotopy deformation retract using a formula similar to (\ref{varpih}), so that $\Bar_\ldot$ in Theorem \ref{theorem:hdr} can also be taken to be the standard bar resolution. The homotopy $h_\ldot$ is inspired by He--Li--Li \cite[Def.~3.3]{helili}.
\end{remark}

\section{Maps in low degrees}
\label{subsection:lowdegrees}

In this section we give explicit expressions for the maps $\partial_\ldot$, $F_\ldot$ and $G_\ldot$ in low degrees, namely for the labelled maps in
\begin{align*}
\begin{tikzpicture}[baseline=-2.6pt,description/.style={fill=white,inner sep=1.75pt}]
\matrix (m) [matrix of math nodes, row sep={5em,between origins}, text height=1.5ex, column sep={5em,between origins}, text depth=0.25ex, ampersand replacement=\&]
{
\cdots \&    P_3 \&    P_2 \&    P_1 \&    P_0 \\
\cdots \& \Bar_3 \& \Bar_2 \& \Bar_1 \& \Bar_0. \\
};
\path[->,line width=.4pt]
(m-1-1) edge (m-1-2)
(m-1-2) edge node[above=-.4ex,font=\scriptsize] {$\partial_3$} (m-1-3)
(m-1-3) edge node[above=-.4ex,font=\scriptsize] {$\partial_2$} (m-1-4)
(m-1-4) edge node[above=-.4ex,font=\scriptsize] {$\partial_1$} (m-1-5)
(m-2-1) edge (m-2-2)
(m-2-2) edge node[above=-.4ex,font=\scriptsize] {} (m-2-3)
(m-2-3) edge node[above=-.4ex,font=\scriptsize] {} (m-2-4)
(m-2-4) edge node[above=-.4ex,font=\scriptsize] {} (m-2-5)
;
\path[->,line width=.4pt,transform canvas={xshift=2pt}]
(m-1-2) edge node[right=-.3ex,font=\scriptsize,pos=.55] {$F_3$} (m-2-2)
(m-1-3) edge node[right=-.3ex,font=\scriptsize,pos=.55] {$F_2$} (m-2-3)
(m-1-4) edge node[right=-.3ex,font=\scriptsize,pos=.55] {$F_1$} (m-2-4)
;
\path[->,line width=.4pt,transform canvas={xshift=-2pt}]
(m-2-2) edge (m-1-2)
(m-2-3) edge node[left=-.3ex,font=\scriptsize,pos=.45] {$G_2$} (m-1-3)
(m-2-4) edge node[left=-.3ex,font=\scriptsize,pos=.45] {$G_1$} (m-1-4)
;
\path[->,line width=.4pt,out=200,in=340,looseness=1]
(m-1-5) edge node[below=-.4ex,font=\scriptsize] {$\rho_0$} (m-1-4)
(m-1-4) edge node[below=-.4ex,font=\scriptsize] {$\rho_1$} (m-1-3)
(m-1-3) edge (m-1-2)
(m-1-2) edge (m-1-1)
;
\draw [double equal sign distance,line width=.4pt]
(m-1-5) to (m-2-5)
;
\end{tikzpicture}
\end{align*}
These maps will be used in Chapter \ref{section:deformations-of-path-algebras} to describe the deformation theory of $A = \Bbbk Q / I$.

Recall from Lemma \ref{lemma:rhosplit} that the differential $\partial_\ldot$ and the homotopy $\rho_\ldot$ are given in low degrees by
\begin{align*}
\partial_1 (a\otimes x\otimes b) & = a\pi(x)\otimes b-a\otimes \pi(x)b\\
\partial_2 (a \otimes s \otimes b) & = a \, \mathrm{split_1} (s-\varphi_s) b \\
\rho_0 (a \otimes \pi(u)) &= -a \, \mathrm{split}_1 (u) \\
\rho_1 (a \otimes x \otimes \pi(u)) &= \begin{cases}
0 & \text{if $xu \in \Irr_S$}\\
a \otimes s \otimes 1 + a \rho_1(\mathrm{split}_1(\phi_s) \pi(v))
 & \text{if $xu = sv$.}
\end{cases}
\end{align*}
(A formula for $\mathrm{split}_1$ was given in (\ref{split1}).)

\begin{lemma}
\label{lemma:low-degree}
In low degrees we have the following formulae for $\partial_\ldot$, $F_\ldot$ and $G_\ldot$.
\begin{enumerate}[itemsep=.75em]
\item \eqmakebox[lowdeg][r]{$G_1 (a \otimes \bar u \otimes b) = {}$}$a \, \mathrm{split}_1 (u) \hair b$ \\[.4em]
      \eqmakebox[lowdeg][r]{$G_2 (a \otimes \bar u \otimes \bar v \otimes b) = {}$}$a \hair \rho_1 (\mathrm{split}_1 (u) \pi(v)) \hair b$
\label{lowdeg1}
\item \eqmakebox[lowdeg][r]{$\partial_3 (a \otimes \makebox[.7em]{$w$} \otimes b) = {}$}$\delta_3 (a \otimes w \otimes b) + G_2 (a \otimes \bar u \otimes \bar \varphi_s \otimes b)$ \\[.4em]
      \eqmakebox[lowdeg][r]{}${} - G_2 (a \otimes \bar \varphi_{s'} \otimes \bar v \otimes b)$
\label{lowdeg2}
\item \eqmakebox[lowdeg][r]{$F_1 (a \otimes \makebox[.7em]{$x$} \otimes b) = {}$}$ a \otimes \bar x \otimes b$ \\[.4em]
      \eqmakebox[lowdeg][r]{$F_2 (a \otimes \makebox[.7em]{$s$} \otimes b) = {}$}$a \hair \varpi_1\mathrm{split_1}(s - \varphi_s)\otimes b$ \\[.4em]
      \eqmakebox[lowdeg][r]{$F_3 (a \otimes \makebox[.7em]{$w$} \otimes b) = {}$}$  a \hair \varpi_1 \mathrm{split}_1(s' - \phi_{s'}) \otimes \bar v \otimes b$ \\[.4em]
      \eqmakebox[lowdeg][r]{}${} - a \hair \varpi_2 F_2 G_2 (1\otimes \bar u \otimes \bar \varphi_s \otimes 1 - 1 \otimes \bar \varphi_{s'} \otimes \bar v \otimes 1 ) \otimes b$
\label{lowdeg3}
\end{enumerate}
where $a, b \in A$, and in \textup{\ref{lowdeg1}} $u, v$ are arbitrary irreducible paths and in \textup{\ref{lowdeg2}--\ref{lowdeg3}} $x \in Q_1$, $s, s' \in S$ and $w \in S_3$ with $w = u s = s' v$ for some (irreducible) paths $u, v$.
\end{lemma}

\begin{proof}
The formulae for  $G_1$ and $F_1, F_2$ follow from the definitions. Let us compute $G_2$:
\begin{align*}
G_2 (a \otimes \bar u \otimes \bar v \otimes b) &= (\rho_1 G_1 d_2 (a \otimes \bar u \otimes \bar v \otimes 1)) b \\
&= a \rho_1 \Big(\pi(u) \mathrm{split}_1 (v) - \mathrm{split}_1 (\sigma \pi (uv)) +  \mathrm{split_1} (u) \pi (v)\Big) b \\
&= a \rho_1 (\mathrm{split}_1 (u) \pi(v)) b
\end{align*} 
where the third identity follows from the identity $\rho_1 (a \hair \mathrm{split}_1 (u)) = 0$ for any irreducible path $u$ (see Lemma \ref{lemma:rhosplit} \ref{rhosplit2}). 

To compute $\partial_3$, let $w \in S_3$. Then
\begin{align*}
\partial_3 (a \otimes w \otimes b) &= \delta_3 (a \otimes w \otimes b) - a(\rho_1 \partial_2 \delta_3 (1 \otimes w \otimes 1)) b \\
&= \delta_3 (a \otimes w \otimes b) - a \rho_1 (\pi(u) \mathrm{split_1} (s)) b + a \rho_1 (\pi (u) \mathrm{split_1} (\varphi_s)) b \\
& \quad + a \rho_1 (\mathrm{split_1} (s') \pi (v)) b - a \rho_1 (\mathrm{split_1} (\varphi_{s'}) \pi (v)) b \\
&= \delta_3 (a \otimes w \otimes b) + G_2 (a \otimes \bar u \otimes \bar \varphi_s \otimes b) - G_2 (a \otimes \bar \varphi_{s'} \otimes \bar v \otimes b)
\end{align*}
where the third identity follows since by Lemma \ref{lemma:rhosplit} \ref{rhosplit2} we have $\rho_1 (\pi (u) \mathrm{split_1} (\varphi_s)) = 0$ and by \eqref{split1pq} we have 
\begin{align*}
 - a \rho_1 (\pi(u) \mathrm{split_1} (s)) b + a \rho_1 (\mathrm{split_1} (s') \pi (v)) b &= a \rho_1(\mathrm{split_1} (u) \pi (\varphi_s)) b \\ &\quad -  a \rho_1( \pi(\phi_{s'}) \mathrm{split_1} (v)) b \\
 & = G_2 (a \otimes \bar u \otimes \bar \varphi_s \otimes b). 
\end{align*} 

The formula of $F_3$ follows straight from the formulae of $\partial_3$ and $F_2$ and the identity $\varpi_2 F_2 (\pi(u) \otimes s \otimes 1) = 0$.
\end{proof}

The following formulae for $G_2$ and $h_2$ will be used in Chapter \ref{section:deformations-of-path-algebras}.

\begin{lemma}\label{lemma:G2}
Let $u$ be an irreducible path.
\begin{enumerate}
\item \label{G2i} For $x \in Q_1$ such that $x u = s v$ with $s \in S$ we have 
\begin{align*}
G_2 (1 \otimes \bar x \otimes \bar u \otimes 1) &= 1 \otimes s \otimes \pi (v) + G_2 (1\otimes \bar \varphi_s \otimes \bar v \otimes 1)\\
h_2 (1 \otimes \bar x \otimes \bar u \otimes 1) & = - \varpi_1 \, \mathrm{split}_1(s) \otimes \bar v \otimes 1 + h_2(1 \otimes \bar \phi_s \otimes \bar v \otimes 1)\\
\intertext{In particular, if $x u = s$ then}
G_2 (1 \otimes \bar x \otimes \bar u \otimes 1) &= 1 \otimes s \otimes 1\\
h_2 (1 \otimes \bar x \otimes \bar u \otimes 1) &= 0. 
\end{align*}
\item \label{G2ii} Write $u = u_1 \cdots u_n$ with $u_i \in Q_1$ for $i = 1, \dotsc, n$. Then for any (irreducible) path $v$ we have
\begin{align*}
G_2 (1 \otimes \bar u \otimes \bar v \otimes 1) & = \sum_{j=1}^n \pi (u_1 \cdots u_{j-1}) G_2 (1 \otimes \bar u_j \otimes \overbar{u_{j+1}\cdots u_n v} \otimes 1)\\
h_2 (1 \otimes \bar u \otimes \bar v \otimes 1) & = \sum_{j=1}^{n-1} \pi(u_1\dotsc u_{j-1}) \otimes \bar u_j \otimes \overbar{u_{j+1}\dotsc u_n} \otimes \bar v \otimes 1 \\
&\quad +\sum_{j=1}^n \pi (u_1 \cdots u_{j-1}) h_2 (1 \otimes \bar u_j \otimes \overbar{u_{j+1}\cdots u_n v} \otimes 1).
\end{align*}
In particular, if $uv$ is also irreducible then we have 
$$G_2 (1 \otimes \bar u \otimes \bar v \otimes 1)  = 0.$$
\end{enumerate}
\end{lemma}
\begin{proof}
Let us prove \ref{G2i}. It follows from Lemma \ref{lemma:low-degree} that 
\begin{align*}
G_2 (1\otimes \bar x \otimes \bar u \otimes 1) &= \rho_1 (\mathrm{split_1} (x) \pi (u)) \\
&= 1 \otimes s \otimes \pi(v) + \rho_1 (\mathrm{split_1} (\varphi_s) \pi (v)) \\
&= 1 \otimes s \otimes \pi(v) + G_2 (1 \otimes \bar \varphi_s \otimes \bar v \otimes 1)
\end{align*}
where the second identity follows from Lemma \ref{lemma:rhosplit} \ref{rhosplit2}.  Combining this with \eqref{varphi1} and Lemma \ref{lemma:low-degree} \ref{lowdeg2}, we obtain the formula for $h_2(1 \otimes \bar x \otimes \bar u \otimes 1)$. Here we use $\varpi_1 F_1G_1(1 \otimes \bar x \otimes 1) = 0$.

Let us prove \ref{G2ii}. We have \begin{align*}
G_2 (1 \otimes \bar u \otimes \bar v \otimes 1) &= \rho_1 (\mathrm{split_1} (u) \pi(v)) \\
&= \sum_{j=1}^n \rho_1 (\pi (u_1 \cdots u_{j-1}) \otimes \bar u_j \otimes \pi(u_{j+1}\cdots u_n v))\\
&= \sum_{j=1}^n \pi (u_1 \cdots u_{j-1}) G_2 (1 \otimes \bar u_j \otimes \overbar{u_{j+1} \cdots u_n v} \otimes 1)
\end{align*}
where the first and the third identities follow from Lemma \ref{lemma:low-degree} \ref{lowdeg1}. Combining this with \eqref{varphi1} and Lemma \ref{lemma:low-degree} \ref{lowdeg1} again, we obtain the formula for $h_2 (1 \otimes \bar u \otimes \bar v \otimes 1)$.
\end{proof}

\chapter{Deformation theory}
\label{section:deformationtheory}

We give a brief review of general deformation theory as used in this book. In \S\ref{subsection:deformationtheoryassociative} we recall the classical deformation theory of associative algebras first studied by Gerstenhaber \cite{gerstenhaber}, which can be phrased in the language of the Maurer--Cartan formalism for DG Lie or L$_\infty$ algebras as explained in \S\ref{subsection:dglielinfinity}. 

\section{Deformations of associative algebras}
\label{subsection:deformationtheoryassociative}

An associative $\Bbbk$-algebra $A$ is a pair $(W, \mu)$ where $W$ is a $\Bbbk$-vector space and $\mu \in \Hom_{\Bbbk} (W \otimes_{\Bbbk} W, W)$ a bilinear map satisfying the associativity condition
\begin{flalign}
\label{associativity}
&& \mu (u \otimes \mu (v \otimes w)) = \mu (\mu (u \otimes v) \otimes w) && \llap{for any $u, v, w \in W$.}
\end{flalign}
This condition of associativity can be expressed in the following form.

\begin{proposition}
\label{mcassociative}
Let $\mu \in \Hom_{\Bbbk} (W \otimes_{\Bbbk} W, W)$. Then $\mu$ is associative if and only if $[\mu, \mu]_{\mathrm G} = 0$.
\end{proposition}

Here $[\blank {,} \blank]_{\mathrm G}$ is the Gerstenhaber bracket on $\Hom_\Bbbk (W^{\otimes_\Bbbk \hdot + 1}, W)$ defined as follows.

\begin{definition}
\label{definitiongerstenhaber}
\index{Gerstenhaber bracket|textbf}
Let $f \in \Hom_{\Bbbk} (W^{\otimes_{\Bbbk} m+1}, W)$ and $g \in \Hom_{\Bbbk} (W^{\otimes_{\Bbbk} n+1}, W)$ be two multilinear maps. Write
\begin{flalign*}
&& f \bullet_i g &= f (\id^{\otimes_{\Bbbk} i} \otimes_{\Bbbk} g \otimes_{\Bbbk} \id^{\otimes_{\Bbbk} m-i}) && \mathllap{0 \leq i \leq m}
\intertext{and also}
&& f \bullet g &= \sum_{0 \leq i \leq m} (-1)^{ni} f \bullet_i g. &&
\intertext{The {\it Gerstenhaber bracket} is then defined by}
&& [f, g]_{\mathrm G} &= f \bullet g - (-1)^{mn} g \bullet f.
\end{flalign*}
\end{definition}

\begin{remark}
\label{remark:restricts}
For later use we note that the Gerstenhaber bracket restricts to $\Hom (\bar A^{\otimes \hdot +1}, A)$, the (shifted) $\Bbbk Q_0$-relative Hochschild cochain complex (see Remark \ref{relativestandard} and Notation \ref{notation:tensorhom}). For simplicity, we still denote the bracket on $\Hom(\bar A^{\otimes \hdot +1}, A)$ by $[\blank {,} \blank]_{\mathrm G}$. 
\end{remark}

\subsection{Varieties of finite-dimensional algebras}
\label{subsubsection:variety}

\index{variety!of finite-dimensional algebras}
We recall the geometric interpretation of ``actual'' deformations of finite-dimensional algebras. In Chapter \ref{section:pbw} a similar geometric description will be given for actual deformations of reduction systems, which can be used to give a more accessible picture of the nontrivial variation of associative structures and which can also be used to describe actual deformations of infinite-dimensional algebras.

Let us just for a moment assume that $W$ (and thus $A$) is finite-dimensional, say $\dim W = d$. Choosing a $\Bbbk$-basis $\{ v_1, \dotsc, v_d \}$ of $W$, the bilinear map $\mu$ is determined by its structure constants $\mu_{ij}^k \in \Bbbk$ defined by
\[
\mu (v_i \otimes v_j) = \sum_{k = 1}^d \mu_{ij}^k \hair v_k.
\]
Thus $\mu \in \Hom_{\Bbbk} (W^{\otimes_{\Bbbk} 2}, W) \simeq \Bbbk^{d^3}$, and the structure constants $\mu_{ij}^k$ may be viewed as coordinate functions. Writing the Gerstenhaber bracket\index{Gerstenhaber bracket} $[\mu, \mu]_{\mathrm G}$ in this basis, the obstruction to associativity of $\mu$ is given by structure constants $\nu_{ijk}^l$ defined by
\begin{align*}
\frac12 [\mu, \mu]_{\mathrm G} (v_i \otimes v_j \otimes v_k) &= \mu (\mu (v_i \otimes v_j) \otimes v_k) - \mu (v_i \otimes \mu (v_j \otimes v_k)) = \sum_{l = 1}^d \nu_{ijk}^l v_l
\end{align*}
where
\begin{flalign}
\label{nu}
&& \nu_{ijk}^l &= \sum_{m = 1}^d \bigl( \mu_{ij}^m \mu_{mk}^l - \mu_{im}^l \mu_{jk}^m \bigr) && \mathllap{1 \leq i, j, k, l \leq d}
\end{flalign}
is quadratic in the $\mu_{ij}^k$'s. In geometric terms we may view $\Bbbk^{d^3}$ as the affine space $\mathbb A^{d^3}$, so that the elements (\ref{nu}) cut out an affine scheme
\[ 
\mathrm{Alg}_W \subset \mathbb A^{d^3} \simeq \Hom_{\Bbbk} (W^{\otimes_{\Bbbk} 2}, W)
\]
of {\it all} associative structures on $W$ given by $d^4 = \dim \Hom_{\Bbbk} (W^{\otimes_{\Bbbk} 3}, W)$ quadratic equations. This scheme is generally not irreducible, and its irreducible components organize the algebras into different types. Note also that \eqref{nu} may contain nilpotents, so that the corresponding scheme may not be reduced. Since we would like to view $\mathrm{Alg}_W$ as a geometric object providing families of algebras, e.g.\ families defined over $\Bbbk [t]$, say, we follow \cite{gabriel} and refer to $\mathrm{Alg}_W$ loosely as a {\it variety} of finite-dimensional algebras, even though for the infinitesimal theory the scheme structure (and not only the structure of the corresponding reduced scheme) is important.

Moreover, there is a natural action of $\GL (W)$ on $\mathrm{Alg}_W$ given by change of basis, i.e.\ for $T \in \GL (W)$ we put
\[
(T \cdot \mu) (v \otimes w) := T^{-1} (\mu (T (v) \otimes T (w)))
\]
and two algebras $A = (W, \mu)$ and $A' = (W, \mu')$ are isomorphic if and only if they lie in the same $\GL (W)$-orbit. In particular, the intuitive notion of ``variation'' of associative structures can be understood as varying an associative structure $\mu$ as a $\Bbbk$-bilinear map in $\mathrm{Alg}_W$. Varying the associative structure inside a $\GL (W)$-orbit will give an isomorphic algebra (a ``trivial deformation'') and passing to a different orbit will give a non-isomorphic algebra (a ``nontrivial deformation'').

Gabriel \cite{gabriel} showed that the variety\index{variety!of finite-dimensional algebras} $\mathrm{Alg}_W$ is always connected, so that any two algebras $A, A' \in \mathrm{Alg}_W$ may be viewed as ``actual deformations'' of each other. In \cite{gabriel,schaps} explicit descriptions of $\mathrm{Alg}_W$ are given for $\dim W \leq 6$. For example, when $\dim W = 4$, the variety $\mathrm{Alg}_W$ has five irreducible components with $18$ $\GL (W)$-orbits of dimensions between $6$ and $15$ and one $1$-dimensional family of $11$-dimensional orbits \cite{gabriel}. An explicit description of $\mathrm{Alg}_W$ becomes less and less accessible in higher dimensions.

Deformation theory is a way to study the local structure of the variety $\mathrm{Alg}_W$ and thus the ``local'' properties of the variation of associative structures, where by ``local'' we mean an {\it infinitesimal} or {\it formal} neighbourhood of a point in $\mathrm{Alg}_W$. In particular, the natural scheme structure of $\mathrm{Alg}_W$ allows us to relate its Zariski tangent space of $\mathrm{Alg}_W$ and its $\mathrm{GL} (W)$-orbits at a point to the Hochschild cohomology of the corresponding algebra.

However, the formal theory (and its relationship to Hochschild cohomology) makes sense for {\it any} associative algebra, not necessarily finite-dimensional.

\subsection{Infinitesimal and formal deformations}

Let $W$ be any $\Bbbk$-vector space, not necessarily finite-dimensional, and let $A = (W, \mu)$ for $\mu \in \Hom_{\Bbbk} (W^{\otimes_{\Bbbk} 2}, W)$ be an associative algebra (i.e.\ $\mu$ satisfies $[\mu, \mu]_{\mathrm G} = 0$).

A {\it first-order deformation} of $A = (W, \mu)$ (over $\Bbbk [t] / (t^2)$) is given by 
\begin{equation}
\label{fo}
\mu^t = \mu + \mu_1 t
\end{equation}
where $\mu_1 \colon A \otimes_{\Bbbk} A \to A$ is a bilinear map such that $(A [t] / (t^2), \mu^t)$ is an associative algebra, where $\mu^t$ is considered a multiplication on $A [t] / (t^2)$ by $t$-linear extension. A first-order deformation is thus determined by the linear map $\mu_1$ and the associativity condition (\ref{associativity}) for $\mu^t$ implies that $\mu_1$ is a Hochschild $2$-cocycle. Two first-order deformations $(A [t]/ (t^2), \mu^t)$ and $(A [t]/ (t^2), {\mu'}^t)$ of $(A, \mu)$ are {\it (gauge) equivalent} if there is a $\Bbbk [t] / (t^2)$-algebra isomorphism $T \colon (A [t]/ (t^2), \mu^t) \to  (A [t]/ (t^2), {\mu'}^t)$ which is the identity modulo $(t)$. Such an isomorphism is thus of the form $T = \id + T_1 t$ for some $\Bbbk$-linear map $T_1 \colon A \to A$ and a straightforward computation shows that $\mu_1 = \mu_1' + d (T_1)$, that is, two first-order deformations are equivalent if and only if they differ by a $2$-coboundary. Thus the set of equivalence classes of first-order deformations can be identified as the cohomology group $\HH^2 (A, A)$ (see \cite{gerstenhaber,gerstenhaberschack}).

\begin{remark}
When $A$ is finite-dimensional one can again interpret this geometrically: a first-order deformation of an algebra $A \in \mathrm{Alg}_W$ gives a tangent vector to $\mathrm{Alg}_W$ at the point $A$. The space of first-order deformations --- which by the above coincides with the space of Hochschild $2$-cocycles --- is the Zariski tangent space of $\mathrm{Alg}_W$ at $A$ and the subspace of $2$-coboundaries is the tangent space to the $\GL (W)$-orbit passing through the point $A$. The dimension of $\HH^2 (A, A)$ thus measures the number of independent tangent directions at the point $A \in \mathrm{Alg}_W$ in which the algebra structure changes.
\end{remark}

The definition of a first-order deformation \eqref{fo} can be extended to all orders as follows. A {\it formal deformation} of $(A, \mu)$ (over $\Bbbk \llrr{t}$) is given by
\[
\mu^t = \mu + \mu_1 t + \mu_2 t^2 + \dotsb
\]
where the maps $\mu_i$ are bilinear maps $A \otimes_{\Bbbk} A \to A$ such that the $\Bbbk \llrr{t}$-linear extension of $\mu^t$ (still denoted by $\mu^t$) makes $(A \llrr{t}, \mu^t)$ into an associative $\Bbbk \llrr{t}$-algebra.

It follows directly from Proposition \ref{mcassociative} that the Gerstenhaber bracket also controls the formal deformation theory of associative algebras.

\begin{corollary}
\label{corollary:gerstenhaber}
Let $A = (W, \mu)$ be an associative algebra and let $\widetilde \mu \in \Hom_{\Bbbk} (W \otimes_{\Bbbk} W, W) \hatotimes (t)$, i.e.\ $\widetilde\mu = \mu_1 t + \mu_2 t^2 + \dotsb$. Then $\widetilde A = (W \llrr{t}, \mu + \widetilde \mu)$ is associative if and only if $\widetilde \mu$ satisfies the Maurer--Cartan equation\index{Maurer--Cartan!equation}
\[
d \widetilde \mu + \tfrac12 [\widetilde \mu, \widetilde \mu]_{\mathrm G} = 0.
\]
\end{corollary}

Recall that the Hochschild differential\index{Hochschild!differential}
\begin{equation}\label{align:differentialhochschildcochain}
d^n \colon  \Hom_{\Bbbk} (A^{\otimes_{\Bbbk} n}, A)  \to  \Hom_{\Bbbk} (A^{\otimes_{\Bbbk} n+1}, A)
\end{equation} 
is given by  $[\mu, \blank]_{\mathrm G}$,  so that
\[
0 = [\mu + \widetilde\mu, \mu + \widetilde\mu]_{\mathrm G} = \underbrace{[\mu, \mu]_{\mathrm G}}_{{} = {} 0} + 2 \underbrace{[\mu, \widetilde\mu]_{\mathrm G}}_{d \widetilde\mu} + [\widetilde\mu, \widetilde\mu]_{\mathrm G}.
\]

In fact, the Gerstenhaber bracket\index{Gerstenhaber bracket} endows the (shifted) Hochschild cochain complex $\Hom_{\Bbbk} (A^{\otimes_{\Bbbk} \hdot +1}, A)$ with the structure of a {\it DG Lie algebra}\index{DG Lie algebra}\index{algebra!DG Lie} and this structure can be used to study both ``actual'' and formal deformations --- not only over $\Bbbk \llrr{t}$ but over any complete local Noetherian $\Bbbk$-algebra.

\section{Deformation theory via DG Lie and L$_\infty$ algebras}
\label{subsection:dglielinfinity}
\index{L$_\infty$!algebra}\index{algebra!L$_\infty$}\index{DG Lie algebra}\index{algebra!DG Lie}

Deformations of associative algebras fit into the general framework of formal deformation theory via DG Lie or L$_\infty$ algebras. One has that deformations are controlled by a cochain complex, where
\begin{align*}
\text{first-order deformations} &\leftrightarrow \text{cochain complex } (C^\hdot, d) \\
\text{higher-order deformations} &\leftrightarrow (C^\hdot, d, \underbrace{\langle \blank {,} \blank \rangle, \langle \blank {,} \blank {,} \blank \rangle, \dotsc}_{\text{``higher structures''}}).
\end{align*}
The ``higher structures'' are multilinear brackets defining a DG Lie\index{DG Lie algebra}\index{algebra!DG Lie} or L$_\infty$ algebra, the latter being a DG Lie algebra ``up to homotopy'' (see Definitions \ref{definition:dglie} and \ref{definition:linfinity} below). The fact that L$_\infty$ algebras\index{L$_\infty$!algebra}\index{algebra!L$_\infty$} play a central role in deformation theory is illustrated in the following result, which might be called the Fundamental Theorem of Formal Deformation Theory.

\begin{theorem}
\label{theorem:quasiisomorphism}
\begin{enumerate}
\item \label{qi1} Let $\mathfrak g$ and $\mathfrak g'$ be two L$_\infty$ algebras. An L$_\infty$-quasi-isomorphism $\Phi_\ldot \colon \mathfrak g \to \mathfrak g'$ induces a natural transformation $\MC_{\mathfrak g} \to \MC_{\mathfrak g'}$ of Maurer--Cartan functors\index{Maurer--Cartan!functor} and a natural isomorphism of deformation functors $\Def_{\mathfrak g} \simeq \Def_{\mathfrak g'}$.
\item \label{qi2} For any (formal) deformation problem in characteristic $0$ with deformation functor $D$, there exists an L$_\infty$ algebra $\mathfrak g$ such that there is an isomorphism of deformation functors $D \simeq \Def_{\mathfrak g}$.
\end{enumerate}
\end{theorem}

(See \S\S\ref{subsubsection:dglie}--\ref{subsubsection:linfinity} for the definition of $\MC_{\mathfrak g}$ and $\Def_{\mathfrak g}$.) For \ref{qi1} see Getzler \cite[Prop.~4.9]{getzler}; \ref{qi2} was established in the more general setting of derived deformation theory by Pridham \cite{pridham} and Lurie \cite{lurie} as an equivalence of the homotopy or $\infty$-categories of formal moduli problems and of DG Lie algebras\index{DG Lie algebra}\index{algebra!DG Lie}, giving a formal proof of the ``philosophy'' formulated by Deligne \cite{deligne} that deformation problems in characteristic $0$ should be controlled by DG Lie algebras\index{DG Lie algebra}\index{algebra!DG Lie}.

The {\it deformation functor} $\Def_{\mathfrak g}$ associated to a DG Lie or L$_\infty$ algebra $\mathfrak g$ captures all the formal information about the (derived) deformation problem controlled by $\mathfrak g$.

In the remainder of this section we give the definitions of DG Lie and L$_\infty [1]$ algebras\index{L$_\infty$!algebra}\index{algebra!L$_\infty$} and their associated deformation functors. In Chapter \ref{section:deformations-of-path-algebras} we obtain an L$_\infty[1]$ algebra structure $\mathbf p (Q, R)$ on the cochain complex $P^{\hdot + 2}$ associated to a reduction-unique reduction system $R$ by homotopy transfer and show that this L$_\infty [1]$ algebra\index{L$_\infty$!algebra}\index{algebra!L$_\infty$} controls the (formal) deformations of $R$, so that by Theorem \ref{theorem:quasiisomorphism} deformations of $A = \Bbbk Q / I$ can be completely described as deformations of any chosen reduction system satisfying ($\diamond$) for $I$.

\subsection{DG Lie algebras}
\label{subsubsection:dglie}
\index{DG Lie algebra|textbf}\index{algebra!DG Lie|textbf}

\begin{definition}
\label{definition:dglie}
\index{DG Lie algebra|textbf}\index{algebra!DG Lie|textbf}
A \textit{differential graded (DG) Lie algebra} $(\mathfrak g, d, [\blank {,} \blank])$ consists of a graded vector space $\mathfrak g = \bigoplus_{n \in \mathbb Z} \mathfrak g^n$ together with a linear map $d \colon \mathfrak g \to \mathfrak g$ of degree $1$ satisfying $d^n d^{n-1} = 0$, the \textit{differential}, and a (graded) Lie bracket $[\blank {,} \blank] \colon \mathfrak g \otimes \mathfrak g \to \mathfrak g$ of degree $0$ satisfying
\begin{enumerate}
\item $[x, y] = (-1)^{\lvert x \rvert \lvert y \rvert + 1} [y, x]$ \hfill \textit{(graded) skew-symmetry}
\item $d [x, y] = [d x, y] + (-1)^{\lvert x \rvert} [x, d y]$ \hfill \textit{(graded) Leibniz rule}
\item $(-1)^{\lvert x \rvert \lvert z \rvert} [x, [y, z]] + (-1)^{\lvert y \rvert \lvert x \rvert} [y, [z, x]] + (-1)^{\lvert z \rvert \lvert y \rvert} [z, [x, y]] = 0$
{\vspace{0em}\leavevmode\unskip\nobreak\hfil\penalty50\hskip2em
  \hbox{}\nobreak\hfil{\it (graded) Jacobi identity\rlap{.}}%
  \parfillskip=0pt \finalhyphendemerits=0 \endgraf}
\end{enumerate}
\end{definition}

The deformation functor associated to a DG Lie algebra is defined as follows. This deformation functor can be defined on the category $\mathfrak{Art}_{\Bbbk}$ of (commutative) local Artinian $\Bbbk$-algebras (e.g.\ $\Bbbk [t] / (t^{n+1})$) and then by completion on the category $\widehat{\mathfrak{Art}}_{\Bbbk}$ of (commutative) complete local Noetherian $\Bbbk$-algebras (e.g.\ $\Bbbk \llrr{t}$). Deformations over a local Artinian $\Bbbk$-algebra are usually called (finite-order) {\it infinitesimal} deformations and deformations over a complete local Noetherian $\Bbbk$-algebra are {\it formal} deformations. (See \S\ref{subsubsection:formalalgebraization} below for a discussion about algebraization of formal deformations.)

We usually work with deformations over $\Bbbk \llrr{t}$ or sometimes $\Bbbk \llrr{t_1, \dotsc, t_n}$ and we write
\begin{align*}
A \llrr{t_1, \dotsc, t_n} = A \hatotimes_{\Bbbk} \Bbbk \llrr{t_1, \dotsc, t_n} = \varprojlim_N A \otimes_{\Bbbk} \Bbbk [t_1, \dotsc, t_n] / \mathfrak m^N
\end{align*}
where $\mathfrak m = (t_1, \dotsc, t_n)$ is the (unique) maximal ideal of $\Bbbk \llrr{t_1, \dotsc, t_n}$ so that $\hatotimes_{\Bbbk}$ denotes the completion of the tensor product with respect to the $\mathfrak m$-adic topology. (Note that if $A$ is finite-dimensional, one has $A \hatotimes_{\Bbbk} \Bbbk \llrr{t_1, \dotsc, t_n} = A \otimes_{\Bbbk} \Bbbk \llrr{t_1, \dotsc, t_n}$.)

\begin{definition}
\label{dgla}
Let $\mathfrak g$ be a DG Lie algebra over a field $\Bbbk$ of characteristic $0$. Define the \textit{gauge functor} $\mathrm G_{\mathfrak g}$ and the \textit{Maurer--Cartan functor}\index{Maurer--Cartan!functor|textbf} $\MC_{\mathfrak g}$ by
\begin{align*}
\begin{tikzpicture}[baseline=-2.6pt,description/.style={fill=white,inner sep=2pt}]
\matrix (m) [matrix of math nodes, row sep=0em, text height=1.5ex, column sep=0em, text depth=0.25ex, ampersand replacement=\&, column 3/.style={anchor=base west}, column 1/.style={anchor=base east}]
{
\mathrm G_{\mathfrak g} \colon \&[-.8em] \mathfrak{Art}_{\Bbbk} \&[2em] \mathfrak{Group} \\
\& (B, \mathfrak m) \& \exp (\mathfrak g^0 \otimes \mathfrak m) \\[1em]
\MC_{\mathfrak g} \colon \&[-.8em] \mathfrak{Art}_{\Bbbk} \&[2em] \mathfrak{Sets} \\
\& (B, \mathfrak m) \& \bigl\{ x \in \mathfrak g^1 \otimes \mathfrak m \bigm| d x + \tfrac12 [x, x] = 0 \bigr\}. \\
};
\path[->,line width=.4pt,font=\scriptsize]
(m-1-2) edge (m-1-3)
(m-3-2) edge (m-3-3)
;
\path[|->,line width=.4pt]
(m-2-2) edge (m-2-3)
(m-4-2) edge (m-4-3)
;
\end{tikzpicture}
\end{align*}
The group $\exp (\mathfrak g^0 \otimes \mathfrak m)$ is called the {\it gauge group} and is defined formally via the Baker--Campbell--Hausdorff formula\footnote{For a nilpotent Lie algebra $\mathfrak n$ one defines the group law $\cdot$ on $\exp (\mathfrak n)$ by $\exp (x) \cdot \exp (y) = \exp (x + y + \tfrac12 [x,y] + \dotsb)$ for $x, y \in \mathfrak n$, where the sum on the right-hand side is the Baker--Campbell--Hausdorff formula. Note that $\mathfrak g \otimes \mathfrak m$ is a nilpotent Lie algebra with Lie bracket $[x \otimes a, y \otimes b] = [x, y] \otimes ab$.}; the equation
\begin{equation}
\label{maurercartandglie}
d x + \tfrac12 [x, x] = 0
\end{equation}
is called the {\it Maurer--Cartan equation}\index{Maurer--Cartan!equation} and elements in $\mathfrak g^1 \otimes \mathfrak m$ satisfying (\ref{maurercartandglie}) are called {\it Maurer--Cartan elements}\index{Maurer--Cartan!elements|textbf}.

The \textit{deformation functor} $\Def_{\mathfrak g}$ associated to $\mathfrak g$ is then defined by $\Def_{\mathfrak g} = {\MC_{\mathfrak g}} \big/ \mathrm G_{\mathfrak g}$, where an element $\exp (w) \in \mathrm G_{\mathfrak g} (B, \mathfrak m)$ acts on $x \in \MC_{\mathfrak g} (B, \mathfrak m)$ by
\begin{flalign}
\label{gauge}
&& \exp (w) \cdot x = x + \sum_{n \geq 1} \frac{\mathrm{ad}_w^{n-1}}{n!} (\mathrm{ad}_w x - dw) && \mathllap{\mathrm{ad}_w = [w, \blank].}
\end{flalign}
\end{definition}

\begin{remark}
The formula \eqref{gauge} can be viewed as a perturbation of the adjoint action depending on the differential. Indeed, if $d = 0$, then $\exp (w) \cdot x = \sum_{n \geq 0} \frac{\mathrm{ad}_w^n}{n!} (x)$ (see Manetti \cite[\S 6.3]{manetti}).
\end{remark}

\begin{remark}
For each complete local Noetherian $\Bbbk$-algebra $(B, \mathfrak m) \in \widehat{\mathfrak{Art}}_{\Bbbk}$ we have that $B / \mathfrak m^n \in \mathfrak{Art}_{\Bbbk}$ for all $n > 0$ and also $B \simeq \varprojlim_n B/ \mathfrak m^n$ so we can define the deformation functor on $\widehat{\mathfrak{Art}}_{\Bbbk}$ simply by $\Def_{\mathfrak g} (B) = \varprojlim_n \Def_{\mathfrak g} (B / \mathfrak m^n)$.

More generally, the DG Lie\index{DG Lie algebra}\index{algebra!DG Lie} or L$_\infty [1]$ algebras\index{L$_\infty$!algebra}\index{algebra!L$_\infty$} give rise to {\it derived} deformation functors which are defined analogously as functors on the category of local commutative DG Artinian rings (concentrated in non-positive chain degrees) with values in the category of simplicial sets (see Pridham \cite{pridham}, also for the equivalence between the various approaches to derived deformation theory). The L$_\infty[1]$ algebra\index{L$_\infty$!algebra}\index{algebra!L$_\infty$} $\mathbf p(Q, R)$ we construct in Chapter \ref{section:deformations-of-path-algebras} can also be used in this more general setting, as it gives rise to a derived deformation functor and the equivalence given in Theorem \ref{theorem:equivalenceformal} is also valid in this setup. However, for this book we will content ourselves with the ``classical'' set-valued deformation functors defined on $\widehat{\mathfrak{Art}}_{\Bbbk}$.
\end{remark}

\subsection{L$_\infty$ algebras}
\label{subsubsection:linfinity}

L$_\infty$ algebras\index{L$_\infty$!algebra}\index{algebra!L$_\infty$} are a more flexible analogue of DG Lie\index{DG Lie algebra}\index{algebra!DG Lie} algebras, which naturally appear in deformation theory (cf.\ Theorem \ref{theorem:quasiisomorphism}). Although the notion of L$_\infty$ algebras is more standard in the literature, in this book we work with the following shifted version of L$_\infty$ algebras, sometimes called L$_\infty[1]$ algebras\index{L$_\infty$!algebra}\index{algebra!L$_\infty$}\index{L$_\infty$!algebra!L$_\infty[1]$}\index{algebra!L$_\infty$!L$_\infty[1]$|textbf}. As we remarked in the introduction, this is purely for sign reasons and the shift $[1]$ can be ignored for the most part --- except of course when carefully looking at the signs.

\begin{definition}
\label{definition:linfinity}
An {\it L$_\infty[1]$ algebra}\index{L$_\infty$!algebra|textbf}\index{algebra!L$_\infty$|textbf} $(\mathbf g, \{ l_n \}_{n \geq 1}) = (\mathbf{g}, \langle \blank \rangle, \langle \blank {,} \blank \rangle, \langle \blank {,} \blank {,} \blank \rangle, \dotsc)$ is a graded $\Bbbk$-vector space $\mathbf{g} = \bigoplus_{m \in \mathbb Z} \mathbf{g}^m$ together with a collection of multilinear maps $l_n \colon \mathbf{g}^{\otimes n} \to \mathbf{g}$ of degree $1$ satisfying for each $n$ the identities
\begin{enumerate}
\item $l_n (x_{\sigma (1)}, \dotsc, x_{\sigma (n)}) = \chi(\sigma) \, l_n (x_1, \dotsc, x_n)$ for any $\sigma \in \mathfrak S_n$ 
\item $\displaystyle\sum_{\substack{i+j=n+1 \\ i,j\geq 1}} \sum_{\sigma \in \mathfrak S_{i,n-i}}  \chi (\sigma) \,  l_j (l_i (x_{\sigma (1)}, \dotsc, x_{\sigma (i)}), x_{\sigma (i+1)}, \dotsc, x_{\sigma (n)}) = 0$
\end{enumerate}
for homogeneous elements $x_1, \dotsc, x_n$. Here
\begin{itemize}
\item $\mathfrak S_n$ is the set of permutations of $n$ elements
\item $\mathfrak S_{i,n-i} \subset \mathfrak S_n$ is the set of {\it shuffles}, {\it i.e.}\ permutations $\sigma \in \mathfrak S_n$ satisfying $\sigma (1) < \dotsb < \sigma (i)$ and $\sigma (i+1) < \dotsb < \sigma (n)$, and
\item $\chi (\sigma) :=  \epsilon(\sigma; x_1, \dotsc, x_n)$, where $\epsilon(\sigma; x_1, \dotsc, x_n)$ is the {\it Koszul sign\index{Koszul!sign}}\footnote{The Koszul sign of a transposition of two homogeneous elements $x_i, x_j$ is defined by $(-1)^{\lvert x_i \rvert \lvert x_j \rvert}$, where $\lvert x_i \rvert$ denotes the degree of $x_i$. This definition is then extended multiplicatively to an arbitrary permutation using a decomposition into transpositions.} of $\sigma$, which also depends on the degrees of the $x_i$.
\end{itemize}
We usually denote the $n$-ary multilinear maps $l_n (\blank, \dotsc {,} \blank )$ by $\langle \blank , \dotsc {,} \blank \rangle$ or sometimes by $\langle \blank, \dotsc {,} \blank \rangle_n$ to indicate the number of entries.
\end{definition}

\begin{remark}
\label{ordinarydgalgebra}
If the $n$-ary brackets are identically zero for $n \geq 2$, one obtains an ordinary cochain complex with the differential given by the unary bracket. We call an L$_\infty[1]$ algebra\index{L$_\infty$!algebra}\index{algebra!L$_\infty$} whose $n$-ary brackets vanish for $n \geq 3$ a {\it DG Lie$[1]$ algebra}.\index{DG Lie algebra}\index{DG Lie algebra!DG Lie$[1]$}\index{algebra!DG Lie}\index{algebra!DG Lie$[1]$|textbf}

There is an isomorphism between the categories of L$_\infty$ algebras and L$_\infty[1]$ algebras\index{L$_\infty$!algebra}\index{algebra!L$_\infty$} via the following rule
\[
(\mathfrak g, [\blank], [\blank {,} \blank], [\blank {,} \blank {,} \blank],  \dotsc)  \mapsto (\mathbf g, \langle\blank \rangle, \langle \blank {,} \blank \rangle, \langle \blank {,} \blank {,} \blank \rangle , \dotsc)
\]
where $\mathbf g = \mathfrak g [1]$  and for each $n \geq 1$
\begin{flalign*}
&& \langle x_1[1], \dotsc, x_n[1] \rangle = (-1)^{\epsilon} [x_1, \dotsc, x_n][1] && \llap{for $x_1, \dotsc, x_n \in \mathfrak g$.}
\end{flalign*}
Here $\epsilon = |x_1|(n-1) + |x_2|(n-2) + \dotsb + |x_{n-1}|$, which is precisely the kind of sign we can avoid by shifting. Note that because of the shift, the $n$-ary map $[\blank, \dotsc, \blank] \colon \mathfrak g^{\otimes n} \to \mathfrak g$ has degree $2-n$. We also note that the shift of any DG Lie algebra $(\mathfrak g, [\blank], [\blank {,} \blank])$ (so that $d = [\blank]$ has degree $1$ and $[\blank {,} \blank]$ has degree $0$) gives a DG Lie$[1]$\index{DG Lie algebra}\index{algebra!DG Lie}\index{DG Lie algebra!DG Lie$[1]$} algebra $(\mathbf{g} = \mathfrak g[1], \langle \blank \rangle, \langle \blank {,} \blank\rangle)$.
\end{remark}

\begin{notation}
\label{notation:h}
The DG Lie$[1]$ algebra\index{DG Lie algebra}\index{algebra!DG Lie}
$$
\mathbf h(A) := (\Hom_{\Bbbk} (A^{\otimes_{\Bbbk} \hdot +2}, A), [\blank], [\blank {,} \blank])
$$
which controls the deformation theory of associative algebras plays a central role in the following. Here $\bullet$ indicates the degree in $\mathbf h (A)$ and for any $f \in \Hom_{\Bbbk} (A^{\otimes_{\Bbbk} m+2}, A)$ and $g \in \Hom_{\Bbbk} (A^{\otimes_{\Bbbk} n +2}, A)$ we have
\begin{align} \label{align:dglie1}
\begin{aligned}
[f] & = - d^\hdot(f) \\
[f, g] & = (-1)^{m+1} [f, g]_{\mathrm G}
\end{aligned}
\end{align}
where $d^\hdot$ is given in \eqref{align:differentialhochschildcochain} and $[\blank {,} \blank]_{\mathrm G}$ is the Gerstenhaber bracket of Definition \ref{definitiongerstenhaber}.

If $A = \Bbbk Q / I$ we also consider the DG Lie$[1]$ subalgebra
\[
(\Hom (\bar A^{\otimes \hdot + 2}, A), [\blank], [\blank {,} \blank])
\]
obtained via \eqref{align:dglie1} by restricting the Gerstenhaber bracket to $\Hom (\bar A^{\otimes \hdot + 1}, A)$ (cf.\ Remark \ref{remark:restricts}) where $\otimes = \otimes_{\Bbbk Q_0}$ and $\Hom = \Hom_{\Bbbk Q_0^\e}$. Since the inclusion is a quasi-isomorphism, we continue to denote this DG Lie$[1]$ subalgebra by $\mathbf h (A)$.
\end{notation}

\begin{definition}
A morphism\index{L$_\infty$!morphism} of L$_\infty[1]$ algebras $\Phi_\ldot \colon (\mathbf g, \{ l_i \}) \to (\mathbf g', \{ l'_i \})$ is given by graded linear maps
\[
\Phi_n \colon \mathbf g^{\otimes n} \to \mathbf g'
\]
of degree zero for all $n \geq 1$ such that 
$\Phi_n (x_{\sigma (1)}, \dotsc, x_{\sigma (n)}) = \chi(\sigma) \, \Phi_n (x_1, \dotsc, x_n)$ for any $\sigma \in \mathfrak S_n$ and 
\begin{align*}
& \displaystyle\sum_{\substack{i+j=n+1 \\ i,j\geq 1}} \sum_{\sigma \in \mathfrak S_{i,n-i}} \chi(\sigma) \Phi_{j} (l_i (x_{\sigma (1)}, \dotsc, x_{\sigma(i)}), x_{\sigma(i+1)}, \dotsc,  x_{\sigma(n)})\\
& = \!\!\! \sum_{\substack{1 \leq r \leq n \\ i_1+\dotsb + i_r = n}} \!\! \sum_{\tau} \chi(\tau)\,  l_r'(\Phi_{i_1}(x_{\tau (1)}, \dotsc, x_{\tau (i_1)}), \dotsc, \Phi_{i_r}(x_{\tau (i_1 + \dotsb + i_{r-1}+1)}, \dotsc, x_{\tau(n)}))
\end{align*}
where $\tau$ runs over all $(i_1, \dotsc, i_r)$-shuffles\footnote{i.e.\ for each $0 \leq j <r$ we have $\tau(i_1+\dotsb+i_j+1)<\tau(i_1+\dotsb+i_j+2)<\dotsb<\tau(i_1+\dotsb+i_{j+1})$} satisfying
\begin{flalign*}
&& \tau (i_1+\dotsb+i_{l-1}+1) < \tau (i_1+\dotsb+i_{l}+1) && \llap{for each $1 \leq l \leq r-1$.}
\end{flalign*}
Here we set $i_0 = 0$.

A morphism of L$_{\infty}[1]$ algebras\index{L$_\infty$!morphism} $\Phi_\ldot$ is a {\it quasi-isomorphism}\index{L$_\infty$!quasi-isomorphism} if $\Phi_1$ is a quasi-isomorphism of complexes. 
\end{definition}

L$_\infty[1]$ algebras (and their morphisms) can also be viewed from a ``dual'' point of view, as a graded vector space $\mathbf g$ together with a coderivation on the reduced cofree cocommutative coalgebra generated by $\mathbf g$ (see \cite{ladastasheff} for more details).

\begin{lemma}[{\cite[Thm.~2.9]{canonaco}}]
\label{lemma:quasi-inverse}
Let $\Phi_\ldot \colon \mathbf g \to \mathbf g'$ be a quasi-isomorphism\index{L$_\infty$!quasi-isomorphism} between L$_\infty[1]$ algebras. Then $\Phi_\ldot$ admits a \emph{quasi-inverse}, i.e.\ there exists a quasi-isomorphism $\Psi_\ldot \colon \mathbf g' \to \mathbf g$ which induces the inverse isomorphism on cohomology.
\end{lemma}

For an L$_\infty[1]$ algebra, one similarly obtains a Maurer--Cartan functor\index{Maurer--Cartan!functor|textbf} $\MC_{\mathbf g}$ by generalizing the Maurer--Cartan equation\index{Maurer--Cartan!equation} as follows.

\begin{definition}
Given an L$_\infty[1]$ algebra $\mathbf g$, a {\it Maurer--Cartan element}\index{Maurer--Cartan!elements|textbf} is an element $x \in \mathbf g^0$ satisfying the Maurer--Cartan equation\index{Maurer--Cartan!equation}
\begin{flalign}
\label{maurercartanlinfinity}
&& \sum_{n=1}^\infty \frac{x^{\langle n \rangle}}{n!} = 0 && \mathllap{x^{\langle n \rangle} = \langle x, \dotsc, x \rangle_n.}
\end{flalign}
\end{definition}

\begin{remark}
\label{remark:adic}
As the sum in (\ref{maurercartanlinfinity}) is infinite, the definition of a Maurer--Cartan element\index{Maurer--Cartan!elements} only makes sense in the right context, i.e.\ if one can establish that the sum is in fact finite or at least converges in some topology.

The former can sometimes be shown for elements in $\mathbf g^0$ which satisfy certain degree conditions. The latter can {\it always} be achieved in the context of formal deformations over any complete local Noetherian algebra $(B, \mathfrak m)$ by working with the $\mathfrak m$-adic topology. Given an L$_\infty [1]$ algebra $(\mathbf g, \{ l_n \}_n)$, formal deformations over $(B, \mathfrak m)$ are described as Maurer--Cartan elements\index{Maurer--Cartan!elements} in the L$_\infty [1]$ algebra $\mathbf g \hatotimes \mathfrak m$ obtained by linearly extending the brackets, i.e.\ by setting
\begin{flalign*}
&& l_n (x_1 \otimes m_1, \dotsc, x_n \otimes m_n) &= l_n (x_1, \dotsc, x_n) \otimes m_1 \cdots m_n && \qquad\qquad\mathllap{x_i \in \mathbf g, m_i \in \mathfrak m.}
\end{flalign*}
For the L$_\infty [1]$ algebra $\mathbf g \hatotimes \mathfrak m$, the sum in (\ref{maurercartanlinfinity}) then converges in the $\mathfrak m$-adic topology.
\end{remark}

In this more general setting of L$_\infty [1]$ algebras, equivalence of Maurer--Cartan elements\index{Maurer--Cartan!elements!equivalence} can no longer be described by the action of a gauge {\it group}. Rather, one says that two Maurer--Cartan elements $x, y \in \MC_{\mathbf g}$ are equivalent\index{Maurer--Cartan!elements!equivalence} if they are {\it homotopic} (see Canonaco \cite{canonaco}, Manetti \cite[\S 5]{manettidgla} and Markl \cite[Ch.~5]{markl} for nice expositions).

The deformation functor $\Def_{\mathbf g}$ associated to an L$_\infty [1]$ algebra $\mathbf g$ can now be defined as
\[
\Def_{\mathbf g} = \MC_{\mathbf g} / {\sim}
\]
where $\sim$ denotes homotopy equivalence.
(Note that two Maurer--Cartan elements in a DG Lie algebra are homotopic\index{Maurer--Cartan!elements!equivalence} if and only if they are related by the action of the gauge group, so that for DG Lie algebras the notions of gauge equivalence and homotopy equivalence coincide \cite[Thm.~5.5]{manettidgla}.)

In Chapter \ref{section:deformations-of-path-algebras} we will give a combinatorial description of the Maurer--Cartan equation\index{Maurer--Cartan!equation} of the L$_\infty [1]$ algebra $\mathbf p (Q, R)$ controlling deformations of the reduction system\index{reduction system!deformation}\index{deformation!of reduction system} $R$. In this context the abstract notion of homotopy equivalence of Maurer--Cartan elements\index{Maurer--Cartan!elements!equivalence} can be described concretely using the notion of equivalence of (formal) reduction systems\index{reduction system!equivalence} (see Definition \ref{definition:equivalenceformalreductionsystem}). In particular, it is sometimes possible to explicitly determine the set of all formal deformations up to equivalence. (See \S\ref{subsubsection:formalbrauertree} for an example.)

\section{Algebraization of formal deformations}
\label{subsubsection:formalalgebraization}
\index{algebraization|textbf}

Given a formal deformation over $\Bbbk \llrr{t}$, say, one may ask to what extent $t$ can be considered an ``actual'' parameter, i.e.\ if it is possible to evaluate $t$ at some value $\lambda \in \Bbbk$ to obtain an ``actual'' deformation.

Given a formal deformation $\widehat A = (A \llrr{t}, \star)$ of an associative algebra $A$, an {\it algebraization}\index{algebraization|textbf} of $\widehat A$ would be an algebra $\widetilde A = (A [t], \star)$ such that $\widehat A$ is isomorphic to the $(t)$-adic completion of $\widetilde A$ as formal deformation, i.e.\ $T \colon \widehat A \tosim \varprojlim_n \widetilde A / (t^n)$, where $T$ restricts to the identity modulo $t$, so that $T = \id + T_1 t + T_2 t^2 + \dotsb$ for some linear maps $T_i \colon A \to A$. For example, given the commutative algebra $A = \Bbbk [x, y] / (xy)$, we have that $\widehat A = \Bbbk [x, y] \llrr{t} \big/ (xy - t)^\compl$ is a formal (commutative) deformation of $A$ over $\Bbbk \llrr{t}$, which coincides with the $(t)$-adic completion of the algebra $\widetilde A = \Bbbk [x, y] [t] / (xy - t)$.

When $A = (W, \mu)$ is a finite-dimensional algebra, the formal deformation theory studies the ``local'' structure of the affine variety\index{variety!of finite-dimensional algebras} $\mathrm{Alg}_W$ of all associative algebra structures on $W$, and an algebraization of a formal deformation may be viewed as reconstructing an affine subvariety passing through $A \in \mathrm{Alg}_W$ from its formal neighbourhood.

When $A$ is infinite-dimensional, algebraizations \index{algebraization|textbf} may not always exist, but when they do one can consider the parameter $t$ as an actual parameter and evaluate at all values of $t$, giving an ``actual'' deformation $\widetilde A_\lambda = \widetilde A / (t - \lambda)$. For the above example, $\{ \widetilde A_\lambda \mid \lambda \in \Bbbk \}$ is a family of commutative algebras $ \widetilde A_\lambda$ which all have  the same $\Bbbk$-basis  as  $\widetilde A_0 = A$. In the context of deformations of $A = \Bbbk Q / I$ algebraizations \index{algebraization|textbf} can be shown to exist under some ``degree conditions'' (see Chapter \ref{section:pbw}).

\chapter{Deformations of path algebras of quivers with relations}
\label{section:deformations-of-path-algebras}

In this section we prove our main results about the formal deformation theory of path algebras of quivers with relations. We start with a summary of the results from the point of view of deformation problems and introduce the notation used in the remainder of this section.

\section{Summary}
\label{subsection:summary}

In the case of formal deformations our results can be summarized as follows.

\begin{theorem}
\label{theorem:deformationreduction}
\index{reduction system!deformation}\index{deformation!of reduction system}\index{deformation!of associative algebra}\index{deformation!of ideal of relations}
Given any finite quiver $Q$ and any two-sided ideal of relations $I \subset \Bbbk Q$, let $A = \Bbbk Q / I$ be the quotient algebra and let $R$ be any reduction system satisfying \textup{($\diamond$)} for $I$.

There is an equivalence of formal deformation problems between
\begin{enumerate}
\item \label{theorem:deformationreduction1} deformations of the associative algebra structure on $A$
\item \label{theorem:deformationreduction2} deformations of the reduction system $R$
\item \label{theorem:deformationreduction3} deformations of the relations $I$.
\end{enumerate}
\end{theorem}

This equivalence can be understood as an isomorphism of the corresponding deformation functors and is proved in \S\ref{subsection:generatorsrelations}.

Our approach for studying the deformations of $A = \Bbbk Q / I$ is based on replacing the bar resolution of $A$ by the projective $A$-bimodule resolution\index{resolution!Chouhy--Solotar} $P_\ldot$ obtained from a reduction system. The homotopy deformation retract between the bar resolution\index{resolution!bar} $\Bar_\ldot$ and the resolution $P_\ldot$ (see Chapter \ref{section:homotopydeformationretract}) allows one to transfer the DG Lie algebra\index{DG Lie algebra}\index{algebra!DG Lie} structure on the Hochschild cochain complex to an L$_\infty$ algebra\index{L$_\infty$!algebra}\index{algebra!L$_\infty$} structure on a cochain complex associated to $P_\ldot$. It turns out that this L$_\infty$ algebra naturally controls formal deformations of the reduction system $R$. Moreover, it also turns out to be rather convenient for concrete computations as one can give a combinatorial criterion for an element to satisfy the Maurer--Cartan equation\index{Maurer--Cartan!equation} (see \S\ref{subsection:computinghigherbrackets}).

We develop the theory for $A = \Bbbk Q / I$ for an arbitrary finite quiver $Q$ and an arbitrary two-sided ideal $I$ of relations. Already the case of a quiver with one vertex and $d$ loops has an interesting application: we can give a combinatorial construction of a quantization of certain algebraic Poisson structures on $\mathbb A^d$ (see Chapter \ref{section:relationtoquantization}).

We shall start by explaining the statement of Theorem \ref{theorem:deformationreduction} and fixing the notation used throughout this section.

\subsection{Outline and notation}
\label{subsubsection:notation}

Let $\Bbbk Q$ be the path algebra of a finite quiver $Q$, let $I \subset \Bbbk Q$ be a two-sided ideal of relations and let $A = \Bbbk Q / I$ denote the quotient algebra under the natural projection $\pi \colon \Bbbk Q \to \Bbbk Q / I$.

Fix a reduction system $R = \{ (s, \varphi_s) \}$ satisfying the condition ($\diamond$) for $I$ (cf.\ Proposition \ref{proposition:existencereduction}). Recall from Remark \ref{remark:f} that $R$ is determined by the set $S = \{ s \mid (s, \varphi_s) \in R \}$ and the $\Bbbk Q_0$-bimodule map $\varphi \in \Hom (\Bbbk S, \Bbbk \Irr_S)$ with $\varphi (s) = \varphi_s$.

Since $R$ satisfies ($\diamond$), we have an $\Bbbk Q_0$-bimodule isomorphism $\sigma \colon A \simeq \Bbbk \Irr_S$ \eqref{sigma} and thus may identify $\Hom (\Bbbk S, A)$ with $\Hom (\Bbbk S, \Bbbk \Irr_S)$ (see Remark \ref{remark:reduction}) and to emphasize the analogy with the Hochschild $2$-cochains $\Hom_{\Bbbk} (A \otimes_{\Bbbk} A, A)$ we view $\phi$ as an element in $\Hom (\Bbbk S, A)$, so that $\phi_s = \sigma \phi (s)$.

To keep the notation simple we usually state the results in the context of one-parameter deformations, i.e.\ as deformations over $\Bbbk \llrr{t}$ with maximal ideal $\mathfrak m = (t)$, but in all results one may replace $\Bbbk \llrr{t}$ by any complete local Noetherian $\Bbbk$-algebra $(B, \mathfrak m)$.

For the rest of the section let
\begin{equation}
\label{g}
\widetilde \varphi \in \Hom (\Bbbk S, A) \hatotimes \mathfrak m
\end{equation}
so that $\widetilde \varphi$ can be written as
\begin{flalign*}
&& \widetilde \varphi &= \widetilde \varphi_1 t + \widetilde \varphi_2 t^2 + \widetilde \varphi_3 t^3 + \dotsb && \mathllap{\widetilde \varphi_i \in \Hom (\Bbbk S, A)}
\end{flalign*}
and let $\widetilde \varphi^{(n)} = \widetilde \varphi_1 t + \dotsb + \widetilde \varphi_n t^n$ denote the image under tensoring by $\otimes_{\Bbbk \llrr{t}} \Bbbk [t] / (t^{n+1})$. We set $\widetilde \varphi_s = \sigma \widetilde \varphi (s)$ for any $s \in S$, where $\sigma$ is extended $\Bbbk \llrr{t}$-linearly. (The subscript on $\widetilde \varphi$ is thus either a natural number $i$ or $1, 2, \dotsc$ used as an index, or an element $s \in S$ used as a shorthand, but the usage should be clear both from the notation and from the context.)

The maps $\widetilde \varphi$ (\ref{g}) are candidates for deformations, deforming $R$ determined by $\varphi \in \Hom (\Bbbk S, A)$ to a (formal) reduction system determined by $\varphi + \widetilde \varphi$. A priori {\it we do not make any additional assumptions on $\widetilde \varphi$}, but we shall show that the objects (bilinear maps, reduction systems, ideals) that can be associated to $\widetilde \varphi$ have the expected nice properties precisely when $\widetilde \varphi$ is a Maurer--Cartan element\index{Maurer--Cartan!elements} of the L$_\infty[1]$ algebra $\mathbf p (Q, R) \hatotimes \mathfrak m$, which we construct via homotopy transfer in \S\ref{subsection:linfinity}. More precisely, to any $\widetilde \varphi \in \Hom (\Bbbk S, A) \hatotimes \mathfrak m$ we can associate
\begin{enumerate}
\item a $\Bbbk \llrr{t}$-bilinear map $\star^{\mathrm H}_{\phi + \widetilde \phi} \colon A \llrr{t} \otimes A \llrr{t} \to A \llrr{t}$
\item \label{ari2} a formal reduction system\index{reduction system!deformation}\index{deformation!of reduction system} $\widehat R_{\phi + \widetilde \phi}$ for $\Bbbk Q \llrr{t}$
\item an ideal $\widehat I_{\phi + \widetilde \phi}$ in $\Bbbk Q \llrr{t}$
\end{enumerate}
such that
\begin{enumerate}
\item \label{ari1} \eqmakebox[ARI][r]{$\widehat A_{\phi + \widetilde \phi}$}${} = (A \llrr{t}, \star^{\mathrm H}_{\phi + \widetilde \phi})$ is a formal deformation of $A$
\item  \eqmakebox[ARI][r]{$\widehat R_{\phi + \widetilde \phi}$}${} = \{ (s, \phi_s + \widetilde \phi_s) \}$ is a formal deformation of $R$
\item \label{ari3} \eqmakebox[ARI][r]{$\widehat I_{\phi + \widetilde \phi}$}${} = \langle s - \phi_s - \widetilde \phi_s \rangle_{s \in S}^\compl$ is a formal deformation of $I$
\end{enumerate}
if and only if $\widetilde \varphi$ is a Maurer--Cartan element\index{Maurer--Cartan!elements} of the L$_\infty[1]$ algebra $\mathbf p (Q, R) \hatotimes \mathfrak m$. In fact, to \ref{ari2} we can associate an (a priori different) {\it combinatorial} star product $\star^{\mathrm C}_{\phi + \widetilde \phi}$ on $A \llrr{t}$ defined by performing rightmost reductions, which turns out to coincide with $\star^{\mathrm H}_{\phi + \widetilde \phi}$ which is defined via {\it homotopy transfer} (see Definitions \ref{definition:star} and \ref{definition:star_2} and Theorem \ref{theorem:interpretation-star}).

When $\widetilde \phi$ is a Maurer--Cartan element\index{Maurer--Cartan!elements} of $\mathbf p (Q, R) \hatotimes \mathfrak m$, we can also define a (generalized) Gutt star product $\star^{\mathrm G}_{\phi + \widetilde \phi}$ on $A \llrr{t}$ associated to \ref{ari3} (see Definition \ref{definition:gutt}). In this case, the three star products $\star^{\mathrm H}_{\phi + \widetilde \phi}, \star^{\mathrm C}_{\phi + \widetilde \phi}$ and $\star^{\mathrm G}_{\phi + \widetilde \phi}$ are associative and coincide (see Theorem \ref{theorem:equivalenceformal}).

The above formal deformations are to be understood in the following sense.

\subsubsection*{Formal deformations of associative algebras}

Recall from Chapter \ref{section:deformationtheory} that a formal deformation of an associative algebra $A$ is given by $\widehat A = (A \llrr{t}, \star)$, where
\begin{itemize}
\item $\star$ is a $\Bbbk \llrr{t}$-bilinear associative product on $A \llrr{t}$ and
\item $\widehat A / (t) \simeq A$ as $\Bbbk$-algebras.
\end{itemize}
We say that two deformations $(A \llrr{t}, \star)$ and $(A\llrr{t}, \star')$ are {\it (gauge) equivalent} if there is an automorphism of $\Bbbk \llrr{t}$-modules $T \colon A\llrr{t} \to A\llrr{t}$ such that
\begin{itemize}
\item $T \otimes_{\Bbbk \llrr{t}} \Bbbk = \id_A$\footnote{here we naturally identify $A\llrr{t} \otimes_{\Bbbk \llrr{t}} \Bbbk$ with $A$} and
\item $T (a \star b) = T (a) \star' T (b)$ for any $a, b \in A$. 
\end{itemize}

\subsubsection*{Formal deformations of reduction systems}
\index{reduction system!deformation}\index{deformation!of reduction system}

A formal deformation of $R = \{ (s, \phi_s) \}_{s \in S}$ is given by
\[
\widehat R_{\phi + \widetilde \phi} = \{ (s, \varphi_s + \widetilde \phi_s) \}_{s \in S}
\]
where $\widetilde \varphi \in \Hom (\Bbbk S, A) \hatotimes \mathfrak m$ such that $\widehat R_{\phi + \widetilde \phi}$ satisfies the {\it formal diamond condition}, by which we mean that for each $n \geq 1$, the reduction system $R_{\phi + \widetilde \phi}^{(n)} = \{ (s, \varphi_s + \widetilde \varphi_s^{(n)}) \}_{s \in S}$ for $\Bbbk Q [t] / (t^{n+1})$ satisfies that every path in $Q$ is reduction-unique.

\begin{remark}
\label{remarkformaldiamondlemma}
In fact, $\widehat R_{\phi + \widetilde \phi}$ can be viewed as an inverse limit of reduction systems for the path algebra of a finite quiver --- which might be denoted $Q [t]$ --- obtained by adding a loop $t_i$ at each vertex $i \in Q_0$ and adding relations $t_{\mathrm{s}(x)} x = x\, t_{\mathrm{t}(x)}$ for each $x \in Q_1$.

It follows again from the Diamond Lemma\index{Diamond Lemma!formal} that $\widehat R_{\phi + \widetilde \phi}$ satisfies the formal diamond condition if and only if all the overlap ambiguities\index{overlap ambiguity}\index{ambiguity!overlap} are resolvable if and only if the natural map $\Bbbk \Irr_S \llrr{t} \to \Bbbk Q \llrr{t} / \widehat I_{\phi + \widetilde \phi} $ is an isomorphism of $\Bbbk \llrr{t}$-modules.
\end{remark}

The following is the formal analogue of the notion of equivalence of reduction systems given in Definition \ref{definition:equivalencereduction}.

\begin{definition}
\label{definition:equivalenceformalreductionsystem}
\index{reduction system!equivalence}
Two formal deformations $\widehat R_{\phi + \widetilde \phi}$ and $\widehat R_{\phi + \widetilde \phi'}$ of $R$ are {\it equivalent} if there is an automorphism $T$ of $\Bbbk \Irr_S \llrr{t}$, viewed as $\Bbbk \llrr{t}$-module, such that
\begin{itemize}
\item $T \otimes_{\Bbbk \llrr{t}} \Bbbk = \id_{\Bbbk \Irr_S}$
\item $T(e_i) = e_i$ for any $i \in Q_0$ and
\item $T (\mathrm{red}_{\phi + \widetilde \phi'}^{(\infty)} (p)) = \mathrm{red}_{\phi + \widetilde \phi}^{(\infty)} (T (p_1) \dotsb T (p_m))$ for any path $p = p_1 \dotsb p_m$ with $p_i \in Q_1$.
\end{itemize}
\end{definition}

\begin{remark}
\label{remark:equivalenceformalreductionsystem}
Similar to Lemma \ref{lemma:equivalencereduction} any equivalence $T$ between $\widehat R_{\phi + \widetilde \phi}$ and $\widehat R_{\phi + \widetilde \phi'}$ is determined by an element $\psi \in \Hom (\Bbbk Q_1, \Bbbk \Irr_S) \hatotimes \mathfrak m$ such that
\begin{itemize}
\item $T (x) = x + \psi (x)$ for any $x \in Q_1$ and
\item $T (\phi (s) + \widetilde \phi' (s)) = \mathrm{red}_{\phi + \widetilde \phi}^{(\infty)} (T (s_1) \dotsb T (s_m))$
for any $s \in S$ with $s = s_1 \dotsb s_m$ and $s_i \in Q_1$.
\end{itemize}
(Here $\mathrm{red}_{\phi + \widetilde \phi}^{(\infty)}$ is the formal analogue of \eqref{align:redinfty}, which in the formal context can be defined by taking $z = 1$ in \eqref{red}.)
\end{remark}

\subsubsection*{Formal deformations of ideals of relations}

Lastly, a formal deformation of an ideal $I \subset \Bbbk Q$ is given by a two-sided ideal $\widehat I$ of $\Bbbk Q \llrr{t}$ such that
\begin{itemize}
\item $\widehat I$ is complete as a $\Bbbk Q \llrr{t}$-bimodule
\item $\widehat I \otimes_{\Bbbk \llrr{t}} \Bbbk = I$ and
\item there is an isomorphism of $\Bbbk \llrr{t}$-modules $\Phi \colon A \llrr{t} \to \Bbbk Q \llrr{t} / \widehat I$ with $\Phi \otimes_{\Bbbk \llrr{t}} \Bbbk = \id_A$.
\end{itemize}
Let $\widehat I$ and ${\widehat I}'$ be two formal deformations of the ideal $I$ and let $\Phi$ and $\Phi'$ denote the corresponding isomorphisms. We say that $\widehat I$ and $\widehat I'$ are {\it equivalent} if there is an automorphism $T$ of $A \llrr{t}$ viewed as $\Bbbk \llrr{t}$-module such that
\begin{itemize}
\item $T \otimes_{\Bbbk \llrr{t}} \Bbbk = \id_A$ and
\item $\Phi'  T  \Phi^{-1} \colon \Bbbk Q \llrr{t} / \widehat I \to \Bbbk Q \llrr{t} / \widehat I'$ is an isomorphism of $\Bbbk \llrr{t}$-algebras.
\end{itemize}

\subsection{Overview}

We give a brief overview of the remainder of this section. In \S\ref{subsection:linfinity} we use the homotopy deformation retract constructed in Chapter \ref{section:homotopydeformationretract} to obtain the L$_\infty[1]$ algebra $\mathbf p (Q, R)$ controlling deformations of reduction systems. In \S\ref{subsection:combinatorialstarproduct} we define a combinatorial star product (see Definition \ref{definition:star_2}), which is defined by performing rightmost reductions and may be used to give a combinatorial criterion for the Maurer--Cartan equation for $\mathbf p (Q, R)$ (see Theorem \ref{theorem:higher-brackets} in \S\ref{subsection:computinghigherbrackets}). This criterion is given by checking associativity of the combinatorial star product on a (usually finite) collection of elements and when this holds, the same formula describes the deformed multiplication for {\it all} elements. Many of the computations are relegated to the proofs of various technical lemmas which can safely be skipped on a first reading.

In \S\ref{subsection:generatorsrelations} we prove the equivalence of the various deformation problems. In \S\ref{subsection:firstorder} we show how the combinatorial star product to first order can be used to compute the second Hochschild cohomology of the algebra.

\section{An L$_{\infty}$ algebra structure}
\label{subsection:linfinity}
\index{L$_\infty$!algebra}\index{algebra!L$_\infty$}

In Chapter \ref{section:homotopydeformationretract} we constructed a homotopy deformation retract
\begin{align}\label{hdr:resolution}
\begin{tikzpicture}[baseline=-2.6pt,description/.style={fill=white,inner sep=1pt,outer sep=0}]
\matrix (m) [matrix of math nodes, row sep=0em, text height=1.5ex, column sep=3em, text depth=0.25ex, ampersand replacement=\&, inner sep=1.5pt]
{
(P_\ldot, \partial_\ldot) \& (\Bar_\ldot, d_\ldot) \\
};
\path[->,line width=.4pt,font=\scriptsize, transform canvas={yshift=.4ex}]
(m-1-1) edge node[above=-.4ex] {$F_\ldot$} (m-1-2)
;
\path[->,line width=.4pt,font=\scriptsize, transform canvas={yshift=-.4ex}]
(m-1-2) edge node[below=-.4ex] {$G_\ldot$} (m-1-1)
;
\path[->,line width=.4pt,font=\scriptsize, looseness=8, in=30, out=330]
(m-1-2.355) edge node[right=-.4ex] {$h_\ldot$} (m-1-2.5)
;
\end{tikzpicture}
\end{align}
between the bar resolution of $A = \Bbbk Q / I$ and the resolution $P_\ldot$ obtained from a reduction system $R$ satisfying the condition ($\diamond$) for $I$.

In this section, we apply the functor $\Hom_{A^\e}(\blank, A)$ to (\ref{hdr:resolution}) and get a homotopy deformation retract for the Hochschild cochain complex\index{Hochschild!cochain complex} of $A$. By the Homotopy Transfer Theorem (see for example \cite[\S 10.3]{lodayvallette}), one can then transfer the DG Lie$[1]$ algebra\index{DG Lie algebra}\index{algebra!DG Lie} structure on $\Hom (\bar A^{\otimes \hdot + 2}, A)$ to an L$_\infty [1]$ algebra structure on $\Hom_{A^\e} (P_{\ldot + 2}, A)$.

First note that for any $\Bbbk Q_0$-bimodule $M$, there is a natural isomorphism
\begin{equation}
\label{naturalisomorphism}
\Hom_{A^\e} (A \otimes M \otimes A, A) \simeq \Hom (M, A)
\end{equation}
given by $\check f \mapsto f$, where $f (m) = \check f (1 \otimes m \otimes 1)$ (cf.\ Notation \ref{notation:tensorhom}).

Consider the complex $P^{\hdot + 2} = \Hom (\Bbbk S_{\ldot + 2}, A)$ whose differential $\langle \blank \rangle$ is given by 
\begin{equation}
\label{shifted}
\langle f \rangle (w) = (-1)^{i} \check f (\partial_{i+1} (1 \otimes w \otimes 1))
\end{equation}
for any $f \in P^i$ and $w \in S_{i+1}$. Recall from Remark \ref{notation:h} the DG Lie$[1]$ algebra\index{DG Lie algebra}\index{algebra!DG Lie}
\[
\mathbf h (A) = (\Hom (\bar A^{\otimes \hdot+2}, A), [\blank], [\blank {,} \blank])
\]
which controls the deformation theory of $A$.

From (\ref{hdr:resolution}), we get the following homotopy deformation retract 
\begin{align}\label{hdr:Hochschild}
\begin{tikzpicture}[baseline=-2.6pt,description/.style={fill=white,inner sep=1pt,outer sep=0}]
\matrix (m) [matrix of math nodes, row sep=0em, column sep=4em, text depth=0.25ex, ampersand replacement=\&, inner sep=2.5pt]
{
P^{\hdot+2} \& \Hom(\bar A^{\otimes \bullet +2}, A)\\
};
\path[->,line width=.4pt,font=\scriptsize, transform canvas={yshift=.4ex}]
(m-1-1) edge node[above=-.4ex] {$ G^{\hdot +2}$} (m-1-2)
;
\path[->,line width=.4pt,font=\scriptsize, transform canvas={yshift=-.4ex}]
(m-1-2) edge node[below=-.4ex] {$ F^{\hdot +2}$} (m-1-1)
;
\path[->,line width=.4pt,font=\scriptsize, looseness=8, in=30, out=330]
(m-1-2.357) edge node[right=-.4ex] {$h^{\hdot+2}$} (m-1-2.3)
;
\end{tikzpicture}
\end{align}
where $F^n$ and $G^n$ are respectively the duals of $F_n$ and $G_n$ under the natural isomorphism (\ref{naturalisomorphism}), and $h^n$ is given by $f \mapsto (-1)^n \check f \circ h_n$ for any $f \in \Hom_{\Bbbk} (A^{\otimes_{\Bbbk} n+1}, A)$. They satisfy
\begin{flalign*}
&&
\begin{aligned}
F^n G^n - \id &= 0 \\
G^n F^n - \id &= h^{n} d' + d' h^{n-1}
\end{aligned}
&& 
\begin{aligned}
F^n h^n &= 0 \\
h^{n-1} G^n &= 0 \\
h^{n-1} h^n &= 0
\end{aligned}
&&
\end{flalign*}
where $d'(f): = [f]$ is the differential of $\mathbf h(A)$.

Thus, applying the Homotopy Transfer Theorem for DG Lie$[1]$ algebras to the homotopy deformation retract (\ref{hdr:Hochschild}) we obtain the following.

\begin{theorem}
\label{theorem:linfinitytransfer}
There exists an L$_\infty[1]$ algebra\index{L$_\infty$!algebra}\index{algebra!L$_\infty$}
\begin{equation}
\label{pqr}
\mathbf p (Q, R) =  (P^{\hdot+2} , \langle \blank \rangle, \langle \blank {,} \blank \rangle, \langle \blank {,} \blank {,} \blank \rangle, \dotsc)
\end{equation}
such that for each $k > 1$ we have 
\[
\langle f_1, \dotsc, f_k \rangle = \sum_{i=1}^{k-1}\sum_{\substack{\sigma \in \mathfrak S_{i,k-i}\\ \sigma(1) < \sigma(i+1)}}  F^\hdot ([G_i^\hdot(f_{\sigma(1)}, \dotsc, f_{\sigma(i)}), G_{k-i}^\hdot(f_{\sigma(i+1)}, \dotsc, f_{\sigma(k)})]).
\]
Moreover, the injection of complexes $G^\hdot \colon \mathbf p (Q, R)  \to \mathbf h(A)$ extends to a quasi-isomorphism of L$_\infty[1]$ algebras
$$
(G^\hdot_1, G^\hdot_2, \dotsc)  \colon \mathbf p (Q, R) \tosim \mathbf h (A)
$$
with $G^\hdot_1 = G^\hdot$ and for $k > 1$ 
\[
G_k^\hdot(f_1, \dotsc, f_k) = \sum_{i=1}^{k-1}\sum_{\substack{\sigma \in \mathfrak S_{i,k-i}\\ \sigma(1) < \sigma(i+1)}}  h^\hdot ([G_i^\hdot(f_{\sigma(1)}, \dotsc, f_{\sigma(i)}), G_{k-i}^\hdot(f_{\sigma(i+1)}, \dotsc, f_{\sigma(k)})]).
\]
\end{theorem}

\begin{remark}
The last formula of Theorem \ref{theorem:linfinitytransfer} is often expressed in terms of graftings of binary trees. For example, the following is a tree appearing in the expression for $G_7^\hdot$
\[
\begin{tikzpicture}[baseline=-2em,x=.35em,y=.35em]
\draw[line width=.4pt] (-6,0) node[draw,shape=circle,scale=.25] (one) {};
\draw[line width=.4pt] (-4,0) node[draw,shape=circle,scale=.25] (two) {};
\draw[line width=.4pt] (-2,0) node[draw,shape=circle,scale=.25] (thr) {};
\draw[line width=.4pt]  (0,0) node[draw,shape=circle,scale=.25] (fou) {};
\draw[line width=.4pt]  (2,0) node[draw,shape=circle,scale=.25] (fiv) {};
\draw[line width=.4pt]  (4,0) node[draw,shape=circle,scale=.25] (six) {};
\draw[line width=.4pt]  (6,0) node[draw,shape=circle,scale=.25] (sev) {};
\foreach \c in {{(-5,-3)},{(-3.5,-6.75)},{(1,-3)},{(5,-3)},{(3,-7)},{(-.25,-12.25)}} {%
\draw[fill,shift={\c},rotate=45] (-.8pt,-.8pt) rectangle ++(1.6pt,1.6pt);
}
\draw[line width=.4pt] (one) -- ++(0,-1) -- ++(1,-1) -- ++(0,-2) -- ++(1.5,-1.5) -- ++(0,-2.5) -- ++(3.25,-3.25) -- ++(0,-2);
\draw[line width=.4pt] (two) -- ++(0,-1) -- ++(-1,-1);
\draw[line width=.4pt] (thr) -- ++(0,-4) -- ++(-1.5,-1.5);
\draw[line width=.4pt] (fou) -- ++(0,-1) -- ++(1,-1) -- ++(0,-2) -- ++(2,-2) -- ++(0,-2) -- ++(-3.25,-3.25);
\draw[line width=.4pt] (fiv) -- ++(0,-1) -- ++(-1,-1);
\draw[line width=.4pt] (six) -- ++(0,-1) -- ++(1,-1) -- ++(0,-2) -- ++(-2,-2);
\draw[line width=.4pt] (sev) -- ++(0,-1) -- ++(-1,-1);
\end{tikzpicture}
= h^\hdot \biggl( \biggl[ \,
\begin{tikzpicture}[baseline=-1.4em,x=.3em,y=.3em]
\draw[line width=.4pt] (-6,0) node[draw,shape=circle,scale=.25] (one) {};
\draw[line width=.4pt] (-4,0) node[draw,shape=circle,scale=.25] (two) {};
\draw[line width=.4pt] (-2,0) node[draw,shape=circle,scale=.25] (thr) {};
\foreach \c in {{(-5,-3)},{(-3.5,-6.75)}} {%
\draw[fill,shift={\c},rotate=45] (-.75pt,-.75pt) rectangle ++(1.5pt,1.5pt);
}
\draw[line width=.4pt] (one) -- ++(0,-1) -- ++(1,-1) -- ++(0,-2) -- ++(1.5,-1.5) -- ++(0,-2.5);
\draw[line width=.4pt] (two) -- ++(0,-1) -- ++(-1,-1);
\draw[line width=.4pt] (thr) -- ++(0,-4) -- ++(-1.5,-1.5);
\end{tikzpicture}
\,,
\begin{tikzpicture}[baseline=-1.4em,x=.3em,y=.3em]
\draw[line width=.4pt]  (0,0) node[draw,shape=circle,scale=.25] (fou) {};
\draw[line width=.4pt]  (2,0) node[draw,shape=circle,scale=.25] (fiv) {};
\draw[line width=.4pt]  (4,0) node[draw,shape=circle,scale=.25] (six) {};
\draw[line width=.4pt]  (6,0) node[draw,shape=circle,scale=.25] (sev) {};
\foreach \c in {{(1,-3)},{(5,-3)},{(3,-7)}} {%
\draw[fill,shift={\c},rotate=45] (-.75pt,-.75pt) rectangle ++(1.5pt,1.5pt);
}
\draw[line width=.4pt] (fou) -- ++(0,-1) -- ++(1,-1) -- ++(0,-2) -- ++(2,-2) -- ++(0,-2);
\draw[line width=.4pt] (fiv) -- ++(0,-1) -- ++(-1,-1);
\draw[line width=.4pt] (six) -- ++(0,-1) -- ++(1,-1) -- ++(0,-2) -- ++(-2,-2);
\draw[line width=.4pt] (sev) -- ++(0,-1) -- ++(-1,-1);
\end{tikzpicture} \, \biggr] \biggr)
\]
with the two trees on the right-hand side appearing in the expressions for $G_3^\hdot$ and $G_4^\hdot$, respectively. (Here $\begin{tikzpicture}[baseline=-.5em,x=.65em,y=.65em]
\draw[line width=.4pt] (0,0) node[draw,shape=circle,scale=.3] (one) {}; \draw[line width=.4pt] (one.south) -- ++(0,-.3em); \end{tikzpicture}$\ denotes an input of $G^\hdot _1 (f_i)$,
$\begin{tikzpicture}[baseline=.2em,x=.65em,y=.65em]
\draw[line width=.4pt] (0,0) -- (0,.7) -- ++(-.8,.8);
\draw[line width=.4pt] (0,.7) -- ++(.8,.8);
\end{tikzpicture}$
the composition given by the Gerstenhaber bracket $[\blank{,}\blank]$
and \,
$\begin{tikzpicture}[baseline=.2em,x=.65em,y=.65em]
\draw[line width=.4pt] (0,1.5) -- (0,0);
\draw[fill,shift={(0,.75)},rotate=45] (-.9pt,-.9pt) rectangle ++(1.8pt,1.8pt);
\end{tikzpicture}$
\, an occurrence of $h^\hdot$.)
\end{remark}

Recall from Remark \ref{ordinarydgalgebra} that there is a one-to-one correspondence between L$_\infty$ algebras\index{L$_\infty$!algebra}\index{algebra!L$_\infty$} and L$_\infty[1]$ algebras\index{L$_\infty$!algebra!L$_\infty[1]$}\index{algebra!L$_\infty$}. Instead of using L$_\infty$ algebras, one of the advantages of using L$_\infty[1]$ algebras is that no signs appear in the formulae of the higher brackets $\langle \blank, \dotsc, \blank \rangle$ and the higher morphisms $G_k^\hdot$ in Theorem \ref{theorem:linfinitytransfer}.

\begin{definition}
\label{definition:star}
{\it The homotopical star product.}
\index{star product!homotopical|textbf}\index{homotopical star product|textbf}
Given $\widetilde \varphi \in \Hom (\Bbbk S, A) \hatotimes \mathfrak m$, one can use the higher morphisms $G^\hdot_k$ of Theorem \ref{theorem:linfinitytransfer} to define a $\Bbbk \llrr{t}$-bilinear map $\star^{\mathrm H}_{\phi + \widetilde \phi} \colon A \llrr{t} \otimes A \llrr{t} \to A \llrr{t}$ by
\begin{equation}
\label{starproduct}
a \star^{\mathrm H}_{\phi + \widetilde \phi} b = ab + \sum_{k \geq 1} \frac{1}{k!} G^\hdot_k (\widetilde \phi^{\otimes k}) (a \otimes b)
\end{equation}
for $a, b \in A$.
\end{definition}

\begin{corollary}
\label{corollary:controls}
The L$_\infty[1]$ algebra $\mathbf p (Q, R)$ controls the formal deformation theory of $A$. More precisely,
\begin{enumerate}
\item given any Maurer--Cartan element\index{Maurer--Cartan!elements} $\widetilde \phi \in \Hom (\Bbbk S, A) \hatotimes \mathfrak m$, the map $\star^{\mathrm H}_{\phi + \widetilde \phi}$ defines an associative multiplication on $A \llrr{t}$
\item if $\widetilde \phi, \widetilde \phi'$ are homotopic Maurer--Cartan elements\index{Maurer--Cartan!elements!equivalence}, then $\star^{\mathrm H}_{\phi + \widetilde \phi}$ and $\star^{\mathrm H}_{\phi + \widetilde \phi'}$ are gauge equivalent
\item any formal deformation $(A \llrr{t}, \star)$ of $A$ is gauge equivalent to $(A \llrr{t}, \star^{\mathrm H}_{\phi + \widetilde \phi})$ for some Maurer--Cartan element $\widetilde \phi$.
\end{enumerate}
\end{corollary}

\begin{proof}
This follows straight from Theorems \ref{theorem:quasiisomorphism} \ref{qi1} and \ref{theorem:linfinitytransfer} and Corollary \ref{corollary:gerstenhaber}.
\end{proof}

\begin{remark}
Note that if $R'$ is any other reduction system\index{reduction system!deformation}\index{deformation!of reduction system} satisfying ($\diamond$) for $I$, then the same construction gives an L$_\infty[1]$ algebra\index{L$_\infty$!algebra}\index{algebra!L$_\infty$} $\mathbf p (Q, R')$ and there is an L$_\infty[1]$ quasi-isomorphism\index{L$_\infty$!quasi-isomorphism}\index{algebra!L$_\infty$} $\mathbf p (Q, R) \simeq \mathbf p (Q, R')$ (cf.\ Lemma \ref{lemma:quasi-inverse}), i.e.\ any reduction system can be used to study the deformation theory of $A$.
\end{remark}

\begin{remark}
\label{remark:nooverlap}
If the reduction system $R$ has no overlaps\index{overlap ambiguity}\index{ambiguity!overlap}, then $P^3 = 0$ and the obstruction space $\HH^3 (A, A)$ vanishes, so that {\it any} $2$-cochain satisfies the Maurer--Cartan equation and thus $\star^{\mathrm H}_{\phi + \widetilde \phi}$ is associative for any $\widetilde \varphi$. In \S\ref{subsection:combinatorialstarproduct} we give an explicit combinatorial formula for this product.
\end{remark}

\begin{remark}
\label{remark:actualtransfer}
Let $\widetilde \varphi$ be a Maurer--Cartan element of $\mathbf p (Q, R)$ such that $G_k^\hdot (\widetilde \varphi^{\otimes k}) = 0$ for $k \gg 0$. Then similar to Corollary \ref{corollary:controls} we have that $\star^{\mathrm H}_{\phi + \widetilde \varphi}$ is well-defined and associative. 
\end{remark}

\section{Combinatorial star product}
\label{subsection:combinatorialstarproduct}
\index{star product!combinatorial}\index{combinatorial star product}

In this section we introduce a combinatorial star product $\star^{\mathrm C}_{\phi + \widetilde \phi}$ and prove that $\star^{\mathrm C}_{\phi + \widetilde \phi}$ always coincides with $\star^{\mathrm H}_{\phi + \widetilde \phi}$ \eqref{starproduct}, even if $\widetilde \varphi$ is not a Maurer--Cartan element of $\mathbf p(Q, R) \hatotimes \mathfrak m$. This gives a combinatorial interpretation of the homotopical star product $\star^{\mathrm H}_{\phi + \widetilde \phi}$ appearing in Corollary \ref{corollary:controls} which turns out to be very useful in giving explicit formulae for the deformed product and studying its properties.

Let $\widetilde \varphi \in \Hom (\Bbbk S, A) \hatotimes \mathfrak m$, viewed as a degree $0$ element in the L$_\infty[1]$ algebra $\mathbf p (Q, R) \hatotimes \mathfrak m$, and denote by $\widehat R_{\phi + \widetilde \phi}$ the formal reduction system
\[
\widehat R_{\phi + \widetilde \phi} = \{ (s, \varphi_s + \widetilde \varphi_s) \mid s \in S \}.
\]
At this point we do not assume that $\widehat R_{\phi + \widetilde \phi}$ is formally reduction-unique.

To define $\star^{\mathrm C}_{\phi + \widetilde \phi}$, corresponding to performing reductions with respect to $\widehat R_{\phi + \widetilde \phi}$, we adapt the definition of the rightmost reduction $\mathrm{red}_{\varphi}$ given in (\ref{rightmost-reduction}) as follows.

Let $z$ be simply a formal (bookkeeping) central variable which will keep track of how many times $\widetilde \varphi$ has been used in the reduction. Similar to (\ref{rightmost-reduction}), we define a $\Bbbk \llrr{t} [z]$-linear map 
$$
\mathrm{red}_{\varphi + z \widetilde \varphi} \colon \Bbbk Q \llrr{t} [z] \to \Bbbk Q\llrr{t} [z] 
$$ 
which is uniquely determined by: for any path $p$ in $Q$
\begin{align*}
\mathrm{red}_{\varphi + z \widetilde \varphi} (p) =
\begin{cases}
p & \text{if $p$ is irreducible}\\
q \varphi_s r + z \, q \hair \widetilde \varphi_s r & \text{if $\widetilde{\mathrm{split}}{}^{\mathrm R}_2 (p) = q \otimes s \otimes r$}.
\end{cases}
\end{align*}
We denote by $\mathrm{red}_{\varphi + z \widetilde \varphi}^k$ the $k$th iterated composition of $\mathrm{red}_{\varphi + z \widetilde \varphi}$.

For any fixed $n \geq 0$, the action of $\mathrm{red}_{\varphi + z \widetilde \varphi}$ on $\Bbbk Q \llrr{t} [z] / (z^{n+1})$ is {\it stable}, i.e.\ $\mathrm{red}_{\varphi + z \widetilde \varphi}^k = \mathrm{red}_{\varphi + z \widetilde \varphi}^{k+1}$ for $k \gg 0$ since $\mathrm{red}_\varphi$ is stable (because $R$ is reduction-finite). This induces a $\Bbbk$-linear map
\begin{align*}
\begin{tikzpicture}[baseline=-2.6pt,description/.style={fill=white,inner sep=1.75pt}]
\matrix (m) [matrix of math nodes, row sep=.25em, text height=1.5ex, column sep=0, text depth=0.25ex, ampersand replacement=\&, column 3/.style={anchor=base west}]
{
\mathrm{red}_{\varphi + z \widetilde \varphi}^{(n)} \colon \&[-.25em] \Bbbk Q \llrr{t} \&[1.5em] \Bbbk Q \llrr{t}  [z] / (z^{n+1}) \\
\& p \& {[\mathrm{red}_{\varphi + z \widetilde \varphi}^k (p)]} \\
};
\path[->,line width=.4pt]
(m-1-2) edge (m-1-3)
;
\path[|->,line width=.4pt]
(m-2-2) edge (m-2-3)
;
\end{tikzpicture}
\end{align*}
where $[\mathrm{red}_{\varphi + z \widetilde \varphi}^k(p)]$ denotes the image of $\mathrm{red}_{\varphi + z \widetilde \varphi}^k (p)$ for $k \gg 0$ (depending on $p$) under the quotient map $\Bbbk Q \llrr{t}  [z] \to \Bbbk Q \llrr{t} [z] / (z^{n+1})$. In fact, the image of $\mathrm{red}_{\varphi + z \widetilde \varphi}^{(n)}$ lies in $\Bbbk \Irr_S [z] / (z^{n+1})$ and we have the following commutative diagram
\begin{align*}
\begin{tikzpicture}[baseline=-.5em,description/.style={fill=white,inner sep=1.75pt}]
\matrix (m) [matrix of math nodes, row sep=3em, text height=1.5ex, column sep=3em, text depth=0.25ex, ampersand replacement=\&]
{
\Bbbk Q \llrr{t}  \\
\Bbbk \Irr_S \llrr{t}  [z] / (z^{n}) \& \Bbbk \Irr_S \llrr{t}  [z] / (z^{n+1}) \rlap{.} \\
};
\path[->,line width=.4pt]
(m-1-1) edge node[left=-.3ex,font=\scriptsize] {$\mathrm{red}_{\varphi + z \widetilde \varphi}^{(n-1)}$} (m-2-1)
(m-1-1) edge node[above=.3ex,font=\scriptsize,pos=.55] {$\mathrm{red}_{\varphi + z \widetilde \varphi}^{(n)}$} (m-2-2)
(m-2-2) edge (m-2-1)
;
\end{tikzpicture}
\end{align*}
Taking the inverse limit, this induces a $\Bbbk \llrr{t}$-linear map
\begin{align}
\label{red}
\mathrm{red}^{(\infty)}_{\varphi + z \widetilde \varphi} \colon \Bbbk Q \llrr{t}   \to \Bbbk \Irr_S \llrr{t}   \llrr{z}.
\end{align}
Note that   $\mathrm{red}^{(\infty)}_{\varphi + z \widetilde \varphi} (u) = u$ for any $u\in \Irr_S$ and $\mathrm{red}^{(\infty)}_{\varphi + z \widetilde \varphi} (s) = \phi_s + z \widetilde \phi_s$ for any $s \in S$.

\begin{definition}
\label{definition:star_2}
\index{star product!combinatorial|textbf}\index{combinatorial star product|textbf}
{\it The combinatorial star product.} 
Let $\widetilde \varphi \in \Hom(\Bbbk S, A) \hatotimes \mathfrak m$. The $\Bbbk \llrr{t}$-bilinear operation $\star^{\mathrm C}_{\phi + \widetilde \phi}$ on $A \llrr{t}$ is given by
\begin{flalign}
\label{star2product}
&& a \star^{\mathrm C}_{\phi + \widetilde \phi} b &= \pi \bigl( \mathrm{red}^{(\infty)}_{\varphi + z \widetilde \varphi} (\sigma (a) \sigma (b))\bigr\rvert_{z=1} \bigr) && \mathllap{a, b \in A.}
\end{flalign}
\end{definition}

The expression $a \star^{\mathrm C}_{\phi + \widetilde \phi} b$ can be understood in a purely combinatorial way as follows. For any $k \geq 0$ let $a \star^{\mathrm C, k}_{\phi + \widetilde \phi} b$ denote the expression obtained by performing rightmost reductions on $\sigma(a) \sigma(b) \in \Bbbk Q$ using $\widetilde \varphi$ exactly $k$ times and $\varphi$ an arbitrary number of times until all elements are irreducible. Then $a \star^{\mathrm C, k}_{\phi + \widetilde \phi} b$ is the coefficient of $z^k$ in \eqref{star2product} (before evaluating $z$ to $1$) and
$$
a \star^{\mathrm C}_{\phi + \widetilde \phi} b = \sum_{k \geq 0} a \star^{\mathrm C, k}_{\phi + \widetilde \phi} b.
$$
Here $a \star^{\mathrm C, 0}_{\phi + \widetilde \phi} b = ab$ and $a \star^{\mathrm C}_{\phi + \widetilde \phi} b$ can be viewed as obtained from performing rightmost reductions with respect to $\widehat R_{\phi + \widetilde \phi}$. Note that $a \star^{\mathrm C}_{\phi + \widetilde \phi} b$ is well-defined for any $a, b \in A$ and any $\widetilde \phi \in \Hom (\Bbbk S, A)  \hatotimes \mathfrak m$ since the coefficient of $t^n$ is a finite sum, as we can use $\widetilde \phi = \widetilde \phi_1 t + \widetilde \phi_2 t^2 + \dotsb$ at most $n$ times. In other words, the coefficient of $t^n$ in $a \star^{\mathrm C, k}_{\phi + \widetilde \phi} b$ is zero for any $k > n$.

The combinatorial star product\index{star product!combinatorial}\index{combinatorial star product} $\star^{\mathrm C}_{\phi + \widetilde \phi}$ has the following properties.
\begin{lemma}
\label{lemma:star}
Let $\widetilde \varphi \in \Hom (\Bbbk S, A) \hatotimes \mathfrak m$ and write $\star^k = \star^{\mathrm C, k}_{\phi + \widetilde \phi}$. Let $u \in \Irr_S$ be any irreducible path.
\begin{enumerate}
\item \label{guttproperty1} If there exists $x \in Q_1$ such that $x u = s v$ for some $s \in S$, then for any $k \geq 1$ we have
$$
\pi(x) \star^k \pi(u) = \pi (\varphi_s) \star^k \pi(v) + \pi (\widetilde \varphi_s) \star^{k-1} \pi(v).
$$
Similarly, if $v$ is an irreducible path such that $u v = u' s$ for some $s \in S$ and some irreducible path $u'$, then we have 
$$
\pi(u) \star^k \pi(v)=\pi(u') \star^k \pi(\varphi_s)+\pi(u') \star^{k-1} \pi (\widetilde \varphi_s).
$$
\item \label{guttproperty2} Let $v$ be an irreducible path such that $uv$ is irreducible. Then for any $k \geq 1$
$$
\pi(u) \star^k \pi (v) = 0.
$$
\item \label{guttproperty3} Let $v$ be an irreducible path and write $u = u_1 \cdots u_n$ for $u_i \in Q_1$. Then for any $k \geq 1$
\begin{align*}
\pi (u) \star^k \pi(v) &= \sum_{i=1}^{k-1} \sum_{j=1}^{n-1} \pi (u_1 \cdots u_{j-1}) (\pi (u_j) \star^i (\pi (u_{j+1} \cdots u_n) \star^{k-i} \pi (v))) \\
& \quad + \sum_{j=1}^n \pi (u_1 \cdots u_{j-1}) (\pi (u_j) \star^k \pi (u_{j+1} \cdots u_n v)).
\end{align*}
 Note that for $k =1$ the first summands are zero. 
\end{enumerate}
\end{lemma}

\begin{proof}
The first and second assertions follow from the combinatorial description of $\star^k$, which is given by performing rightmost reductions using $\widetilde \phi$ exactly $k$ times and $\phi$ an arbitrary number of times.

To see \ref{guttproperty3}, the combinatorial description implies that for any $1\leq j \leq n-1$ we have
\begin{align*}
\pi(u_j\cdots u_n) \star^k \pi(v) =
 \sum_{i=0}^{k} \pi(u_j) \star^i (\pi(u_{j+1}\cdots u_n) \star^{k-i} \pi(v)).
\end{align*}
Multiplying $\pi(u_1\cdots u_{j-1})$ and taking the sum for $1\leq j\leq n-1$, we get \ref{guttproperty3}. Here we need to use the fact that $a \star^0 b = ab$. 
\end{proof}

\subsection{Relation to the homotopical star product}
\index{star product!homotopical}\index{homotopical star product}

For a general element $\widetilde \varphi \in \Hom (\Bbbk S, A) \hatotimes \mathfrak m$, the bilinear map $\star^{\mathrm C}_{\phi + \widetilde \phi}$ need not be associative. Regardless, the next result shows that $\star^{\mathrm C}_{\phi + \widetilde \phi}$ always coincides with $\star^{\mathrm H}_{\phi + \widetilde \phi}$ (cf.\ Definition \ref{definition:star}).

\begin{theorem}
\label{theorem:interpretation-star}
Let $\widetilde \varphi \in \Hom (\Bbbk S, A) \hatotimes \mathfrak m$ be arbitrary. Let $a, b \in A$ and write $G_0^\hdot (a \otimes b) := ab$. For any $k \geq 0$ we have
\begin{flalign}
\label{align:G_i0}
 \frac1{k!} G^\hdot _k (\widetilde \varphi^{\otimes k}) (a \otimes b) = a \star^{\mathrm C, k}_{\phi + \widetilde \phi} b.
\end{flalign}
Thus, $\star^{\mathrm C}_{\phi + \widetilde \varphi} = \star^{\mathrm H}_{\phi + \widetilde \varphi}$.
\end{theorem}

In the remainder of this subsection we give a proof of Theorem \ref{theorem:interpretation-star}. In the proof of this theorem, as well as in Lemmas \ref{lemma:technical} and \ref{lemma:higher-brackets}, we use the following shorthand notation.

\begin{notation}
For $k \geq 2$ and $1 \leq i < k$ let
\[
G_{i,k-i}^\hdot (\widetilde \phi^{\otimes k}) := G_i^\hdot (\widetilde \phi^{\otimes i}) \bullet G_{k-i}^\hdot(\widetilde \phi^{\otimes k-i}) 
\in \Hom (\bar A^{\otimes 3}, A) \hatotimes \mathfrak m
\]
where $\bullet$ is the Gerstenhaber circle product (cf.\ Definition \ref{definitiongerstenhaber}) and let $\check G_{i,k-i}^\hdot (\widetilde \phi^{\otimes k})$ denote the preimage of $G_{i,k-i}^\hdot(\widetilde \phi^{\otimes k})$ under the isomorphism (\ref{naturalisomorphism}).
\end{notation}

The following lemma is very useful to compute the terms $G_k^\hdot (\widetilde \phi^{\otimes k})$ appearing in $\star^{\mathrm H}_{\phi + \widetilde \phi}$.

\begin{lemma}
\label{lemma:technical}
Let $u, v$ be irreducible paths.
\begin{enumerate}
\item \label{technical1} We have the following recursive formula for any $k \geq 2$
\begin{align*}
G_k^\hdot (\widetilde \phi^{\otimes k})(\bar u \otimes \bar v) = -\sum_{i=1}^{k-1} {k \choose i} \check G_{i, k-i}^{\hdot} (\widetilde \phi^{\otimes k}) (h_2 (1 \otimes \bar u \otimes \bar v \otimes 1))
\end{align*}
where $h_2$ is given in \eqref{varpih}.
\item \label{technical2} If $u v$ is irreducible, then for any $k \geq 1$ we have 
$$
G_k^\hdot (\widetilde \phi^{\otimes k}) (\bar u \otimes \bar v) = 0.
$$
\end{enumerate} 
\end{lemma}

\begin{proof}
We first prove \ref{technical1}. By Theorem \ref{theorem:linfinitytransfer} we have
$$
G_k^\hdot (\widetilde \phi^{\otimes k}) = \sum_{i=1}^{k-1} \frac{1}{2} {k \choose i} h^2 ([G_i^{\hdot} (\widetilde \phi^{\otimes i}), G_{k-i}^{\hdot} (\widetilde \varphi^{\otimes k-i})])=- \sum_{i=1}^{k-1} {k \choose i} h^2 (G_{i, k-i}^{\hdot} (\widetilde \phi^{\otimes k})),
$$
where the minus sign comes from \eqref{align:dglie1}. 

Next we prove \ref{technical2} by induction on $k$. It follows from Lemma \ref{lemma:G2} \ref{G2ii} that $G_\ldot (1 \otimes \bar u \otimes \bar v \otimes 1) = 0$, which dually implies that $G^\hdot _1 (\widetilde \phi) (\bar u \otimes \bar v) = 0$.

For $k \geq 2$, and writing $u = u_1 \dotsb u_n$ for $u_i \in Q_1$, we have
\begin{align*}
G_k^\hdot (\widetilde \phi^{\otimes k}) (\bar u \otimes \bar v) &= - \sum_{i=1}^{k-1} {k \choose i} \check G_{i, k-i}^{\hdot} (\widetilde \phi^{\otimes k}) (\varpi_1 \mathrm{split_1} (u) \otimes \bar v \otimes 1) \\
&=0
\end{align*}
where the first identity follows from \ref{technical1} and (\ref{varphi1}): 
\[ 
h_2(1 \otimes \bar u \otimes \bar v \otimes 1) =  \varpi_1 F_1G_1(1 \otimes \bar u\otimes 1) \otimes \bar v \otimes 1= \varpi_1 \mathrm{split_1} (u) \otimes \bar v \otimes 1
\]
and  the second  identity from the induction hypothesis since any subpath of $uv$ is still irreducible.
\end{proof}

\begin{proof}[Proof of Theorem \ref{theorem:interpretation-star}]
For the proof, we shall use the relation $\preceq$ introduced in \cite[\S2]{chouhysolotar} for any reduction system $R$ satisfying ($\diamond$) for an ideal $I$. Here $\preceq$ is a relation on the set $\{ \lambda p \mid \lambda \in \Bbbk \setminus \{ 0 \}, p \in Q_\ldot \}$ defined as the least reflexive and transitive relation such that $\lambda p\preceq \mu q$ if there is a reduction $\mathfrak r$ such that $\mathfrak r (\mu q) = \lambda p + r$, where $r$ is a linear combination of paths and $p$ does not appear as a summand of $r$. We write $\lambda p\prec \mu q$ if $\lambda p\preceq \mu q $ and $\lambda p\neq \mu q$. It follows from \cite[Lem.~2.11]{chouhysolotar} that $\preceq$ satisfies the descending chain condition since $R$ satisfies the condition ($\diamond$).

We now prove (\ref{align:G_i0}) for $a = \pi (u)$ and $b = \pi (v)$ by induction on $k$. (Note that for $k = 0$ (\ref{align:G_i0}) holds by definition so the base case is $k = 1$.) Let us write $\star^k = \star^{\mathrm C,k}_{\phi + \widetilde \phi}$.

Let $k = 1$. We use the induction on the order $\prec$ for $u v$. If $uv$ is irreducible (i.e.\ $uv$ is minimal with respect to $\prec$) then by Lemmas \ref{lemma:star} \ref{guttproperty2} and \ref{lemma:technical} \ref{technical2} 
\[
\pi (u) \star^1 \pi (v) = 0 =  G_1^\hdot(\widetilde \varphi) (\bar u \otimes \bar v).
\]
If $uv$ is not irreducible, we consider the following two cases. In the first case, we  assume that $u \in Q_1$.  Write $uv = sw$ for $s \in S$ and $w$ irreducible.  It follows from Lemmas \ref{lemma:star} \ref{guttproperty1} and \ref{lemma:G2} \ref{G2i} that
\begin{align*}
\pi (u) \star^1 \pi (v) &= \pi(\widetilde \phi_s w) + \pi(\phi_s) \star^1 \pi(w)\\
G_1^\hdot(\widetilde \varphi) (\bar u \otimes \bar v) & =  \pi(\widetilde \phi_s w) + G_1^\hdot(\widetilde \varphi) (\bar \phi_s \otimes \bar w).
\end{align*}
Since $\varphi_s w \prec u v$ (i.e.\ each summand of $\varphi_sw$ is ``smaller'' than $uv$ with respect to the order $\prec$), by the induction hypothesis we have 
$\pi(\phi_s) \star^1 \pi(w) =  G_1^\hdot(\widetilde \varphi) (\bar \phi_s \otimes \bar w).$ This yields $ \pi (u) \star^1 \pi (v) = G_1^\hdot(\widetilde \varphi) (\bar u \otimes \bar v)$. 

In the second case, we assume that  $u = u_1\dotsb u_n$ for $u_i \in Q_1$. By Lemmas \ref{lemma:star} \ref{guttproperty3} and \ref{lemma:G2} \ref{G2ii} 
\begin{align*}
\pi(u) \star^1 \pi(v) &= \sum_{j=1}^n \pi(u_1\dotsb u_{j-1}) ( \pi(u_j) \star^1 \pi(u_{j+1}\dotsb u_nv))\\
G_1^\hdot(\widetilde \phi) ( \bar u \otimes \bar v) & = \sum_{j=1}^n \pi(u_1\dotsb u_{j-1})  G_1^\hdot(\widetilde \phi) ( \bar u_j \otimes \overbar{u_{j+1}\dotsb u_nv}).
\end{align*}
Then from the first case it follows  that $\pi(u) \star^1 \pi(v) = G_1^\hdot(\widetilde \phi) ( \bar u \otimes \bar v)$.

Let $k \geq 2$.  We also use the induction on the order $\prec$ for $u v$. By Lemmas \ref{lemma:star} \ref{guttproperty2} and \ref{lemma:technical} \ref{technical2}, the identity (\ref{align:G_i0}) holds if $u v$ is irreducible. If $u v$ is not irreducible,  we consider the following two cases. In the first case, we assume that  $u \in Q_1$. Then we may write $u v = s w$ with $s \in S$ and $w$ irreducible. On the one hand, we have
\begin{equation}
\label{align:G_i}
\begin{aligned}
\pi (u) \star^k \pi (v) &= \pi (\varphi_s) \star^k \pi (w) + \pi (\widetilde \varphi_s) \star^{k-1} \pi (w) \\
&= \frac{1}{k!} G_k^\hdot (\widetilde \varphi^{\otimes k}) (\bar \varphi_s \otimes \bar w) + \pi (\widetilde \varphi_s) \star^{k-1} \pi (w)
\end{aligned}
\end{equation}
where the first identity follows from Lemma \ref{lemma:star} \ref{guttproperty1} and the second identity from the induction hypothesis since $\varphi_s w \prec u v$.

On the other hand, by Lemmas \ref{lemma:technical} \ref{technical1} and \ref{lemma:G2} \ref{G2i} we have
\begin{align}
\label{align:G_i1}
G_k^\hdot (\widetilde \varphi^{\otimes k}) (\bar u \otimes \bar v) = \sum_{i=1}^{k-1} {k \choose i}  \check G_{i, k-i}^{\hdot} (\widetilde \varphi^{\otimes k}) \bigl( \varpi_1 \, \mathrm{split}_1(s) \otimes \bar w \otimes 1  \bigr)  + G_k^\hdot (\widetilde \varphi^{\otimes k}) (\bar \varphi_s \otimes \bar w).
\end{align}
Combining (\ref{align:G_i}) and (\ref{align:G_i1}), we observe that in order to prove (\ref{align:G_i0}), it suffices to prove the following identity:
\begin{equation} 
\label{align:G_i2}
\pi (\widetilde \varphi_s) \star^{k-1} \pi (w) = \frac{1}{k!} \sum_{i=1}^{k-1} {k \choose i}  \check G_{i, k-i}^{\hdot} (\widetilde \varphi^{\otimes k}) \bigl( \varpi_1 \, \mathrm{split}_1(s) \otimes \bar w \otimes 1  \bigr)  .
\end{equation}

Writing $s = s_1 \dotsb s_m$ with $s_1, \dotsc, s_m \in Q_1$. By Definition \ref{definitiongerstenhaber} for each $1 \leq i<k$ we have 
\begin{align}\label{align:RHS}
\check G_{i, k-i}^{\hdot} (\widetilde \varphi^{\otimes k}) \bigl( \varpi_1 \, \mathrm{split}_1(s) \otimes \bar w \otimes 1  \bigr) =  G_{i}^\hdot (\widetilde \varphi^{\otimes i})(G_{k-i}^\hdot (\widetilde \varphi^{\otimes k-i})(\bar s_1 \otimes \overbar{s_{2}\dotsb s_m}) \otimes \bar w )  
\end{align}
since by Lemma  \ref{lemma:technical} \ref{technical2}  for each $1 < j<m$ we have 
\begin{align*}
G_{k-i}^\hdot (\widetilde \varphi^{\otimes k-i}) (\bar s_j \otimes \overbar{s_{j+1}\dotsb s_m})  = 0= G_{k-i}^\hdot (\widetilde \varphi^{\otimes k-i}) (\overbar{s_{j+1}\dotsb s_m} \otimes \bar w).
\end{align*}
From \eqref{align:RHS} it follows that the right-hand side of \eqref{align:G_i2} is equal to 
\begin{align}\label{align:RHS1}
\frac{1}{(k-1)!} G_{k-1}^\hdot (\widetilde \varphi^{\otimes k-1}) (\overbar{\widetilde \varphi_s} \otimes \bar w)
\end{align}
since by the induction hypothesis ($k-i<k$) we have
\[
G_{k-i}^\hdot (\widetilde \varphi^{\otimes k-i})(\bar s_1 \otimes \overbar{s_{2}\dotsb s_m})= \pi(s_1) \star^{k-i} \pi(s_2\dotsb s_m) = 
\begin{cases}
 \widetilde \phi_s & \text{if $k - i = 1$}\\
 0 & \text{if $k - i > 1$}
\end{cases} 
\] 
By the induction hypothesis  again, \eqref{align:RHS1} is further equal to $\pi (\widetilde \varphi_s) \star^{k-1} \pi (w).$ This proves \eqref{align:G_i2}.

Now we consider the second case where $u = u_1 \dotsb u_n$ for $u_i \in Q_1$ is an arbitrary irreducible path. By Lemmas \ref{lemma:technical} \ref{technical1} and \ref{lemma:G2} \ref{G2ii} we have
\begin{align}\label{align:729u}
\begin{aligned}
G_k^\hdot (\widetilde \varphi^{\otimes k}) (\bar u \otimes \bar v) ={}& - \sum_{i=1}^{k-1}\sum_{j=1}^{n-1}  { k \choose i}\pi(u_1\dotsc u_{j-1}) G_{i, k-i}^{\hdot}(\widetilde \varphi^{\otimes k}) ( \bar u_j \otimes \overbar{u_{j+1}\dotsc u_n} \otimes \bar v)\\
& +\sum_{j=1}^{n} \pi(u_1\dotsc u_{j-1}) G_k^\hdot (\widetilde \varphi^{\otimes k}) ( \bar u_j \otimes \overbar{u_{j+1}\dotsc u_nv}).
\end{aligned}
\end{align}

By Lemma \ref{lemma:technical} \ref{technical2} we have $G_{k-i}^\hdot (\widetilde \varphi^{\otimes k-i}) (\bar u_j \otimes \overbar{u_{j+1}\dotsc u_n} )= 0$  it follows that 
\[
G_{i, k-i}^{\hdot}(\widetilde \varphi^{\otimes k}) ( \bar u_j \otimes \overbar{u_{j+1}\dotsc u_n} \otimes \bar v) = - G_i^\hdot (\widetilde \varphi^{\otimes i}) ( \bar u_j \otimes G_{k-i} ^\hdot (\widetilde \varphi^{\otimes k-i})( \overbar{u_{j+1}\dotsc u_n} \otimes \bar v)).
\]
Substituting this into \eqref{align:729u}  we obtain
\begin{align*}
 \frac{1}{k!} G_k^\hdot (\widetilde \varphi^{\otimes k}) (\bar u \otimes \bar v) 
={} &\sum_{i=1}^{k-1}   \sum_{j=1}^{n-1}\pi (u_1 \cdots u_{j-1}) (\pi (u_j) \star^i (\pi(u_{j+1}\cdots u_n)\star^{k-i} \pi(v)))\\
& + \sum_{j=1}^n \pi(u_1\cdots u_{j-1}) (\pi(u_j) \star^k \pi(u_{j+1}\cdots u_n v)) \\
={} &\pi (u) \star^k \pi(v)
\end{align*}
where the first identity follows from the first case and  the induction hypothesis ($k-i<k$), and the second identity from Lemma \ref{lemma:star} \ref{guttproperty3}. 
\end{proof}

\begin{corollary}
\label{theorem-summary}
Let $A = \Bbbk Q / I$ and let $R$ be a reduction system satisfying \textup{($\diamond$)}. Then any formal deformation of $A$ over a complete local Noetherian $\Bbbk$-algebra $(B, \mathfrak m)$ is gauge equivalent to $(A \hatotimes B, \star^{\mathrm C}_{\phi + \widetilde \phi})$ for some Maurer--Cartan element $\widetilde \varphi$ of $\mathbf p (Q, R) \hatotimes \mathfrak m$.
\end{corollary}

\begin{proof}
This follows from Corollary \ref{corollary:controls} and Theorem \ref{theorem:interpretation-star}.
\end{proof}

To understand all star products up to gauge equivalence, it thus suffices to look at all combinatorial star products up to gauge equivalence. The group of all gauge equivalences is very large, but Theorem \ref{theorem:equivalenceformal} will show that a gauge equivalence between two combinatorial star products is given by an equivalence of formal reduction systems (cf.\ \S\ref{subsubsection:notation}) and as such is already determined by a map in $\Hom (\Bbbk Q_1, A) \hatotimes \mathfrak m$.

In fact, the combinatorial star product\index{star product!combinatorial}\index{combinatorial star product} is not only useful for giving explicit formulae for formal deformations, it can also be used for giving a combinatorial criterion for the Maurer--Cartan equation\index{Maurer--Cartan!equation}, as will be shown in \S\ref{subsection:computinghigherbrackets}.

\section{Combinatorial criterion for the Maurer--Cartan equation}
\label{subsection:computinghigherbrackets}
\index{Maurer--Cartan!equation}

The following lemma gives a combinatorial way to compute the higher brackets $\langle \widetilde \varphi,  \dotsc, \widetilde \varphi \rangle$ which appear in the Maurer--Cartan equation\index{Maurer--Cartan!equation} for the L$_\infty[1]$ algebra $\mathbf p (Q, R) \hatotimes \mathfrak m$.

\begin{lemma}
\label{lemma:higher-brackets}
Let $\widetilde \varphi \in \Hom (\Bbbk S, A) \hatotimes \mathfrak{m}$ and write $\star^k = \star^{\mathrm C, k}_{\phi + \widetilde \phi}$. Then for $k \geq 1$, the $k$th bracket $\langle \widetilde \varphi, \dotsc, \widetilde \varphi \rangle_k \in \Hom (\Bbbk S_3, A) \hatotimes \mathfrak{m}$ is given by
\begin{align*}
& \frac{1}{k!}\langle \widetilde \varphi, \dotsc, \widetilde \varphi \rangle_k(uvw) \\ 
={} &  \pi(u) \star^k \pi(\varphi_s) + \pi (u) \star^{k-1} \pi (\widetilde \varphi_s) - \pi (\varphi_{s'}) \star^k \pi(w)  - \pi (\widetilde \varphi_{s'}) \star^{k-1} \pi(w) 
\end{align*}
where $u v w \in S_3$ with $uv, vw \in S$ and we write $s = v w$ and $s' = u v$.
\end{lemma}

\begin{proof}
For $k = 1$ this identity follows from Lemma \ref{lemma:low-degree} \ref{lowdeg2} and Theorem \ref{theorem:interpretation-star}. Let us assume that $k \geq 2$. For any $u v w \in S_3$, by Theorem \ref{theorem:linfinitytransfer} we have 
\begin{align*}
\langle \widetilde \varphi, \dotsc, \widetilde \varphi \rangle_k (uvw) & = \sum_{i=1}^{k-1} \frac{1}{2} {k \choose i} F^3 (\langle G_i^{\hdot} (\widetilde \phi^{\otimes i}), G_{k-i}^{\hdot} (\widetilde \varphi^{\otimes k-i})\rangle) \\
& = - \sum_{i=1}^{k-1} {k \choose i} \check G_{i, k-i}^{\hdot} (\widetilde \phi^{\otimes k}) (F_3 (1 \otimes uvw \otimes 1))
\end{align*}
where the minus sign comes from \eqref{align:dglie1}. 
Thus, by Lemma \ref{lemma:low-degree} \ref{lowdeg3} and Theorem \ref{theorem:interpretation-star} we have 
\begin{align*}
\frac1{k!} \langle \widetilde \varphi, \dotsc, \widetilde \varphi \rangle_k (uvw) &= \pi(u) \star^k \pi (\varphi_s) -  \pi(\varphi_{s'}) \star^k \pi(w) \\
&\; - \frac1{k!} \sum_{i=1}^{k-1} {k \choose i} \check G_{i, k-i}^{\hdot} (\widetilde \phi^{\otimes k}) 
\left( \left\lceil
{\small
\begin{aligned}
{} &  \quad \varpi_1 \mathrm{split_1} (s' - \varphi_{s'}) \otimes \bar w \otimes 1 \\
{} &- \varpi_1 F_1 G_1 (1 \otimes \bar u \otimes 1) \otimes \bar \varphi_s \otimes 1 \\
{} & + \varpi_1 F_1 G_1 (1 \otimes \bar \varphi_{s'} \otimes 1) \otimes \bar w \otimes 1
\end{aligned}
}
\right\rfloor \right).
\end{align*}
(Here we use again $\lceil \blank \rfloor$ to indicate a line break within the parentheses.)
Thus it remains to verify that 
\begin{multline}
\label{align:F3}
\frac{1}{k!} \sum_{i=1}^{k-1} {k \choose i} \check G_{i, k-i}^{\hdot} (\widetilde \phi^{\otimes k}) 
(\varpi_1 \mathrm{split_1} (s') \otimes \bar w \otimes 1
-\varpi_1 \mathrm{split_1}(u) \otimes \bar \varphi_s \otimes 1) \\
 = \pi (\widetilde \varphi_{s'}) \star^{k-1} \pi (w) - \pi (u) \star^{k-1} \pi (\widetilde \varphi_s)
\end{multline}
where we use $F_1 G_1 (1 \otimes \bar u \otimes 1)  = \mathrm{split_1}(u)$. 
For this, we write $uvw = x_1 \cdots x_n$ for $x_i \in Q_1$ so that
\[
\begin{tikzpicture}[x=3em,y=1em]
\node[left] at (0,0) {$uvw = {}$};
\foreach \c/\n in {0/0,1/1,1.25/15,2.25/25,3.25/35,3.5/4,4.5/5,5.5/6,5.75/65,6.75/75} {
\node[inner sep=0] (\n) at (\c,0) {};
}
\path[->,line width=.4pt] (0)  edge node[above=-.4ex,font=\scriptsize] {$x_1$} (1);
\path[->,line width=.4pt] (15) edge node[above=-.4ex,font=\scriptsize] {$x_l$} (25);
\path[->,line width=.4pt] (25) edge node[above=-.4ex,font=\scriptsize] {$x_{l+1}$} (35);
\path[->,line width=.4pt] (4)  edge node[above=-.4ex,font=\scriptsize] {$x_m$} (5);
\path[->,line width=.4pt] (5)  edge node[above=-.4ex,font=\scriptsize] {$x_{m+1}$} (6);
\path[->,line width=.4pt] (65) edge node[above=-.4ex,font=\scriptsize] {$x_n$} (75);
\node at (1.125,1.3) {$\overbracket[.3pt]{\hspace{6.58em}}^{u}$};
\node at (4.5,1.3) {$\overbracket[.3pt]{\hspace{13.35em}}^{s}$};
\node at (2.25,-1.2) {$\underbracket[.3pt]{\hspace{13.35em}}_{s'}$};
\node at (5.625,-1.2) {$\underbracket[.3pt]{\hspace{6.58em}}_{w}$};
\node[font=\scriptsize] at (1.125,0) {$...$};
\node[font=\scriptsize] at (3.375,0) {$...$};
\node[font=\scriptsize] at (5.625,0) {$...$};
\end{tikzpicture}
\]
for some $0 < l < m < n$. Then by Lemma \ref{lemma:star} \ref{guttproperty1}, for any $ i \geq 1$ and $ 1 \leq  j < m$ we have 
\begin{equation}
\label{align:theorem7.173}
\pi (x_{j+1} \dotsc x_m) \star^i \pi (w) \\
= 
\begin{cases}
\begin{aligned} &\pi (x_{j+1} \dotsc x_l) \star^i \pi (\varphi_s) \\ &\quad{} + \pi (x_{j+1} \dotsc x_l) \star^{i-1} \pi (\widetilde \varphi_s) \end{aligned} & \text{if $ j\leq l-1$} \\
\pi(\widetilde \phi_s) & \text{if $j = l$ and $i =1$} \\
0 & \text{otherwise.}
\end{cases}
\end{equation}
Note that we have
\begin{equation}
\label{align:theorem7.171}
\begin{aligned}
&\frac{1}{k!} \sum_{i=1}^{k-1} {k \choose i} \check G_{i, k-i}^{\hdot} (\widetilde \phi^{\otimes k}) (\varpi_1 \mathrm{split_1}(s')\otimes \bar w \otimes 1) \\
= {} & \frac{1}{(k-1)!} G_{k-1}^\hdot (\widetilde \phi^{\otimes k-1}) (G_1^\hdot (\widetilde \varphi) (\bar x_1 \otimes \overbar{x_2 \cdots x_m}) \otimes \bar w) \\
& - \frac{1}{k!} \sum_{j=1}^{m-1} \sum_{i=1}^{k-1} {k \choose i} \pi(x_1 \cdots x_{j-1}) G_i^\hdot (\widetilde \phi^{\otimes i}) (\bar x_j \otimes G_{k-i}^\hdot (\widetilde \phi^{\otimes k-i}) (\overbar{x_{j+1} \cdots x_m} \otimes \bar w))\\
={} & \pi (\widetilde \varphi_{s'}) \star^{k-1} \pi (w) - \sum_{j=1}^{m-1} \sum_{i=1}^{k-1} \pi (x_1 \cdots x_{j-1}) (\pi (x_j) \star^i (\pi(x_{j+1}\cdots x_{m}) \star^{k-i} \pi (w)))
\end{aligned}
\end{equation}
where in the first identity we need to use Lemma \ref{lemma:technical} \ref{technical2} and the minus sign comes from the Gerstenhaber circle product. 
By a similar computation, we have 
\begin{equation}
\label{align:theorem7.172}
\begin{aligned}
&\frac{1}{k!} \sum_{i=1}^{k-1} {k \choose i} \check G_{i, k-i}^{\hdot}(\widetilde \phi^{\otimes k}) (\varpi_1 \mathrm{split_1}(u) \otimes \bar \varphi_s \otimes 1 ) \\
 ={} & - \sum_{j=1}^{l-1} \sum_{i=1}^{k-1} \pi (x_1 \cdots x_{j-1}) (\pi (x_j) \star^i  (\pi (x_{j+1}\cdots x_l) \star^{k-i} \pi (\varphi_s))).
\end{aligned}
\end{equation}
Combining (\ref{align:theorem7.171}), (\ref{align:theorem7.172}) and (\ref{align:theorem7.173}), we obtain 
\begin{align*}
&\frac{1}{k!} \sum_{i=1}^{k-1} {k \choose i} \check G_{i, k-i}^{\hdot} (\widetilde \phi^{\otimes k}) \biggl( \biggl\lceil
\begin{aligned}
& \ \ \varpi_1 \mathrm{split_1} (s') \otimes \bar w \otimes 1 \\
&{} - \varpi_1 \mathrm{split_1}(u) \otimes \bar \varphi_s \otimes 1
\end{aligned}
\biggr\rfloor \biggr) \\
={} & \pi (\widetilde \varphi_{s'}) \star^{k-1} \pi(w) - \sum_{j=1}^l \pi (x_1 \cdots x_{j-1}) (\pi (x_j) \star^{k-1} \pi (x_{j+1}\cdots x_l \widetilde \varphi_s)) \\
&\ - \sum_{j=1}^{l-1}\sum_{i=1}^{k-2} \pi(x_1 \cdots x_{j-1}) (\pi(x_j) \star^i (\pi (x_{j+1} \cdots x_l) \star^{k-i-1} \pi (\widetilde \varphi_s)) \\
={} & \pi (\widetilde \varphi_{s'}) \star^{k-1} \pi (w) - \pi (u) \star^{k-1} \pi (\widetilde \varphi_s)
\end{align*}
where the second identity follows from  Lemma \ref{lemma:star} \ref{guttproperty3}. This verifies (\ref{align:F3}).
\end{proof}

Recall from Theorem \ref{theorem:interpretation-star} that the homotopical star product\index{star product!homotopical}\index{homotopical star product} $\star^{\mathrm H}_{\phi + \widetilde \phi}$ admits a combinatorial description in terms of reductions, namely $\star^{\mathrm H}_{\phi + \widetilde \phi} = \star^{\mathrm C}_{\phi + \widetilde \phi}$. The following theorem shows that this combinatorial description can be used to check the Maurer--Cartan equation.\index{Maurer--Cartan!equation}

\begin{theorem}
\label{theorem:higher-brackets}
Let $\widetilde \varphi \in \Hom (\Bbbk S, A) \hatotimes \mathfrak m$ and write $\star = \star^{\mathrm C}_{\phi + \widetilde \phi}$. Then $\widetilde \varphi$ satisfies the Maurer--Cartan equation\index{Maurer--Cartan!equation} for $\mathbf p (Q, R) \hatotimes \mathfrak m$ if and only if for any $u v w \in S_3$ with $u v, v w \in S$, we have
\begin{equation}
\label{conditionmc}
\pi (u) \star (\pi(v) \star \pi (w)) = (\pi (u) \star \pi (v)) \star \pi (w).
\end{equation}
\end{theorem}

Namely, to check that $\widetilde \varphi$ satisfies the Maurer--Cartan equation, it suffices to check that $\star$ is associative on elements in $S_3$. In the terminology of Definition \ref{definition:overlap}, the condition (\ref{conditionmc}) is equivalent to the overlap\index{overlap ambiguity}\index{ambiguity!overlap} $u v w \in S_3$ being resolvable using rightmost reductions with respect to $\widehat R_{\phi + \widetilde \phi}$.

\begin{proof}[Proof of Theorem \ref{theorem:higher-brackets}]
The Maurer--Cartan equation for $\widetilde \varphi$ is given by
\[
\sum_{k \geq 1} \frac{1}{k!} \langle \widetilde \varphi, \dotsc, \widetilde \varphi \rangle_k = 0 \in \Hom (\Bbbk S_3, A) \hatotimes \mathfrak m
\]
(see (\ref{maurercartanlinfinity})). Writing $s = vw$ and $s' = u v$ we have
\begin{align*}
\sum_{k \geq 1} \frac{1}{k!} \langle \widetilde \varphi, \dotsc, \widetilde \varphi \rangle_k (u v w)
 &= \pi (u) \star \pi (\varphi_s + \widetilde \varphi_s) - \pi (\varphi_{s'} + \widetilde \varphi_{s'}) \star \pi (w)\\
 &= \pi (u) \star (\pi (v) \star \pi(w)) - (\pi (u) \star \pi(v)) \star \pi(w)
\end{align*}
where the first identity follows from Lemma \ref{lemma:higher-brackets} by summing over $k \geq 1$ and the second identity from Lemma \ref{lemma:star} \ref{guttproperty1}.
\end{proof}

\section{Equivalence of deformations}
\label{subsection:generatorsrelations}
\index{reduction system!deformation}\index{deformation!of reduction system}\index{deformation!of associative algebra}\index{deformation!of ideal of relations}

We will now prove the equivalences between deformations of the ideal of relations $I$, deformations of the reduction system $R$ and deformations of the algebra $A = \Bbbk Q / I$. Before we do so, let us also introduce the following notion, which naturally associates a star product to a formal deformation of $I$.

\begin{definition}
\label{definition:gutt}
\index{Gutt star product|textbf}\index{star product!Gutt|textbf}
Let $\widehat I$ be a formal deformation of the ideal $I \subset \Bbbk Q$ and $\Phi \colon A \llrr{t} \to \Bbbk Q \llrr{t} / \widehat I$ be the corresponding isomorphism of $\Bbbk \llrr{t}$-modules (see \S\ref{subsubsection:notation}). 

We associate to $\widehat I$ a (generalized) {\it Gutt star product} $\star^{\mathrm G}$ on $A\llrr{t}$ defined by
$$
a \star^{\mathrm G} b = \Phi^{-1} (\Phi (a) \cdot \Phi (b))
$$
for any $a, b \in A$, where $\cdot$ indicates the multiplication of $\Bbbk Q \llrr{t} / \widehat I$. (The original Gutt star product was introduced in Gutt \cite{gutt} as a quantization of the linear Lie--Kirillov Poisson structure\index{Poisson structure!linear}\index{Poisson structure} on the dual of a Lie algebra $\mathfrak g$, defining a star product on $\Sym (\mathfrak g) \llrr{t}$ by pulling back the associative structure on the universal enveloping algebra $\U_t (\mathfrak g) = \T (\mathfrak g) \llrr{t} / \langle x \otimes y - y \otimes x - [x, y] t\rangle_{x, y \in \mathfrak g}$ along the classical PBW isomorphism.)
\end{definition}

Since $\star^{\mathrm G}$ is associative by definition, $(A \llrr{t}, \star^{\mathrm G})$ is a formal deformation of $A = \Bbbk Q / I$. Note that, up to gauge equivalence, $\star^{\mathrm G}$ does not depend on the choice of $\Phi$.

\begin{proposition}
\label{proposition:formalideal}
Let $\widehat I$ and $\widehat I'$ be formal deformations of $I$. Then the following are equivalent:
\begin{enumerate}
\item $\widehat I$ and $\widehat I'$ are equivalent
\item $(A \llrr{t}, \star^{\mathrm G})$ and $(A \llrr{t}, \star'^{\mathrm G})$ are gauge equivalent.
\end{enumerate}
\end{proposition}

\begin{proof}
Let $\Phi \colon A \llrr{t} \to \Bbbk Q \llrr{t} / \widehat I$ and $\Phi' \colon A \llrr{t} \to \Bbbk Q \llrr{t} / \widehat I'$ be the corresponding isomorphisms of $\Bbbk \llrr{t}$-modules and let $a \star'^{\mathrm G} b = \Phi'^{-1} (\Phi'(a) \cdot \Phi'(b))$. By definition both $\Phi$ and $\Phi'$ become isomorphisms of $\Bbbk \llrr{t}$-algebras with respect to $\star^{\mathrm G}$ and $\star'^{\mathrm G}$, respectively.

If $\widehat I$ and $\widehat I'$ are equivalent, then by definition there is an isomorphism of $\Bbbk \llrr{t}$-modules $T \colon A \llrr{t} \to A \llrr{t}$ such that $\Phi' T \Phi^{-1}$ is an isomorphism of $\Bbbk \llrr{t}$-algebras. We infer that $T$ induces an $\Bbbk \llrr{t}$-algebra isomorphism between $(A \llrr{t}, \star^{\mathrm G})$ and $(A \llrr{t}, \star'^{\mathrm G})$. The converse follows similarly.
\end{proof}

Let $\widetilde \varphi \in \Hom (\Bbbk S, A) \hatotimes \mathfrak m$ and define the following two-sided ideal
\[
\widehat I_{\phi + \widetilde \phi} = \langle s - \varphi_s - \widetilde \varphi_s \rangle_{s \in S}^\compl \subset \Bbbk Q \llrr{t}.
\]
There is a natural $\Bbbk\llrr{t}$-linear map 
\begin{align}
\label{alignvarphi}
\Phi \colon A \llrr{t} \to \Bbbk Q \llrr{t} / \widehat I_{\phi + \widetilde \phi}
\end{align}
sending $\pi(u)$ to $[u]$ for any irreducible path $u$. If $\Phi$ is an isomorphism of $\Bbbk \llrr{t}$-modules, we denote the corresponding (generalized) Gutt star product\index{Gutt star product}\index{star product!Gutt} by $\star^{\mathrm G}_{\phi + \widetilde \phi}$.

We now have the following general result.

\begin{theorem}
\label{theorem:equivalenceformal}
The following conditions are equivalent
\begin{enumerate}
\item \label{cor1} $\widetilde \phi$ is a Maurer--Cartan element\index{Maurer--Cartan!elements} of $\mathbf p (Q, R) \hatotimes \mathfrak m$
\item \label{cor2} $\widehat R_{\phi + \widetilde \phi}$ is a formal deformation of $R$
\item \label{cor3} $(A \llrr{t}, \star^{\mathrm C}_{\phi + \widetilde \phi})$ is associative
\item \label{cor4} $\widehat I_{\phi + \widetilde \phi}$ is a formal deformation of $I$.
\end{enumerate}
In this case, $\star^{\mathrm G}_{\phi + \widetilde \phi} = \star^{\mathrm C}_{\phi + \widetilde \phi} = \star^{\mathrm H}_{\phi + \widetilde \phi}$.

Moreover, if $\widetilde \phi$ and $\widetilde \phi'$ are two Maurer--Cartan elements, the following are equivalent.
\begin{enumerate}
\item \label{core1} $\widetilde \phi$ and $\widetilde \phi'$ are homotopic \index{Maurer--Cartan!elements!equivalence}
\item \label{core2} $\widehat R_{\phi + \widetilde \phi}$ and $\widehat R_{\phi + \widetilde \phi'}$ are equivalent
\item \label{core3} $(A \llrr{t}, \star^{\mathrm C}_{\phi + \widetilde \phi})$ and $(A \llrr{t}, \star^{\mathrm C}_{\phi + \widetilde \phi'})$ are gauge equivalent
\item \label{core4} $\widehat I_{\phi + \widetilde \phi}$ and $\widehat I_{\phi + \widetilde \phi'}$ are equivalent.
\end{enumerate}
\end{theorem}

\begin{proof}
Let us prove the first assertion. \ref{cor1} $\Leftrightarrow$ \ref{cor3} follows from Theorems \ref{theorem:linfinitytransfer}, \ref{theorem:interpretation-star} and \ref{theorem:higher-brackets}. To show \ref{cor1} $\Rightarrow$ \ref{cor2} let $\widetilde \varphi$ be a Maurer--Cartan element of $\mathbf p (Q, R) \hatotimes \mathfrak m$. It follows from Theorem \ref{theorem:higher-brackets} that all overlap ambiguities of $S$ are resolvable via rightmost reductions with respect to $\widehat R_{\phi + \widetilde \phi}$. This shows that $\widehat R_{\phi + \widetilde \phi}$ is a formal deformation of $R$, see Remark \ref{remarkformaldiamondlemma}. To show \ref{cor2} $\Rightarrow$ \ref{cor1} let $\widehat R_{\phi + \widetilde \phi}$ be a formal deformation of $R$. Then by definition each path in $Q$ is reduction-unique with respect to $\widehat R_{\phi + \widetilde \phi}$. In particular, we have \eqref{conditionmc} so \ref{cor1} follows from Theorem \ref{theorem:higher-brackets}. Finally, \ref{cor2} $\Leftrightarrow$ \ref{cor4} follows from the formal version of the Diamond Lemma, see again Remark \ref{remarkformaldiamondlemma}.

That $\star^{\mathrm C}_{\phi + \widetilde \phi} = \star^{\mathrm H}_{\phi + \widetilde \phi}$ was proved in Theorem \ref{theorem:interpretation-star}. To show that $\star^{\mathrm G}_{\phi + \widetilde \phi} = \star^{\mathrm C}_{\phi + \widetilde \phi}$ it is enough to show that they coincide for $\pi (u), \pi (v)$, where $u, v$ are irreducible paths. Assuming \ref{cor4} we have
\begin{align*}
\pi(u) \star^{\mathrm G}_{\phi + \widetilde \phi} \pi(v) = \Phi^{-1} (\Phi (\pi(u)) \Phi (\pi(v))) = \Phi^{-1} ([u] [v]) = \pi(u) \star^{\mathrm C}_{\phi + \widetilde \phi} \pi(v)
\end{align*} 
where the second identity follows from the definition of $\Phi$ (\ref{alignvarphi}) and the third identity follows because $[u] [v] = \Phi (\pi (u) \star^{\mathrm C}_{\phi + \widetilde \phi} \pi (v))$ since reductions (using $s \mapsto \varphi_s + \widetilde \varphi_s$) induce the identity on elements in the quotient $\Bbbk Q \llrr{t} / \widehat I_{\phi + \widetilde \phi}$.

Let us prove the last assertion. \ref{core1} $\Leftrightarrow$ \ref{core3} follows again from Theorems \ref{theorem:linfinitytransfer} and \ref{theorem:interpretation-star}. \ref{core3} $\Leftrightarrow$ \ref{core4} follows from Proposition \ref{proposition:formalideal} together with the fact that $\star^{\mathrm G}_{\phi + \widetilde \phi} = \star^{\mathrm C}_{\phi + \widetilde \phi}$. Finally \ref{core2} $\Leftrightarrow$ \ref{core4} can be proved exactly as the equivalence of \ref{equivred1} and \ref{equivred3} of Lemma \ref{lemma:equivalencereduction}.
\end{proof}

The equivalence between \ref{cor1} and \ref{cor2} in Theorem \ref{theorem:equivalenceformal} can be phrased as follows.

\begin{corollary}
\label{corollary:natural}
The L$_\infty[1]$ algebra $\mathbf p (Q, R)$ naturally controls the formal deformation theory of the reduction system $R$\index{reduction system!deformation}\index{deformation!of reduction system}.\index{L$_\infty$!algebra}\index{algebra!L$_\infty$}
\end{corollary}

We can now give the proof of Theorem \ref{theorem:deformationreduction}.

\begin{proof}[Proof of Theorem \ref{theorem:deformationreduction}]
The equivalence of \ref{theorem:deformationreduction1} and \ref{theorem:deformationreduction2} follows from Theorem \ref{theorem:linfinitytransfer} and Corollary \ref{corollary:natural}, as both the deformation theory of $A$ and the deformation theory of $R$ are controlled by $\mathbf p (Q, R)$. To prove the equivalence of \ref{theorem:deformationreduction2} and \ref{theorem:deformationreduction3}, by Theorem \ref{theorem:equivalenceformal} it only remains to show that any formal deformation $\widehat I$ of $I_\phi$ is equivalent to $\widehat I_{\phi + \widetilde \phi}$ for some Maurer--Cartan element $\widetilde \phi$ of $\mathbf p (Q, R) \otimes \mathfrak m$. We have an algebra isomorphism $(A \llrr{t}, \star^{\mathrm G}) \simeq \Bbbk Q \llrr{t} / \widehat I$ (see Definition \ref{definition:gutt}). But $(A \llrr{t}, \star^{\mathrm G})$ is a formal deformation of $A$ and by Corollary \ref{theorem-summary} gauge equivalent to $\star^{\mathrm C}_{\phi + \widetilde \phi} = \star^{\mathrm G}_{\phi + \widetilde \phi}$ for some Maurer--Cartan element $\widetilde \phi$. Then $\widehat I$ is equivalent to $\widehat I_{\phi + \widetilde \phi}$ by Proposition \ref{proposition:formalideal}.
\end{proof}

\begin{subappendices}

\section{Computing Hochschild cohomology in degree $2$}
\label{subsection:firstorder}
\index{Hochschild!cohomology $\HH^2$}

In \S\ref{subsection:deformationtheoryassociative} we mentioned that for any algebra $A$ the space of first-order deformations up to equivalence coincides with the second Hochschild cohomology group $\HH^2 (A, A)$, where the associative algebra structure is given by an element $\mu \in \Hom_{\Bbbk} (A^{\otimes_{\Bbbk} 2}, A)$. It may be useful to point out that as a consequence of Theorem \ref{theorem:linfinitytransfer} this still holds true if we identify deformations of the associative multiplication on $A = \Bbbk Q / I$ with deformations of a reduction system (cf.\ Theorem \ref{theorem:deformationreduction}).

\begin{corollary}
\label{corollary:firstorder}
Let $A = \Bbbk Q / I$ and let $R$ be any reduction system satisfying \textup{($\diamond$)} for $I$. Then $\HH^2 (A, A)$ is isomorphic to the space of first-order deformations of $R$ modulo equivalence.
\end{corollary}

In practice $\HH^2 (A, A)$ can thus be computed using only the combinatorial star product\index{star product!combinatorial}\index{combinatorial star product}, as follows. Let $\widetilde \phi \in \Hom (\Bbbk S, A) \simeq \Hom (\Bbbk S, \Bbbk \Irr_S)$. Then $\widetilde \phi$ is a $2$-cocycle if and only if for any $u v w \in S_3$ with $u v, v w \in S$ we have
\[
(\pi (u) \star \pi(v)) \star \pi(w) = \pi (u) \star (\pi (v) \star \pi(w)) \; \mathrm{mod} \ t^2
\]
where $\star = \star^{\mathrm C}_{\phi + \widetilde \phi t}$.

Moreover, two $2$-cocycles $\widetilde \phi, \widetilde \phi'$ are cohomologous, i.e.\ $\widetilde \phi' - \widetilde \phi = \langle \psi \rangle$ for some $\psi \in \Hom (\Bbbk Q_1, A) \simeq \Hom (\Bbbk Q_1, \Bbbk \Irr_S)$ where $\langle \blank \rangle = -\partial^1$ \eqref{shifted} is the differential of $\mathbf p (Q, R)$, if and only if $T \colon \Bbbk \Irr_S [t] / (t^2) \to \Bbbk \Irr_S [t] / (t^2)$ determined by $T (x) = x + \psi (x) t$ for $x \in Q_1$ satisfies
\begin{align*}
T (\phi (s)) + \widetilde \phi' (s) t = T (s_1) \star \dotsb \star T (s_m) \; \mathrm{mod} \ t^2
\end{align*}
for any $s \in S$ with $s = s_1 \dotsb s_m$ for $s_i \in Q_1$. Here we set
\begin{align}\label{align:tudetermined}
T (u) = T (u_1) \star \dotsb \star T (u_n) \; \mathrm{mod} \ t^2
\end{align}
for any irreducible path $u = u_1 \dotsb u_n$ with $u_i \in Q_1$. 
Here, note that by \eqref{shifted} we have
\begin{equation}
\label{align:partial1}
\begin{aligned}
\partial^1\psi(s) &= \sum_{i=1}^m s_1 \dotsb s_{i-1} \psi(s_i) s_{i+1} \dotsb s_m \\ &\quad{}- \sum_j\sum_{i=1}^{k_j} \lambda_j x_{j,1} \dotsb x_{j,i-1} \psi(x_{j,i}) x_{j,i+1} \dotsb x_{j,k_j}
\end{aligned}
\end{equation}
where we write $\phi(s) = \sum_j \lambda_j x_{j,1} \dotsb x_{j,k_j}$ for $\lambda_j \in \Bbbk$ and $x_{j,1} \dotsb x_{j,k_j} \in \Irr_S$ with $x_{j,i} \in Q_1$.

See Example \ref{example:weyl} and Lemma \ref{Lemma:hh24dimensional} for examples.
\end{subappendices}

\chapter{Deformations in the non-formal setting}
\label{section:nonformal}

In Chapter \ref{section:deformations-of-path-algebras} we proved that for any algebra $A = \Bbbk Q / I$ and any reduction system $R$ satisfying the condition ($\diamond$), there is an equivalence of {\it formal} deformation problems between deformations of $A$, deformations of $R$ and deformations of $I$ (Theorems \ref{theorem:deformationreduction} and \ref{theorem:equivalenceformal}), where a formal deformation could be given by a Maurer--Cartan element of the L$_\infty[1]$ algebra $\mathbf p (Q, R) \otimes \mathfrak m$. Here ``formal'' refers to the base of deformation, which could be any {\it complete local} Noetherian $\Bbbk$-algebra $(B, \mathfrak m)$.

In this section we will describe under which conditions these results hold in the non-formal setting, i.e.\ for $\mathbf p (Q, R)$ itself.

\section{Finiteness conditions}

The Maurer--Cartan equation\index{Maurer--Cartan!equation|textbf} of an L$_\infty[1]$ algebra involves an infinite sum which need not be well-defined in general (cf.\ Remark \ref{remark:adic}). To study the role of the Maurer--Cartan equation of $\mathbf p (Q, R)$ in the non-formal case we may thus consider the following increasingly restrictive finiteness conditions on elements $\widetilde \phi \in \Hom (\Bbbk S, A)$
\begin{enumerate}[leftmargin=*,label=(FC\arabic*),start=0]
\item \label{fc0} the Maurer--Cartan equation\index{Maurer--Cartan!equation} for $\widetilde \phi$ is well-defined 
\item \label{fc1} the combinatorial star product\index{star product!combinatorial}\index{combinatorial star product} $\star^{\mathrm C}_{\phi + \widetilde \phi}$ is well-defined
\item \label{fc2} the reduction system $R_{\phi + \widetilde \phi}$ is reduction-finite.
\end{enumerate}

\begin{remark}
\ref{fc0} is the minimal requirement for considering questions about the Maurer--Cartan equation\index{Maurer--Cartan!equation}, or about Maurer--Cartan elements, in the non-formal setting. Similarly, \ref{fc1} is the minimal requirement to consider relations to actual deformations of algebras. However, it is not immediately clear if the sets of elements satisfying \ref{fc0}--\ref{fc2} have any nice geometric properties. For instance, Example \ref{example:nf} shows that the sets of elements satisfying \ref{fc1} and \ref{fc2} are not vector subspaces of $\Hom (\Bbbk S, A)$. In Chapter \ref{section:pbw} we will introduce certain degree conditions, which will allow us to identify vector subspaces of $\Hom (\Bbbk S, A)$ satisfying \ref{fc2}. This allows us to obtain affine varieties of ``actual'' deformations, similar to the variety of associative structures on a finite-dimensional vector space.
\end{remark}

\begin{lemma}
\textup{\ref{fc2}} $\Rightarrow$ \textup{\ref{fc1}} $\Rightarrow$ \textup{\ref{fc0}}.
\end{lemma}

\begin{proof}
\textup{\ref{fc2}} $\Rightarrow$ \textup{\ref{fc1}}. Let us denote
\[
a \star^{\mathrm C, k}_{\phi + \widetilde \phi} b = \bigl( a \star^{\mathrm C, k}_{\phi + \widetilde \phi t} b \bigr) \bigr\rvert_{t=1}
\]
where the right-hand side is defined in \S\ref{subsection:combinatorialstarproduct}. Since $R_{\phi + \widetilde \phi}$ is reduction-finite, we have that for any $a, b \in A$ the element $a \star_{\phi + \widetilde \phi}^{\mathrm C, k} b = 0$ for $k \gg 0$. It follows that $a \star_{\phi + \widetilde \phi}^{\mathrm C} b = \sum_{k \geq 0} a \star_{\phi + \widetilde \phi}^{\mathrm C, k} b$ is a finite sum and thus well-defined. 

\noindent \textup{\ref{fc1}} $\Rightarrow$ \textup{\ref{fc0}}. Note that Lemma \ref{lemma:higher-brackets} still holds, with the same proof, for the non-formal element $\widetilde \phi \in \Hom (\Bbbk S, A)$. By assumption $a \star^{\mathrm C, k}_{\phi + \widetilde \phi} b = 0$ for $k \gg 0$. Then Lemma \ref{lemma:higher-brackets} implies that for any $uvw \in S_3$ we have that $\frac{1}{k!}\langle \widetilde \varphi, \dotsc, \widetilde \varphi \rangle_k(uvw) = 0$ for $k \gg 0$. 
\end{proof}

Note that in Definition \ref{definition:equivalencereduction} we denoted by $\mathrm{MC} \subset \Hom (\Bbbk S, A)$ the subset of Maurer--Cartan elements satisfying \ref{fc2} and showed that the equivalence of the associated reduction systems\index{reduction system!equivalence} could be viewed as a groupoid $G \rightrightarrows \mathrm{MC}$. The following theorem is the non-formal analogue of Theorem \ref{theorem:equivalenceformal}.

\begin{theorem}
\label{theorem:nonformal}
Let $\widetilde \phi \in \Hom (\Bbbk S, A)$ satisfy \textup{\ref{fc2}}. Then the following are equivalent
\begin{enumerate}
\item \label{nf1} $\widetilde \phi$ satisfies the Maurer--Cartan equation\index{Maurer--Cartan!equation}
\item \label{nf2} all overlaps\index{overlap ambiguity}\index{ambiguity!overlap} of $R_{\phi + \widetilde \phi}$ resolve via rightmost reductions
\item \label{nf3} $(A, \star^{\mathrm C}_{\phi + \widetilde \phi})$ is associative
\item \label{nf4} the image of the set $\Irr_S$ under the projection $\Bbbk Q \to \Bbbk Q / \langle s - \phi_s - \widetilde \phi_s \rangle_{s \in S}$ forms a $\Bbbk$-basis.
\end{enumerate}
If $\widetilde \phi$ satisfies only \textup{\ref{fc1}}, then the above equivalences hold between \textup{\ref{nf1}}--\textup{\ref{nf3}}, and if $\widetilde \phi$ satisfies only \textup{\ref{fc0}}, then the equivalences hold between \textup{\ref{nf1}} and \textup{\ref{nf2}}.

Moreover, if $\widetilde \phi$ and $\widetilde \phi'$ are two Maurer--Cartan elements satisfying \textup{\ref{fc2}}, then the following are equivalent
\begin{enumerate}
\item \label{nfmc1} $\widetilde \phi$ and $\widetilde \phi'$ lie in the same orbit of the groupoid $G \rightrightarrows \mathrm{MC}$
\item \label{nfmc2} $R_{\phi + \widetilde \phi}$ and $R_{\phi + \widetilde \phi'}$ are equivalent
\item \label{nfmc3} $(A, \star^{\mathrm C}_{\phi + \widetilde \phi})$ and $(A, \star^{\mathrm C}_{\phi + \widetilde \phi'})$ are isomorphic.
\end{enumerate}
\end{theorem}

\begin{proof}
Let $\widetilde \phi$ satisfy \textup{\ref{fc2}}. The equivalence of \ref{nf2} and \ref{nf4} follows from the Diamond Lemma \ref{lemma:basis}. The equivalence of \ref{nf1} and \ref{nf2} follows from Theorem \ref{theorem:higher-brackets}, which also holds in the non-formal case, with the same proof. That \ref{nf3} implies \ref{nf2} is clear. Finally, Theorem \ref{theorem:interpretation-star} holds in the non-formal case, again with the same proof, showing that \ref{nf1} implies \ref{nf3}. The equivalences of \ref{nfmc1}--\ref{nfmc3} for two Maurer--Cartan elements $\widetilde \phi, \widetilde \phi'$ are the content of Lemma \ref{lemma:equivalencereduction}.

The same argument proves the claimed statements in the cases when $\widetilde \phi$ satisfies only \ref{fc1} or \ref{fc0}.
\end{proof}

\begin{remark}
Note that the Diamond Lemma \ref{lemma:basis} is stated under the condition \ref{fc2} of reduction-finiteness, and Theorem \ref{theorem:nonformal} gives a generalization of the Diamond Lemma\index{Diamond Lemma!generalization} under the weaker finiteness conditions \ref{fc0} and \ref{fc1}.
\end{remark}

We now give an example to show that the conditions \ref{fc0}--\ref{fc2} are not equivalent in general, which also illustrates why the other equivalences in Theorem \ref{theorem:nonformal} may fail in the cases \ref{fc0} and \ref{fc1}.

\begin{figure}
\centering
\begin{tikzpicture}[x=1em,y=1em,line width=.4pt]
\node at (7,0) {$\supset$};
\node at (19,0) {$\supset$};
\begin{scope}[xshift=24em]
\draw[fill=black] (0,0) circle(.2ex);
\draw[line cap=round] (-2,-3) -- (-3,-3) -- (-3,3) -- (3,3) -- (3,-3) -- (0,-3) -- (0,3) (2,3) -- (2,4) -- (-2,2) -- (-2,-4) -- (0,-3);
\node at (0,-5) {(FC2)};
\end{scope}
\begin{scope}[xshift=13em]
\draw[fill=black] (0,0) circle(.2ex);
\draw[line cap=round] (-2,-3) -- (-3,-3) -- (-3,-1) (-2,0) -- (-3,0) -- (-3,3) -- (3,3) -- (3,-3) -- (0,-3) -- (0,-1) (0,-3) -- (-2,-4) -- (-2,2) -- (2,4) -- (2,3) (-3,0) -- (-5,-1) -- (1,-1) -- (5,1) -- (3,1) (-2,-1) -- (0,0) -- (3,0) (0,0) -- (0,3);
\node at (0,-5) {(FC1)};
\end{scope}
\begin{scope}[xshift=0em]
\node[font=\scriptsize] at (2.4,4.7) {$\Hom (\Bbbk S, \Bbbk \Irr_S)$};
\draw[fill=black] (0,0) circle(.2ex);
\node[font=\scriptsize] at (0,-.6) {$0$};
\path[dash pattern=on 0pt off 1.3pt, line width=.6pt, line cap=round] (1,-4) edge (5,-2) (5,-2) edge (5,4) (5,4) edge (-1,4) (-1,4) edge (-5,2) (-5,2) edge (-5,-4) (-5,-4) edge (1,-4) (1,-4) edge (1,2) (1,2) edge (-5,2) edge (5,4) (-1,4) edge (-1,2.2) (-1,-2) edge (-1,1.8) (-1,-2) edge (-5,-4) (-1,-2) edge (.8,-2) (5,-2) edge (1.2,-2);
\node at (0,-5) {(FC0)};
\end{scope}
\end{tikzpicture}
\caption{Finiteness conditions for elements of $\Hom (\Bbbk S, \Bbbk \Irr_S)$ for Example \ref{example:nf}}
\label{figure:fc}
\end{figure}

\begin{example}
\label{example:nf}
Let $Q$ be the following quiver
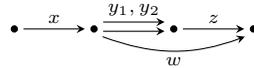
\[
\begin{tikzpicture}[baseline=-.25em,x=3em,y=1em]
\draw[line width=1pt, fill=black] (1,0) circle(0.2ex);
\draw[line width=1pt, fill=black] (2,0) circle(0.2ex);
\draw[line width=1pt, fill=black] (3,0) circle(0.2ex);
\draw[line width=1pt, fill=black] (4,0) circle(0.2ex);
\node[shape=circle, scale=0.7](N1) at (1,0) {};
\node[shape=circle, scale=0.7](N2) at (2,0) {};
\node[shape=circle, scale=0.7](N3) at (3,0) {};
\node[shape=circle, scale=0.7](N4) at (4,0) {};
\path[->, line width=.4pt]
(N1) edge node[above=-.3ex, font=\scriptsize] {$x$} (N2)
(N2) edge[transform canvas={yshift=.6ex}]  node[above=-.3ex, font=\scriptsize] {$y_1, y_2$}  (N3)
(N2) edge[transform canvas={yshift=-.2ex}] (N3)
(N3) edge node[above=-.3ex, font=\scriptsize] {$z$} (N4)
;
\path[->, line width=.4pt, out=-20, in=200]
(N2) edge[transform canvas={yshift=-.4ex}] node[below=-.3ex, font=\scriptsize] {$w$} (N4)
;
\end{tikzpicture}
\]
let $S = \{ x y_1, y_2 z \}$ and let $\phi$ be the zero element of $\Hom (\Bbbk S, \Bbbk \Irr_S)$ whose associated reduction system is $R = R_{\mathrm{mon}} = \{ (x y_1, 0), (y_2 z, 0) \}$. Note that $A_{\mathrm{mon}} = \Bbbk Q / \langle x y_1, y_2 z \rangle$ is $12$-dimensional.

Each $\widetilde \phi \in \Hom (\Bbbk S, \Bbbk \Irr_S)$ now gives a reduction system
\begin{flalign*}
&& R_{\phi + \widetilde \phi} &= \{ (x y_1, \lambda x y_2), (y_2 z, \mu y_1 z + \nu w) \} && \mathllap{\lambda, \mu, \nu \in \Bbbk}
\end{flalign*}
so that $\Hom (\Bbbk S, \Bbbk \Irr_S) \simeq \mathbb A^3$ and performing reductions with respect to $R_{\phi + \widetilde \phi}$ gives
\begin{multline}
\label{xyz}
x y_1 z \mapsto \lambda x y_2 z \mapsto \lambda \mu x y_1 z + \lambda \nu x w \\ \mapsto \lambda^2 \mu x y_2 z + \lambda \nu x w \mapsto \lambda^2 \mu^2 x y_1 z + \lambda \nu (1 + \lambda \mu) x w \mapsto \dotsb
\end{multline}
and we have
\begin{itemize}
\item $R_{\phi + \widetilde \phi}$ is reduction-finite if and only if $\lambda$ or $\mu$ vanish, so that the set of elements satisfying \ref{fc2} is cut out by $\lambda \mu = 0$
\item when $\nu = 0$, then $\star_{\phi + \widetilde \phi}^{\mathrm C}$ is always well-defined, even when $\lambda, \mu \neq 0$ and the set of elements satisfying \ref{fc1} is cut out by $\lambda \mu \nu = 0$
\item since $S_3 = \emptyset$ the Maurer--Cartan equation vacuously holds for any element in $\Hom (\Bbbk S, \Bbbk \Irr_S)$, so the set of elements satisfying \ref{fc0} is all of $\mathbb A^3$. 
\end{itemize}
See Fig.~\ref{figure:fc} for an illustration. In particular, this example shows that in general \ref{fc0} $\not\Rightarrow$ \ref{fc1} $\not\Rightarrow$ \ref{fc2}.

Moreover, when $\lambda \mu \neq 0$ and $\nu = 0$, then $(A, \star_{\phi + \widetilde \phi})$ is {\it not} isomorphic to the quotient $\Bbbk Q / I_{\phi + \widetilde \phi}$, where $I_{\phi + \widetilde \phi} = \langle x y_1 - \lambda x y_2, y_2 z - \mu y_1 z \rangle$. Indeed, $\Bbbk Q / I_{\phi + \widetilde \phi}$ is $13$-dimensional, showing that under condition \ref{fc1}, the condition \ref{nf4} in Theorem \ref{theorem:nonformal} is not equivalent to \ref{nf1}--\ref{nf3}.

Similarly, when $\lambda \mu \nu \neq 0$, then $\star$ is not defined on all of $A$, so it does not make sense to consider \ref{nf3} of Theorem \ref{theorem:nonformal} and, moreover, the image of irreducible paths does not form a basis of $\Bbbk Q / I_{\phi + \widetilde \phi}$, so that also \ref{nf4} fails. That is, under condition \ref{fc0}, \ref{nf3} and \ref{nf4} of Theorem \ref{theorem:nonformal} are not equivalent to \ref{nf1}--\ref{nf2}.
\end{example}

In Example \ref{example:nf} the sets of elements satisfying \ref{fc0}--\ref{fc2} were closed. This suggests the following question.

\begin{question}
\label{question}
For a finite quiver $Q$ and any set $S \subset Q_{\geq 2}$ such that $\dim \Hom (\Bbbk S, \Bbbk \Irr_S) = N < \infty$, are the subsets of elements in $\Hom (\Bbbk S, \Bbbk \Irr_S) \simeq \mathbb A^N$ satisfying \textup{\ref{fc0}}--\textup{\ref{fc2}} always closed? In particular, is reduction-finiteness \textup{\ref{fc2}} a closed condition?
\end{question}

\section{The Diamond Lemma from a deformation-theoretic viewpoint}
\label{subsection:proofdiamond}
\index{Diamond Lemma}

In the non-formal context, Theorem \ref{theorem:nonformal} showed that reduction-finiteness ensures that we have an equivalence between actual deformations of reduction systems, algebras and ideals. From this deformation-theoretic point of view, the Diamond Lemma\index{Diamond Lemma} \ref{lemma:basis} can be understood as follows.

Given an algebra $A = \Bbbk Q / I$, a reduction system satisfying the condition ($\diamond$) for $I$ is given by a set of pairs $R = \{ (s, \varphi_s) \}$. One may then consider the reduction system $R_{\mathrm{mon}} = \{ (s, 0) \mid (s, \varphi_s) \in R \}$ and the associated monomial algebra $A_{\mathrm{mon}} = \Bbbk Q / \langle s \rangle_{s \in S}$, where $S = \{ s \mid (s, \varphi_s) \in R \}$. As the irreducible paths only depend on the set $S$, the notion of irreducible paths is the same for $R$ and $R_{\mathrm{mon}}$. Indeed, $R$ is determined by $S$ and a function $\varphi \in \Hom (\Bbbk S, A_{\mathrm{mon}})$ (cf.\ Remark \ref{remark:reduction}). Clearly, the irreducible paths form a $\Bbbk$-basis of $A_{\mathrm{mon}}$ (cf.\ Example \ref{example:monomialreduction}) and $R_{\mathrm{mon}}$ clearly satisfies ($\diamond$) for the ideal $\langle s \rangle_{s \in S}$.

We now have the following deformation-theoretic interpretation:

\begin{none*}
\noindent\begin{minipage}[t]{.48\textwidth}
{\it Diamond Lemma}\index{Diamond Lemma}
\begin{enumerate}
\item all overlap ambiguities\index{overlap ambiguity}\index{ambiguity!overlap} of $R$ are resolvable
\item $R$ satisfies ($\diamond$) for $I$
\item the image of the set $\Irr_S$ of irreducible paths under the projection $\Bbbk Q \toarg{\pi} A$ is a $\Bbbk$-basis of $A$.
\end{enumerate}
\end{minipage}%
\hspace{.04\textwidth}%
\begin{minipage}[t]{.48\textwidth}
{\it Deformation-theoretic viewpoint}
\begin{enumerate}
\item \label{deformationtheoretic1} $\varphi$ satisfies the Maurer--Car\-tan equation of $\mathbf p (Q, R_{\mathrm{mon}})$
\item $R$ is a deformation of $R_{\mathrm{mon}}$
\item $A$ is a (flat) deformation of $A_{\mathrm{mon}}$.
\end{enumerate}
\end{minipage}
\end{none*}

\vspace{\parskip}
Here \ref{deformationtheoretic1} of the deformation-theoretic viewpoint says that if $R = \{ (s, \varphi_s) \}$ satisfies ($\diamond$), then $\varphi$ can be viewed as a Maurer--Cartan element of the L$_\infty[1]$ algebra $\mathbf p (Q, R_{\mathrm{mon}})$. 

\begin{remark}
After posting our preprint to the arXiv, the L$_\infty$ algebra\index{L$_\infty$!algebra}\index{algebra!L$_\infty$} structure for the monomial case also appeared in independent work by M.\hair J.~Redondo and F.~Rossi Bertone \cite{redondorossibertone} and a similar ``homotopical'' interpretation for the Diamond Lemma\index{Diamond Lemma} was given by V.~Dotsenko and P.~Tamaroff \cite{dotsenkotamaroff}. 
\end{remark}

\subsection{Maurer--Cartan twisting}
\index{Maurer--Cartan!twisting}

Since the map $\phi$ is a Maurer--Cartan element of $\mathbf p (Q, R_{\mathrm{mon}})$, one may also obtain the L$_\infty[1]$ algebra\index{L$_\infty$!algebra}\index{algebra!L$_\infty$} $\mathbf p (Q, R)$ from the L$_\infty[1]$ algebra $\mathbf p (Q, R_{\mathrm{mon}})$ by {\it Maurer--Cartan twisting}\index{Maurer--Cartan!twisting} (see for example \cite[Thm.~2.6]{chuanglazarev}). This process may be given by the following explicit formulae: denoting the $n$-ary bracket of $\mathbf p (Q, R_{\mathrm{mon}})$ by $\langle \blank , \dotsc {,} \blank \rangle_{\mathrm{mon}}$, the $n$-ary bracket $\langle \blank, \dotsc {,} \blank \rangle$ of $\mathbf p (Q, R)$ may be given by
\[
\langle \psi_1, \dotsc, \psi_n \rangle = \langle \psi_1, \dotsc, \psi_n \rangle_{\mathrm{mon}} + \sum_{m \geq 1} \frac{1}{m!} \langle \underbrace{\phi, \dotsc, \phi}_m, \psi_1, \dotsc, \psi_n \rangle_{\mathrm{mon}}.
\]

This is conceptually a neat way to understand the higher brackets of $\mathbf p (Q, R)$. However, the proofs of the main results in Chapter \ref{section:deformations-of-path-algebras} can not really be shortened by adopting this point of view. Moreover, in practice one usually deals with algebras which are not monomial, and the constructions in Chapters \ref{section:resolution}--\ref{section:deformations-of-path-algebras} are presented in the general, not necessarily monomial case.

\chapter[Varieties of reduction systems and PBW deformations]{Varieties of reduction systems and Poincaré--Birkhoff--Witt \\ deformations}
\label{section:pbw}

In this section we define certain subspaces $\Hom (\Bbbk S, A)_\prec$ of $\Hom (\Bbbk S, A)$ on which the Maurer--Cartan equation\index{Maurer--Cartan!equation} is well-defined and cuts out an algebraic variety of reduction systems, which can be viewed as an analogue of the varieties of finite-dimensional algebras (see \S\ref{subsubsection:variety}). Here $\Hom (\Bbbk S, A)_\prec$ is the subspace of elements in $\Hom (\Bbbk S, A)$ which satisfy a certain degree condition (\ref{degreecondition3}) to be defined in Definition \ref{definition:prec} below. The elements in $\Hom (\Bbbk S, A)_\prec$ satisfy the condition \ref{fc2} of Chapter \ref{section:nonformal}.

This geometric point of view can now be used to study actual deformations of infinite-dimensional algebras, which we also relate to PBW (Poincaré--Birkhoff--Witt) deformations of graded algebras (\S\ref{subsection:PBWdeformations}). In fact, the Maurer--Cartan equation\index{Maurer--Cartan!equation} of $\mathbf p (Q, R)$ can be viewed as a generalization of the classical criterion for PBW deformations of $N$-Koszul algebras (see Proposition \ref{proposition:pbw}).

The geometric formulation of deformations via varieties of reduction systems also provides a local geometric picture of the deformations of infinite-dimensional algebras, which can also be understood as a way to obtain algebraizations of certain formal deformations (see \S\ref{subsection:algebraizations}). 

\section{Algebraic varieties of reduction systems}
\label{subsection:varieties}
\index{reduction system!variety of}

\index{variety!of reduction systems}\index{admissible order}
Let $A = \Bbbk Q / I$ and let $R = \{ (s, \phi_s) \mid s \in S \}$ be a reduction system satisfying ($\diamond$) for $I$ obtained from a noncommutative Gröbner basis\index{Gröbner basis!noncommutative} with respect to some admissible order $\prec$ on $Q_\ldot$ so that $\langle S \rangle = \langle \tip_\prec (I) \rangle$ (see \S\ref{subsubsection:reduction}). Throughout this subsection, we fix this order $\prec$. 

\begin{definition}
\label{definition:prec}
For any $\widetilde \phi \in \Hom (\Bbbk S, A)$ we introduce the following degree condition:
\begin{flalign}
\label{degreecondition3}
\tag{$\prec$}
&& \widetilde \varphi_s = \sigma \widetilde \varphi (s) \in \Bbbk Q_{\prec s} && \mathllap{\text{for any $s \in S$}}
\end{flalign}
where $\Bbbk Q_{\prec s}$ is the $\Bbbk$-linear span of all paths which are ``smaller'' than $s$ with respect to the order $\prec$.  (Here $\sigma \colon A \to \Bbbk \Irr_S \subset \Bbbk Q$ is the $\Bbbk$-linear map defined in \eqref{sigma}.) Denote by $\Hom (\Bbbk S, A)_\prec \subset \Hom (\Bbbk S, A)$ the subspace of elements $\widetilde \varphi$ satisfying (\ref{degreecondition3}).
\end{definition}

Under the assumption that $\dim \Hom (\Bbbk S, A)_\prec < \infty$, Green--Hille--Schroll \cite{greenhilleschroll} introduced the following sets
\begin{align*}
\mathrm{Alg}_{\prec} &= \bigl\{ \Bbbk Q / J \bigm| \text{$J$ ideal of $\Bbbk Q$ and $\langle \tip_\prec (J) \rangle = \langle \tip_\prec (I) \rangle$} \bigr\} \\
V_{\prec} &= \bigl\{ \widetilde \varphi \in \Hom (\Bbbk S, A)_\prec \bigm| \text{$R_{\phi + \widetilde \phi}$ satisfies the condition ($\diamond$)} \bigr\}
\end{align*}
and proved the following result.

\begin{theorem}[{\cite[Thm.~4.3 \& Thm.~4.4]{greenhilleschroll}}]
\mbox{}
\begin{enumerate}
\item There is a one-to-one correspondence between the sets $V_{\prec}$ and $\mathrm{Alg}_\prec$. 
\item $V_{\prec}$ admits the structure of an affine algebraic variety.
\end{enumerate}
\end{theorem}

\begin{remark}
As for the case of $\mathrm{Alg}_W$ of associative multiplications on a (finite-dimensional) vector space $W$ (see \S\ref{subsubsection:variety}, the equations defining $V_\prec$ naturally give rise to an affine {\it scheme} structure on $V_\prec$ which is necessary when relating its Zariski tangent space to Hochschild cohomology. However, we continue to refer to $V_\prec$ as a variety, even if at times we need to take the scheme structure (notably, nilpotent elements) into account.
\end{remark}

By definition $V_{\prec}$ contains the origin $0$ of $\Hom (\Bbbk S, A)$, which corresponds to the reduction system determined by $\phi$ and thus to the algebra $A$ itself. A general point $\widetilde \phi \in V_\prec$ corresponds to
\begin{align*}
R_{\phi + \widetilde \phi} &= \{ (s, \phi_s + \widetilde \phi_s) \}_{s \in S} \\
A_{\phi + \widetilde \phi} &= \Bbbk Q / \langle s - \phi_s - \widetilde \phi_s \rangle_{s \in S}.
\end{align*}

The following result gives a complete description of the variety $V_\prec$ from the point of view of deformations of reduction systems. Recall from Definition \ref{definition:equivalencereduction} the groupoid $G \rightrightarrows \mathrm{MC}$ over the set $\MC \subset \Hom (\Bbbk S, A)$ of all elements {\it whose associated reduction system is reduction-unique}. 

\begin{theorem}
\label{theorem:variety}
\index{reduction system!equivalence}\index{variety!of reduction systems}\index{reduction system!variety of}
Let $\dim \Hom (\Bbbk S, A)_\prec = N < \infty$.
\begin{enumerate}
\item The Maurer--Cartan equation\index{Maurer--Cartan!equation} of $\mathbf p (Q, R)$ gives the explicit algebraic equations for the scheme $V_{\prec} \subset \mathbb A^N$.
\item \label{variety2} The groupoid $G \rightrightarrows \MC$ restricts to a groupoid $G_\prec \rightrightarrows V_\prec$ and two reduction systems in $V_\prec$ are equivalent if and only if they lie in the same orbit of $G_\prec \rightrightarrows V_\prec$.
\item The Zariski tangent space of $V_\prec$ at the point $\widetilde \phi \in V_\prec$ is isomorphic to the space of $2$-cocycles in $\Hom (\Bbbk S, A_{\phi + \widetilde \phi})_\prec$. In particular, we have a \emph{Kodaira--Spencer map} $\mathrm{KS} \colon \mathrm T_{\widetilde \phi} V_\prec \to \HH^2 (A_{\phi + \widetilde \phi}, A_{\phi + \widetilde \phi})$.
\item \label{variety4} The Zariski tangent space to the orbit of the groupoid $G_\prec \rightrightarrows V_\prec$ at $\widetilde \phi$ is contained in the space of $2$-co\-bound\-aries in $\Hom (\Bbbk S, A_{\phi + \widetilde \phi})_\prec$.
\end{enumerate}
\end{theorem}

\begin{remark}
This description is analogous to the variety $\mathrm{Alg}_W$ of all associative algebra structures on a finite-dimensional vector space $W$ (see \S\ref{subsubsection:variety}). For finite-dimensional algebras, $V_\prec$ gives a low-dimensional blueprint of the variety $\mathrm{Alg}_W$ (see \S\ref{subsubsection:finitedimensional}). Moreover, the variety $V_\prec$ also makes sense for infinite-dimensional algebras, giving a geometric description of certain algebraizable formal deformations. As Green--Hille--Schroll \cite{greenhilleschroll} already note, there does not appear to be a natural action of a group on the variety $V_\prec$. However, Theorem \ref{theorem:variety} shows that there is a natural {\it groupoid} action by restricting the groupoid $G \rightrightarrows \MC$ given in Definition \ref{definition:equivalencereduction}. Recall that the algebras corresponding to equivalent reduction systems are isomorphic, so the orbits of the variety $V_\prec$ under this groupoid action give a low-dimensional picture of the orbits of $\GL (W)$ acting on $\mathrm{Alg}_W$ (see \S\ref{subsubsection:finitedimensional} and \S\ref{subsection:brauertree} for examples of $V_\prec$).
\end{remark}

\begin{proof}[Proof of Theorem \ref{theorem:variety}]
Let us prove the first assertion. It follows from Theorem \ref{theorem:nonformal} that the Maurer--Cartan equation of $\mathbf p (Q, R)$ is well-defined on $\Hom (\Bbbk S, A)_\prec \subset \Hom (\Bbbk S, A)$ since the elements in $\Hom (\Bbbk S, A)_\prec$ satisfy the condition \ref{fc2}. Then $\widetilde \phi \in V_\prec$ if and only if
\begin{equation}
\label{w123}
(u \star v) \star w = u \star (v \star w)
\end{equation}
for any $u v w \in S_3$ with $uv, vw \in S$, where $\star = \star^{\mathrm C}_{\phi + \widetilde \phi}$.

The space $\Hom (\Bbbk S, A)_\prec$ has a natural $\Bbbk$-basis given by pairs of parallel paths $(s, p)$ with $s \in S$, $p \in \Irr_S$ and $p \prec s$, so that any element $\widetilde \phi \in \Hom (\Bbbk S, A)_\prec$ is determined by its coefficients $\lambda_{s,p} \in \Bbbk$ in this basis. For each $uvw \in S_3$, \eqref{w123} gives explicit algebraic equations on the coefficients $\lambda_{s,p}$.

The second assertion is clear from the definition. Let us prove the third assertion. By \S\ref{subsection:firstorder} an element  $ \widetilde \varphi' \in P^2(Q, R_{\phi + \widetilde \varphi})_\prec$ is a $2$-cocycle if and only if for any $uvw \in S_3$, we have 
\[
(\pi (u) \star \pi(v)) \star \pi(w) = \pi (u) \star (\pi (v) \star \pi(w)) \; \mathrm{mod} \ t^2,
\]
where $\star = \star^{\mathrm C}_{\phi + \widetilde \varphi + \widetilde \phi' t}$. The latter is equivalent to $\widetilde \phi + \widetilde \phi' t$ lying in the fiber of the map $V_{\prec}(\Bbbk [t]/(t^2)) \to V_{\prec}(\Bbbk) = V_{\prec}$ at $ \widetilde \phi$, which is by definition the Zariski tangent space of $V_{\prec}$ at $\widetilde \phi$.    Here, $V_{\prec}(B)$ denotes the set of $B$-points of the variety $V_{\prec}$ for any commutative local ring $B$. Similarly, the first part of the fourth assertion may be deduced from \S\ref{subsection:firstorder} and Definition \ref{definition:equivalencereduction}.
\end{proof}

Note that if the Maurer--Cartan equation\index{Maurer--Cartan!equation} holds for all $2$-cochains, for example if $P^3 = 0$, then $V_\prec$ is simply the affine space $\mathbb A^N$.

The following example shows that the inclusion in Theorem \ref{theorem:variety} \ref{variety4} may be strict.

\begin{example}
\label{example:weyl}
Let $A = \Bbbk \langle x, y \rangle / (y x - x y - 1)$ be the first Weyl algebra and let $R = \{ (y x, x y + 1) \}$ be the reduction system obtained from the degree--lexicographic order extending $x \prec y$ and let $S = \{ y x \}$ as usual.\index{admissible order!degree--lexicographic}\index{degree--lexicographic order}

Since $R$ has no overlaps, any element in $\Hom (\Bbbk S, A) \simeq A$ is a $2$-cocycle. On the other hand, one can check directly that every element in $\Hom (\Bbbk S, A)$ is also a $2$-coboundary, so that $\HH^2 (A, A) = 0$. Concretely, $A$ has $\Bbbk$-basis $\Irr_S = \{ x^k y^l \}_{k, l \geq 0}$ and defining $\psi \in \Hom (\Bbbk Q_1, A)$ by $\psi (x) = x^k y^l$ and $\psi (y) = 0$, by \eqref{align:partial1} we have that for any $k, l \geq 0$
\begin{align*}
\partial^1 \psi (y x) &= y \psi (x) + \psi (y) x - \psi (x) y - x \psi (y) \\
&= y x^k y^l - x^k y^{l+1} \\
&= k x^{k-1} y^l.
\end{align*}

Now $V_\prec = \Hom (\Bbbk S, A)_\prec \simeq \mathbb A^5$ with coordinates $(\kappa, \lambda, \mu, \nu, \xi)$ corresponding to elements $\widetilde \phi \colon yx \mapsto \kappa + \lambda x + \mu y + \nu x^2 + \xi xy$. Note that the points in the closed subvariety $W = \{ (0, 0, 0, 0, \xi) \} \simeq \mathbb A^1$ correspond to the quantum Weyl algebras $A_q = \Bbbk \langle x, y \rangle / (y x - q x y - 1)$, where $q = 1 + \xi$, which are mutually non-isomorphic. Since the orbit of $0 \in V_\prec$ under the groupoid action by $G_\prec$ contains only isomorphic algebras, $W \cap (G_\prec \cdot 0) = \{ 0 \}$ so that $\mathrm T_0 W \cap \mathrm T_0 (G_\prec \cdot 0) = 0$.

Thus in this example the space of $2$-coboundaries contained in $\Hom (\Bbbk S, A)_\prec$ {\it strictly} contains the tangent space to the orbit under the groupoid action.
\end{example}

\begin{remark}
\label{remark:gerstenhaber}
Gerstenhaber and Giaquinto \cite{gerstenhabergiaquinto} observed that the first Weyl algebra $A$ from Example \ref{example:weyl} is infinitesimally ``rigid'' since $\HH^2 (A, A) = 0$, but it nevertheless admits a nontrivial family of deformations, namely the quantum Weyl algebras $A_q$ for $q \in \Bbbk$. (This type of phenomenon can only occur for infinite-dimensional algebras.) 
 
Gerstenhaber and Giaquinto showed that the deformation theory of {\it diagrams of algebras} could be used to obtain the family $A_q$ from deformations of the diagram $\Bbbk [x] \to A \leftarrow \Bbbk [y]$. It seems to be an interesting feature of the approach via deformations of reduction systems that the quantum Weyl algebras $A_q$ can also be obtained as actual deformations of the Weyl algebra $A$, as in Example \ref{example:weyl}.
\end{remark}

\section{Poincaré--Birkhoff--Witt deformations}\label{subsection:PBWdeformations}
\index{PBW deformation|textbf}

Recall that the classical context of PBW deformations is the Poincaré--Birkhoff--Witt isomorphism (of $\Bbbk$-vector spaces) between the universal enveloping algebra
\[
\U (\mathfrak g) = \Bbbk \langle x_1, \dotsc, x_d \rangle / (x_i \otimes x_j - x_j \otimes x_i - [x_i, x_j])_{1 \leq i < j \leq d}
\]
of a finite-dimensional Lie algebra (graded by tensor degree) and its associated graded algebra
\[
\gr \U (\mathfrak g) \simeq \Sym (\mathfrak g) = \Bbbk \langle x_1, \dotsc, x_d \rangle / (x_i \otimes x_j - x_j \otimes x_i)_{1 \leq i < j \leq d} \simeq \Bbbk [x_1, \dotsc, x_d].
\]
Here $\U (\mathfrak g)$ is viewed as a PBW deformation of $\gr \U (\mathfrak g) \simeq \Sym (\mathfrak g)$, where the (homogeneous quadratic) commutativity relations $x_i \otimes x_j - x_j \otimes x_i$ are deformed to the relations $x_i \otimes x_j - x_i \otimes x_j - [x_i, x_j]$ which are quadratic but nonhomogeneous in the tensor degree grading. This notion was generalized to Koszul algebras in \cite{bravermangaitsgory,polishchukpositselski}.

In this subsection, we consider (a slightly more general notion of) PBW deformations of graded algebras in the following sense. Viewing $\Bbbk Q$ as graded by path length (i.e.\ by tensor degree when viewing $\Bbbk Q$ as the tensor algebra of $\Bbbk Q_1$ over $\Bbbk Q_0$ as in Chapter \ref{section:quivers}), we can view any quotient algebra $A = \Bbbk Q / I$ as filtered by path length and consider the associated graded algebra $\gr A \simeq \Bbbk Q / I'$. Note that the relations $I'$ of $\gr A$ are homogeneous, but they may have varying degrees. A {\it PBW deformation} (or {\it filtered deformation}) of $\gr A$ is any quotient $\widetilde A = \Bbbk Q / \widetilde I$ such that $\gr \widetilde A \simeq \gr A$.

\subsection{Varieties of PBW deformations}

Similar to the variety of actual deformations $V_\prec$ associated to the degree condition \eqref{degreecondition3} in \S\ref{subsection:varieties}, the set of PBW deformations of a graded algebra $A$ can be viewed as a variety $V_<$ for the following notion of degree condition.

\index{variety!of PBW deformations}
\begin{definition}
\label{definition:degree}
Let $R = \{ (s, \varphi_s) \mid s \in S \}$ be a finite reduction system satisfying the condition ($\diamond$) for an ideal $I$ of $\Bbbk Q$ and let $A = \Bbbk Q / I$ as usual.

For any $\widetilde \phi \in \Hom (\Bbbk S, A)$ we introduce the following two degree conditions:
\begin{flalign}
\label{degreecondition1}
\tag{$<$}
&& \widetilde \phi_s = \sigma \widetilde \phi (s) \in \Bbbk Q_{< \lvert s \rvert} && \mathllap{\text{for any $s \in S$}} \\
\label{degreecondition2}
\tag{$\leq$}
&& \widetilde \phi_s = \sigma \widetilde \phi (s) \in \Bbbk Q_{\leq \lvert s \rvert} && \mathllap{\text{for any $s \in S$}}
\end{flalign}
which are to be understood as conditions on $\widetilde \phi$. 

Let us also denote $\Hom (\Bbbk S, A)_< \subset \Hom (\Bbbk S, A)_\leq \subset \Hom (\Bbbk S, A)$ the subspaces of elements satisfying (\ref{degreecondition1}) and (\ref{degreecondition2}), respectively.
\end{definition}

Assume that $\varphi \in \Hom (\Bbbk S, A)_\leq$. Then $A$ is filtered by path length and denote by $\gr A$ the associated graded algebra.

\begin{proposition}
\label{proposition:pbwvariety}
The subspace $\Hom (\Bbbk S, A)_< \subset \Hom (\Bbbk S, A)$ is a finite-dimensional affine space and the Maurer--Cartan equation\index{Maurer--Cartan!equation} for $\mathbf p (Q, R)$ cuts out an affine variety $V_<$ whose points correspond to PBW deformations of $\gr A$.

Moreover, there is a natural groupoid action on $V_<$ such that two PBW deformations\index{PBW deformation}\index{variety!of PBW deformations} are isomorphic if and only if they lie in the same orbit.
\end{proposition}

\begin{proof}
The elements in $\Hom (\Bbbk S, A)_<$ satisfy \ref{fc2} in Chapter \ref{section:nonformal}. 
Then it follows from Theorem \ref{theorem:nonformal} that the Maurer--Cartan elements of $\mathbf p (Q, R)$ satisfying the degree condition \textup{(\ref{degreecondition1})} correspond to the PBW deformations. Since $S$ and $S_3$ are finite sets, the space $\Hom (\Bbbk S, A)_<$ is finite-dimensional and the Maurer--Cartan equation gives finitely many algebraic equations in $\Hom (\Bbbk S, A)_<$ via the identity (\ref{conditionmc}).

The second statement follows from Theorem \ref{theorem:variety} \ref{variety2} by restricting the groupoid $G \rightrightarrows \mathrm{MC}$ to $V_< \subset \mathrm{MC}$.
\end{proof}

The following theorem due to Berger--Ginzburg \cite{bergerginzburg} and Fløystad--Vatne \cite{floystadvatne} gives a necessary and sufficient criterion for PBW deformations\index{PBW deformation} of $N$-Koszul algebras, generalizing the corresponding result of Braverman--Gaitsgory \cite{bravermangaitsgory} and Polishchuk--Positselski \cite{polishchukpositselski} for Koszul algebras\index{Koszul!algebra} (i.e.\ the case $N = 2$).

\begin{theorem}[\cite{bergerginzburg,bravermangaitsgory,floystadvatne,polishchukpositselski}]
\label{theorem:pbwcriteria}
Let $A = \T (V) / I$ be an $N$-Koszul algebra\index{Koszul!algebra} where $I = \langle X \rangle$ is a homogeneous ideal generated by a set of relations $X \subset V^{\otimes_\Bbbk N}$. Let $\psi \colon \Bbbk X \to \T^{< N} (V)$ be a linear map with $\psi (x) = \psi_1 (x) + \dotsb + \psi_{N} (x)$ for each $x \in X$, where $\psi_i (x) \in V^{\otimes_\Bbbk N- i}$ and let $I' = \langle x - \psi (x) \rangle_{x \in X}$. Then $A' = \T (V) / I'$ is a PBW deformation of $A$ if and only if
\[
I' \cap \T^{\leq N} (V) = \Bbbk \{ x - \psi (x) \}_{x \in X}.
\]
Moreover, this condition is equivalent to the following two conditions:
\begin{enumerate}
\item \label{pbwcond1} $\im \big( (\Bbbk X \otimes_{\Bbbk} V) \cap (V \otimes_{\Bbbk} \Bbbk X) \mathrel{\tikz[baseline]\path[->,line width=.4pt] (0ex,0.65ex) edge node[above=-.4ex, overlay, font=\scriptsize] {$[\id_V, \psi_1]$} (8ex,.65ex);} V^{\otimes_{\Bbbk} N} \big) \subset \Bbbk X$
\item \label{pbwcond2} $\psi_i \circ [\id_V, \psi_1] = - [\id_V, \psi_{i+1}]$ for $1 \leq i \leq N$ (setting $\psi_{N+1} = 0$)
\end{enumerate}
where $[\id_V, \psi_i] = \id_V \otimes \psi_i - \psi_i \otimes \id_V$ as maps $(\Bbbk X \otimes_\Bbbk V) \cap (V \otimes_\Bbbk \Bbbk X) \to V^{\otimes_\Bbbk N - i + 1}$.
\end{theorem}

The criterion of Theorem \ref{theorem:pbwcriteria} is sometimes called the {\it Braverman--Gaitsgory criterion}\index{PBW deformation}. (See also Cassidy--Shelton \cite{cassidyshelton} for a criterion for PBW deformations of graded algebras using central extensions.)

The following result relates the above criterion to the Maurer--Cartan equation\index{Maurer--Cartan!equation} of $\mathbf p(Q, R)$.

\begin{proposition}
\label{proposition:pbw}
Let $A =  \Bbbk Q / I = \T (V) / I$ be a finitely generated algebra graded by path length, where $Q$ has a single vertex and $V = \Bbbk Q_1$. Let $R = \{ (s, \varphi_s) \mid s \in S \}$ be a reduction system satisfying \textup{($\diamond$)} for $I$ such that $S \subset V^{\otimes_\Bbbk N}$ and $\varphi_s \in V^{\otimes_\Bbbk N}$ for all $s \in S$ and assume that $S_3 \subset V^{\otimes_\Bbbk N+1}$. 
 
Then $A$ is an $N$-Koszul algebra\index{Koszul!algebra} and the degree components of the Maurer--Cartan equation\index{Maurer--Cartan!equation} of $\mathbf p(Q, R)$ can naturally be identified with the conditions \textup{\ref{pbwcond1}} and \textup{\ref{pbwcond2}} in Theorem \ref{theorem:pbwcriteria}.
\end{proposition}

\begin{proof}
Let us first prove that $A$ is $N$-Koszul. Let $A_{\mathrm{mon}} = \T (V) / \langle S \rangle$ denote the associated monomial algebra. Since $ S_3 \subset V^{\otimes N+1}$, it follows from \cite[Thm.~10.2]{greenmarcosmartinezvillazhang} that $A_{\mathrm{mon}}$ is $N$-Koszul. By \cite[Prop.~8.2]{chouhysolotar} we obtain that $A$ is also $N$-Koszul. 

Throughout we use the natural isomorphism $\Bbbk Q \simeq \T (V)$, where a path $x_1 \dotsb x_n \in Q_\ldot$ corresponds to $x_1 \otimes \dotsb \otimes x_n \in \T (V)$ for $x_i \in Q_1$.

We first prove the second statement in case $A = A_{\mathrm{mon}}$.  Let $\widetilde \varphi \in \Hom (\Bbbk S, A_{\mathrm{mon}})_{<}$ be a Maurer--Cartan element of $\mathbf p (Q, R_{\mathrm{mon}})$. For each $s \in S$ we may write $\widetilde \phi (s) = \widetilde \phi_1 (s) + \dotsb + \widetilde \phi_N (s)$, where $\psi_i(s): = \sigma \widetilde \phi_i (s) \in \Bbbk Q_{N-i}$.  We need to show that $\psi_i$ satisfies the two conditions \ref{pbwcond1} and \ref{pbwcond2} in Theorem \ref{theorem:pbwcriteria}, where we take $X = S$. Note that  
\begin{equation}
\label{xs3}
(\Bbbk X \otimes_{\Bbbk} V) \cap (V \otimes_{\Bbbk} \Bbbk X) = \Bbbk S_3 \subset V^{\otimes N + 1}.
\end{equation}
Let $u v w \in S_3$ with $u v, v w \in S$. Since $S_3 \subset V^{\otimes N + 1}$ we have $u, w \in V = \Bbbk Q_1$.  Writing $\star = \star^{\mathrm C}_{\phi + \widetilde \phi}$, Theorem \ref{theorem:higher-brackets} shows that
\[
(\pi(u) \star \pi(v)) \star \pi(w) - \pi(u) \star (\pi(v) \star \pi(w)) = 0,
\] 
Restricting to the component $\pi(\Bbbk Q_N) \subset A$, we have the following identity
\begin{align*}
0={} & \bigl( (\pi(u) \star \pi(v)) \star \pi(w) - \pi(u) \star (\pi(v) \star \pi(w) \bigr)\bigr\rvert_{\pi(\Bbbk Q_N)}\\
 ={} & \widetilde \phi_1(u v) \pi(w) - \pi(u) \widetilde \phi_1 (vw).
\end{align*}
It follows that $[\id_V, \psi_1] (uvw) = - \psi_1(uv) w + u \psi_1 (vw)$ must be in $I_{\mathrm{mon}} \cap \Bbbk Q_N = \Bbbk S = \Bbbk X$. 
Similarly, for $0 < i < N$ we may obtain 
\begin{align*}
0 & =  \bigl( (\pi(u) \star \pi(v)) \star \pi(w) - \pi(u) \star (\pi(v) \star \pi(w) \bigr) \bigr\rvert_{\pi(\Bbbk Q_{N-i})} \\
& =\widetilde \phi_i (  \psi_1(uv) w - u \psi_1 (vw)) + \widetilde \phi_{i+1} (uv) \pi(w) - \pi(u) \widetilde \phi_{i+1}(vw) \in \Bbbk \Irr_S
\end{align*}
which yields $ \psi_i [\id_V, \psi_1] = -[\id_V, \psi_{i+1}]$.
Therefore, the Maurer--Cartan equation coincides with the two conditions \ref{pbwcond1} and \ref{pbwcond2} in Theorem \ref{theorem:pbwcriteria}.

Let us consider the general case $A = \Bbbk Q /I$, where $I = \langle X \rangle$ with $X = \{ s - \varphi_s \mid s \in S \} \subset V^{\otimes N}$. The proof is very similar to the monomial case, except that the equality in \eqref{xs3} need not hold. Instead, one should consider the natural $\Bbbk$-linear isomorphism
\begin{equation}
\label{wwbar}
\begin{aligned}
\Bbbk S_3 &\tosim (\Bbbk X \otimes_{\Bbbk} V) \cap (V \otimes_{\Bbbk} \Bbbk X) \\
uvw &\mapsto \overbar{uvw}
\end{aligned}
\end{equation}
with $uv, vw \in S$ and the image $\overbar{uvw}$ of $uvw$ is given as follows. Since $S_3 \subset V^{\otimes N + 1}$ we have $u,w \in V = \Bbbk Q_1$. For each $i \geq 1$ we introduce the maps $\Phi_i, \Phi_i' \colon \Bbbk S_3 \to V^{\otimes N + 1}$ with
\begin{align*}
\Phi_i (uvw) &=\underbrace{\dotsb (\phi \otimes \id_V) (\id_V \otimes \phi) (\phi \otimes \id_V)}_i (uvw) \\
\Phi_i'(uvw) &=\underbrace{\dotsb (\id_V \otimes \phi) (\phi \otimes \id_V) (\id_V \otimes \phi)}_i (uvw)
\end{align*}
which perform exactly $i$ reductions on $uvw \in S_3$, starting with the leftmost and rightmost reduction, respectively. Since $R$ satisfies \textup{($\diamond$)} it follows that there exists a minimal number $i_0 > 0$ (resp.\ $i'_0 > 0$) such that $\Phi_{i_0} (uvw)$ (resp.\ $\Phi_{i'_0}' (uvw)$) is irreducible and $\Phi_{i_0} (uvw) = \Phi_{i'_0}' (uvw)$ in $V^{\otimes N+1}$.
The image $\overbar{uvw}$ of $uvw$ is given by 
\begin{equation*}
\overbar{uvw} := 
\begin{cases}
\displaystyle uvw + \sum_{i = 1}^{i_0-1} (-1)^i \Phi_i (uvw) + \sum_{j = 1}^{i'_0-1} (-1)^j \Phi_j' (uvw) & \text{if $i_0 + i'_0$ is odd}\\
\displaystyle uvw + \sum_{i = 1}^{i_0} (-1)^i \Phi_i (uvw) + \sum_{j = 1}^{i'_0-1} (-1)^j \Phi_j' (uvw) & \text{if $i_0 + i'_0$ is even.}
\end{cases}
\end{equation*}
Then similar to the monomial case, one checks that under the isomorphism \eqref{wwbar} the Maurer--Cartan equation gives precisely the two conditions \ref{pbwcond1} and \ref{pbwcond2} in Theorem \ref{theorem:pbwcriteria}.
\end{proof}

Thus the Maurer--Cartan equation\index{Maurer--Cartan!equation} of $\mathbf p (Q, R)$ can be viewed as a generalization of the Braverman--Gaitsgory criterion for PBW deformations\index{PBW deformation} to the setting of non-Koszul algebras and non-PBW deformations.

\begin{remark}
\label{remark:strong}
When an $N$-Koszul algebra\index{Koszul!algebra!strong} has a reduction system\index{reduction system!for Koszul algebras} satisfying the assumption in Proposition \ref{proposition:pbw}, the criterion for PBW deformations\index{PBW deformation} in Theorem \ref{theorem:pbwcriteria} becomes a combinatorial criterion. Such a reduction system exists for {\it strong} $N$-Koszul algebras (cf.\ \cite[Def.~2.2]{green2}). However, even if an $N$-Koszul algebra is not strong, it may still admit a reduction system satisfying the hypotheses of Proposition \ref{proposition:pbw}. (See for example \cite[Ex.~2.10.1]{chouhysolotar} which is $3$-Koszul, but not strong $3$-Koszul, admitting a reduction system satisfying the hypotheses of the proposition. This reduction system may be obtained by applying Heuristic \ref{heuristic}.) Note that for $N = 2$, the assumption $S_3 \subset V^{\otimes_\Bbbk N+1}$ is automatic.  
\end{remark}

We now relate the varieties $V_\prec$ defined in \S\ref{subsection:varieties} to the variety $V_<$ of PBW deformations of Proposition \ref{proposition:pbwvariety}.

As is apparent from the notation, the affine variety $V_{\prec}$ depends on the chosen admissible order. However, if we restrict our attention to a certain class of admissible orders, we may consider a variety of algebras that does not depend on the particular admissible order in this class.

We say that an admissible order $\prec$ is {\it compatible with the path length} \index{admissible order!compatible with path length} if for any two paths $p, q \in \Bbbk Q$ such that $\lvert p \rvert < \lvert q \rvert$, we have $p \prec q$.

\begin{theorem}
\label{theorem:unionvariety}
\index{admissible order!compatible with path length}
Let $A = \Bbbk Q / \langle S \rangle$. Let $T$ be the set of admissible orders compatible with the path length. 
\begin{enumerate}
\item \label{compatibleorder1} $\bigcup_{{\prec} \in T} V_{\prec}$ is an algebraic variety. 
\item \label{compatibleorder2} $\bigcap_{{\prec} \in T} V_{\prec} = V_<$ is the variety of PBW deformations\index{variety!of PBW deformations}\index{PBW deformation}.
\end{enumerate}
\end{theorem}

\begin{proof}
Let $N = \dim \Hom (\Bbbk S, A)_{\leq}$. To prove \ref{compatibleorder1} note that for each ${\prec} \in T$ we have $\Hom (\Bbbk S, A)_\prec \subset \Hom (\Bbbk S, A)_\leq$, so that
\[
W := \bigcup_{{\prec} \in T} \Hom (\Bbbk S, A)_\prec \subset \Hom (\Bbbk S, A)_\leq \simeq \mathbb A^N
\]
is an arrangement of subspaces which may be viewed as an affine variety in $\mathbb A^N$. By Theorem \ref{theorem:variety} the Maurer--Cartan equation is well-defined on each $\Hom (\Bbbk S, A)_\prec$ and thus also on $W$, cutting out an affine variety $\bigcup_{{\prec} \in T} V_{\prec} \subset \mathbb A^N$.

To prove \ref{compatibleorder2}, note that since ${\prec} \in T$ is compatible with the path length, we have $\Hom (\Bbbk S, A)_< \subset \Hom (\Bbbk S, A)_\prec$ for all ${\prec} \in T$. Thus the variety $V_<$ of PBW deformations given in Proposition \ref{proposition:pbwvariety} is contained in the intersection. Conversely, given any order on arrows $x_i \in Q_1$, the corresponding degree--lexicographic order \index{admissible order!degree--lexicographic}\index{degree--lexicographic order} is compatible with the path length, so for each $s \in S$ and each irreducible path $u \in Q_{\lvert s \rvert}$ parallel to $s$, there is always some (degree--lexicographic) admissible order such that $s \prec u$. Thus $\bigcap_{{\prec} \in T} \Hom (\Bbbk S, A)_\prec = \Hom (\Bbbk S, A)_<$ and consequently the Maurer--Cartan equation cuts out precisely the variety $V_<$ of PBW deformations.
\end{proof}

The degree conditions (\ref{degreecondition1}) and (\ref{degreecondition3}) are very useful in practice to obtain actual deformations.  In general,  $\Hom (\Bbbk S, A)$ may be an infinite-dimensional vector space and it is not immediately clear what geometric structure (if any) $V \subset \Hom (\Bbbk S, A)$ has. The degree conditions (\ref{degreecondition1}) and (\ref{degreecondition3}) ensure that the Maurer--Cartan equation\index{Maurer--Cartan!equation}  does make sense and then Proposition \ref{proposition:pbwvariety}, Theorems \ref{theorem:variety} and \ref{theorem:unionvariety} identify concrete affine varieties of actual deformations
\begin{equation}
\label{inclvar}
V_< \subset \bigcup_{{\prec} \in T} V_\prec \subset V.
\end{equation}
(See Figs.~\ref{figure:orbits} and \ref{figure:brauervariety} for concrete examples of $V_\prec$ and $V_<$.)

We now illustrate the theory developed up to this point in a simple example.

\begin{example}[(A very simple example)]
\label{example:simple}
The algebra
\[
A = \Bbbk \bigl(
\begin{tikzpicture}[baseline=-2.75pt,x=2.5em,y=1em]
\draw[line width=1pt, fill=black] (1,0) circle(0.2ex);
\draw[line width=1pt, fill=black] (0,0) circle(0.2ex);
\node[shape=circle, scale=0.7](L) at (0,0) {};
\node[shape=circle, scale=0.7](R) at (1,0) {};
\path[->, line width=.4pt, transform canvas={yshift=.4ex}, font=\scriptsize] (L) edge node[above=-.4ex] {$x$} (R);
\path[->, line width=.4pt, transform canvas={yshift=-.4ex}, font=\scriptsize] (R) edge node[below=-.4ex] {$y$} (L);
\end{tikzpicture}
\bigr) \big/ \langle xy, yx \rangle
\]
is a self-injective algebra of dimension $4$ with $\gldim A = \infty$. We first consider formal deformations of $A$ over $\Bbbk \llrr{t}$, given as Maurer--Cartan elements of the L$_\infty[1]$ algebra $\mathbf p (Q, R) \hatotimes \mathfrak m$, where $\mathfrak m = (t)$.

\paragraph*{The reduction system} The ideal $\langle xy, yx \rangle$ admits a reduction system
\begin{align*}
R = \{ (xy, 0), (yx, 0) \}
\end{align*}
so that
\begin{flalign*}
&& S_0 &= \eqmakebox[QS]{$Q_0$} = \{ e_1, e_2 \} & S_3 &= \{ xyx, yxy \} && \\
&& S_1 &= \eqmakebox[QS]{$Q_1$} = \{ x, y \}     & S_4 &= \{ xyxy, yxyx \} && \\
&& S_2 &= \eqmakebox[QS]{$S$}   = \{ xy, yx \}   &     & \hspace{.52em} \smash{\vdots}
\end{flalign*}
Note that the only irreducible paths parallel to $xy$ and $yx$ are the idempotents $e_1$ and $e_2$, so a general element $\widetilde \varphi$ of $\Hom (\Bbbk S, A) \hatotimes \mathfrak m$ is of the form
\begin{flalign*}
&& \widetilde \varphi (xy) &= \lambda e_1 \qquad \text{and} \qquad \widetilde \varphi(yx) = \mu e_2 && \llap{for $\lambda, \mu \in \mathfrak m$.}
\end{flalign*}
Note that any $\widetilde \phi$ satisfies \eqref{degreecondition1}, so that $\Hom (\Bbbk S, A)_< = \Hom (\Bbbk S, A)$.

\paragraph*{The Maurer--Cartan equation}\index{Maurer--Cartan!equation}
Writing $\star = \star^{\mathrm C}_{\phi + \widetilde \phi}$, note that $\widetilde \varphi$ is a Maurer--Cartan element of the L$_\infty[1]$ algebra $\mathbf p (Q, R) \hatotimes \mathfrak m$ if and only if $\lambda = \mu$, since for $xyx \in S_3$ we have
\begin{align*}
x \star (y \star x) = \mu x \star e_2 = \mu x \qquad \text{and} \qquad (x \star y) \star x = \lambda e_1 \star x = \lambda x
\end{align*}
and similarly for $yxy \in S_3$. By Theorem \ref{theorem:higher-brackets} we thus must have $\lambda = \mu$.

\paragraph*{Algebraization}\index{algebraization}
If $\mu = \lambda = \lambda_1 t + \dotsb + \lambda_n t^n \in \Bbbk [t]$, the corresponding formal deformation $(A \llrr{t}, \star)$ admits an algebraization $(A [t], \star) = \Bbbk Q [t] / \langle xy - \lambda e_1, yx - \lambda e_2 \rangle$. For instance, we may take $\lambda = t$ and evaluate $t = 1$ giving $A' := \Bbbk Q / \langle xy - e_1, yx - e_2 \rangle \simeq \mathrm M_2 (\Bbbk)$, i.e.\ the $4$-dimensional algebra of $2 \times 2$ matrices.

On the other hand, the element $\widetilde \varphi' \in \Hom (\Bbbk S, A) \hatotimes \mathfrak m$ defined by
\begin{align*}
\widetilde \varphi' (xy) = 0 \quad \text{and} \quad \widetilde \varphi' (yx) = e_2 t
\end{align*}
is {\it not} a Maurer--Cartan element and the algebra $(A \llrr{t}, \star') = \Bbbk Q \llrr{t} / \langle xy, yx - e_2 t \rangle$ where $\star' = \star^{\mathrm C}_{\phi + \widetilde \phi'}$ is {\it not} a flat formal deformation of $A$. Indeed, $(A \llrr{t}, \star')$ has $t$-torsion, as using the deformed relations we have for example $x t = xyx = 0$. Here $\widetilde \varphi'$ satisfies the degree condition (\ref{degreecondition1}), but $\Bbbk Q / \langle xy, yx - e_2 \rangle \simeq \Bbbk \{ e_1 \}$ is not a PBW deformation of $A$.

\paragraph*{Algebraic variety} Since $\Hom (\Bbbk S, A) = \Hom (\Bbbk S, A)_< \simeq \Bbbk^2$, the Maurer--Cartan equation\index{Maurer--Cartan!equation} of $\mathbf p (Q, R)$, given by $\lambda = \mu$ cuts out the variety $V_< \subset \mathbb A^2 = \{ (\lambda, \mu) \}$ of PBW deformations.
\end{example}

\begin{remark}
\index{admissible ideal}\index{algebra!finite-dimensional}\index{finite-dimensional algebra}
Any finite-dimensional algebra over an algebraically closed field $\Bbbk$ is Morita equivalent to the path algebra of a quiver modulo an admissible ideal $I$ (i.e.\ $\Bbbk Q_{\geq N} \subset I \subset \Bbbk Q_{\geq 2}$ for some $N \geq 2$). In Example \ref{example:simple} the deformation $A'$ is a PBW deformation\index{PBW deformation} of $A$ and has the same dimension, but the ideal $\langle xy - e_1, yx - e_2 \rangle$ is no longer admissible and $A' \simeq \mathrm M_2 (\Bbbk)$ is Morita equivalent to $\Bbbk$, so $\gldim A' = \gldim \Bbbk = 0$.

In other words, deformations of path algebras of quivers need not preserve the admissibility of the ideal and may be Morita equivalent to the path algebra of a quiver with less (or more) vertices and less arrows.
\end{remark}

\section{Algebraizations}
\label{subsection:algebraizations}
\index{algebraization}

In the context of algebraizations one should look at Maurer--Cartan elements in
\[
\mathbf p (Q, R) \otimes \mathfrak m
\]
where now $\mathfrak m$ is a maximal ideal in a commutative Noetherian $\Bbbk$-algebra $B$ which is not necessarily local or and not necessarily $\mathfrak m$-adically complete. For example $\mathfrak m = (t) \subset \Bbbk [t]$. Even though $\mathbf p (Q, R) \otimes \mathfrak m$ is still an L$_\infty[1]$ algebra, the Maurer--Cartan equation\index{Maurer--Cartan!equation} (being an infinite sum) does not necessarily make sense for {\it all} elements in $\Hom (\Bbbk S, A) \otimes \mathfrak m$ (cf.\ Remark \ref{remark:adic}) and the degree conditions \eqref{degreecondition1} and \eqref{degreecondition3} ensure that for the elements in $\Hom (\Bbbk S, A)_< \otimes \mathfrak m$ and $\Hom (\Bbbk S, A)_\prec \otimes \mathfrak m$ it does.

\begin{proposition}
\label{proposition:Groebner}
Let $\widetilde \varphi \in \Hom (\Bbbk S, A)_\prec \otimes \mathfrak m$ or $\Hom (\Bbbk S, A)_{<} \otimes \mathfrak m$.
\begin{enumerate}
\item \label{lemgr1} Then for any $a, b \in A$ we have that $a \star^{\mathrm C}_{\phi + \widetilde \varphi} b$ is a finite sum.
\item \label{lemgr2} If $\widetilde \phi$ satisfies the Maurer--Cartan equation\index{Maurer--Cartan!equation}  of $\mathbf p (Q, R) \hatotimes \mathfrak m$, then the formal deformation $(A \llrr{t}, \star^{\mathrm C}_{\phi + \widetilde \varphi})$ admits $(A [t], \star^{\mathrm C}_{\phi + \widetilde \phi})$ as an algebraization.\index{algebraization}
\end{enumerate}
\end{proposition}

\begin{proof}
Assertion \ref{lemgr1} follows from the observation that $\star^{\mathrm C}_{\phi + \widetilde \phi}$ is defined by performing (rightmost) reductions with respect to $R_{\phi + \widetilde \phi} = \{ (s, \phi_s + \widetilde \phi_s) \}$ (see Definition \ref{definition:star_2}) and such reductions strictly decrease the order of paths with respect to $\prec$, and \ref{lemgr2} directly follows from \ref{lemgr1} since $\star^{\mathrm C}_{\phi + \widetilde \phi}$ is well-defined on $A [t] \subset A \llrr{t}$.
\end{proof}

Proposition \ref{proposition:Groebner} is very useful in practice, since it allows one to pass from formal to actual deformations in contexts where one deals with deformations which do not satisfy the PBW condition (e.g.\ certain homogeneous deformations), but still satisfy the degree condition (\ref{degreecondition3}) for some suitable choice of \index{admissible order} admissible order. In geometric contexts this can for example be used to obtain algebraizations\index{algebraization} of quantizations of polynomial Poisson structures, deformed quantum preprojective algebras and noncommutative projective planes \cite{barmeierwang}.

\subsection{Geometric interpretation}
\label{subsubsection:geometricinterpretation}

The variety $V_\prec$ of \S\ref{subsection:varieties} can be viewed as providing algebraizations\index{algebraization} of formal deformations in the following sense. Consider the point $\phi$ viewed as the origin of the variety $V_\prec$, corresponding to the reduction system $R = R_\phi$, or equivalently to the algebra $A$. If $C \subset V_\prec$ is a curve passing through $\phi$ such that $C$ is smooth at $\phi$, then the completion $\widehat{\mathcal O}_{C,\phi}$ is a formal deformation of $R$ over $\Bbbk \llrr{t}$.

Let $(A \llrr{t}, \star)$ be the corresponding formal deformation of $A$. Then $C$ can be viewed as an algebraization\index{algebraization} of $(A \llrr{t}, \star)$.

Let $\alpha \in \HH^2 (A, A)$ be a cohomology class. If there is a curve $C$ in $V_\prec$ passing through the origin which is smooth at the origin, and if $\widetilde \phi_1 \in \mathrm T_\phi C$ is such that the image under the Kodaira--Spencer map $\mathrm{KS} (\widetilde \phi_1) = \alpha$ (cf.\ Theorem \ref{theorem:variety}), then there is a formal deformation of $R$ over $\Bbbk \llrr{t}$ whose first-order term is $\widetilde \phi_1$ which admits $C$ as an algebraization, providing an algebraization\index{algebraization} of the formal deformation of $A$ associated to $\alpha$.

\section{Finite-dimensional algebras}
\label{subsubsection:finitedimensional}
\index{finite-dimensional algebra}

Deformations of finite-dimensional algebras can in principle be determined already from the classical theory, i.e.\ as Maurer--Cartan elements in the Hochschild complex with its DG Lie algebra structure given by the Gerstenhaber bracket. The deformations can be understood by studying the variety $\mathrm{Alg}_W$ of all associative algebra structures on a finite-dimensional vector space $W$, and its stratification into $\GL (W)$-orbits (see \S\ref{subsubsection:variety}).

When $\dim W = d < \infty$, one has that $\mathrm{Alg}_W$ is an affine variety in $\mathbb A^{d^3}$ cut out by $d^4$ quadratic equations, and $\GL (W)$ is an algebraic group of dimension $d^2$. The dimensions quickly grow too large to allow a detailed analysis of the orbits.

From a computational perspective, the approach via reduction systems is often much more accessible --- even for higher-dimensional algebras (see \S\ref{subsection:brauertree} for deformations of Brauer tree algebras). Using the degree conditions (\ref{degreecondition1}) and (\ref{degreecondition3}), the varieties $V_<$ or $V_\prec$ can be used to give a low-dimensional blueprint of the variety $\mathrm{Alg}_W$. Let us illustrate this in the case of $4$-dimensional algebras.

Gabriel \cite{gabriel} showed that for $\dim W = 4$, the variety $\mathrm{Alg}_W \subset \mathbb A^{64}$ consists of five irreducible components with $18$ $\GL (W)$-orbits of dimensions between $6$ and $15$ and one $1$-dimensional family of $11$-dimensional orbits.

The $2 \times 2$ matrix algebra $\mathrm M_2 (\Bbbk)$ corresponds to an open $12$-dimensional orbit inside one of the five irreducible components of $\mathrm{Alg}_W$ and the $4$-dimensional self-injective algebra
\[
A = \Bbbk \bigl(
\begin{tikzpicture}[baseline=-2.75pt,x=2.5em,y=1em]
\draw[line width=1pt, fill=black] (1,0) circle(0.2ex);
\draw[line width=1pt, fill=black] (0,0) circle(0.2ex);
\node[shape=circle, scale=0.7](L) at (0,0) {};
\node[shape=circle, scale=0.7](R) at (1,0) {};
\path[->, line width=.4pt, transform canvas={yshift=.4ex}, font=\scriptsize] (L) edge node[above=-.4ex] {$x$} (R);
\path[->, line width=.4pt, transform canvas={yshift=-.4ex}, font=\scriptsize] (R) edge node[below=-.4ex] {$y$} (L);
\end{tikzpicture}
\bigr) \big/ \langle xy, yx \rangle
\]
of Example \ref{example:simple} is an $11$-dimensional orbit lying in its closure. Describing the deformations of $A$ using the approach via reduction systems, the space of $2$-cochains $\Hom (\Bbbk S, A)$ is only $2$-dimensional and the Maurer--Cartan equation \index{Maurer--Cartan!equation} cuts out the variety $V_< \subset \mathbb A^2$ of PBW deformations\index{PBW deformation}, with $V_< \simeq \mathbb A^1$ with $0 \in \mathbb A^1$ corresponding to $A$ and $\lambda \in \mathbb A^1 \setminus \{ 0 \}$ corresponding to $\mathrm M_2 (\Bbbk)$.

Now let $A$ be the $4$-dimensional algebra
\[
A = \Bbbk \bigl(
\begin{tikzpicture}[baseline=-2.75pt,x=2.5em,y=1em]
\draw[line width=1pt, fill=black] (0,0) circle(0.2ex);
\node[shape=circle, scale=0.7](L) at (0,0) {};
\path[->,line width=.4pt,font=\scriptsize, looseness=16, in=215, out=145]
(L.160) edge node[pos=.8,below=-.2ex] {$x$} (L.200)
;
\path[->,line width=.4pt,font=\scriptsize, looseness=16, in=35, out=325]
(L.340) edge node[pos=.15,below=-.2ex] {$y$} (L.20)
;
\end{tikzpicture}
\bigr) \big/ \langle yx, x^2, y^2 \rangle.
\]
Then $A$ lies in the intersection of another, $13$-dimensional, irreducible component and a $12$-dimensional component and the $\GL (W)$-orbit through $A$ has dimension $11$. We consider the obvious monomial reduction system $R = \{(yx, 0),  (x^2, 0), (y^2, 0) \}$ so that $S = \{ yx,  x^2, y^2 \}$. 

\begin{lemma}\label{Lemma:hh24dimensional}
$\dim_\Bbbk \HH^2(A, A) = 3$.
\end{lemma}
\begin{proof}
Any element $\widetilde \phi \in \Hom(\Bbbk S, A)$ is of the following general form 
\begin{alignat*}{7}
&& \widetilde \phi (yx)  &{}={}& \lambda_e &{}+{}& \lambda_x x &{}+{}& \lambda_y y &{}+{}& \lambda_{xy} xy& \\
&& \widetilde \phi (x^2) &{}={}&     \mu_e &{}+{}&     \mu_x x &{}+{}&     \mu_y y &{}+{}&     \mu_{xy} xy& \\
&& \widetilde \phi (y^2) &{}={}&     \nu_e &{}+{}&     \nu_x x &{}+{}&     \nu_y y &{}+{}&     \nu_{xy} xy&
\end{alignat*}
with $\lambda_e, \dotsc, \nu_{xy} \in \Bbbk$. By Corollary \ref{corollary:firstorder} $\widetilde \phi$ is a $2$-cocycle if and only if for each $uvw \in S_3 = \{ yx^2, y^2x, x^3, y^3\}$ with $uv, vw \in S$ we have 
\[
(\pi (u) \star \pi(v)) \star \pi(w) = \pi (u) \star (\pi (v) \star \pi(w)) \; \mathrm{mod} \ t^2
\]
where $\star = \star^{\mathrm C}_{\phi + \widetilde \phi t}$ and $\phi = 0$. This is equivalent to  
\[
\lambda_e = 0,\quad \mu_e = 0 = \mu_y, \quad \nu_e = 0 = \nu_x.
\]
By Corollary \ref{corollary:firstorder} again two $2$-cocycles $\widetilde \phi, \widetilde \phi'$ are cohomologous, i.e.\ $\widetilde \phi' - \widetilde \phi = \langle \psi \rangle$ for some $\psi \in \Hom (\Bbbk Q_1, A) \simeq \Hom (\Bbbk Q_1, \Bbbk \Irr_S)$, if and only if the isomorphism $T \colon \Bbbk \Irr_S[t]/(t^2) \to \Bbbk \Irr_S[t]/(t^2)$ determined by $T (x) = x + \psi (x) t$ for $x \in Q_1$ via \eqref{align:tudetermined},  satisfies (note that $\phi = 0$)
\begin{align}\label{align:examplecoboundary}
\widetilde \phi' (s) t = T (s_1) \star \dotsb \star T (s_m) \; \mathrm{mod} \ t^2
\end{align}
for any $s \in S$ with $s = s_1 \dotsb s_m$ for $s_i \in Q_1$. Write $\psi$ as the following general form 
\begin{alignat*}{7}
&& \psi (x) &{}={}& \alpha_e &{}+{}& \alpha_x x &{}+{}& \alpha_y y &{}+{}& \alpha_{xy} xy& \\
&& \psi (y) &{}={}&  \beta_e &{}+{}&  \beta_x x &{}+{}&  \beta_y y &{}+{}&  \beta_{xy} xy&
\end{alignat*}
with $\alpha_e, \dotsc, \beta_{xy} \in \Bbbk$. Then \eqref{align:examplecoboundary} is equivalent to    
\begin{flalign*}
&&
\begin{aligned}
\lambda_x' - \lambda_x &= \beta_e  \\
\lambda_y' - \lambda_y &= \alpha_e \\
\lambda_{xy}' - \lambda_{xy} &= 0
\end{aligned}
&&
\begin{aligned}
\mu_x' - \mu_x &= 2 \alpha_e \\
\mu_{xy}' - \mu_{xy} &= \alpha_{xy} \\
\end{aligned}
&&
\begin{aligned}
\nu_y' - \nu_y &= 2 \beta_e \\
\nu_{xy}' - \nu_{xy} &= \beta_y.
\end{aligned}
&&
\end{flalign*}
This implies that $\HH^2 (A, A)$ is $3$-dimensional and any cohomology class can be represented by a $2$-cocycle of the form
\begin{equation}
\label{spanning}
\begin{aligned}
y x &\mapsto \lambda \hair x y \\
x^2 &\mapsto \mu \hair x \\
y^2 &\mapsto \nu \hair y
\end{aligned}
\end{equation}
for some $\lambda, \mu, \nu \in \Bbbk$. 
\end{proof}

Letting $\prec$ be the degree--lexicographic order \index{admissible order!degree--lexicographic}\index{degree--lexicographic order} obtained from $x \prec y$, we have that $\dim \Hom (\Bbbk S, A)_\prec = 11$. Let us look at the subspace $Y \subset \Hom (\Bbbk S, A)_\prec$ spanned by the elements of the form \eqref{spanning}, which will give the full local picture of the deformations of $A$. Evaluating the Maurer--Cartan equation \index{Maurer--Cartan!equation} of $\mathbf p (Q, R)$ on the overlaps $S_3 = \{ y^3, x^3, y^2 x, y x^2 \}$ (see Theorem \ref{theorem:higher-brackets}), the variety $X = V_\prec \cap Y \subset Y \simeq \mathbb A^3$ is cut out by two equations
\begin{align*}
\lambda (\lambda - 1) \mu &= 0 \\
\lambda (\lambda - 1) \nu &= 0
\end{align*}
so that $X = \{ (\lambda, 0, 0) \} \cup \{ (0, \mu, \nu) \} \cup \{ (1, \mu, \nu) \}$, where $0 \in X$ corresponds to the algebra $A$.

The orbits of the groupoid action $G \rightrightarrows X$ capture precisely the isomorphism classes of the algebras contained in $X$. One can show by Definition \ref{definition:equivalencereduction} that  each point in $\{ (\lambda, 0, 0) \}$ is its own orbit; the planes $\{ (0, \mu, \nu) \}$ and $\{ (1, \mu, \nu) \}$ decompose into four and three orbits, respectively, as illustrated in Fig.~\ref{figure:orbits} (cf.\ the diagram \cite[p.~153]{gabriel}). Note that the Zariski tangent space $\mathrm T_0 X \simeq \Bbbk^3$ and the natural Kodaira--Spencer map $\mathrm T_0 X \to \HH^2 (A, A)$ (see Theorem \ref{theorem:variety}) is an isomorphism, so that $X$ captures all deformations of $A$ parametrized by $\HH^2 (A, A)$.

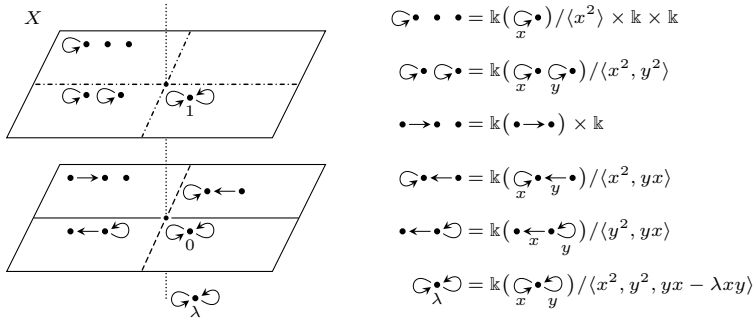
\begin{figure}
\centering
\begin{tikzpicture}[x=1em,y=1em,line width=.4pt]
\draw[line cap=round] (4,-2) -- (6,2) -- (-4,2) -- (-6,-2) -- cycle;
\draw[line cap=round] (4,3) -- (6,7) -- (-4,7) -- (-6,3) -- cycle;
\draw[fill=black] (0,0) circle(.17ex);
\draw[fill=black] (0,5) circle(.17ex);
\path[dash pattern=on 0pt off 1.3pt, line width=.6pt, line cap=round] (0,-2.08) edge (0,-3);
\path[dash pattern=on 0pt off 1.3pt, line width=.6pt, line cap=round] (0,2.92) edge (0,.1);
\path[dash pattern=on 0pt off 1.3pt, line width=.6pt, line cap=round] (0,5.1) edge (0,8);
\path[dash pattern=on 2pt off 1.5pt, line width=.5pt, line cap=round] (-.1,-.2) edge (-.9,-1.95);
\path[dash pattern=on 2pt off 1.5pt, line width=.5pt, line cap=round] (.1,.2) edge (.9,1.95);
\draw (-.2,0) -- (-4.99,0);
\draw (.2,0) -- (4.99,0);
\path[dash pattern=on 1.5pt off 1.5pt on .05pt off 1.5pt, line width=.55pt, line cap=round] (.1,5.2) edge (.9,6.95);
\path[dash pattern=on 1.5pt off 1.5pt on .05pt off 1.5pt, line width=.55pt, line cap=round] (-.1,4.8) edge (-.9,3.05);
\path[dash pattern=on 1.5pt off 1.5pt on .05pt off 1.5pt, line width=.55pt, line cap=round] (-.2,5) edge (-4.99,5);
\path[dash pattern=on 1.5pt off 1.5pt on .05pt off 1.5pt, line width=.55pt, line cap=round] (.2,5) edge (4.99,5);
\node[font=\scriptsize] at (-5,7.5) {$X$};
\node[font=\scriptsize,right] at (11,7.5) {$= \Bbbk \bigl( \hspace{1.6em} \bigr) / \langle x^2 \rangle \times \Bbbk \times \Bbbk$};
\node[font=\scriptsize,right] at (11,5.5) {$= \Bbbk \bigl( \hspace{3.2em} \bigr) / \langle x^2, y^2 \rangle$};
\node[font=\scriptsize,right] at (11,3.5) {$= \Bbbk \bigl( \hspace{2.2em} \bigr) \times \Bbbk$};
\node[font=\scriptsize,right] at (11,1.5) {$= \Bbbk \bigl( \hspace{3.2em} \bigr) / \langle x^2, yx \rangle$};
\node[font=\scriptsize,right] at (11,-.5) {$= \Bbbk \bigl( \hspace{3.2em} \bigr) / \langle y^2, yx \rangle$};
\node[font=\scriptsize,right] at (11,-2.5){$= \Bbbk \bigl( \hspace{2.65em} \bigr) / \langle x^2, y^2, y x - \lambda x y \rangle$};
\begin{scope}[shift={(10.2em,-2.5em)}]
\draw[line width=1pt, fill=black] (0,0) circle(0.15ex);
\node[shape=circle, scale=0.5] (L1) at (0,0) {};
\node[font=\tiny] at (0,-.5) {$\lambda$};
\path[->,line width=.35pt,font=\tiny, looseness=16, in=215, out=145]
(L1.160) edge (L1.200);
\path[->,line width=.35pt,font=\tiny, looseness=16, in=35, out=325]
(L1.340) edge (L1.20);
\end{scope}
\begin{scope}[shift={(14em,-2.5em)}]
\draw[line width=1pt, fill=black] (0,0) circle(0.15ex);
\node[shape=circle, scale=0.5] (L1) at (0,0) {};
\path[->,line width=.35pt,font=\tiny, looseness=16, in=215, out=145]
(L1.160) edge node[pos=.8,below=-.2ex] {$x$} (L1.200);
\path[->,line width=.35pt,font=\tiny, looseness=16, in=35, out=325]
(L1.340) edge node[pos=.15,below=-.2ex] {$y$} (L1.20);
\end{scope}
\begin{scope}[shift={(8.9em,-.5em)}]
\draw[line width=1pt, fill=black] (0,0) circle(0.15ex);
\draw[line width=1pt, fill=black] (1.3,0) circle(0.15ex);
\node[shape=circle, scale=0.5] (L3) at (0,0) {};
\node[shape=circle, scale=0.5] (R3) at (1.3,0) {};
\path[->,line width=.35pt,font=\tiny, looseness=16, in=35, out=325]
(R3.340) edge (R3.20);
\path[<-,line width=.35pt,font=\tiny]
(L3) edge (R3);
\end{scope}
\begin{scope}[shift={(13.2em,-.5em)}]
\draw[line width=1pt, fill=black] (0,0) circle(0.15ex);
\draw[line width=1pt, fill=black] (1.3,0) circle(0.15ex);
\node[shape=circle, scale=0.5] (L3) at (0,0) {};
\node[shape=circle, scale=0.5] (R3) at (1.3,0) {};
\path[->,line width=.35pt,font=\tiny, looseness=16, in=35, out=325]
(R3.340) edge node[pos=.15,below=-.2ex] {$y$} (R3.20);
\path[<-,line width=.35pt,font=\tiny]
(L3) edge node[below=-.2ex] {$x$} (R3);
\end{scope}
\begin{scope}[shift={(9.7em,1.5em)}]
\draw[line width=1pt, fill=black] (0,0) circle(0.15ex);
\draw[line width=1pt, fill=black] (1.3,0) circle(0.15ex);
\node[shape=circle, scale=0.5] (L2) at (0,0) {};
\node[shape=circle, scale=0.5] (R2) at (1.3,0) {};
\path[->,line width=.35pt,font=\tiny, looseness=16, in=215, out=145]
(L2.160) edge (L2.200);
\path[<-,line width=.35pt,font=\tiny]
(L2) edge (R2);
\end{scope}
\begin{scope}[shift={(14em,1.5em)}]
\draw[line width=1pt, fill=black] (0,0) circle(0.15ex);
\draw[line width=1pt, fill=black] (1.3,0) circle(0.15ex);
\node[shape=circle, scale=0.5] (L2) at (0,0) {};
\node[shape=circle, scale=0.5] (R2) at (1.3,0) {};
\path[->,line width=.35pt,font=\tiny, looseness=16, in=215, out=145]
(L2.160) edge node[pos=.8,below=-.2ex] {$x$} (L2.200);
\path[<-,line width=.35pt,font=\tiny]
(L2) edge node[below=-.2ex] {$y$} (R2);
\end{scope}
\begin{scope}[shift={(8.9em,3.5em)}]
\draw[line width=1pt, fill=black] (0,0) circle(0.15ex);
\draw[line width=1pt, fill=black] (1.3,0) circle(0.15ex);
\draw[line width=1pt, fill=black] (2.1,0) circle(0.15ex);
\node[shape=circle, scale=0.5] (L4) at (0,0) {};
\node[shape=circle, scale=0.5] (R4) at (1.3,0) {};
\path[->,line width=.35pt,font=\tiny]
(L4) edge (R4);
\end{scope}
\begin{scope}[shift={(13.2em,3.5em)}]
\draw[line width=1pt, fill=black] (0,0) circle(0.15ex);
\draw[line width=1pt, fill=black] (1.3,0) circle(0.15ex);
\node[shape=circle, scale=0.5] (L4) at (0,0) {};
\node[shape=circle, scale=0.5] (R4) at (1.3,0) {};
\path[->,line width=.35pt,font=\tiny]
(L4) edge (R4);
\end{scope}
\begin{scope}[shift={(9.7em,5.5em)}]
\draw[line width=1pt, fill=black] (0,0) circle(0.15ex);
\draw[line width=1pt, fill=black] (1.3,0) circle(0.15ex);
\node[shape=circle, scale=0.5] (L5) at (0,0) {};
\node[shape=circle, scale=0.5] (R5) at (1.3,0) {};
\path[->,line width=.35pt,font=\tiny, looseness=16, in=215, out=145]
(L5.160) edge (L5.200);
\path[->,line width=.35pt,font=\tiny, looseness=16, in=215, out=145]
(R5.160) edge (R5.200);
\end{scope}
\begin{scope}[shift={(14em,5.5em)}]
\draw[line width=1pt, fill=black] (0,0) circle(0.15ex);
\draw[line width=1pt, fill=black] (1.3,0) circle(0.15ex);
\node[shape=circle, scale=0.5] (L5) at (0,0) {};
\node[shape=circle, scale=0.5] (R5) at (1.3,0) {};
\path[->,line width=.35pt,font=\tiny, looseness=16, in=215, out=145]
(L5.160) edge node[pos=.8,below=-.2ex] {$x$} (L5.200);
\path[->,line width=.35pt,font=\tiny, looseness=16, in=215, out=145]
(R5.160) edge node[pos=.8,below=-.2ex] {$y$} (R5.200);
\end{scope}
\begin{scope}[shift={(9.4em,7.5em)}]
\draw[line width=1pt, fill=black] (0,0) circle(0.15ex);
\draw[line width=1pt, fill=black] (.8,0) circle(0.15ex);
\draw[line width=1pt, fill=black] (1.6,0) circle(0.15ex);
\node[shape=circle, scale=0.5] (L6) at (0,0) {};
\path[->,line width=.35pt,font=\tiny, looseness=16, in=215, out=145]
(L6.160) edge (L6.200);
\end{scope}
\begin{scope}[shift={(14em,7.5em)}]
\draw[line width=1pt, fill=black] (0,0) circle(0.15ex);
\node[shape=circle, scale=0.5] (L6) at (0,0) {};
\path[->,line width=.35pt,font=\tiny, looseness=16, in=215, out=145]
(L6.160) edge node[pos=.8,below=-.2ex] {$x$} (L6.200);
\end{scope}
%
\begin{scope}[shift={(1.1em,-3em)}]
\draw[line width=1pt, fill=black] (0,0) circle(0.15ex);
\node[shape=circle, scale=0.5] (L1) at (0,0) {};
\node[font=\tiny] at (0,-.5) {$\lambda$};
\path[->,line width=.35pt,font=\tiny, looseness=16, in=215, out=145]
(L1.160) edge (L1.200);
\path[->,line width=.35pt,font=\tiny, looseness=16, in=35, out=325]
(L1.340) edge (L1.20);
\end{scope}
\begin{scope}[shift={(.9em,-.5em)}]
\draw[line width=1pt, fill=black] (0,0) circle(0.15ex);
\node[shape=circle, scale=0.5] (L1) at (0,0) {};
\node[font=\tiny] at (0,-.5) {$0$};
\path[->,line width=.35pt,font=\tiny, looseness=16, in=215, out=145]
(L1.160) edge (L1.200);
\path[->,line width=.35pt,font=\tiny, looseness=16, in=35, out=325]
(L1.340) edge (L1.20);
\end{scope}
\begin{scope}[shift={(.9em,4.5em)}]
\draw[line width=1pt, fill=black] (0,0) circle(0.15ex);
\node[shape=circle, scale=0.5] (L1) at (0,0) {};
\node[font=\tiny] at (0,-.5) {$1$};
\path[->,line width=.35pt,font=\tiny, looseness=16, in=215, out=145]
(L1.160) edge (L1.200);
\path[->,line width=.35pt,font=\tiny, looseness=16, in=35, out=325]
(L1.340) edge (L1.20);
\end{scope}
\begin{scope}[shift={(-3.6em,-.5em)}]
\draw[line width=1pt, fill=black] (0,0) circle(0.15ex);
\draw[line width=1pt, fill=black] (1.3,0) circle(0.15ex);
\node[shape=circle, scale=0.5] (L2) at (0,0) {};
\node[shape=circle, scale=0.5] (R2) at (1.3,0) {};
\path[->,line width=.35pt,font=\tiny, looseness=16, in=35, out=325]
(R2.340) edge (R2.20);
\path[<-,line width=.35pt,font=\tiny]
(L2) edge (R2);
\end{scope}
\begin{scope}[shift={(1.55em,1em)}]
\draw[line width=1pt, fill=black] (0,0) circle(0.15ex);
\draw[line width=1pt, fill=black] (1.3,0) circle(0.15ex);
\node[shape=circle, scale=0.5] (L2) at (0,0) {};
\node[shape=circle, scale=0.5] (R2) at (1.3,0) {};
\path[->,line width=.35pt,font=\tiny, looseness=16, in=215, out=145]
(L2.160) edge (L2.200);
\path[<-,line width=.35pt,font=\tiny]
(L2) edge (R2);
\end{scope}
\begin{scope}[shift={(-3.6em,1.5em)}]
\draw[line width=1pt, fill=black] (0,0) circle(0.15ex);
\draw[line width=1pt, fill=black] (1.3,0) circle(0.15ex);
\draw[line width=1pt, fill=black] (2.1,0) circle(0.15ex);
\node[shape=circle, scale=0.5] (L4) at (0,0) {};
\node[shape=circle, scale=0.5] (R4) at (1.3,0) {};
\path[->,line width=.35pt,font=\tiny]
(L4) edge (R4);
\end{scope}
\begin{scope}[shift={(-3em,4.6em)}]
\draw[line width=1pt, fill=black] (0,0) circle(0.15ex);
\draw[line width=1pt, fill=black] (1.3,0) circle(0.15ex);
\node[shape=circle, scale=0.5] (L5) at (0,0) {};
\node[shape=circle, scale=0.5] (R5) at (1.3,0) {};
\path[->,line width=.35pt,font=\tiny, looseness=16, in=215, out=145]
(L5.160) edge (L5.200);
\path[->,line width=.35pt,font=\tiny, looseness=16, in=215, out=145]
(R5.160) edge (R5.200);
\end{scope}
\begin{scope}[shift={(-3em,6.5em)}]
\draw[line width=1pt, fill=black] (0,0) circle(0.15ex);
\draw[line width=1pt, fill=black] (.8,0) circle(0.15ex);
\draw[line width=1pt, fill=black] (1.6,0) circle(0.15ex);
\node[shape=circle, scale=0.5] (L6) at (0,0) {};
\path[->,line width=.35pt,font=\tiny, looseness=16, in=215, out=145]
(L6.160) edge (L6.200);
\end{scope}
\end{tikzpicture}
\caption{A variety of actual deformations of $\Bbbk \langle x, y \rangle / \langle x^2, y^2, yx \rangle$}
\label{figure:orbits}
\end{figure}

This sort of low-dimensional local picture can be obtained for any finite-dimen\-sional algebra. Note that if $A = \Bbbk Q / I$ is finite-dimensional, then $I$ admits a {\it finite} Gröbner basis\index{Gröbner basis!noncommutative!finite} \cite[Prop.~2.11]{green1}, which can be used to obtain a finite reduction system satisfying ($\diamond$) for $I$ (cf.\ \S\ref{subsubsection:reduction}).

\begin{subappendices}

\section{Deformations of Brauer tree algebras}
\label{subsection:brauertree}

In this section we describe deformations of certain Brauer tree algebras which are special biserial algebras and are examples of (generalized) Khovanov arc algebras which appear in various different contexts from low dimensional topology and topological quantum field theory to representation theory and symplectic geometry (see e.g.\ \cite{khovanov,stroppel,brundanstroppel,abouzaidsmith,boenakanowiesner}). Deformations of these Brauer tree algebras were studied in Mazorchuk--Stroppel \cite{mazorchukstroppel} in the context of the representation theory of $\mathfrak{sl}_n$. Here we illustrate how to recover the deformations of these Brauer tree algebras using the approach via reduction systems and the combinatorial criterion for the Maurer--Cartan equation \index{Maurer--Cartan!equation} given in Theorem \ref{theorem:higher-brackets}. We compute both the variety of {\it actual} deformations as well as the equivalence classes of formal deformations.

Let $n \geq 4$ be fixed. Let $Q$ be the quiver
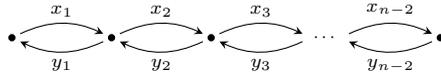
\[
\begin{tikzpicture}[baseline=-2.75pt,x=3.75em,y=1em]
\draw[line width=1pt, fill=black] (0,0) circle(0.2ex);
\draw[line width=1pt, fill=black] (1,0) circle(0.2ex);
\draw[line width=1pt, fill=black] (2,0) circle(0.2ex);
\draw[line width=1pt, fill=black] (4.3,0) circle(0.2ex);
\node[shape=circle, scale=0.7](0) at (0,0) {};
\node[shape=circle, scale=0.7](1) at (1,0) {};
\node[shape=circle, scale=0.7](2) at (2,0) {};
\node[shape=circle, scale=0.7](31) at (3,0) {};
\node[shape=circle, scale=0.7](32) at (3.3,0) {};
\node[shape=circle, scale=0.7](dots) at (3.15,0) {$\cdots$};
\node[shape=circle, scale=0.7](4) at (4.3,0) {};
\path[->,line width=.4pt,font=\scriptsize, looseness=1, in=155, out=25]
(0) edge node[above=-.2ex] {$x_1$} (1)
(1) edge node[above=-.2ex] {$x_2$} (2)
(2) edge node[above=-.2ex] {$x_3$} (31)
(32) edge node[above=-.2ex] {$x_{n-2}$} (4)
;
\path[<-,line width=.4pt,font=\scriptsize, looseness=1, in=-155, out=-25]
(0) edge node[below=-.2ex] {$y_1$} (1)
(1) edge node[below=-.2ex] {$y_2$} (2)
(2) edge node[below=-.2ex] {$y_3$} (31)
(32) edge node[below=-.2ex] {$y_{n-2}$} (4)
;
\end{tikzpicture}
\]
The {\it Brauer tree algebra} is the quotient algebra $A = \Bbbk Q / I$ where 
\[
I = \langle x_i x_{i+1}, y_{i+1}y_i, x_{i+1}y_{i+1}- y_{i}x_{i}\rangle_{1\leq i \leq n-3}
\]
so that $A$ is a finite-dimensional algebra of dimension $4n-6$.

We have a reduction system 
$$
R = \{ (y_1 x_1 y_1, 0), (x_1 y_1 x_1, 0)\} \cup \{ (x_i x_{i+1}, 0), (y_{i+1} y_i, 0), (x_{i+1} y_{i+1}, y_i x_i) \}_{1 \leq i < n-2}
$$
satisfying ($\diamond$) for $I$. The extra two terms $(y_1 x_1 y_1, 0)$ and $(x_1 y_1 x_1, 0)$ are added to the reduction system to resolve the overlap ambiguities\index{overlap ambiguity}\index{ambiguity!overlap} $x_1 x_2 y_2$ and $x_2 y_2 y_1$, respectively (cf.\ Heuristic \ref{heuristic}). We have
\begin{align*}
S   &= \{ x_1 y_1 x_1, y_1 x_1 y_1\} \cup \{ x_i x_{i+1}, y_{i+1} y_i, x_{i+1} y_{i+1} \}_{1 \leq i \leq n-3} \\
S_3 &= \{ x_1 y_1 x_1 y_1, y_1 x_1 y_1 x_1, x_1 y_1 x_1 x_2, y_2 y_1 x_1 y_1 \} \\
&\qquad{} \cup \{ x_{i+1} y_{i+1} y_i, x_i x_{i+1} y_{i+1} \}_{1 \leq i \leq n-3} \\
&\qquad{} \cup \{x_ix_{i+1}x_{i+2}, y_{i+2}y_{i+1}y_i\}_{1 \leq i \leq n-4}.
\end{align*}

Observe that the reduction system $R$ can be obtained from the noncommutative Gröbner basis\index{Gröbner basis!noncommutative} induced by the degree--lexicographic order $\prec$ \index{admissible order!degree--lexicographic}\index{degree--lexicographic order} which extends $y_1 \prec \dotsb \prec y_{n-2} \prec x_{n-2} \prec \dotsb \prec x_1$ and for this order we have $\Hom (\Bbbk S, A)_\prec = \Hom(\Bbbk S, A) \simeq \Bbbk^{2n - 4}$.

Note that the irreducible path parallel to $x_1 y_1 x_1$ is $x_1$, the irreducible path parallel to $y_1 x_1 y_1$ is $y_1$, the irreducible paths parallel to $x_{i+1} y_{i+1}$ are $e_{i+1}$ and $y_i x_i$, and there are no irreducible paths parallel to $x_i x_{i+1}$ or $y_{i+1} y_{i}$. This immediately gives the following lemma.

\begin{lemma}
\label{lemma:brauerg}
We have that $\Hom (\Bbbk S, A) \simeq \Bbbk^{2n-4}$ with an element $\widetilde \varphi \in \Hom(\Bbbk S, A)$ of the following general form
\begin{flalign}
&& \widetilde \varphi_{x_1y_1x_1} & = \lambda_1 x_1 && \notag \\
&& \widetilde \varphi_{y_1x_1y_1} & = \mu_1 y_1 && \notag \\
&& \widetilde \varphi_{x_{i+1}y_{i+1}} & = \lambda_{i+1} e_{i+1} + \mu_{i+1} y_i x_i && \mathllap{\text{for $1 \leq i \leq n-3$}} \label{brauerg} \\
&& \widetilde \varphi_{x_ix_{i+1}} & = 0 && \notag \\
&& \widetilde \varphi_{y_{i+1}y_i} & = 0 && \notag
\end{flalign}
for $\lambda_j, \mu_j \in \Bbbk$ for $1 \leq j \leq n-2$.

Moreover, the subspace $\Hom (\Bbbk S, A)_< \subset \Hom (\Bbbk S, A)$ is given by the elements $\widetilde \varphi$ of the form \textup{(\ref{brauerg})} with $\mu_2 = \dotsb = \mu_{n-2} = 0$.
\end{lemma}

\subsection{The Maurer--Cartan equation}
\index{Maurer--Cartan!equation}

The Maurer--Cartan equation of $\mathbf p (Q, R)$ allows us to compute all formal deformations up to equivalence (Theorem \ref{theorem:equivalenceformal}) and it gives the equations for the varieties of actual deformations (Proposition \ref{proposition:pbwvariety} and Theorem \ref{theorem:variety}).

Thus the whole deformation theory of $A$ is essentially contained in a simple computation which we record in the following lemma.

\begin{lemma}
\label{lemma:brauermc}
Let $\widetilde \varphi \in \Hom (\Bbbk S, A)$ be given by $\lambda_i, \mu_j \in \Bbbk$ for $1 \leq i, j \leq n - 2$ as in Lemma \ref{lemma:brauerg}. Then the Maurer--Cartan equation \index{Maurer--Cartan!equation} for $\widetilde \varphi$ is well-defined for all $\lambda_i, \mu_j$ in $\Bbbk$ and is equivalent to
\[
\lambda_1 = \mu_1 \qquad\text{and}\qquad \lambda_i = (-1)^{i+1} \mu_1 (1 + \mu_2) \dotsb (1 + \mu_i) \quad \text{for $1 < i \leq n-2$.}
\]

The same equation is valid for formal deformations over any complete local Noetherian $\Bbbk$-algebra $(B, \mathfrak m)$ by viewing $\lambda_i, \mu_j \in \mathfrak m$ so that $\widetilde \varphi \in \Hom (\Bbbk S, A) \hatotimes \mathfrak m$.
\end{lemma}

\begin{proof}
By Theorem \ref{theorem:higher-brackets} the Maurer--Cartan equation is equivalent to the condition $\pi(u) \star (\pi(v) \star \pi(w)) = (\pi(u) \star \pi(v)) \star \pi(w)$ for every $uvw \in S_3$ with $uv, vw \in S$, where $\star = \star^{\mathrm C}_{\phi + \widetilde \phi}$. For example the condition $\lambda_1 = \mu_1$ is equivalent to $(x_1 \star y_1 x_1) \star y_1 = x_1 \star (y_1 x_1 \star y_1)$. The other equations follow from computing associativity of $\star$ for other elements in $S_3$. (For all $n \geq 4$ it suffices to check essentially six equations.)

That the Maurer--Cartan equation is well-defined for all $\widetilde \varphi \in \Hom (\Bbbk S, A)$ follows from Proposition \ref{proposition:Groebner} since $\Hom (\Bbbk S, A)_\prec = \Hom(\Bbbk S, A).$
\end{proof}

From Lemma \ref{lemma:brauermc} and Corollary \ref{corollary:firstorder} it is now easy to see that $\HH^2 (A, A) \simeq \Bbbk$ (cf. \cite[Prop.~31]{mazorchukstroppel}).

\subsection{The varieties of actual deformations}

The equations in Lemma \ref{lemma:brauermc} cut out the variety $V_\prec \subset \mathbb A^{2n-4}$ which contains the variety $V_< \simeq \mathbb A^1$ of PBW deformations (cf.\ (\ref{inclvar})) parametrized by $\lambda \in \Bbbk$, where
\[
\lambda = \mu_1 = \lambda_1 = - \lambda_2 = \dotsb = (-1)^{n-1} \lambda_{n-2} \qquad \text{and} \qquad \mu_2 = \dotsb = \mu_{n-2} = 0.
\]
The algebras $A_\lambda$ in this family are given by
$$ 
A_{\lambda} = \Bbbk Q / \langle x_{i+1}y_{i+1} - y_i x_i - (-1)^{i} \lambda e_{i+1}, x_{i}x_{i+1}, y_{i+1}y_i\rangle_{1 \leq i < n-2}
$$
where $A_0 = A$ is the Brauer tree algebra. (Note that the two relations $x_1 y_1 x_1 - \lambda x_1$ and $y_1 x_1 y_1 - \lambda y_1$ are implied by the other relations.) Using the basis of irreducible paths for $A$, it would be straightforward to write down a full multiplication table of $A_\lambda$ using  $\star$.

Since $\HH^2 (A, A) \simeq \Bbbk$ we could already stop here as we have found a $1$-dimensional family of nontrivial deformations, i.e.\ for any cohomology class in $\HH^2 (A, A)$ we have found an algebraizable deformation whose first order term is a cocycle representing this cohomology class. The other deformations come from coboundaries and it can be shown directly that the algebra corresponding to the point $(\lambda_1, \dotsc, \lambda_{n-2}, \mu_1, \dotsc, \mu_{n-2}) \in V_\prec$ is isomorphic to the algebra determined by $\mu_1$ with $\mu_2 = \dotsb = \mu_{n-2} = 0$. In fact, it is straightforward to construct an explicit equivalence between the corresponding reduction systems giving the following result.

\begin{proposition}
\label{proposition:orbits}
The groupoid $G \rightrightarrows \MC = V_\prec$ has two orbits, the closed point corresponding to the origin of $V_\prec \subset \Hom (\Bbbk S, A)_\prec$ and its complement in $V_\prec$.
\end{proposition}

\noindent (Note that here $\Hom (\Bbbk S, A)_\prec = \Hom (\Bbbk S, A)$, so all corresponding reduction systems are reduction-finite and we have $V_\prec = \MC$.)

\begin{remark}
\label{remark:brauerdim}
Note that in the classical picture, we have $\dim_{\Bbbk} \Hom_{\Bbbk} (A \otimes_{\Bbbk} A, A) = 8 (2n-3)^3$ so that
\begin{none*}
\begin{center}
\renewcommand{\arraystretch}{1.25}
\begin{tabular}{c||c|c|c|c|c}
$n$ & $4$ & $5$ & $6$ & $7$ & $\cdots$ \\
\hline
\rlap{$\dim A$}\phantom{$\dim \Hom (A^{\otimes 2}, A)$} & $10$ & $14$ & $18$ & $22$ & $\cdots$ \\[.5em]
$\dim \Hom (A^{\otimes 2}, A)$ & \eqmakebox[ctr]{$1000$} & \eqmakebox[ctr]{$2744$} & \eqmakebox[ctr]{$5832$} & \eqmakebox[ctr]{$10648$} & $\cdots$ \\
$\dim \Hom (A^{\otimes 3}, A)$ & \eqmakebox[ctr]{$10000$} & \eqmakebox[ctr]{$38416$} & \eqmakebox[ctr]{$104976$} & \eqmakebox[ctr]{$234256$} & $\cdots$ \\[.5em]
\rlap{$\dim \Hom (\Bbbk S, A)$}\phantom{$\dim \Hom (A^{\otimes 2}, A)$} & \eqmakebox[ctr]{$4$} & \eqmakebox[ctr]{$6$} & \eqmakebox[ctr]{$8$} & \eqmakebox[ctr]{$10$} & $\cdots$ \\
\rlap{$\dim \Hom (\Bbbk S_3, A)$}\phantom{$\dim \Hom (A^{\otimes 2}, A)$} & \eqmakebox[ctr]{$6$} & \eqmakebox[ctr]{$10$} & \eqmakebox[ctr]{$14$} & \eqmakebox[ctr]{$18$} & $\cdots$ \\
\end{tabular}
\end{center}
\end{none*}
\noindent In principle, the deformations of the $10$-dimensional Brauer tree algebra can be obtained by studying a certain irreducible component of the variety of all associative algebra structures on a $10$-dimensional vector space --- which is cut out by $10000$ equations in $\mathbb A^{1000}$ (cf.\ \S\ref{subsubsection:variety}). An explicit description of this variety seems rather inaccessible in such high dimensions. However, the variety $V_{\prec}$ can be viewed as a suitable subvariety, which may already capture all nontrivial deformations. Indeed for the $10$-dimensional Brauer tree algebra, we find a {\it surface} $V_\prec$ whose coordinate ring is $\Bbbk [x, y, z] / (x - yz)$ which contains an affine {\it line} $V_<$ (see Fig.~\ref{figure:brauervariety}) already capturing all nontrivial deformations by Proposition \ref{proposition:orbits}. 
\end{remark}

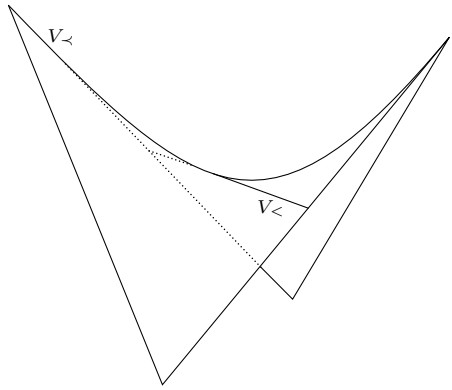
\begin{figure}
\centering
\begin{tikzpicture}[x=1pt,y=1pt,line width=.4pt]
\draw (-58,142) -- (0,0) -- (108,130) -- (49,32) -- (37,44);
\draw (55,66) -- (19,79);
\path[dash pattern=on 0pt off 1.5pt, line width=.6pt, line cap=round] (37,44) edge (-36.5,120);
\path[dash pattern=on 0pt off 1.5pt, line width=.6pt, line cap=round] (-5,87.5) edge (12,81.75);
\node[font=\scriptsize] at (-38,130) {$V_\prec$};
\node[font=\scriptsize] at (41,66) {$V_<$};
\path (-58,142) edge[out=-46,in=230,looseness=1.63] (108,130);
\end{tikzpicture}
\caption{The variety of actual deformations of the reduction system for the $10$-dimensional Brauer tree algebra and its subvariety of PBW deformations}
\label{figure:brauervariety}
\end{figure}

\subsection{Formal deformations up to equivalence}
\label{subsubsection:formalbrauertree}

Intuitively, the tangent space of the ``deformation space'' should be given by the second Hochschild cohomology, which for the Brauer tree algebras is $\HH^2 (A, A) \simeq \Bbbk$. We found above that the variety $V_< \simeq \mathbb A^1$ is a geometric model of the orbit space of the groupoid of actual deformations and it gives an explicit family of nontrivial deformations whose tangent space at $0 \in \mathbb A^1$ (corresponding to $A$) is indeed $1$-dimensional.

From the above point of view we expect that the deformation functor $\Def_A$ encoding the formal deformation theory of $A$ is prorepresented by the formal neighbourhood of $0 \in \mathbb A^1$ whose ring of functions is just the formal power series ring $\Bbbk \llrr{t}$ (cf.\ \S\ref{subsubsection:geometricinterpretation}). Indeed, this can be shown directly by determining the Maurer--Cartan elements up to equivalence in the formal setting (see Lemma \ref{lemma:brauermc} and Definition \ref{definition:equivalenceformalreductionsystem}) giving the following result, which is the formal analogue of Proposition \ref{proposition:orbits}.

\begin{proposition} \label{proposition:deformationfunctorbrauer}
The deformation functor $\Def_A$  of the Brauer tree algebra $A$ is prorepresented by $\widehat{\mathcal O}_{\mathbb A^1, 0} \simeq (\Bbbk \llrr{t}, (t)) \in \widehat{\mathfrak{Art}}_{\Bbbk}$.
\end{proposition}

\end{subappendices}

\chapter[A graphical calculus]{A graphical calculus for the combinatorial star product}
\label{section:relationtoquantization}
\index{deformation quantization}
\index{star product!combinatorial!graphical calculus}\index{combinatorial star product!graphical calculus}

In this section we give a graphical description of the combinatorial star product for deformations of the polynomial algebra
\[
A = \Bbbk [x_1, \dotsc, x_d] = \Bbbk \langle x_1, \dotsc, x_d \rangle / \langle x_j x_i - x_i x_j\rangle_{1 \leq i < j \leq d}.
\]
This graphical description turns out to be very similar to Kontsevich's graphical calculus used in his universal formula for the deformation quantization of Poisson structures on $\mathbb R^d$ \cite{kontsevich1} --- with a number of key differences.\index{deformation quantization}

Let $Q$ be the quiver with one vertex and $d$ loops $x_1, \dotsc, x_d$. Imposing the commutativity relations, one obtains the polynomial algebra
\begin{equation}
\label{polynomialring}
A = \Bbbk [x_1, \dotsc, x_d] = \Bbbk \biggl( 
\!
\begin{tikzpicture}[baseline=-2.75pt,x=2.2em,y=1em]
\draw[line width=1pt, fill=black] (0,0) circle(0.2ex);
\node[shape=circle, scale=0.7](L) at (0,0) {};
\node[shape=circle, scale=0.9](LL) at (0,0) {};
\path[->,line width=.4pt,font=\scriptsize, looseness=16, in=35, out=325,transform canvas={xshift=-.5pt,yshift=-.3pt}]
(L.340) edge (L.20)
;
\path[->,line width=.4pt,font=\scriptsize, looseness=18, in=40, out=320,transform canvas={xshift=-3pt},overlay]
(LL.320) edge (LL.40)
;
\node[font=\scriptsize] at (0.56,-1.3) {$x_1, \dotsc, x_d$}
;
\node[font=\scriptsize] at (.8,0) {$...$}
;
\end{tikzpicture}
\!\!
\biggr) \Big/ \bigl\langle x_j x_i - x_i x_j \bigr\rangle_{1 \leq i < j \leq d}.
\end{equation}
We have a reduction system
\begin{equation}
\label{reductionsymmetric}
\begin{aligned}
R &= \bigl\{ (x_j x_i, x_i x_j) \bigr\}_{1 \leq i < j \leq d} \\
\llap{with} \quad S &= \{ x_j x_i \}_{1 \leq i < j \leq d}
\end{aligned}
\end{equation}
(cf.\ Example \ref{example:reductiongroebnersymmetric}).

Working with deformations over $\Bbbk \llrr{\hbar}$ with $\mathfrak m = (\hbar)$, let $\widetilde \phi \in \Hom (\Bbbk S, A) \hatotimes \mathfrak m$. Then we have the following result.

\begin{theorem}
\label{theorem:loopless}
The combinatorial star product\index{star product!combinatorial!graphical calculus}\index{combinatorial star product!graphical calculus} $\star = \star^{\rm C}_{\phi + \widetilde \phi}$ can be given as
\begin{flalign}
\label{combinatorialstar}
&& f \star g &= \sum_{k \geq 0} \sum_{\Pi \in \mathfrak G_{k,2}^{\mathrm C}} C_\Pi (f, g) && \mathllap{f, g \in A}
\end{flalign}
where $\mathfrak G_{k,2}^{\mathrm C}$ is a set of graphs and each $C_\Pi$ is a formal power series of bidifferential operators starting in order $\hbar^k$ associated to a graph $\Pi$, corresponding to performing reductions.
\end{theorem}

We will compare the formula \eqref{combinatorialstar} to Kontsevich's universal formula in \S\ref{subsection:relationdeformationquantization} and explain how the combinatorial star product\index{star product!combinatorial}\index{combinatorial star product} and the L$_\infty[1]$ algebra\index{L$_\infty$!algebra}\index{algebra!L$_\infty$} $\mathbf p (Q, R)$ of Chapter \ref{section:deformations-of-path-algebras} can be used as a combinatorial approach to the problem of deformation quantization of Poisson structures in \S\ref{subsection:combinatorialquantization}.\index{deformation quantization}

The set of graphs $\mathfrak G_{k,2}^{\mathrm C}$ and the formal power series of bidifferential operators $C_\Pi$ will be given in Definition \ref{definition:admissible}.

We now turn to the proof of Theorem \ref{theorem:loopless}. For simplicity, let us assume that $\widetilde \phi_{x_j x_i}$ is a monomial times $\hbar^m$ for some $m$ --- the general case follows by extending everything linearly. We shall now describe $f \star g$, where $f = x_{i_1} \cdots x_{i_p}$ and $g = x_{j_1} \cdots x_{j_q}$ with $1 \leq i_1 \leq \dotsb \leq i_p \leq d$ and $1 \leq j_1 \leq \dotsb \leq j_q \leq d$ are two arbitrary monomials of degree $p$ and $q$, respectively.

Recall that we have
\[
f \star g = fg + \sum_{k \geq 1} f \star^k g
\]
where $f \star^k g$ is given by performing rightmost reductions on the concatenation $x_{i_1} \cdots x_{i_p} x_{j_1} \cdots x_{j_q}$, using $\widetilde \phi$ exactly $k$ times and then projecting to $A\llrr{\hbar}$ via $\pi \colon \Bbbk Q\llrr{\hbar} \to (\Bbbk Q / I) \llrr{\hbar} = A \llrr{\hbar}$ (see \S\ref{subsection:combinatorialstarproduct}).

If the indices appearing in the concatenation $x_{i_1} \cdots x_{i_p} x_{j_1} \cdots x_{j_q}$ are increasing (i.e.\ if $i_p \leq j_1$), then there are no reductions to be performed, and we have $f \star^k g = 0$ for all $k > 0$, so that $f \star g = fg$ (cf.\ Lemma \ref{lemma:star} \ref{guttproperty2}). In general, we have to perform reductions on $x_1 \cdots x_{i_p} x_{j_1} \cdots x_{j_q}$, which we shall now describe in detail.

If $i_p > j_1$, then there may be several terms $i_l$ for $1 \leq l \leq p$ and $j_m$ for $1 \leq m \leq q$ such that $i_l > j_m$ and then
\begin{equation}
\label{starg1}
f \star^1 g = \sum_{i_l > j_m} \pi (x_{i_1} \cdots x_{i_{l-1}} x_{j_1} \cdots x_{j_{m-1}} \widetilde \phi_{x_{i_l} x_{j_m}} \overbar{x_{i_{l+1}} \cdots x_{i_p} x_{j_{m+1}} \cdots x_{j_q}})
\end{equation}
where each summand is obtained by using the usual commutativity relations from the right until $x_{i_l}$ and $x_{j_m}$ are next to each other and using $\widetilde \phi$ (once) when commuting $x_{i_l}$ past $x_{j_m}$, i.e.\ replacing $x_{i_l} x_{j_m}$ not by $x_{j_m} x_{i_l}$ but by $\widetilde \phi_{x_{i_l} x_{j_m}}$. Since we always perform the rightmost reduction first, the terms to the right of $\widetilde \phi_{x_{i_l} x_{j_m}}$ must have been put into the correct order already, which we indicate by the bar. At this stage there are further reductions to be performed, but in $\star^1$ these have to use the usual commutativity relations since we have used $\widetilde \phi$ once already.

Representing linear monomials by \begin{tikzpicture}[baseline=.1em,x=.8em,y=.8em]\draw[line width=.4pt] (0,0) -- (1,0) -- (1,1) -- (0,1) -- (0,0); \end{tikzpicture}\hair, each term appearing as an argument of $\pi$ in (\ref{starg1}) can be represented graphically as follows
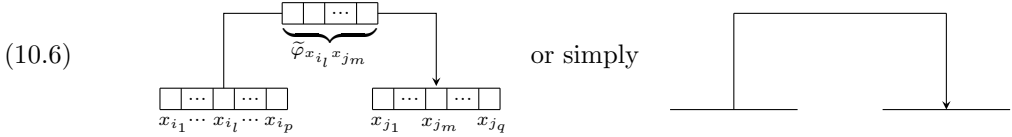
\begin{equation}
\label{graph1}
\begin{tikzpicture}[baseline=1.75em,x=.8em,y=.8em]
\path[line width=.4pt]
(1,0) edge (7,0)
(7,0) edge (7,1) 
(7,1) edge (1,1)
(1,1) edge (1,0)
(2,0) edge (2,1)
(3.5,0) edge (3.5,1)
(4.5,0) edge (4.5,1)
(6,0) edge (6,1)
;
\node[font=\scriptsize] at (1.65,-.7) {$x_{i_1}$};
\node[font=\scriptsize] at (2.75,-.5) {$...$};
\node[font=\scriptsize] at (4.15,-.7) {$x_{i_l}$};
\node[font=\scriptsize] at (5.25,-.5) {$...$};
\node[font=\scriptsize] at (6.65,-.7) {$x_{i_p}$};
\node[font=\scriptsize] at (2.75,.5) {$...$};
\node[font=\scriptsize] at (5.25,.5) {$...$};
\begin{scope}[shift={(10,0)}]
\path[line width=.4pt]
(1,0) edge (7,0)
(7,0) edge (7,1) 
(7,1) edge (1,1)
(1,1) edge (1,0)
(2,0) edge (2,1)
(3.5,0) edge (3.5,1)
(4.5,0) edge (4.5,1)
(6,0) edge (6,1)
;
\node[font=\scriptsize] at (1.65,-.7) {$x_{j_1}$};
\node[font=\scriptsize] at (4.2,-.7) {$x_{j_m}$};
\node[font=\scriptsize] at (6.65,-.7) {$x_{j_q}$};
\node[font=\scriptsize] at (2.75,.5) {$...$};
\node[font=\scriptsize] at (5.25,.5) {$...$};
\end{scope}
\begin{scope}[shift={(6.75,4)}]
\path[line width=.4pt]
(0,0) edge (4.5,0)
(4.5,0) edge (4.5,1) 
(4.5,1) edge (0,1)
(0,1) edge (0,0)
(1,0) edge (1,1)
(2,0) edge (2,1)
(3.5,0) edge (3.5,1)
;
\node at (2.25,-.9) {$\underbrace{\hspace{3.5em}}_{}$};
\node[font=\scriptsize] at (2.25,-1.4) {$\widetilde \phi_{x_{i_l} x_{j_m}}$};
\node[font=\scriptsize] at (2.75,.5) {$...$};
\end{scope}
\draw[->,line width=.4pt]
(6.75,4.5) -- (4,4.5) -- (4,1.1);
\draw[->,line width=.4pt]
(11.25,4.5) -- (14,4.5) -- (14,1.1);
\node[anchor=base] at (21,1.75em) {or simply};
\begin{scope}[shift={(24,0)}]
\path[line width=.4pt]
(1,0) edge (7,0)
(4,4.5) edge (14,4.5)
(11,0) edge (17,0)
;
\path[->,line width=.4pt]
(14,4.5) edge (14,.1)
(4,4.5) edge (4,.1)
;
\end{scope}
\end{tikzpicture}
\end{equation}
Here the ordering of the monomials and the ingoing arrows is important, so it is convenient to represent the monomials by line segments, rather than points.

The higher terms $f \star^k g$ for $k \geq 1$ can be obtained recursively, by performing reductions on all terms (before using the projection $\pi$) appearing in the expression of $\star^{k-1}$. In particular, to describe the terms in $f \star^2 g$ let us choose a monomial appearing as an argument of $\pi$ in (\ref{starg1}) and relabel the linear monomials as
\[
x_{h_1} \cdots x_{h_r} := x_{i_1} \cdots x_{i_{l-1}} x_{j_1} \cdots x_{j_{m-1}} \widetilde \phi_{x_{i_l} x_{j_m}} \overbar{x_{i_{l+1}} \cdots x_{i_p} x_{j_{m+1}} \cdots x_{j_q}}.
\]
Now choose linear monomials $x_{h_u}$ and $x_{h_v}$ for $1 \leq u < v \leq r$ with $h_u > h_v$ and use $\widetilde \phi$ (once) when commuting $x_{h_u}$ past $x_{h_v}$, giving
\[
x_{h_1} \cdots x_{h_{u-1}} \widetilde \phi_{x_{h_u} x_{h_v}} \overbar{x_{h_{u+1}} \cdots x_{h_{v-1}} x_{h_{v+1}} \cdots x_{h_r}}.
\]

Graphically this can be represented as in (\ref{graph1}) by
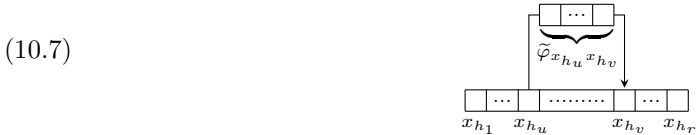
\begin{equation}
\label{graph}
\begin{tikzpicture}[baseline=2em,x=.8em,y=.8em]
\path[line width=.4pt]
(0,0) edge (10.5,0)
(10.5,0) edge (10.5,1) 
(10.5,1) edge (0,1)
(0,1) edge (0,0)
(1,0) edge (1,1)
(2.5,0) edge (2.5,1)
(3.5,0) edge (3.5,1)
(7,0) edge (7,1)
(8,0) edge (8,1)
(9.5,0) edge (9.5,1)
;
\node[font=\scriptsize] at (.65,-.7) {$x_{h_1}$};
\node[font=\scriptsize] at (1.75,.5) {$...$};
\node[font=\scriptsize] at (3.15,-.7) {$x_{h_u}$};
\node[font=\scriptsize] at (5.25,.5) {$.........$};
\node[font=\scriptsize] at (7.75,-.7) {$x_{h_v}$};
\node[font=\scriptsize] at (8.75,.5) {$...$};
\node[font=\scriptsize] at (10.15,-.7) {$x_{h_r}$};
\begin{scope}[shift={(3.5,4)}]
\path[line width=.4pt]
(0,0) edge (3.5,0)
(3.5,0) edge (3.5,1) 
(3.5,1) edge (0,1)
(0,1) edge (0,0)
(1,0) edge (1,1)
(2.5,0) edge (2.5,1)
;
\node at (1.75,-.9) {$\underbrace{\hspace{2.7em}}_{}$};
\node[font=\scriptsize] at (1.85,-1.4) {$\widetilde \phi_{x_{h_u} x_{h_v}}$};
\node[font=\scriptsize] at (1.75,.5) {$...$};
\end{scope}
\draw[->,line width=.4pt]
(3.5,4.5) -- (3,4.5) -- (3,1.1);
\draw[->,line width=.4pt]
(7,4.5) -- (7.5,4.5) -- (7.5,1.1);
\end{tikzpicture}
\end{equation}
but (\ref{graph}) can also be represented in terms of the original graph (\ref{graph1}). For example, if $x_{h_u}$ is a linear monomial appearing in $\widetilde \phi_{x_{i_l} x_{j_m}}$ and $x_{h_v}$ a linear monomial appearing in $x_{j_1}, \dotsc, x_{j_{m-1}}$ such that $h_u > h_v$, we can represent (\ref{graph}) as
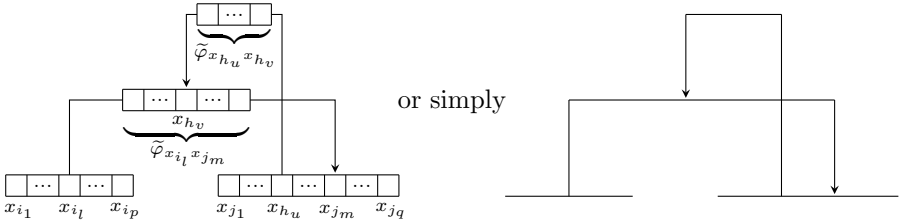
\begin{equation*}
\begin{tikzpicture}[baseline=3.2em,x=.8em,y=.8em]
\path[line width=.4pt]
(0,0) edge (6,0)
(6,0) edge (6,1) 
(6,1) edge (0,1)
(0,1) edge (0,0)
(1,0) edge (1,1)
(2.5,0) edge (2.5,1)
(3.5,0) edge (3.5,1)
(5,0) edge (5,1)
(10,0) edge (18.5,0)
(18.5,0) edge (18.5,1)
(18.5,1) edge (10,1)
(10,1) edge (10,0)
(11,0) edge (11,1)
(12.5,0) edge (12.5,1)
(13.5,0) edge (13.5,1)
(15,0) edge (15,1)
(16,0) edge (16,1)
(17.5,0) edge (17.5,1)
(5.5,4) edge (11.5,4)
(11.5,4) edge (11.5,5)
(11.5,5) edge (5.5,5)
(5.5,5) edge (5.5,4)
(6.5,4) edge (6.5,5)
(8,4) edge (8,5)
(9,4) edge (9,5)
(10.5,4) edge (10.5,5)
(9,8) edge (12.5,8)
(12.5,8) edge (12.5,9)
(12.5,9) edge (9,9)
(9,9) edge (9,8)
(10,8) edge (10,9)
(11.5,8) edge (11.5,9)
;
\node[font=\scriptsize] at (1.75,.5) {$...$};
\node[font=\scriptsize] at (4.25,.5) {$...$};
\node[font=\scriptsize] at (11.75,.5) {$...$};
\node[font=\scriptsize] at (14.25,.5) {$...$};
\node[font=\scriptsize] at (16.75,.5) {$...$};
\node[font=\scriptsize] at (7.25,4.5) {$...$};
\node[font=\scriptsize] at (9.75,4.5) {$...$};
\node[font=\scriptsize] at (10.75,8.5) {$...$};
\node[font=\scriptsize] at (.65,-.7) {$x_{i_1}$};
\node[font=\scriptsize] at (3.15,-.7) {$x_{i_l}$};
\node[font=\scriptsize] at (5.65,-.7) {$x_{i_p}$};
\node[font=\scriptsize] at (8.65,3.45) {$x_{h_v}$};
\node[font=\scriptsize] at (10.65,-.7) {$x_{j_1}$};
\node[font=\scriptsize] at (13.15,-.7) {$x_{h_u}$};
\node[font=\scriptsize] at (15.65,-.7) {$x_{j_m}$};
\node[font=\scriptsize] at (18.15,-.7) {$x_{j_q}$};
\node at (10.75,7.1) {$\underbrace{\hspace{2.7em}}_{}$};
\node[font=\scriptsize] at (10.85,6.6) {$\widetilde \phi_{x_{h_u} x_{h_v}}$};
\node at (8.5,2.5) {$\underbrace{\hspace{4.7em}}_{}$};
\node[font=\scriptsize] at (8.65,2) {$\widetilde \phi_{x_{i_l} x_{j_m}}$};
\draw[->,line width=.4pt]
(5.5,4.5) -- (3,4.5) -- (3,1.1);
\draw[->,line width=.4pt]
(9,8.5) -- (8.5,8.5) -- (8.5,5.1);
\draw[->,line width=.4pt]
(11.5,4.5) -- (15.5,4.5) -- (15.5,1.1);
\draw[->,line width=.4pt]
(12.5,8.5) -- (13,8.5) -- (13,1.1);
\node[anchor=base] at (21,3.3em) {or simply};
\begin{scope}[shift={(23.5,0)}]
\path[line width=.4pt]
(0,0) edge (6,0)
(10,0) edge (18.5,0)
(3,4.5) edge (15.5,4.5)
(8.5,8.5) edge (13,8.5)
;
\path[->,line width=.4pt]
(15.5,4.5) edge (15.5,.1)
(13,8.5) edge (13,.1)
;
\path[<-,line width=.4pt]
(8.5,4.6) edge (8.5,8.5)
(3,.1) edge (3,4.5)
;
\end{scope}
\end{tikzpicture}
\end{equation*}

The following six graphs are the graphs which can appear in $\star^2$
\[
\begin{tikzpicture}[x=.5em,y=.5em]
\draw[line width=.4pt]
(0,0) -- (4,0);
\draw[line width=.4pt]
(6,0) -- (10,0);
\draw[<->,line width=.4pt]
(2,.1) -- (2,3) -- (8,3) -- (8,.1);
\draw[<->,line width=.4pt]
(1,.1) -- (1,4) -- (7,4) -- (7,.1);
\begin{scope}[shift={(12,0)}]
\draw[line width=.4pt]
(0,0) -- (4,0);
\draw[line width=.4pt]
(6,0) -- (10,0);
\draw[<->,line width=.4pt]
(2,.1) -- (2,3) -- (8,3) -- (8,.1);
\draw[<->,line width=.4pt]
(1,.1) -- (1,4) -- (9,4) -- (9,.1);
\end{scope}
\begin{scope}[shift={(24,0)}]
\draw[line width=.4pt]
(0,0) -- (4,0);
\draw[line width=.4pt]
(6,0) -- (10,0);
\draw[<->,line width=.4pt]
(2,.1) -- (2,3) -- (8,3) -- (8,.1);
\draw[<->,line width=.4pt]
(1,.1) -- (1,4) -- (5,4) -- (5,3.1);
\end{scope}
\begin{scope}[shift={(36,0)}]
\draw[line width=.4pt]
(0,0) -- (4,0);
\draw[line width=.4pt]
(6,0) -- (10,0);
\draw[<->,line width=.4pt]
(2,.1) -- (2,3) -- (8,3) -- (8,.1);
\draw[<->,line width=.4pt]
(3,.1) -- (3,4) -- (5,4) -- (5,3.1);
\end{scope}
\begin{scope}[shift={(48,0)}]
\draw[line width=.4pt]
(0,0) -- (4,0);
\draw[line width=.4pt]
(6,0) -- (10,0);
\draw[<->,line width=.4pt]
(2,.1) -- (2,3) -- (8,3) -- (8,.1);
\draw[<->,line width=.4pt]
(5,3.1) -- (5,4) -- (7,4) -- (7,.1);
\end{scope}
\begin{scope}[shift={(60,0)}]
\draw[line width=.4pt]
(0,0) -- (4,0);
\draw[line width=.4pt]
(6,0) -- (10,0);
\draw[<->,line width=.4pt]
(2,.1) -- (2,3) -- (8,3) -- (8,.1);
\draw[<->,line width=.4pt]
(5,3.1) -- (5,4) -- (9,4) -- (9,.1);
\end{scope}
\end{tikzpicture}
\]

Using (\ref{graph1}) and (\ref{graph}) recursively, each term in $\star^k$ can be represented by a graph. 

\begin{definition}
\label{definition:admissible}
\index{admissible graph!for the combinatorial star product}
\index{star product!combinatorial!admissible graph}\index{combinatorial star product!admissible graph}
Denote by $\mathfrak G_{k,2}^{\mathrm C}$ the set of graphs appearing for $\star^k$ and for each $\Pi \in \mathfrak G_{k,2}^{\mathrm C}$, let $C_\Pi (f, g)$ be the sum of terms in $f \star^k g$ corresponding to reductions whose corresponding graph is of the form $\Pi$.
\end{definition}

(See Proposition \ref{proposition:graphs} for the relation to Kontsevich's admissible graphs.)

\begin{lemma}
\label{lemma:operator}
For each $\Pi \in \mathfrak G^{\mathrm C}_{k,2}$ the term $C_\Pi (f, g)$, obtained by performing reductions, can be expressed explicitly as a formal power series $C_\Pi$ of bidifferential operators acting on $f \otimes g$.
\end{lemma}

\begin{proof}
For each shape
\hair\begin{tikzpicture}[baseline=.1em,x=.8em,y=.8em]
\draw[<->,line width=.4pt]
(0,.35em) -- (0,2ex) -- (1.5em,2ex) -- (1.5em,-.1ex);
\draw[-,line width=.4pt]
(-.3em,.3em) -- (.3em,.3em);
\draw[-,line width=.4pt]
(1.2em,-.2ex) -- (1.8em,-.2ex);
\end{tikzpicture}\hair{}
appearing in $\Pi$ label the horizontal line segment by $\widetilde \phi_{x_j x_i}$ for some $x_i$ and $x_j$ with $j > i$ and label the legs by $i$ and $j$, such that the labelling of the legs agrees with the order of the reductions and the labelling along the horizontal line segment respects the ordering, so that the parts of a labelled graph will look as follows
\begin{flalign*}
&&
\begin{tikzpicture}[baseline=.8em,x=.8em,y=.8em]
\draw[<->,line width=.4pt]
(0,.65em) -- (0,4ex) -- (3em,4ex) -- (3em,-.4ex);
\draw[-,line width=.4pt]
(-.6em,.6em) -- (.6em,.6em);
\draw[-,line width=.4pt]
(2.4em,-.55ex) -- (3.6em,-.55ex);
\draw[densely dotted,line width=.4pt]
(-.9em,.6em) -- (-.6em,.6em);
\draw[densely dotted,line width=.4pt]
(.9em,.6em) -- (.6em,.6em);
\draw[densely dotted,line width=.4pt]
(2.1em,-.55ex) -- (2.4em,-.55ex);
\draw[densely dotted,line width=.4pt]
(3.9em,-.55ex) -- (3.6em,-.55ex);
\node[font=\scriptsize] at (-.4em,2.8ex) {$j$};
\node[font=\scriptsize] at (3.4em,2ex) {$i$};
\node[font=\scriptsize] at (1.5em,5.2ex) {$\widetilde \phi_{x_j x_i}$};
\end{tikzpicture}
\quad\text{for $i < j$}
&&
\begin{tikzpicture}[baseline=3ex,x=.8em,y=.8em]
\draw[-,line width=.4pt]
(.3em,0) -- (4.7em,0);
\draw[->,line width=.4pt]
(.5em,4ex) -- (1em,4ex) -- (1em,.1em);
\draw[->,line width=.4pt]
(1.2em,6ex) -- (2em,6ex) -- (2em,.1em);
\draw[->,line width=.4pt]
(4.5em,2.5ex) -- (3em,2.5ex) -- (3em,.1em);
\draw[->,line width=.4pt]
(3em,8ex) -- (4em,8ex) -- (4em,.1em);
\foreach \x/\y in {{(0,0)}/{(.3em,0)}, {(5em,0)}/{(4.7em,0)}, {(.2em,4ex)}/{(.5em,4ex)}, {(.9em,6ex)}/{(1.2em,6ex)}, {(4.8em,2.5ex)}/{(4.5em,2.5ex)}, {(2.7em,8ex)}/{(3em,8ex)}}
{
\draw[densely dotted,line width=.4pt]
\x -- \y;
}
\node[font=\scriptsize] at (.7em,2.2ex) {$i$};
\node[font=\scriptsize] at (1.7em,3.1ex) {$j$};
\node[font=\scriptsize] at (2.7em,1.8ex) {$k$};
\node[font=\scriptsize] at (3.7em,5.4ex) {$l$};
\end{tikzpicture}
\quad\text{for $i \leq j \leq k \leq l$.}
&&
\end{flalign*}

A leg labelled $i$ shall correspond to the operator $\partial_i = \frac{\partial}{\partial x_i}$ acting on the monomial represented by the line segment. Once the leg is labelled by $i$, say, the corresponding operator $\partial_i$ can only act on $x_i$, but $x_i$ may appear multiple times. For example, if $x_i$ appears $k$ times and $l$ of the incoming legs are labelled by $i$, then there are $k \choose l$ ways of choosing which instances of $x_i$ should be replaced in the reductions. But
\[
\frac{\partial^l x_i^k}{\partial x_i^l} = k (k-1) \cdots (k-l+1) x_i^{k-l} = \frac{k!}{(k-l)!} x^{k-l} = l! {k \choose l} x_i^{k-l}
\]
so performing reductions can be viewed as taking {\it divided partial derivatives}, i.e.\ we may define the product of $\partial_i$ and $\partial_j$ by
\[
\partial_i \partial_j = \partial_j \partial_i \quad\text{ if } i \neq j \qquad\text{and}\qquad \underbrace{\partial_i \cdots \partial_i}_k = \frac{1}{k!} \frac{\partial^k}{\partial x_i^k}
\]
and in general we write
\begin{flalign*}
&& \frac{\partial^{[k]} f} {\partial x_1^{k_1} \dotsb \partial x_d^{k_d}} &= \frac{1}{k_1!\dotsb k_d!} \frac{\partial^k f}{\partial x_1^{k_1} \dotsb \partial x_d^{k_d}}
&& \mathllap{k = k_1 + \dotsb + k_d.}
\end{flalign*}

For each labelling we obtain an explicit formula for the bidifferential operator corresponding to the labelling (see Example \ref{example:lambdapi}) and $C_\Pi$ is defined as the sum of these bidifferential operators for all such labellings.
\end{proof}

This completes the proof of Theorem \ref{theorem:loopless}.

\begin{remark}
Theorem \ref{theorem:loopless} gives a graphical formula for the combinatorial star product\index{star product!combinatorial}\index{combinatorial star product} in terms of bidifferential operators. When working over $\Bbbk = \mathbb R$, bidifferential operators are continuous with respect to the locally convex topology of $C^\infty (\mathbb R^d)$ defined by the seminorms $\lVert f \rVert_{K, n} = \sum_{\lvert k \rvert \leq n} \sup_{x \in K} \lvert \partial_k f (x) \rvert$, where $K \subset \mathbb R^d$ is a compact set, $n \in \mathbb N$ and $\partial_k = \frac{\partial^{k_1}}{\partial x_1^{k_1}} \dotsb \frac{\partial^{k_d}}{\partial x_d^{k_d}}$ for a multi-index $k = (k_1, \dotsc, k_d) \in \mathbb N^d$ with $\lvert k \rvert = k_1 + \dotsb + k_d$ its sum of components. Thus $\star = \star^{\mathrm C}_{\phi + \widetilde \phi}$ extends uniquely to the algebra $C^\infty (\mathbb R^d) \llrr{t}$ of all smooth functions. (See \cite{barmeierschmitt} for further applications in strict deformation quantization\index{deformation quantization!strict}.)
\end{remark}

\begin{example}
\label{example:constantpoisson}
\index{Poisson structure!constant}\index{Poisson structure}
Let $\eta_{ji} \in \Bbbk$ for $1 \leq i < j \leq d$ be any constants and define $\widetilde \phi \in \Hom (\Bbbk S, A) \otimes \mathfrak m$ by $\widetilde \varphi (x_j x_i) = \eta_{ji} \hbar$. Then the explicit formula for $\star = \star^{\rm C}_{\phi + \widetilde \phi}$ is given by
\begin{align*}
f \star g &= fg +  \sum_{k \geq 1}  \sum_{\substack{i_1 \leq \dotsb \leq i_k \\ j_1 \leq \dotsb \leq j_k}} \sum_{s \in \mathfrak S_k} \eta_{j_{s(1)}i_1} \dotsb \eta_{j_{s(k)}i_k} \frac{\partial^{[k]} f}{\partial x_{j_1}\dotsb\partial x_{j_k}} \frac{\partial^{[k]} g} {\partial x_{i_1} \dotsb \partial x_{i_k}} \hbar^k \\
&= \exp \Bigl(\hbar \sum_{i < j} \eta_{ji} \, \partial_j \otimes \partial_i \Bigr) (f \otimes g) \big\rvert_{\mathrm{Diag}}
\end{align*}
where we set $\eta_{ji} = 0$ if $i \geq j$ and the second identity may be verified by comparing the coefficients of $\hbar^k$ for each $k \geq 0$. 

In fact, the constants $\eta_{ji}$ define a constant Poisson structure
\[
\eta = \sum_{1\leq i < j \leq d} \eta_{ji} \frac{\partial}{\partial x_j} \wedge \frac{\partial}{\partial x_i}
\]
on $\mathbb A^d$. Moreover, using Theorem \ref{theorem:higher-brackets} one easily checks that $\widetilde \varphi$ is a Maurer--Cartan element of $\mathbf p (Q, R) \hatotimes \mathfrak m$, and $\star$ gives a deformation quantization\index{deformation quantization} of $\eta$; see Proposition \ref{proposition:quantizations} below.

Note that the standard quantization\index{deformation quantization} of $\eta$ is given by the Moyal product $\ast$
\[
f \ast g = \exp \Bigl( \frac{\hbar}2 \sum\limits_{i < j} \eta_{ji} (\partial_j \otimes \partial_i - \partial_i \otimes \partial_j) \Bigr) (f \otimes g) \big\rvert_{\mathrm{Diag}}
\]
and an explicit gauge transformation between $\star$ and $\ast$ can be given by the formula $\Phi (f) \ast \Phi (g) = \Phi (f \star g)$ where $\Phi (f) = \exp \bigl( \frac{\hbar}2 \sum_{i<j} \eta_{ji} \frac{\partial^2}{\partial x_i \partial x_j} \bigr) (f)$.
\end{example}

\section{Deformation quantization and Kontsevich's Formality Theorem}
\label{subsection:relationdeformationquantization}
\index{Kontsevich!formality}

We give a brief overview of deformation quantization\index{deformation quantization|textbf} and the results of Kontsevich's seminal paper \cite{kontsevich1} (also see \cite{kontsevich1,keller2,cattaneo} for details). We then directly relate the graphical description of the combinatorial star product given in Theorem \ref{theorem:loopless} to Kontsevich's construction.

\begin{definition}
Let $(X, \{ \blank {,} \blank \})$ be a Poisson manifold and let $A = C^\infty (X)$ be the algebra of smooth functions. A {\it star product} $\star$ is an $\mathbb R \llrr{\hbar}$-bilinear associative product on $A \llrr{\hbar}$ of the form
\[
f \star g = fg + \sum_{n \geq 1} B_n (f, g) \hbar^n
\]
such that
\begin{enumerate}
\item each $B_n$ is a bidifferential operator
\item $1 \star f = f = f \star 1$ for each $f \in C^\infty (X)$.
\end{enumerate}

We say that $\star$ is a {\it quantization} of $\{ \blank {,} \blank \}$ (i.e.\ a deformation quantization of $(X, \{ \blank {,} \blank \})$)\index{deformation quantization} if
\[
\frac{f \star g - g \star f}\hbar \Big\rvert_{\hbar = 0} = B_1 (f, g) - B_1 (g, f) = \{ f, g \}
\]
i.e.\ if the skew-symmetrization of the first-order term $B_1$ of $\star$ coincides with the Poisson bracket.
\end{definition}

Deformation quantization\index{deformation quantization} was proposed in Bayen--Flato--Frønsdal--Lichnerowicz--Sternheimer \cite{bayen} as a way to quantize a classical mechanical system, given by a Poisson manifold, by viewing quantum mechanics as a kind of deformation of classical mechanics. The existence of a deformation quantization\index{deformation quantization} was first shown for symplectic manifolds (i.e.\ the non-degenerate case) by De~Wilde--Lecomte \cite{dewildelecomte} with a different explicit construction given by Fedosov \cite{fedosov}. The deformation quantization\index{deformation quantization} of arbitrary Poisson structures was solved by Kontsevich \cite{kontsevich1}, giving the following general result.

\begin{theorem}[(Kontsevich \cite{kontsevich1})]
Any Poisson manifold $(X, \{ \blank {,} \blank \})$ admits a formal deformation quantization.\index{deformation quantization}
\end{theorem}

Kontsevich's proof relies on the construction of an L$_\infty$ quasi-iso\-mor\-phism\index{L$_\infty$!quasi-isomorphism}, the {\it Kontsevich formality quasi-iso\-mor\-phism}\index{Kontsevich!formality|textbf}, between the DG Lie algebra\index{DG Lie algebra}\index{algebra!DG Lie} $\mathfrak v (X)$ of polyvector fields on $X$ (with the Schouten--Nijenhuis bracket\index{Schouten--Nijenhuis bracket} and trivial differential) and the DG Lie algebra $\mathfrak d (X)$ of multi-differential operators (viewed as a DG Lie subalgebra of the Hochschild DG Lie algebra of $A = C^\infty (X)$). This quasi-isomorphism gives an equivalence of Maurer--Cartan elements up to gauge equivalence in these two DG Lie algebras (cf.~Definition \ref{dgla} and Theorem \ref{theorem:quasiisomorphism}), which for $\mathfrak v (X)$ are precisely (formal) Poisson structures and for $\mathfrak d (X)$ star products.

This formality quasi-isomorphism was constructed using an explicit universal formula for any Poisson structure on $\mathbb R^d$ and the statement for arbitrary Poisson manifolds follows from a globalization argument (see \cite[\S7]{kontsevich1} and \cite{cattaneofeldertomassini}). In particular, to each Poisson structure on $\mathbb R^d$, viewed as a Maurer--Cartan element of $\mathfrak v (\mathbb R^d)$, one can thus associate an explicit star product.

Kontsevich's explicit formula for Poisson structures on $\mathbb R^d$ takes the following form\index{Kontsevich!star product|textbf}\index{star product!Kontsevich|textbf}
\begin{equation}
\label{kontsevichstar}
f \ast^{\mathrm K} g = f g + \sum_{k \geq 1} \frac{\hbar^k}{k!} \sum_{\Gamma \in \mathfrak G_{k,2}} w_\Gamma B_\Gamma (f, g)
\end{equation}
and for constant Poisson structures $\ast^{\mathrm K}$ coincides with the Moyal product $\ast$ (see Example \ref{example:constantpoisson}).
The graphs $\Gamma$ and their associated bidifferential operators $B_\Gamma$ are defined as follows.

\begin{definition}[{\cite[\S 6.1]{kontsevich1}}]
\label{definition:admissiblek}
\index{admissible graph!for the Kontsevich star product}\index{Kontsevich!graphs|textbf}
Let $\Gamma$ be a directed graph with $k + 2$ vertices, $k$ of which are said to be of ``first type'' and $2$ of ``second type'', and $2k$ arrows such that each vertex of first type has exactly two outgoing arrows. Then $\Gamma$ is said to be {\it admissible} if $\Gamma$ has no parallel arrows and no loops and the set of all admissible graphs with $k$ vertices of first type and $2$ vertices of second type is denoted by $\mathfrak G_{k,2}$.\footnote{One may also allow a larger class of admissible graphs in the definition, but they would not enter the formula of the Kontsevich star product, as their associated weights would be $0$.}

To each graph $\Gamma \in \mathfrak G_{k,2}$ one can associate a certain bidifferential operator $B_\Gamma$, which is built from $k$ copies of the Poisson bivector field $\eta = \sum_{1 \leq i < j \leq d} \eta_{ji} \frac{\partial}{\partial x_j} {\hair\wedge\hair} \frac{\partial}{\partial x_i}$, by decorating each vertex of first type by $\eta_{ji}$ for some $i, j$ and the two outgoing arrows by $\frac{\partial}{\partial x_i}$ and $\frac{\partial}{\partial x_j}$. Then $B_\Gamma (f, g)$ is the sum over all such labellings acting on $f, g$ which are placed at the two vertices of second type (see e.g.\ \cite[\S 2.2]{keller2}).
\end{definition}

The subtlety in the formula (\ref{kontsevichstar}) arguably lies in the choice of weights $w_\Gamma$ which make $\ast^{\mathrm K}$ an {\it associative} product. These {\it Kontsevich weights} are obtained by looking at geodesic embeddings of $\Gamma$ into the upper half plane (with the hyperbolic metric) and integrating a certain differential form (associated to the angles $\theta_e$ of the individual edges) over the compactification of a configuration space of points:
\[
w_\Gamma = \frac{1}{(2 \pi)^{2k}} \int_{\overline{\mathbb H}_k} \bigwedge_{e \in \Gamma} \mathrm d \theta_e.
\]
These weights are rather mysterious and some of them are conjecturally irrational (cf.\ Felder--Willwacher \cite{felderwillwacher}), but the integrals have recently been computed as integer-linear combinations of multiple zeta values in Banks--Panzer--Pym \cite{bankspanzerpym}.

\begin{remark}
Locally, the same formula can also be used to quantize holomorphic Poisson structures on $\mathbb C^d$. However, the globalization proceduce works rather differently as one should replace the algebra of smooth functions $C^\infty (X)$ by the structure sheaf $\mathcal O_X$. In particular, one should consider deformations of $\mathcal O_X$ as (twisted) presheaf, which naturally include ``commutative'' deformations of $\mathcal O_X$ corresponding to deformations of the complex structure of $X$. For more details see for example \cite{barmeierwang,kontsevich2,nesttsygan,palamodov1,palamodov2,vandenbergh,yekutieli2}.
\end{remark}

We now relate the set of graphs $\mathfrak G^{\mathrm C}_{k,2}$ for the combinatorial star product\index{star product!combinatorial}\index{combinatorial star product} appearing in Theorem \ref{theorem:loopless} to the graphs of Kontsevich's universal formula \eqref{kontsevichstar} (see Definition \ref{definition:admissiblek}).

\begin{proposition}
\label{proposition:graphs}
\index{Kontsevich!graphs}
There is a bijection between the set $\mathfrak G_{k,2}^{\mathrm C}$ and the set of pairs
\[
\bigl( \Lambda, \{ <_v \}_{v \in V (\Lambda)} \bigr)
\]
where $\Lambda \in \mathfrak G_{k,2}$ is a Kontsevich graph without oriented cycles (``wheels''), $V (\Lambda)$ denotes the set of vertices of the graph $\Lambda$, and $<_v$ is a total order on the set of incoming edges at the vertex $v \in V (\Lambda)$.
\end{proposition}

\begin{proof}
Given a graph $\Pi \in \mathfrak G_{k,2}^{\mathrm C}$ one may contract each horizontal line segment to a vertex $v$ and set $<_v$ as the ordering (from left to right) of the incoming edges along this line segment. Conversely, any pair $(\Lambda, \{ <_v \}_v)$ may appear in $\mathfrak G_{k,2}^{\mathrm C}$. This can be seen by drawing a Kontsevich graph $\Lambda \in \mathfrak G_{k,2}$ without cycles in a way that reflects the ordering $<_v$ at each vertex and then replacing each vertex by a horizontal line segment. One now observes that there is a unique way to order the line segments horizontally such that the lower line segments correspond to the reductions performed earlier.
\end{proof}

\begin{example}
\label{example:lambdapi}
Consider the following example of a graph $\Lambda \in \mathfrak G_{3,2}$ and consider two different orderings at the left vertex of second type. The corresponding graphs $\Pi = (\Lambda, \{ <_v \}) \in \mathfrak G_{3,2}^{\mathrm C}$ are then given as follows
\begin{equation*}
\begin{tikzpicture}[baseline=0,x=4em,y=4em]
\node at (.5,1.5) {$(\Lambda, \{ <_v \})$};
\node at (11.5em,1.5) {$\Pi$};
\draw[line width=.4pt, fill=black]         (27:1) circle(0.25ex);
\node[shape=circle, scale=0.8](right)  at  (27:1) {};
\draw[line width=.4pt, fill=black]         (54:1) circle(0.25ex);
\node[shape=circle, scale=0.8](middle) at  (54:1) {};
\draw[line width=.4pt, fill=black]        (81:1) circle(0.25ex);
\node[shape=circle, scale=0.8](left)   at (81:1) {};
\draw[line width=.4pt, fill=white]        (0,0) circle(0.3ex);
\node[shape=circle, scale=0.9](f)      at (0,0) {};
\draw[line width=.4pt, fill=white]        (1,0) circle(0.3ex);
\node[shape=circle, scale=0.8](g)      at (1,0) {};
\path[-stealth, line width=.4pt, line cap=round] (left) edge (f.90);
\path[-stealth, line width=.4pt, line cap=round] (left) edge (middle);
\path[-stealth, line width=.4pt, line cap=round] (middle) edge (f);
\path[-stealth, line width=.4pt, line cap=round] (middle) edge (right);
\path[-stealth, line width=.4pt, line cap=round] (right) edge (f.18);
\path[-stealth, line width=.4pt, line cap=round] (right) edge (g);
\begin{scope}[xshift=9em,x=1em,y=1.3em]
\draw[line width=.4pt]
(-.5,0) -- (2.5,0);
\draw[line width=.4pt]
(3.5,0) -- (5.5,0);
\draw[<->,line width=.4pt]
(0,.07) -- (0,3) -- (2.13,3) -- (2.13,2.07);
\draw[<->,line width=.4pt]
(1,.07) -- (1,2) -- (3.25,2) -- (3.25,1.07);
\draw[<->,line width=.4pt]
(2,.07) -- (2,1) -- (4.5,1) -- (4.5,.07);
\end{scope}
\begin{scope}[yshift=-6em]
\draw[line width=.4pt, fill=black]         (27:1) circle(0.25ex);
\node[shape=circle, scale=0.8](right)  at  (27:1) {};
\draw[line width=.4pt, fill=black]         (54:1) circle(0.25ex);
\node[shape=circle, scale=0.8](middle) at  (54:1) {};
\draw[line width=.4pt, fill=black]        (81:1) circle(0.25ex);
\node[shape=circle, scale=0.8](left)   at (81:1) {};
\draw[line width=.4pt, fill=white]        (0,0) circle(0.3ex);
\node[shape=circle, scale=0.9](f)      at (0,0) {};
\draw[line width=.4pt, fill=white]        (1,0) circle(0.3ex);
\node[shape=circle, scale=0.8](g)      at (1,0) {};
\path[-stealth, line width=.4pt, line cap=round, in=54, out=261] (left) edge (f.54);
\path[-stealth, line width=.4pt, line cap=round] (left) edge (middle);
\path[-stealth, line width=.4pt, line cap=round, in=81, out=234] (middle) edge (f.90);
\path[-stealth, line width=.4pt, line cap=round] (middle) edge (right);
\path[-stealth, line width=.4pt, line cap=round] (right) edge (f.18);
\path[-stealth, line width=.4pt, line cap=round] (right) edge (g);
\begin{scope}[xshift=9em,x=1em,y=1.3em]
\draw[line width=.4pt]
(-.5,0) -- (2.5,0);
\draw[line width=.4pt]
(3.5,0) -- (5.5,0);
\draw[<->,line width=.4pt]
(0,.07) -- (0,2) -- (3.25,2) -- (3.25,1.07);
\draw[<->,line width=.4pt]
(1,.07) -- (1,3) -- (2.25,3) -- (2.25,2.07);
\draw[<->,line width=.4pt]
(2,.07) -- (2,1) -- (4.5,1) -- (4.5,.07);
\end{scope}
\end{scope}
\end{tikzpicture}
\end{equation*}
where the ordering $<_v$ is as illustrated.

For $1 \leq i, j, k, l, m, n \leq d$ with $i < j$, $k < l$, $m < n$ and $j \leq l \leq n$ each labelling of the first graph corresponds to a bidifferential operators as follows:
\begin{equation*}
\begin{tikzpicture}[baseline=0,x=1em,y=1.3em]
\draw[line width=.4pt]
(-.5,0) -- (2.5,0);
\draw[line width=.4pt]
(3.5,0) -- (5.5,0);
\draw[<->,line width=.4pt]
(0,.07) -- (0,3) -- (2.13,3) -- (2.13,2.07);
\draw[<->,line width=.4pt]
(1,.07) -- (1,2) -- (3.25,2) -- (3.25,1.07);
\draw[<->,line width=.4pt]
(2,.07) -- (2,1) -- (4.5,1) -- (4.5,.07);
\draw[<->,line width=.4pt] (6,1.5) -- (8,1.5);
\node[font=\scriptsize] at (-.3,1.05) {$j$};
\node[font=\scriptsize] at (3.65,1.55) {$k$};
\node[font=\scriptsize] at (.7,1.05) {\strut$l$};
\node[font=\scriptsize] at (2.5,2.6) {$i$};
\node[font=\scriptsize] at (2.4,.55) {$n$};
\node[font=\scriptsize] at (5,.55) {$m$};
\node at (18,1.5) {$\displaystyle\widetilde \phi_{x_j x_i} \frac{\partial \widetilde \phi_{x_l x_k}}{\partial x_i} \frac{\partial \widetilde \phi_{x_n x_m}}{\partial x_k} \frac{\partial^{[3]}}{\partial x_j \partial x_l \partial x_n} \otimes \frac{\partial}{\partial x_m}$.};
\end{tikzpicture}
\end{equation*}
Similarly, for the second graph, where now $j \leq k \leq n$
\begin{equation*}
\begin{tikzpicture}[baseline=0,x=1em,y=1.3em]
\draw[line width=.4pt]
(-.5,0) -- (2.5,0);
\draw[line width=.4pt]
(3.5,0) -- (5.5,0);
\draw[<->,line width=.4pt]
(0,.07) -- (0,2) -- (3.25,2) -- (3.25,1.07);
\draw[<->,line width=.4pt]
(1,.07) -- (1,3) -- (2.25,3) -- (2.25,2.07);
\draw[<->,line width=.4pt]
(2,.07) -- (2,1) -- (4.5,1) -- (4.5,.07);
\draw[<->,line width=.4pt] (6,1.5) -- (8,1.5);
\node[font=\scriptsize] at (-.3,1.05) {$j$};
\node[font=\scriptsize] at (3.5,1.65) {$i$};
\node[font=\scriptsize] at (.7,1.05) {\strut$k$};
\node[font=\scriptsize] at (2.55,2.55) {$l$};
\node[font=\scriptsize] at (2.4,.55) {$n$};
\node[font=\scriptsize] at (5,.55) {$m$};
\node at (18,1.5) {$\displaystyle\widetilde \phi_{x_l x_k} \frac{\partial \widetilde \phi_{x_n x_m}}{\partial x_i} \frac{\partial \widetilde \phi_{x_j x_i}}{\partial x_l} \frac{\partial^{[3]}}{\partial x_j \partial x_k \partial x_n} \otimes \frac{\partial}{\partial x_m}$.};
\end{tikzpicture}
\end{equation*}
For both these examples, the bidifferential operator $C_\Pi$ would be the sum over all such labellings arising from reductions.
\end{example}

\section{Combinatorial deformation quantization}
\label{subsection:combinatorialquantization}

In this section we will show when the combinatorial star product\index{star product!combinatorial}\index{combinatorial star product} constructed in \S\ref{subsection:combinatorialstarproduct} can be used to produce explicit formulae for deformation quantizations\index{deformation quantization} of algebraic Poisson structures on $\mathbb A^d$. The idea of using the Diamond Lemma\index{Diamond Lemma} to produce explicit formal deformation quantizations\index{deformation quantization} was used for algebraic Poisson structures on $\mathbb R^3$ in \cite{arnlindbordemannhoferhoppeshimada,doninmakarlimanov,nowak} and the results in this section may be viewed as a general account of this approach.

For this subsection we revert to the classical language of (unshifted) DG Lie algebras and L$_\infty$ algebras\index{L$_\infty$!algebra}\index{algebra!L$_\infty$}, denoting by $\mathbf p (Q, R) [-1]$ the L$_\infty$ algebra obtained from the L$_\infty [1]$ algebra $\mathbf p (Q, R)$ by a negative shift (cf.\ Remark \ref{ordinarydgalgebra}).

Recall from \S\ref{subsection:relationdeformationquantization} that $\mathfrak v (\mathbb A^d)$ denotes the polyvector fields with its DG Lie algebra\index{DG Lie algebra}\index{algebra!DG Lie} structure given by the Schouten--Nijenhuis bracket. Relying on Kontsevich's formality\index{Kontsevich!formality} quasi-isomorphism we have the following theoretical result.

\begin{theorem}
\label{theorem:theoretical}
There is an L$_\infty$ quasi-isomorphism\index{L$_\infty$!quasi-isomorphism}\index{algebra!L$_\infty$} $\mathfrak v (\mathbb A^d) \simeq \mathbf p (Q, R) [-1]$. In particular, any Poisson structure\index{Poisson structure} can be quantized using the combinatorial star product.\index{star product!combinatorial}\index{combinatorial star product}
\end{theorem}

\begin{proof}
Combining Kontsevich's formality quasi-isomorphism $\mathfrak v (\mathbb A^d) \tosim \mathfrak d (\mathbb A^d)$, the quasi-isomorphism given by the inclusion $\mathfrak d (\mathbb A^d) \subset \mathbf h (A) [-1]$, the quasi-isomorphism $\mathbf p (Q, R) \tosim \mathbf h (A)$ (Theorem \ref{theorem:linfinitytransfer}) and the invertibility of L$_\infty$ quasi-iso\-mor\-phisms (cf.\ Lemma \ref{lemma:quasi-inverse}) we obtain the following diagram of DG Lie and L$_\infty$ algebras with L$_\infty$ quasi-iso\-mor\-phisms
\begin{equation}
\label{quasi}
\begin{tikzpicture}[baseline=0,description/.style={fill=white,inner sep=1.75pt}]
\matrix (m) [matrix of math nodes, row sep=2em, text height=1.5ex, column sep=3em, text depth=0.25ex, ampersand replacement=\&]
{
\mathfrak v (\mathbb A^d)   \& \mathfrak d (\mathbb A^d) \\
\mathbf p (Q, R) [-1]       \& \mathbf h (A) [-1] \\
};
\path[->,line width=.4pt]
(m-1-1) edge (m-1-2)
(m-1-2) edge (m-2-2)
(m-2-2) edge (m-2-1)
;
\end{tikzpicture}
\end{equation}
By Corollary \ref{theorem-summary} any star product on $\mathbb A^d$ (quantizing some Poisson structure $\eta$) is gauge equivalent to the combinatorial star product\index{star product!combinatorial}\index{combinatorial star product} for some Maurer--Cartan element $\widetilde \varphi \in \Hom (\Bbbk S, A) \hatotimes \mathfrak m$.
\end{proof}

Although this theoretical result depends on Kontsevich's formality\index{Kontsevich!formality} quasi-iso\-mor\-phism, it recasts the problem of quantizing Poisson structures\index{Poisson structure} in a combinatorial language, as quantizations can always be obtained as Maurer--Cartan elements of $\mathbf p (Q, R) \hatotimes \mathfrak m$. Indeed, Proposition \ref{proposition:quantizations} below shows that the first-order term of a Maurer--Cartan element is a Poisson structure\index{Poisson structure} and Theorem \ref{theorem:loopless} shows that the associated combinatorial star product\index{star product!combinatorial}\index{combinatorial star product} gives an explicit formula for the quantization\index{deformation quantization} which does not use any conjecturally transcendental coefficients (cf.\ Remarks \ref{remarks:quantization}).

Let us write $\widetilde \phi = \widetilde \varphi_1 \hbar + \widetilde \varphi_2 \hbar^2 + \dotsb \in \Hom (\Bbbk S, A) \hatotimes \mathfrak m$ as usual. To $\widetilde \phi$ we can associate a bivector field
\begin{equation}
\label{eta}
\eta = \sum_{1 \leq i < j \leq d} \eta_{ji} \frac{\partial}{\partial x_j} {\hair\wedge\hair} \frac{\partial}{\partial x_i}
\end{equation}
by setting $\eta_{ji} = \widetilde \phi_1 (x_j x_i) \in A$. 

\begin{lemma}
\label{lemma:koszulquantization}
\begin{enumerate}
\item \label{kq1} The L$_{\infty}$ algebra $\mathbf p (Q, R) [-1]$ has trivial differential and the underlying graded vector space is isomorphic to the polyvector fields on $\mathbb A^d$.
\item \label{kq2} The binary bracket $[\blank {,} \blank]$ of $\mathbf p (Q, R) [-1]$ coincides with the Schouten--Nijenhuis\index{Schouten--Nijenhuis bracket} graded Lie bracket on polyvector fields.
\end{enumerate}
\end{lemma}

\begin{proof} 
\ref{kq1} was shown in Example \ref{example:polynomialkoszul}. To prove \ref{kq2}, note that
\[
P^k \simeq \bigoplus_{1 \leq i_1 < \dotsb < i_k \leq d} A x_{i_k} x_{i_{k-1}} \cdots x_{i_1}.
\]
The maps $F^\hdot \colon \Hom_{\Bbbk} (A^{\otimes_{\Bbbk} \hdot}, A) \to P^\hdot$ and $G^\hdot \colon P^\hdot \to \Hom_{\Bbbk} (A^{\otimes_{\Bbbk} \hdot}, A)$ are given as follows. For any $\psi \in \Hom_{\Bbbk} (A^{\otimes_{\Bbbk} k}, A)$, we have 
\begin{align*}
F^k (\psi) =\sum_{i_1 < \dotsb < i_k} \sum_{s \in \mathfrak S_k} \mathrm{sgn}(s) \psi(x_{i_{s(k)}} \otimes x_{i_{s(k-1)}} \otimes \dotsb \otimes x_{i_{s(1)}}) x_{i_k} x_{i_{k-1}} \cdots x_{i_1}.
\end{align*}

To give a formula for $G^k$ note that any element in $P^k$ is a linear combination of elements of the form $\alpha = a x_{i_k} x_{i_{k-1}} \cdots x_{i_1}$ for some $1 \leq i_1 < \dotsb < i_k \leq d$ and some $a \in A$ and thus $G^k$ is determined by its value on such elements for which we have
\begin{flalign*}
&& G^k (\alpha) (a_k \otimes a_{k-1} \otimes \dotsb \otimes a_1) &= a \frac{\partial a_k}{\partial x_{i_k}} \frac{\partial a_{k-1}}{\partial x_{i_{k-1}}} \dotsb \frac{\partial a_1}{\partial x_{i_1}} && \mathllap{a_1, \dotsc, a_k \in A.}
\end{flalign*}

Now let $\beta = b x_{j_l} x_{j_{l-1}} \cdots x_{j_1} \in P^l$. Then the Lie bracket $[\blank {,} \blank]$ on $P^{\hdot + 1}$ is determined by the following formula:
\begin{align*}
[\alpha, \beta] &= F^{k+l-1} ([G^k (\alpha), G^l (\beta)]_{\mathrm G}) \\
& = \sum_{r=1}^k (-1)^{r-1} a \frac{\partial b}{\partial x_{i_{k-r+1}}} x_{j_l} \cdots x_{j_1} x_{i_k} \cdots \widehat x_{i_{k-r+1}} \cdots x_{i_1} \\
&\quad - \sum_{r=1}^l (-1)^{l-r} b \frac{\partial a}{\partial x_{j_{l-r+1}}} x_{j_l} \cdots \widehat x_{j_{l-r+1}} \cdots x_{j_1} x_{i_k} \cdots x_{i_1}
\end{align*} 
where $[\blank {,} \blank]_{\mathrm G}$ is the Gerstenhaber bracket given in Definition \ref{definitiongerstenhaber}. Note that the right-hand side of the second identity is exactly the Schouten--Nijenhuis bracket $[\alpha, \beta]_{\mathrm{SN}}$ of the multivector fields corresponding to $\alpha$ and $\beta$. 
\end{proof}

By Theorem \ref{theorem:equivalenceformal}, $\star^{\rm C}_{\phi+\widetilde \varphi}$ is associative whenever $\widetilde \varphi$ is a Maurer--Cartan element of $\mathbf p (Q, R) \hatotimes \mathfrak m$. The following proposition shows that the first-order term of a Maurer--Cartan element is always a Poisson structure\index{Poisson structure}. As the Maurer--Cartan equation \index{Maurer--Cartan!equation} can be checked combinatorially, this gives a criterion for which Poisson structures\index{Poisson structure} may be quantized via the combinatorial star product\index{star product!combinatorial}\index{combinatorial star product} {\it without} any higher correction terms $\widetilde \varphi_2 \hbar^2 + \widetilde \varphi_3 \hbar^3 + \dotsb$.

\begin{proposition}
\label{proposition:quantizations}
\begin{enumerate}
\item If $\widetilde \varphi = \sum_{i \geq 1} \widetilde \varphi_i \hbar^i \in \mathbf p (Q, R) \hatotimes \mathfrak m$ is a Maurer--Cartan element, then the associated bivector field $\eta$ \eqref{eta} is a Poisson structure on $\mathbb A^d$ and the combinatorial star product\index{star product!combinatorial}\index{combinatorial star product} $\star = \star^{\rm C}_{\phi + \widetilde \phi}$ gives an explicit deformation quantization of $\eta$.\index{deformation quantization}

\item \label{quant2} Let $\eta$ be any algebraic Poisson structure\index{Poisson structure} on $\mathbb A^d$ for $d \geq 3$ and let $\widetilde \varphi_1 \in \Hom (\Bbbk S, A)$ be the corresponding element. Let $\widetilde \varphi = \widetilde \varphi_1 \hbar$ and denote by $\star = \star^{\rm C}_{\phi + \widetilde \phi}$.

If $x_k \star (x_j \star x_i) = (x_k \star x_j) \star x_i$ for all $1 \leq i < j < k \leq d$, then $\star$ gives an explicit deformation quantization of $\eta$.\index{deformation quantization}
\end{enumerate}
\end{proposition}

\begin{proof}
This follows directly from Theorem \ref{theorem:higher-brackets}, since by Lemma \ref{lemma:koszulquantization} the binary bracket coincides with the Schouten--Nijenhuis bracket.
\end{proof}

From a combinatorial point of view, the problem of quantizing Poisson structures\index{Poisson structure} can thus be solved by finding Maurer--Cartan elements of $\mathbf p (Q, R) \hatotimes \mathfrak m$.

\begin{remark}
The following classes of Poisson structures satisfy the condition of Proposition \ref{proposition:quantizations} \ref{quant2}: constant\index{Poisson structure!constant} and linear\index{Poisson structure!linear} Poisson structures on $\mathbb A^d$ for any $d$, algebraic Poisson structures of arbitrary degree on $\mathbb A^2$, almost all quadratic Poisson structures on $\mathbb A^3$ (up to Poisson automorphism), which was already observed by \cite{doninmakarlimanov}.

Moreover, it is not difficult to come up with particular examples of arbitrarily high degree on $\mathbb A^d$ for any $d$. Also, the condition of Proposition \ref{proposition:quantizations} \ref{quant2} is easy to check for any {\it given} Poisson structure\index{Poisson structure} on $\mathbb A^d$.
\end{remark}

We now give two concrete examples for quantizing nonlinear Poisson structures in dimension $3$. (In \cite{barmeierschmitt} these combinatorial star products\index{star product!combinatorial}\index{combinatorial star product} are used to obtain {\it strict} deformation quantizations\index{deformation quantization!strict}, i.e.\ formal quantizations which converge for constant values of $\hbar$ and can be extended from polynomial functions to larger function spaces, such as the space of analytic complex-valued functions on $\mathbb R^d$ with infinite radius of convergence.)

The following first example is a quadratic Poisson structure\index{Poisson structure} on $\mathbb A^3$ which is a Maurer--Cartan element of $\mathbf p(Q, R)$ and thus can be quantized using the combinatorial star product.\index{star product!combinatorial}\index{combinatorial star product}

\begin{example}
Consider the Poisson structure $\eta = -(x^2 + \lambda y z) \frac{\partial}{\partial z} {\hair\wedge\hair} \frac{\partial}{\partial y}$ on $\mathbb A^3$ with $\lambda \in \Bbbk$. Note that the associated element $\widetilde \varphi \in \mathbf p (Q, R) \hatotimes \mathfrak m$ where
\[
\widetilde \varphi(y x) = 0 = \widetilde \varphi(z x) \qquad \text{and}\qquad \widetilde \varphi(zy) = -(x^2+ \lambda yz) \hbar
\]
is a Maurer--Cartan element since $(z \star y) \star x = xyz - (x^3 + \lambda xyz) \hbar = z \star (y \star x)$, where $\star = \star^{\rm C}_{\phi + \widetilde \phi}$. It follows that $(A \llrr{\hbar}, \star)$ is a deformation quantization\index{deformation quantization} of $\eta$. Note that by Manchon--Masmoudi--Roux \cite[Prop.~III.2]{manchonmasmoudiroux} the Poisson structure\index{Poisson structure} $\eta$ is not the image of a classical $r$-matrix when $\lambda \neq 0$ and thus cannot be quantized using Drinfel'd twists.
\end{example}

The following example is an exact Poisson structure of degree $k \geq 2$, which is {\it not} a Maurer--Cartan element of $\mathbf p (Q, R) \hatotimes \mathfrak m$, but which may be corrected by adding higher-order terms. (We obtained the higher-order terms by considering graphs with cycles and leave a precise description of this process for future work.) After adding these higher-order terms, the combinatorial star product\index{star product!combinatorial}\index{combinatorial star product} gives again an explicit formula for the quantization\index{deformation quantization} of this Poisson structure.

\begin{example}
\label{example:higherterms}
Consider the exact Poisson structure\index{Poisson structure}\index{Poisson structure!exact} $\eta$ associated to the polynomial $-xyz - \frac{1}{k+1} x^{k+1}$ for $k \geq 2$, i.e.\ 
$\eta = (yz + x^k) \frac{\partial}{\partial z} {\hair\wedge\hair} \frac{\partial}{\partial y} - xz \frac{\partial}{\partial z} {\hair\wedge\hair} \frac{\partial}{\partial x} + xy \frac{\partial}{\partial y} {\hair\wedge\hair} \frac{\partial}{\partial x}$ on $\mathbb A^3$. Let $\widetilde \varphi_1 \in \Hom(\Bbbk S, A)$ be the associated element to $\eta$: 
\[ 
\widetilde \varphi_1 (yx) = xy, \quad \widetilde \varphi_1 (zx) = -xz, \quad \text{and} \quad \widetilde \varphi_1 (zy) = yz + x^k.
\]
The element $\widetilde \varphi_1 \hbar$ is {\it not} a Maurer--Cartan element of $\mathbf p (Q, R) \hatotimes \mathfrak m$. However, by adding higher-order correction terms to we obtain a Maurer--Cartan element $\widetilde \varphi = \widetilde \varphi_1 \hbar + \widetilde \varphi_2 \hbar^2 + \dotsb$ of $\mathbf p (Q, R) \hatotimes \mathfrak m$ given as
\begin{align*}
\widetilde \varphi (yx) & = xy \hbar \\  
\widetilde \varphi (zx) & = xz (-\hbar + \hbar^2 - \hbar^3 + \dotsb) = -xz \frac{\hbar}{1 + \hbar}  \\
\widetilde \varphi (zy) & = (yz + x^k) \hbar.
\end{align*}
Indeed one easily verifies that for $\star = \star^{\rm C}_{\phi + \widetilde \phi}$ we have $(z \star y) \star x = z \star (y \star x)$ and thus $(A \llrr{\hbar}, \star)$ is a deformation quantization\index{deformation quantization} of the Poisson structure $\eta$.
\end{example}

\begin{remarks}
\label{remarks:quantization}
\begin{enumerate}
\item {\it Rational deformation quantization\index{deformation quantization}.} Given a Maurer--Cartan element of $\mathbf p (Q, R) \hatotimes \mathfrak m$ the formula (\ref{combinatorialstar}) appears to be a rather natural way to define a star product since the expression does not need any weights and thus works over $\mathbb Z$.
\item {\it ``Loopless'' deformation quantization.} It is not yet clear how to define the higher correction terms $\widetilde \varphi_2 \hbar^2 + \dotsb$ in general, without using the Kontsevich formality\index{Kontsevich!formality} map (see Theorem \ref{theorem:theoretical}), when $\widetilde \varphi_1 \hbar$ does not satisfy the Maurer--Cartan equation \index{Maurer--Cartan!equation} of $\mathbf p (Q, R)$. However, in examples this can be done by considering graphs with cycles. Indeed, Dito \cite{dito2} and Willwacher \cite{willwacher} have shown that a universal formula without cycles (``wheels'') cannot exist. The ``smallest'' example of a Poisson structure\index{Poisson structure} which by itself does not satisfy the Maurer--Cartan equation \index{Maurer--Cartan!equation} is a quadratic Poisson structure on $\mathbb A^3$ (see Example \ref{example:higherterms}). However, one can find higher order terms to ``correct'' a Poisson structure to a Maurer--Cartan element by considering graphs with cycles. We hope that a general combinatorial formula can be found, which would give a rational universal quantization formula\index{deformation quantization}. 
\end{enumerate}
\end{remarks}

\subsection*{Acknowledgements}

We would like to thank Pieter Belmans, Martin Bordemann, William Crawley-Boevey, Lutz Hille, Martin Kalck, Bernhard Keller, Travis Schedler, Catharina Stroppel, Michel Van den Bergh, Matt Young and Guodong Zhou for many helpful comments and discussions during the writing of this book. We also would like to thank Murray Gerstenhaber for valuable comments and bringing the reference \cite{gerstenhabergiaquinto} to our attention, and we thank the algebra group at Beijing Normal University, in particular Wei Hu, Nengqun Li, Yuming Liu and Ziheng Liu, for carefully reading this book and making many valuable suggestions for improvement. We are also grateful to Philipp Schmitt for pointing out an error in the description of the graphical calculus in a previous version.

A large part of this work was carried out while both authors were supported by the Max Planck Institute for Mathematics in Bonn and later by the Hausdorff Research Institute for Mathematics in Bonn as part of the Junior Trimester Program ``New Trends in Representation Theory'', funded by the Deutsche Forschungsgemeinschaft (DFG, German Research Foundation) under Germany's Excellence Strategy -- EXC-2047/1 -- 390685813, and we would like to thank both institutes for the welcoming atmosphere and the excellent working environment.

\backmatter

\printindex


\begin{thebibliography}{BFFLS78}

\bibitem[AS16]{abouzaidsmith}
M.~Abouzaid, I.~Smith,
The symplectic arc algebra is formal,
Duke Math.\ J.\ 165 (2016) 985--1060.

\bibitem[Ani86]{anick}
D.J.~Anick, 
On the homology of associative algebras, 
Trans.\ Amer.\ Math.\ Soc.\ 296 (1986) 2823--2844.

\bibitem[ABHHS09]{arnlindbordemannhoferhoppeshimada}
J.~Arnlind, M.~Bordemann, L.~Hofer, J.~Hoppe, H.~Shimada,
Noncommutative Riemann surfaces by embeddings in $\mathbb R^3$,
Comm.\ Math.\ Phys.\ 288 (2009) 403--429.

\bibitem[BPP20]{bankspanzerpym}
P.~Banks, E.~Panzer, B.~Pym,
Multiple zeta values in deformation quantization,
Invent.\ Math.\ 222 (2020) 79--159.

\bibitem[Bar96]{bardzell1}
M.J.~Bardzell,
Resolutions and cohomology of finite dimensional algebras,
PhD thesis, Virginia Polytechnic Institute and State University (1996)

\bibitem[Bar97]{bardzell2}
M.J.~Bardzell,
The alternating syzygy behavior of monomial algebras,
J.\ Algebra 188 (1997) 69--89.

\bibitem[BS22]{barmeierschmitt}
S.~Barmeier, P.~Schmitt,
Strict quantization of polynomial Poisson structures,
Comm.\ Math.\ Phys.\ 398 (2023) 1085--1127.

\bibitem[BSW]{barmeierschrollwang}
S.~Barmeier, S.~Schroll, Z.~Wang,
A$_\infty$ deformations of graded gentle algebras and Fukaya categories of surfaces,
in preparation.

\bibitem[BW21]{barmeierwang}
S.~Barmeier, Z.~Wang,
Deformations of categories of coherent sheaves via quivers with relations,
forthcoming in Algebraic Geom.,
arXiv:2107.07490 (2021)

\bibitem[BW22]{barmeierwang2}
S.~Barmeier, Z.~Wang,
A$_\infty$ deformations of extended Khovanov arc algebras and Stroppel's Conjecture,
arXiv:2211.03354 (2022)

\bibitem[BFFLS78]{bayen}
F.~Bayen, M.~Flato, C.~Fronsdal, A.~Lichnerowicz, D.~Sternheimer,
Deformation theory and quantization I: Deformations of symplectic structures,
Ann.\ Physics 111 (1978) 61--110.

\bibitem[BG06]{bergerginzburg}
R.~Berger, V.~Ginzburg,
Higher symplectic reflection algebras and non-homogeneous $N$-Koszul property,
J.\ Algebra 304 (2006) 577--601.

\bibitem[Ber78]{bergman}
G.M.~Bergman,
The Diamond Lemma for ring theory,
Adv.\ Math.\ 29 (1978) 178--218.

\bibitem[BNK08]{boenakanowiesner} 
B.~Boe, D.~Nakano, E.~Wiesner,
Category $\mathcal O$ for the Virasoro algebra: cohomology and Koszulity,
Pacific J.\ Math.\ 234 (2008) 535--548.

\bibitem[BC14]{bokutchen}
L.A.~Bokut, Y.~Chen,
Gröbner--Shirshov bases and their calculation,
Bull.\ Math.\ Sci.\ 4 (2014) 325--395.

\bibitem[BK03]{bokutkolesnikov}
L.A.~Bokut', P.S.~Kolesnikov,
Gröbner--Shirshov bases: From their incipiency to the present,
J.~Math.~Sci.\ 116 (2003) 2894--2916.

\bibitem[BG96]{bravermangaitsgory}
A.~Braverman, D.~Gaitsgory,
Poincaré--Birkhoff--Witt theorem for quadratic algebras of Koszul type,
J.\ Algebra 181 (1996) 315--328.

\bibitem[BD16]{bremnerdotsenko}
M.R.~Bremner, V.~Dotsenko,
Algebraic operads: An algorithmic companion,
CRC Press, Boca Raton, FL (2016)

\bibitem[BS11]{brundanstroppel}
J.~Brundan, C.~Stroppel,
Highest weight categories arising from Khovanov's diagram algebra I: Cellularity,
Moscow Math.\ J.\ 11 (2011) 685--722.

\bibitem[BGV03]{bueso}
J.L.~Bueso, J.~Gómez-Torrecillas, A.~Verschoren,
Algorithmic methods in non-commutative algebra: Applications to quantum groups,
Kluwer Academic Publishers, Dordrecht (2003)

\bibitem[CW99]{cannasdasilvaweinstein}
A.~Cannas da Silva, A.~Weinstein, 
Geometric models for noncommutative algebra, 
Berkeley Mathematics Lectures Notes 10,
American Mathematical Society, Providence, RI (1999)

\bibitem[Can99]{canonaco}
A.~Canonaco,
L$_\infty$-algebras and quasi-isomorphisms,
in: Seminari di Geometria Algebrica 1998--1999, 67--86,
Scuola Normale Superiore, Pisa (1999)

\bibitem[CS07]{cassidyshelton}
T.~Cassidy, B.~Shelton,
PBW-deformation theory and regular central extensions,
J.\ reine angew.\ Math.\ 610 (2007) 1--12.

\bibitem[Cat05]{cattaneo}
A.S.~Cattaneo,
Formality and star products,
in: S.~Gutt, J.~Rawnsley, D.~Sternheimer (eds.), Poisson geometry, deformation quantisation and group representations, 79--144,
Cambridge University Press, Cambridge (2005)

\bibitem[CF00]{cattaneofelder}
A.S.~Cattaneo, G.~Felder,
A path integral approach to the Kontsevich Quantization Formula,
Comm.\ Math.\ Phys.\ 212 (2000) 591--611.

\bibitem[CFT02]{cattaneofeldertomassini}
A.S.~Cattaneo, G.~Felder, L.~Tomassini,
From local to global deformation quantization of Poisson manifolds,
Duke Math.\ J.\ 115 (2002) 329--352.

\bibitem[CS15]{chouhysolotar}
S.~Chouhy, A.~Solotar,
Projective resolutions of associative algebras and ambiguities,
J.\ Algebra 432 (2015) 22--61.

\bibitem[CL11]{chuanglazarev}
J.~Chuang, A.~Lazarev,
L-infinity maps and twistings,
Homology Homotopy Appl.\ 13 (2011) 175--195.

\bibitem[Del87]{deligne}
P.~Deligne,
Letter to J.J.~Millson, available at \\ \verb+https://publications.ias.edu/sites/default/files/millson.pdf+ (1987)

\bibitem[DL83]{dewildelecomte}
M.~De Wilde, P.B.A.~Lecomte, 
Existence of star-products and of formal deformations of the Poisson Lie algebra of arbitrary symplectic manifolds,
Lett.\ Math.\ Phys.\ 7 (1983) 487--496. 

\bibitem[Dit15]{dito2}
G.~Dito,
The necessity of wheels in universal quantization formulas,
Comm.\ Math.\ Phys.\ 338 (2015) 523--532.

\bibitem[DM98]{doninmakarlimanov}
J.~Donin, L.~Makar-Limanov,
Quantization of quadratic Poisson structures on a polynomial algebra of three variables,
J.\ Pure Appl.\ Algebra 129 (1998) 247--261.

\bibitem[DT20]{dotsenkotamaroff}
V.~Dotsenko, P.~Tamaroff,
Tangent complexes and the Diamond Lemma,
arXiv:2010.14792 (2020)

\bibitem[EPS98]{eisenbudpeevasturmfels}
D.~Eisenbud, I.~Peeva, B.~Sturmfels,
Non-commutative Gröbner bases for commutative algebras,
Proc.\ Amer.\ Math.\ Soc.\ 126 (1998) 687--691.

\bibitem[Fed94]{fedosov}
B.V.~Fedosov,
A simple geometrical construction of deformation quantization,
J.\ Differential Geom.\ 40 (1994) 213--238.

\bibitem[FW09]{felderwillwacher}
G.~Felder, T.~Willwacher,
On the (ir)rationality of Kontsevich weights,
Int.\ Math.\ Res.\ Not.\ 2010 (2009) 701--716.

\bibitem[FV06]{floystadvatne}
G.~Fløystad, J.E.~Vatne,
PBW-deformations of $N$-Koszul algebras,
J.\ Algebra 302 (2006) 116--155.

\bibitem[Gab74]{gabriel}
P.~Gabriel,
Finite representation type is open,
in: Proceedings of the International Conference on Representations of Algebras, 3--7 September 1974, Ottawa, Carleton Math.\ Lecture Notes 9, 132--155,
Carleton University, Ottawa (1974)

\bibitem[Ger64]{gerstenhaber}
M.~Gerstenhaber,
On the deformation of rings and algebras,
Ann.\ Math.\ 79 (1964) 59--103.

\bibitem[Ger66]{gerstenhaber2}
M.~Gerstenhaber,
On the deformation of rings and algebras II,
Ann.\ Math.\ 84 (1966) 1--19.

\bibitem[Ger68]{gerstenhaber3}
M.~Gerstenhaber,
On the deformation of rings and algebras III,
Ann.\ Math.\ 88 (1968) 1--34.

\bibitem[Ger74]{gerstenhaber4}
M.~Gerstenhaber,
On the deformation of rings and algebras IV,
Ann.\ Math.\ 99 (1974) 257--276.

\bibitem[GG15]{gerstenhabergiaquinto}
M.~Gerstenhaber, A.~Giaquinto,
Deformations associated with rigid algebras,
J.\ Homotopy Relat.\ Struct.\ 10 (2015) 437--458.

\bibitem[GS88]{gerstenhaberschack}
M.~Gerstenhaber, S.D.~Schack,
Algebraic cohomology and deformation theory,
in: Deformation theory of algebras and structures and applications, Il Ciocco, 1--14 June 1986, NATO Adv.\ Sci.\ Inst.\ Ser.\ C Math.\ Phys.\ Sci.\ 247, 11--264,
Kluwer Academic Publishers, Dordrecht (1988)

\bibitem[Get09]{getzler}
E.~Getzler,
Lie theory for nilpotent L$_{\infty}$-algebras,
Ann.\ Math.\ 170 (2009) 271--301.

\bibitem[Gre99]{green1}
E.L.~Green,
Noncommutative Gröbner bases, and projective resolutions,
in: P.~Dräxler, C.M.~Ringel, G.O.~Michler (eds.), Computational Methods for Representations of Groups and Algebras, 29--60,
Birkhäuser, Basel (1999)

\bibitem[Gre17]{green2}
E.L.~Green,
The geometry of strong Koszul algebras,
arXiv:1702.02918 (2017)

\bibitem[GHS21]{greenhilleschroll}
E.L.~Green, L.~Hille, S.~Schroll,
Algebras and varieties,
Algebr.\ Represent.\ Theory 24 (2021) 367--388.

\bibitem[GMMZ04]{greenmarcosmartinezvillazhang}
E.L.~Green, E.N.~Marcos, R.~Mart\'{\i}nez-Villa, P. Zhang, 
$D$-Koszul algebra, 
J.\ Pure Appl.\ Algebra 193 (2004) 141--162.

\bibitem[Gut83]{gutt}
S.~Gutt,
An explicit $*$-product on the cotangent bundle of a Lie group,
Lett.\ Math.\ Phys.\ 7 (1983) 249--258.

\bibitem[HLL20]{helili}
W.~He, S.~Li, Y.~Li, 
$G$-twisted braces and orbifold Landau--Ginzburg models, 
Comm.\ Math.\ Phys.\ 373 (2020) 175--217.

\bibitem[IIVZ15]{ivanovivanovvolkovzhou}
A.~Ivanov, S.O.~Ivanov, Y.~Volkov, G.~Zhou,
BV structure on Hochschild cohomology of the group ring of quaternion group of order eight in characteristic two,
J.\ Algebra 435 (2015) 174--203.

\bibitem[Kel03]{keller1}
B.~Keller,
Derived invariance of higher structures on the Hochschild complex, available at \\ \verb+https://webusers.imj-prg.fr/~bernhard.keller/publ/dih.pdf+ (2003)

\bibitem[Kel05]{keller2}
B.~Keller,
Deformation quantization after Kontsevich and Tamarkin,
in: A.~Cattaneo, B.~Keller, C.~Torossian, A.~Bruguières (eds.), Déformation, quantification, théorie de Lie, 19--62,
Société mathématique de France, Paris (2005)

\bibitem[Kho00]{khovanov}
M.~Khovanov,
A categorification of the Jones polynomial,
Duke Math.\ J.\ 101 (2000) 359--426.

\bibitem[Kon03]{kontsevich1}
M.~Kontsevich,
Deformation quantization of Poisson manifolds,
Lett.\ Math.\ Phys.\ 66 (2003) 157--216.

\bibitem[Kon01]{kontsevich2}
M.~Kontsevich,
Deformation quantization of algebraic varieties,
Lett.\ Math.\ Phys.\ 56 (2001) 271--294.

\bibitem[LS93]{ladastasheff}
T.~Lada, J.~Stasheff, 
Introduction to SH Lie algebras for physicists, 
Int.\ J.\ Theor.\ Phys. 32 (1993) 1087--1103.

\bibitem[Lod98]{loday}
J.-L.~Loday,
Cyclic homology, 2nd edition,
Grundlehren der mathematischen Wissenschaften 301,
Springer-Verlag, Berlin (1998)

\bibitem[LV12]{lodayvallette}
J.-L.~Loday, B.~Vallette,
Algebraic operads,
Grundlehren der mathematischen Wissenschaften 346,
Springer-Verlag, Berlin (2012)

\bibitem[LV05]{lowenvandenbergh1}
W.~Lowen, M.~Van den Bergh,
Hochschild cohomology of Abelian categories and ringed spaces,
Adv.\ Math.\ 198 (2005) 172--221.

\bibitem[LV06]{lowenvandenbergh2}
W.~Lowen, M.~Van den Bergh,
Deformation theory of Abelian categories,
Trans.\ Amer.\ Math.\ Soc.\ 358 (2006) 5441--5483.

\bibitem[Lur11]{lurie}
J.~Lurie,
Moduli problems for ring spectra,
in: R.~Bhatia, A.~Pal, G.~Rangarajan, V.~Srinivas, M.~Vanninathan (eds.), Proceedings of the International Congress of Mathematicians (Hyderabad 2010), Vol.\ II, 1099--1125,
Hindustan Book Agency, New Delhi (2010)

\bibitem[MMR02]{manchonmasmoudiroux}
D.~Manchon, M.~Masmoudi, A.~Roux,
On quantization of quadratic Poisson structures,
Comm.\ Math.\ Phys.\ 225 (2002) 121--130.

\bibitem[Man99]{manettidgla}
M.~Manetti,
Deformation theory via differential graded Lie algebras,
in: Seminari di Geometria Algebrica 1998--1999, 21--48,
Scuola Normale Superiore, Pisa (1999)

\bibitem[Man22]{manetti}
M.~Manetti,
Lie methods in deformation theory,
Springer Monographs in Mathematics,
Springer Nature, Singapore (2022)

\bibitem[Mar12]{markl}
M.~Markl,
Deformation theory of algebras and their diagrams,
Regional Conference Series in Mathematics 116,
American Mathematical Society, Providence, RI (2012)

\bibitem[MS11]{mazorchukstroppel}
V.~Mazorchuk, C.~Stroppel,
Cuspidal $\mathfrak{sl}_n$-modules and deformations of certain Brauer tree algebras,
Adv.\ Math.\ 228 (2011) 1008--1042.

\bibitem[Mor86]{fmora}
F.~Mora,
Groebner bases for non-commutative polynomial rings,
Lecture Notes in Comput.\ Sci.\ 229,
Springer-Verlag, Berlin (1986)

\bibitem[Mor94]{tmora}
T.~Mora,
An introduction to commutative and noncommutative Gröbner bases,
in: Second International Colloquium on Words, Languages and Combinatorics, 25--28 August 1992, Kyoto,
Theoret.\ Comput.\ Sci.\ 134 (1994) 131--173.

\bibitem[NT01]{nesttsygan}
R.~Nest, B.~Tsygan,
Deformations of symplectic Lie algebroids, deformations of holomorphic symplectic structures, and index theorems,
Asian J.\ Math.\ 5 (2001) 599--636.

\bibitem[Now97]{nowak}
C.~Nowak,
Star products for integrable Poisson structures on $\mathbb R^3$,
arXiv:q-alg/9708012 (1997)

\bibitem[Pal07a]{palamodov1}
V.P.~Palamodov,
Infinitesimal deformation quantization of complex analytic spaces,
Lett.\ Math.\ Phys.\ 79 (2007) 131--142.

\bibitem[Pal07b]{palamodov2}
V.P.~Palamodov,
Associative deformations of complex analytic spaces,
Lett.\ Math.\ Phys.\ 82 (2007) 191--217.

\bibitem[PP05]{polishchukpositselski}
A.~Polishchuk, L.~Positselski,
Quadratic algebras,
University Lecture Series 37,
American Mathematical Society, Providence, RI (2005)

\bibitem[Pri10]{pridham}
J.P.~Pridham,
Unifying deformation theories,
Adv.\ Math.\ 224 (2010) 772--826.

\bibitem[RR18]{redondoroman}
M.J.~Redondo, L.~Román,
Comparison morphisms between two projective resolutions of monomial algebras,
Rev.\ Un.\ Mat.\ Argentina 59 (2018) 1--31.

\bibitem[RR22]{redondorossibertone}
M.J.~Redondo, F.~Rossi Bertone,
L$_\infty$-structure on Bardzell's complex for monomial algebras,
arXiv:2008.08122, J.\ Pure Appl.\ Algebra 226 (2022) 106935.

\bibitem[Sch88]{schaps}
M.~Schaps,
Deformations of finite-dimensional algebras and their idempotents,
Trans.\ Amer.\ Math.\ Soc.\ 307 (1988) 843--856.

\bibitem[Sch16a]{schedler}
T.~Schedler,
Zeroth Hochschild homology of preprojective algebras over the integers,
Adv.\ Math.\ 299 (2016) 451--542.

\bibitem[Sch16b]{schedler2}
T.~Schedler,
Deformations of algebras in noncommutative geometry,
in: G.~Bellamy, D.~Rogalski, T.~Schedler, J.~T.~Stafford, M.~Wemyss (eds.),
Noncommutative algebraic geometry, 71--165,
Math.\ Sci.\ Res.\ Inst.\ Publ.\ 64,
Cambridge University Press, New York (2016)

\bibitem[Skö08]{skoldberg}
E.~Sköldberg,
A contracting homotopy for Bardzell's resolution,
Math.\ Proc.\ R.\ Ir.\ Acad.\ 108 (2008) 111--117.

\bibitem[Str09]{stroppel}
C.~Stroppel,
Parabolic category $\mathcal O$, perverse sheaves on Grassmannians, Springer fibres and Khovanov homology,
Compos.\ Math.\ 145 (2009) 954--992.

\bibitem[Str10]{stroppel2}
C.~Stroppel,
Schur--Weyl dualities and link homologies,
in: R.~Bhatia, A.~Pal, G.~Rangarajan, V.~Srinivas, M.~Vanninathan (eds.), Proceedings of the International Congress of Mathematicians (Hyderabad 2010), Vol.~III, 1344--1365,
Hindustan Book Agency, New Delhi (2010)

\bibitem[Van07]{vandenbergh}
M.~Van den Bergh,
On global deformation quantization in the algebraic case,
J.\ Algebra 315 (2007) 326--395.

\bibitem[Wil14]{willwacher}
T.~Willwacher,
The obstruction to the existence of a loopless star product,
C.\ R.\ Acad.\ Sci.\ Paris Ser.\ I 352 (2014) 881--883.

\bibitem[Yek15]{yekutieli2}
A.~Yekutieli,
Twisted deformation quantization of algebraic varieties,
Adv.\ Math.\ 268 (2015) 241--305.

\end{thebibliography}
\end{document}